\documentclass[a4paper, reqno, 10pt]{amsart}
\usepackage{amssymb}
\usepackage{extarrows}
\usepackage{bm}
\usepackage[centering]{geometry}
\usepackage{enumitem}
\usepackage{longtable}
\usepackage{multirow}
\usepackage{array}
\usepackage{etex}
\usepackage{pictex}
\usepackage{graphics}

\usepackage[all]{xy}

\newcommand{\tabincell}[2]{\begin{tabular}{@{}#1@{}}#2\end{tabular}}
\numberwithin{equation}{section}
\theoremstyle{definition}
\newtheorem{defn}{Definition}[section]
\newtheorem{rem}[defn]{Remark}

\newtheorem{exm}[defn]{Example}

\theoremstyle{plain}
\newtheorem{cor}[defn]{Corollary}
\newtheorem{thm}[defn]{Theorem}
\newtheorem{lem}[defn]{Lemma}
\newtheorem{exmrem}[defn]{Example-Remark}
\newtheorem{Notation}[defn]{Notation}
\newtheorem{prop}[defn]{Proposition}
\newtheorem{fact}[defn]{Fact}
\newtheorem{defn-thm}[defn]{Definition-Theorem}

\newtheorem*{Freyd-Mitchell}{Freyd-Mitchell embedding Theorem}
\def\Fib{\operatorname{Fib}}
\def\Cofib{\operatorname{Cofib}}
\def\Weq{\operatorname{Weq}}
\def\Coker{\operatorname{Coker}}

\def\Hom{\operatorname{Hom}}
\def\Ext{\operatorname{Ext}}
\def\Tor{\operatorname{Tor}}
\def\Ker{\operatorname{Ker}}
\def\Id{\operatorname{Id}}
\def\Im{\operatorname{Im}}

\begin{document}

\title[Cotorsion pairs and model structures on Morita rings] {Cotorsion pairs and model structures \\ on  Morita rings}
\author[Pu Zhang, Jian Cui, Shi Rong] {Pu Zhang$^*$,  Jian Cui, Shi Rong \\ \\  School of Mathematical Sciences \\
Shanghai Jiao Tong University,  \ Shanghai 200240, \ China }
\thanks{$^*$ Corresponding author}
\thanks{pzhang$\symbol{64}$sjtu.edu.cn \ \ \ \  provinceanying$\symbol{64}$sjtu.edu.cn \ \ \ \ rongshi$\symbol{64}$sjtu.edu.cn}
\thanks{\it 2020 Mathematics Subject Classification. Primary 18N40, 16D90, 16E30; Secondary 16E65, 16G50, 16G20}
\thanks{Supported by National Natural Science Foundation of China, Grant No. 12131015, 11971304; and Natural Science Foundation of Shanghai, Grant No. 23ZR1435100.}
\maketitle
\begin{abstract} We study cotorsion pairs and abelian model structures on Morita rings \ $\Lambda =\left(\begin{smallmatrix} A & {}_AN_B \\
	{}_BM_A & B\end{smallmatrix}\right)$ which are Artin algebras. Given cotorsion pairs $(\mathcal U, \mathcal X)$ and $(\mathcal V, \mathcal Y)$ in  $A$-Mod and $B$-Mod, respectively, one can construct four cotorsion pairs in $\Lambda$-Mod:
$$({}^\perp\left(\begin{smallmatrix}\mathcal X\\ \mathcal Y\end{smallmatrix}\right), \ \left(\begin{smallmatrix}\mathcal X\\ \mathcal Y\end{smallmatrix}\right)),
\ \ \ \ \ (\Delta(\mathcal U, \ \mathcal V), \ \Delta(\mathcal U, \ \mathcal V)^\perp), \ \ \ \ \ \
(\left(\begin{smallmatrix}\mathcal U\\ \mathcal V\end{smallmatrix}\right), \ \left(\begin{smallmatrix}\mathcal U\\ \mathcal V\end{smallmatrix}\right)^\perp),
\ \ \ \ \ \  (^{\perp}\nabla(\mathcal X, \ \mathcal Y), \ \nabla(\mathcal X, \ \mathcal Y)).$$
These cotorsion pairs have relations:
$$\Delta(\mathcal U, \ \mathcal V)^{\bot} \subseteq \left(\begin{smallmatrix}\mathcal X\\ \mathcal Y\end{smallmatrix}\right), \ \ \ \ \ \ \ \ ^{\bot}\nabla(\mathcal X, \ \mathcal Y) \subseteq \left(\begin{smallmatrix}\mathcal U\\ \mathcal V\end{smallmatrix}\right).$$
An important feature  is that they are not equal, in general. In fact, there even exists a Morita algebra $\Lambda$, such that the four cotorsion pairs are pairwise different.
The problem of identifications, i.e., when these inclusions are the same,  are studied;
the heredity and completeness of these cotorsion pairs are investigated; and finally, various model structures on $\Lambda$\mbox{-}{\rm Mod} are obtained,
by explicitly giving the corresponding Hovey triples and Quillen's homotopy categories.
In particular, cofibrantly generated Hovey triples,
and the Gillespie-Hovey triples induced by compatible generalized projective (respectively, injective) cotorsion pairs, are explicitly constructed.
All these Hovey triples obtained are pairwise different and ``new" in some sense.
Some results are even new  when $M = 0$ or $N = 0$.

\vskip5pt

Key words: Morita ring, cotorsion pair, model structure, Hovey triple, Quillen's homotopy category, Gorenstein-projective module, monomorphism category

\end{abstract}

\vskip10pt

\section{\bf Introduction}

This paper is to study cotorsion pairs and abelian model structures on some Morita rings.

\vskip5pt

Morita rings $\Lambda = \left(\begin{smallmatrix} A & {}_AN_B \\
{}_BM_A & B\end{smallmatrix}\right)$, originated from equivalences of module categories ([M]), and formulated in [Bas],
are also called the rings of Morita contexts, and the formal matrix rings. They are widely used in various aspects of mathematics; and for more information we refer to [C], [G], [MR], [KT] and [GrP].
Throughout this paper, we assume that the considered Morita rings are Artin algebras.

\vskip5pt

Model structures, introduced by D. Quillen [Q1, Q2],  provide common ideas and framework for many branches of mathematics.
A triple \ $(\mathcal C, \ \mathcal F, \ \mathcal W)$  of classes of objects of abelian category $\mathcal A$ is {\it a Hovey triple},
if $\mathcal W$ is thick and $(\mathcal C\cap \mathcal W, \ \mathcal F)$ and $(\mathcal C, \ \mathcal F\cap \mathcal W)$ are complete cotorsion pairs in $\mathcal A$;
and it is {\it hereditary}, if both the
cotorsion pairs are hereditary.
By M. Hovey [H2] (see also [BR]), abelian model structures on $\mathcal A$ and the Hovey triples in $\mathcal A$ are in one-to-one correspondence.

\vskip5pt

Of special interest are hereditary Hovey triples. In this case, $\mathcal C\cap \mathcal F$ is a Frobenius category,  $\mathcal C \cap \mathcal F\cap \mathcal W$ is
the class of projective-injective objects, and Quillen's homotopy category is exactly
the stable category $(\mathcal C\cap \mathcal F)/(\mathcal C\cap \mathcal F\cap \mathcal W)$. See  [BR],  [Bec], [Gil4].

\vskip5pt

J. Gillespie [Gil3] gives an approach to construct a hereditary Hovey triple \ $(^\perp\Upsilon, \ \Theta^\perp, \ \mathcal W)$,
from two compatible complete hereditary cotorsion pairs \ $(\Theta, \ \Theta^\perp)$ and \ $(^\perp\Upsilon, \ \Upsilon)$, where
$\mathcal W$ is given as in Theorem \ref{GHtriple}; and conversely, any hereditary Hovey triple in an abelian category $\mathcal A$ is obtained in this way.
This general construction   \ $(^\perp\Upsilon, \ \Theta^\perp, \ \mathcal W)$  of hereditary Hovey triples will be called {\it the Gillespie-Hovey triples}.
See Subsection 2.9 for details.

\vskip5pt

The module categories of Morita rings have been described ([G]), and
cotorsion pairs and abelian model structures on the special case of triangular matrix rings (i.e., $M = 0$) have been studied ([ZPD]). Nevertheless, a general study of cotorsion pairs and abelian model structures on some Morita rings meets difficulties and induces a lot of new phenomena, even under the assumption of
$M\otimes_AN = 0 = N\otimes_BM$.

\vskip5pt

From cotorsion pairs $(\mathcal U, \mathcal X)$ and $(\mathcal V, \mathcal Y)$,
respectively in  $A$-Mod and $B$-Mod, one can construct four kinds of cotorsion pairs in Morita rings.
Quite different from the case of $M = 0$ or $N = 0$, the four cotorsion pairs are pairwise different, in general.
The heredity,  the problem of identifications, the completeness, and the specializations,  of these cotorsion pairs are studied. It turns out that Morita rings are rich in producing cotorsion pairs. Even if one takes $(\mathcal U, \mathcal X)$ and $(\mathcal V, \mathcal Y)$ to be the projective or the injective cotorsion pair, what one gets in $\Lambda$\mbox{-}{\rm Mod} are pairwise generally different and ``new" cotorsion pairs.

\vskip5pt

Based on these,
various model structures on $\Lambda$\mbox{-}{\rm Mod} are obtained,
by explicitly giving the Hovey triples and Quillen's homotopy categories.
In particular, cofibrantly generated Hovey triples,
and the Gillespie-Hovey triples induced by compatible generalized projective (respectively, injective) cotorsion pairs, are explicitly constructed. All these Hovey triples obtained are pairwise different and ``new" in some sense.
Some results are new even for $M = 0$ or $N = 0$.

\vskip5pt

The paper is organized as follows.

\vskip5pt

1. \ Introduction

2. \ Preliminaries

3. \ (Hereditary) cotorsion pairs in Morita rings

4. \ Identifications

5. \ Completeness

6. \ Realizations

7. \ Abelian model structures on Morita rings

\vskip5pt

\subsection{(Hereditary) cotorsion pairs in Morita rings}  For a ring $R$, let $R$-Mod be the category of left $R$-modules.
For a class $\mathcal C$ of objects in abelian category $\mathcal A$ and $X\in \mathcal A$, by \ $\Ext^1_\mathcal A(X, \mathcal C)=0$ we mean
\ $\Ext^1_\mathcal A(X, C)=0$ for all $C\in \mathcal C$. Let $^\perp\mathcal C$ be the full subcategory of objects $X$ with \ $\Ext^1_\mathcal A(X, \mathcal C)=0$.
Similarly for $\mathcal C^\perp$.

\vskip5pt

Given a class \ $\mathcal X$ \ of \ $A$-modules and a class \ $\mathcal Y$ \ of \ $B$-modules, three classes
$$\left(\begin{smallmatrix}\mathcal X\\ \mathcal Y\end{smallmatrix}\right), \ \ \ \ \ \Delta(\mathcal X, \ \mathcal Y),  \ \ \ \ \ \nabla(\mathcal X, \ \mathcal Y)$$  of modules over Morita ring \ $\Lambda$ \ are defined. See Subsection 3.1.
In particular, one has  the monomorphism category \ ${\rm Mon}(\Lambda) = \Delta(A\mbox{\rm-Mod},  \ B\mbox{\rm-Mod})$, and the epimorphism category \
${\rm Epi}(\Lambda)= \nabla(A\mbox{\rm-Mod},  \ B\mbox{\rm-Mod})$.

\vskip5pt

By \ $\Tor^A_1(M, \ \mathcal U)$ $=0$ we mean \ $\Tor^A_1(M, \ U)=0$ for all $U\in\mathcal U$. Constructions of (hereditary) cotorsion pairs in $\Lambda$-Mod are given as follows.

\vskip5pt

\begin{thm}\label{mainin31} \ {\rm (Theorem \ref{ctp1})} \ Let \ $\Lambda = \left(\begin{smallmatrix} A & N \\ M & B\end{smallmatrix}\right)$  be a Morita ring  with $\phi = 0=\psi$,  \ $(\mathcal U, \ \mathcal X)$ and \ $(\mathcal V, \ \mathcal Y)$  cotorsion pairs in $A\mbox{-}{\rm Mod}$ and $B$\mbox{\rm-Mod}, respectively.

\vskip5pt

$(1)$ \ If \ $\Tor^A_1(M, \ \mathcal U)=0 = \Tor^B_1(N, \ \mathcal V)$,
then \ $({}^\perp\left(\begin{smallmatrix}\mathcal X\\ \mathcal Y\end{smallmatrix}\right), \ \left(\begin{smallmatrix}\mathcal X\\ \mathcal Y\end{smallmatrix}\right))$ is a cotorsion pair in $\Lambda${\rm-Mod}$;$ and it is hereditary if and only if  \ so are \ $(\mathcal U, \ \mathcal X)$ and  \ $(\mathcal V, \ \mathcal Y)$.

\vskip5pt

$(2)$ \ If \ $\Ext_A^1(N, \ \mathcal X) =0 = \Ext_B^1(M, \ \mathcal Y)$,
then \ $(\left(\begin{smallmatrix}\mathcal U\\ \mathcal V\end{smallmatrix}\right), \ \left(\begin{smallmatrix}\mathcal U\\ \mathcal V\end{smallmatrix}\right)^\perp)$ is a cotorsion pair in $\Lambda${\rm-Mod}$;$ and
it is hereditary if and only if so are
\ $(\mathcal U, \ \mathcal X)$  and  $(\mathcal V, \ \mathcal Y)$.
\end{thm}

\vskip5pt

\begin{thm}\label{mainin32} \ \ {\rm (Theorem \ref{ctp6})} \ Let $\Lambda = \left(\begin{smallmatrix} A & N \\	M & B\end{smallmatrix}\right)$ be a Morita ring  with $M\otimes_A N = 0 = N\otimes_BM$,  \ $(\mathcal U, \ \mathcal X)$ and $(\mathcal V, \ \mathcal Y)$  cotorsion pairs in $A\mbox{-}{\rm Mod}$ and  $B$\mbox{\rm-Mod}, respectively.  Then

\vskip5pt

$(1)$ \ \ $(\Delta(\mathcal U, \ \mathcal V), \ \Delta(\mathcal U, \ \mathcal V)^\perp)$ is a cotorsion pair in $\Lambda${\rm-Mod}$;$ and if \ $M_A$ and \ $N_B$ are flat, then it is hereditary if and only if so are \ $(\mathcal U, \ \mathcal X)$  and \ $(\mathcal V, \ \mathcal Y)$.

\vskip5pt

$(2)$ \ \ $(^{\perp}\nabla(\mathcal X, \ \mathcal Y), \ \nabla(\mathcal X, \ \mathcal Y))$ is a cotorsion pair in $\Lambda${\rm-Mod}$;$ and if \ $_BM$ and \ $_AN$ are projective, then
it is hereditary if and only if so are \ $(\mathcal U, \ \mathcal X)$  and  $(\mathcal V, \ \mathcal Y)$.
\end{thm}

We stress that, the condition \ ``$M\otimes_A N = 0 = N\otimes_BM$" in {\rm Theorem \ref{mainin32}}, can not be weakened as \ ``$\phi = 0 = \psi$" in general, as Example \ref{irem1} shows.

\vskip5pt

The four cotorsion pairs $$({}^\perp\left(\begin{smallmatrix}\mathcal X\\ \mathcal Y\end{smallmatrix}\right), \ \left(\begin{smallmatrix}\mathcal X\\ \mathcal Y\end{smallmatrix}\right)),
\ \ \ \ \ \ \ \ (\Delta(\mathcal U, \ \mathcal V), \ \Delta(\mathcal U, \ \mathcal V)^\perp), \ \ \ \ \ \ \ (\left(\begin{smallmatrix}\mathcal U\\ \mathcal V\end{smallmatrix}\right), \ \left(\begin{smallmatrix}\mathcal U\\ \mathcal V\end{smallmatrix}\right)^\perp),
\ \ \ \ \ \ \ \ (^{\perp}\nabla(\mathcal X, \ \mathcal Y), \ \nabla(\mathcal X, \ \mathcal Y))$$
given in Theorems \ref{mainin31}  and  \ref{mainin32} have relations:
$$\Delta(\mathcal U, \ \mathcal V)^{\bot} \subseteq \left(\begin{smallmatrix}\mathcal X\\ \mathcal Y\end{smallmatrix}\right); \ \ \ \ \ \ \ ^{\bot}\nabla(\mathcal X, \ \mathcal Y) \subseteq \left(\begin{smallmatrix}\mathcal U\\ \mathcal V\end{smallmatrix}\right).$$ See Theorem \ref{compare} for details. An important and interesting feature is that,  {\bf they are not equal}, in general. In fact, there even exists an algebra $\Lambda$, such that the four cotorsion pairs above are pairwise different. Such an example has been given in Example \ref{ie}.

\subsection{Identifications}  If \ $M=0$ or $N = 0$, \ Theorems \ref{mainin31}  and  \ref{mainin32} have been obtained by {\rm R. M. Zhu}, {\rm Y. Y. Peng} and {\rm N. Q. Ding} [ZPD,  3.4, 3.6].
Moreover, they prove that (see {\rm [ZPD, 3.7]}) $$({}^\perp\left(\begin{smallmatrix}\mathcal X\\ \mathcal Y\end{smallmatrix}\right), \ \left(\begin{smallmatrix}\mathcal X\\ \mathcal Y\end{smallmatrix}\right)) = (\Delta(\mathcal U, \ \mathcal V), \ \Delta(\mathcal U, \ \mathcal V)^\perp)$$ and
$$(\left(\begin{smallmatrix} \mathcal U\\ \mathcal V\end{smallmatrix}\right), \ \left(\begin{smallmatrix} \mathcal U\\ \mathcal V\end{smallmatrix}\right)^\perp) =(^{\perp}\nabla(\mathcal X, \ \mathcal Y), \ \nabla(\mathcal X, \ \mathcal Y)).$$
As pointed out above, in general, these are not true! We will study the problem of identifications, i.e., when the two equalities hold true. If they are equal, then one has the cotorsion pairs
$$(\Delta(\mathcal U, \ \mathcal V), \ \left(\begin{smallmatrix}\mathcal X\\ \mathcal Y\end{smallmatrix}\right))
\ \ \ \ \mbox{and} \ \ \ \ (\left(\begin{smallmatrix} \mathcal U\\ \mathcal V\end{smallmatrix}\right), \ \nabla(\mathcal X, \ \mathcal Y)),$$
both are explicitly given. Since \ $^\perp\left(\begin{smallmatrix}\mathcal X\\ \mathcal Y\end{smallmatrix}\right)$ and \
$\Delta(\mathcal U, \ \mathcal V)^\perp$ are usually difficult to determine,
these identifications are of significance, in explicitly finding abelian model structures in Morita rings.

\vskip5pt

In the rest of this section,  \ $$\Lambda = \left(\begin{smallmatrix} A & N \\ M & B \end{smallmatrix}\right)$$ is a Morita ring which is an Artin algebra with \ $M\otimes_A N = 0 = N\otimes_BM$. We will not state this every time.
For functors \ ${\rm T}_A, \ {\rm H}_A: A\mbox{-Mod} \longrightarrow \Lambda\mbox{-Mod}$ and
\ ${\rm T}_B, \ {\rm H}_B: B\mbox{-Mod} \longrightarrow \Lambda\mbox{-Mod}$,  we refer to Subsection 2.4. By \ $M\otimes_A\mathcal U \subseteq \mathcal Y$ we mean
$M\otimes_A U \in \mathcal Y$ for all $U\in\mathcal U$.

\vskip5pt

\begin{thm}\label{mainin41} \ {\rm (Theorem \ref{identify1})}  \ Let \ $(\mathcal U, \ \mathcal X)$ and $(\mathcal V, \ \mathcal Y)$ be cotorsion pairs in $A\mbox{-}{\rm Mod}$ and in $B$\mbox{\rm-Mod}, respectively.

\vskip5pt

$(1)$ \ Assume that \ $\Tor^A_1(M, \ \mathcal U) =0 = \Tor^B_1(N, \ \mathcal V)$. If \
$M\otimes_A\mathcal U \subseteq \mathcal Y$ \  and \ $N\otimes_B\mathcal V \subseteq \mathcal X$, then
 \ $\Delta(\mathcal U, \ \mathcal V) =
\ ^\perp\left(\begin{smallmatrix} \mathcal X\\ \mathcal Y\end{smallmatrix}\right) = {\rm T}_A(\mathcal U)\oplus {\rm T}_B(\mathcal V)$, and thus \ $({\rm T}_A(\mathcal U)\oplus {\rm T}_B(\mathcal V), \ \left(\begin{smallmatrix} \mathcal X\\ \mathcal{Y}\end{smallmatrix}\right))$ \ is a cotorsion pair.
\vskip10pt

$(2)$ Assume that $\Ext_B^1(M,  \mathcal Y) =0 = \Ext_A^1(N,  \mathcal X)$. If $\Hom_B(M, \ \mathcal Y)\subseteq \mathcal U$  and  $\Hom_A(N,  \mathcal X) \subseteq \mathcal V$, then
$\nabla(\mathcal X,  \mathcal Y) =
\left(\begin{smallmatrix} \mathcal U \\ \mathcal V\end{smallmatrix}\right)^\perp = {\rm H}_A(\mathcal X)\oplus {\rm H}_B(\mathcal Y)$,   and thus  $(\left(\begin{smallmatrix} \mathcal U \\ \mathcal V\end{smallmatrix}\right),  {\rm H}_A(\mathcal X)\oplus {\rm H}_B(\mathcal Y))$  is a cotorsion pair.
\end{thm}

Even if the two cotorsion pairs are not equal in general,
there are possibilities that they can be equal for some special $A$, $B$, $M$ and $N$.
The following result provide such important cases:
cotorsion pairs  $$(^\perp\binom{_A\mathcal I}{_B\mathcal I}, \ \binom{_A\mathcal I}{_B\mathcal I}) \ \ \  \mbox{and} \ \ \ \ ({\rm Mon}(\Lambda), \ {\rm Mon}(\Lambda)^\bot) = (\Delta (A\mbox{-}{\rm Mod}, \ B\mbox{-}{\rm Mod}), \ \Delta (A\mbox{-}{\rm Mod}, \ B\mbox{-}{\rm Mod})^\perp)$$
are not equal in general (cf. Example \ref{ie}); but the following result claims that they can be the same in some special cases.

\vskip5pt

For a ring $R$, let $_R\mathcal P$  (respectively,  $_R\mathcal I$) be the full subcategory of $R$-Mod of projective (respectively, injective) modules, $_R \mathcal P^{\le 1}$  (respectively,  $_R\mathcal I^{\le 1}$) the full subcategory  of modules
with projective (respectively, injective) dimension $\le 1$.

\begin{thm}\label{mainin42} \ {\rm (Theorem \ref{ctp4})} \ Assume that $A$ and $B$ are quasi-Frobenius rings, \ $_AN$ and $_BM$ are projective, and that $M_A$ and $N_B$ are flat. Then

\vskip5pt

${\rm (1)}$ \ \ $\Lambda$ is a Gorenstein ring with \ ${\rm inj.dim} _\Lambda\Lambda\le 1$.

\vskip5pt

${\rm (2)}$ \ \ $(^\perp\binom{_A\mathcal I}{_B\mathcal I}, \ \binom{_A\mathcal I}{_B\mathcal I}) = ({\rm Mon}(\Lambda), \ {\rm Mon}(\Lambda)^\bot);$ and it is
exactly the Gorenstein-projective cotorsion pair  $({\rm GP}(\Lambda), \ _\Lambda\mathcal P^{\le 1})$.  So, it is complete and hereditary, and
$${\rm GP}(\Lambda) = {\rm Mon}(\Lambda) = \ {}^\perp \ _\Lambda\mathcal P, \ \ \ \ {\rm Mon}(\Lambda)^\bot = \ _\Lambda \mathcal P^{\le 1}.$$

\vskip5pt

${\rm (2)'}$ \ \ $(\binom{_A\mathcal P}{_B\mathcal P}, \ \binom{_A\mathcal P}{_B\mathcal P}^\perp) = (^\bot{\rm Epi}(\Lambda), \ {\rm Epi}(\Lambda));$ and it
is exactly the Gorenstein-injective cotorsion pair \ $(_\Lambda \mathcal P^{\le 1}, \ {\rm GI}(\Lambda))$.  So, it is  complete and hereditary, and
$$ {\rm GI}(\Lambda) = {\rm Epi}(\Lambda) = \ _\Lambda\mathcal I{}^\perp, \ \ \ \ ^\bot{\rm Epi}(\Lambda) = \ _\Lambda \mathcal P^{\le 1}.$$
\end{thm}

\vskip5pt

The conditions of {\rm Theorem \ref{mainin42}} really and often occur. See Example \ref{examctp4}.
Note that \ ${\rm GP}(\Lambda) = {\rm Mon}(\Lambda)$  is  a new result, as an application of
cotorsion theory and monomorphism category to Gorenstein-projective modules. See Remark \ref{remctp4}.

\subsection{Completeness} \ Completeness of a cotorsion pair is important, not only in the theory itself,  but also in finding abelian model structures ([H2]. See Theorem \ref {hoveycorrespondence}).
In view of identifications, we only discuss the completeness of cotorsion pairs in Theorem \ref{mainin31}.

\vskip5pt If cotorsion pairs \ $(\mathcal U,  \mathcal X)$ and \ $(\mathcal V,  \mathcal Y)$  are cogenerated by sets $S_1$ and $S_2$, respectively,
then the cotorsion pair \ $(^\perp\left(\begin{smallmatrix}\mathcal X \\ \mathcal Y\end{smallmatrix}\right), \ \left(\begin{smallmatrix}\mathcal X \\ \mathcal Y\end{smallmatrix}\right))$
is  cogenerated by the set \ ${\rm T}_A(S_1) \cup {\rm T}_B(S_2)$, and hence complete, by a well-known theorem [ET2, Theorem 10]  (see Proposition \ref{cogenerated}) of P. C. Eklof and J. Trlifaj. See Proposition \ref{generatingcomplete}.

\vskip5pt

However, since the theorem of Eklof and Trlifaj has no dual versions, in general\footnote{For some special cases, there are indeed dual versions. See e.g. [ET1, Theorem 14(ii)].}, there is no information on the completeness of \ $(\left(\begin{smallmatrix}\mathcal U\\ \mathcal V\end{smallmatrix}\right), \ \left(\begin{smallmatrix}\mathcal U\\ \mathcal V\end{smallmatrix}\right)^\perp)$.
Also, it is more natural to start from the completeness of  $(\mathcal U,  \mathcal X)$ and \ $(\mathcal V,  \mathcal Y)$.
Thus, we need module-theoretical methods to the completeness of cotorsion pairs in Morita rings.

\vskip5pt

Take \ $(\mathcal V, \ \mathcal Y)$ to be an arbitrary complete cotorsion pair in $B$-Mod.
Taking \ $(\mathcal U, \ \mathcal X) = (_A\mathcal P, \ A\mbox{-}{\rm Mod})$ and applying Theorem \ref{mainin31}(1),
we have Theorem \ref{mainin51}(1) below;
taking \ $(\mathcal U, \ \mathcal X) = (A\mbox{-}{\rm Mod}, \ _A\mathcal I)$ and applying Theorem \ref{mainin31}(2),  we have Theorem \ref{mainin51}(2) below.

\begin{thm} \label{mainin51} \ {\rm (Theorem \ref{ctp2})} \ Assume that \ $N_B$ is flat and \ $_BM$ is projective. Let \  $(\mathcal V, \ \mathcal Y)$ be a complete cotorsion pair in $B\mbox{-}{\rm Mod}$.
	
\vskip5pt
	
$(1)$  \  If \ $M\otimes_A\mathcal P\subseteq \mathcal Y$,
then \ $({\rm T}_A(_A\mathcal P)\oplus {\rm T}_B(\mathcal V), \ \left (\begin{smallmatrix} A\mbox{-}{\rm Mod} \\ \mathcal Y\end{smallmatrix}\right))$ \ is a complete cotorsion pair.

\vskip10pt

$(2)$ \ If \ $\Hom_A(N, \ _A\mathcal I) \subseteq \mathcal V$, then \ $(\left(\begin{smallmatrix} A\mbox{-}{\rm Mod} \\ \mathcal V\end{smallmatrix}\right), \ {\rm H}_A(_A\mathcal I)\oplus {\rm H}_B(\mathcal Y))$ \  is a complete cotorsion pair.
\end{thm}

\vskip5pt

We stress that

$(1)$ \ \  If  \ $_BM$ is injective, then \ $M\otimes_A\mathcal P\subseteq \mathcal Y$ always
holds;

$(2)$ \ \ If \ $N_B$ is flat, then \ $\Hom_A(N, \ _A\mathcal I) \subseteq \mathcal V$ always holds.

\vskip10pt

Similarly, let \ $(\mathcal U, \ \mathcal X)$  be an arbitrary complete cotorsion pair in $A$-Mod.
Taking $(\mathcal V, \ \mathcal Y)  = (_B\mathcal P, \ B\mbox{-}{\rm Mod})$ and applying Theorem \ref{mainin31}(1), we get Theorem \ref{ctp3}(1);
and taking $(\mathcal V, \ \mathcal Y) = (B\mbox{-}{\rm Mod}, \ _B\mathcal I)$ and applying Theorem \ref{mainin31}(2), we get Theorem \ref{ctp3}(2).

\subsection{Realizations} \ It turns out that Morita rings are rich in cotorsion  pairs. In Theorem \ref{mainin31} (respectively, Theorem \ref{mainin32}),
even if one starts form the projective or the injective cotorsion pair in $A$-Mod and $B$-Mod,
what one gets in $\Lambda$-Mod are already {\it pairwise generally different} (see Definition \ref{difference}) and {\it``new"} cotorsion pairs.
Here, by a ``new" cotorsion pair we mean that it is generally different from
the projective and the injective cotorsion pair, the Gorenstein-projective and the Gorenstein-injective cotorsion pair, and the flat cotorsion pair ([EJ, Lemma 7.1.4]).
For details see Definition \ref{new}, Propositions \ref{different}, \ref{newI},  \ref{different2} and  \ref{newII}.

\subsection{Abelian model structures on Morita rings}

A natural method of getting abelian model structures on Morita rings, is to see how abelian model structures on $A$-Mod and $B$-Mod can induce the ones on $\Lambda$-Mod.

\vskip5pt

A cofibrantly generated model category ([H1, 2.1.17])  enjoys nice properties.
A Hovey triple $(\mathcal C, \mathcal F, \mathcal W)$ will be called {\it cofibrantly generated}, if both the cotorsion pairs $(\mathcal C\cap \mathcal W, \mathcal F)$ and $(\mathcal C, \mathcal F\cap\mathcal W)$ are cogenerated by sets. For a Grothendieck category  $\mathcal A$ with enough projective objects, a Hovey triple $(\mathcal C, \mathcal F, \mathcal W)$  in $\mathcal A$
is cofibrantly generated if and only if the corresponding model category is cofibrantly generated. See Proposition \ref {cofibrantly}.

\vskip5pt

\begin{thm}\label{main71} \ {\rm (Theorem \ref{cofibrantlygenHtriple})} \ Let \ $(\mathcal U', \ \mathcal X, \ \mathcal W_1)$ and \ $(\mathcal V', \ \mathcal Y, \ \mathcal W_2)$ be cofibrantly generated
Hovey triples in $A\mbox{-}{\rm Mod}$ and  $B\mbox{-}{\rm Mod}$, respectively.

\vskip5pt

$(1)$ \ Assume that  \ ${\rm Tor}^A_1(M, \ \mathcal U') = 0 = {\rm Tor}^B_1(N, \ \mathcal V')$,  \  $M\otimes_A\mathcal U' \subseteq \mathcal Y\cap \mathcal W_2$ and \ $N\otimes_B\mathcal V' \subseteq \mathcal X\cap \mathcal W_1.$ Then
$$({\rm T}_A(\mathcal U')\oplus {\rm T}_B(\mathcal V'), \ \left(\begin{smallmatrix} \mathcal X \\ \mathcal Y\end{smallmatrix}\right), \ \left(\begin{smallmatrix} \mathcal W_1 \\ \mathcal W_2\end{smallmatrix}\right))$$
is a cofibrantly generated Hovey triple in $\Lambda\mbox{-}{\rm Mod};$ and it is hereditary with \ ${\rm Ho}(\Lambda) \cong   {\rm Ho}(A)\oplus {\rm Ho}(B),$ if \ $(\mathcal U',  \mathcal X,  \mathcal W_1)$ and \ $(\mathcal V',  \mathcal Y,  \mathcal W_2)$ are hereditary.

\vskip5pt

$(2)$ \ Assume that   ${\rm Ext}_B^1(M,  \mathcal Y) = 0 = {\rm Ext}_A^1(N,  \mathcal X)$, $\Hom_B(M, \mathcal Y) \subseteq \mathcal U'\cap \mathcal W_1$ and  $\Hom_A(N, \mathcal X) \subseteq \mathcal V'\cap \mathcal W_2$. Then
$$(\left(\begin{smallmatrix} \mathcal U' \\ \mathcal V'\end{smallmatrix}\right),  \ {\rm H}_A(\mathcal X)\oplus {\rm H}_B(\mathcal Y), \ \left(\begin{smallmatrix} \mathcal W_1 \\ \mathcal W_2\end{smallmatrix}\right))$$
is a cofibrantly generated
Hovey triple$;$ and it is hereditary with  ${\rm Ho}(\Lambda) \cong   {\rm Ho}(A)\oplus {\rm Ho}(B)$, if  $(\mathcal U', \mathcal X,  \mathcal W_1)$ and  $(\mathcal V',  \mathcal Y, \mathcal W_2)$ are hereditary.
\end{thm}

For general Hovey triples (not assumed to be cofibrantly generated), we have

\begin{thm} \label{main72} \ {\rm (Theorem \ref{Htriple1})} \  Let \ $N_B$ be flat,  $_BM$ projective,  and \ $(\mathcal V',  \mathcal Y,  \mathcal W)$ a Hovey triple in $B$\mbox{-}{\rm Mod}.

\vskip5pt

$(1)$  \  If \ $M\otimes_A\mathcal P \subseteq \mathcal Y\cap \mathcal W$, then
$$({\rm T}_A(_A\mathcal P)\oplus {\rm T}_B(\mathcal V'), \ \left(\begin{smallmatrix} A\text{\rm\rm-Mod} \\ \mathcal Y\end{smallmatrix}\right), \ \left(\begin{smallmatrix} A\text{\rm\rm-Mod} \\ \mathcal W\end{smallmatrix}\right))$$
is a Hovey triple in $\Lambda$\mbox{-}{\rm Mod}$;$ and it is hereditary with \ ${\rm Ho}(\Lambda) \cong {\rm Ho}(B)$, if
 \ $(\mathcal V',  \mathcal Y,  \mathcal W)$ is hereditary.

\vskip5pt

$(2)$  \ If \ $\Hom_A(N, \ _A\mathcal I)\subseteq \mathcal V'\cap \mathcal W$, then
$$(\left(\begin{smallmatrix} A\mbox{-}{\rm Mod} \\ \mathcal V'\end{smallmatrix}\right), \ \ {\rm H}_A(_A\mathcal I)\oplus {\rm H}_B(\mathcal Y), \ \ \left(\begin{smallmatrix}A\text{\rm\rm-Mod}\\ \mathcal W\end{smallmatrix}\right))$$
is a Hovey triple in $\Lambda$\mbox{-}{\rm Mod}$;$ and it is hereditary with \ ${\rm Ho}(\Lambda) \cong {\rm Ho}(B)$, if
 \ $(\mathcal V',  \mathcal Y,  \mathcal W)$ is hereditary.
\end{thm}

Theorem \ref{main72} is not a corollary of Theorem \ref{main71}, it needs to use Theorem \ref{mainin51},
a module-theoretical argument on the completeness of cotorsion pairs in Morita rings.

\vskip5pt

Similarly, starting from a Hovey triple in $A$-{\rm Mod}, we get Theorem \ref{Htriple2}.
By Theorems  \ref{main72} and \ref{Htriple2}, we in fact get four kinds of abelian model structures on $\Lambda$\mbox{-}{\rm Mod}.

\subsection{Generalized projective cotorsion pairs (gpctps) and projective models} \ A complete cotorsion pair \ $(\mathcal U, \ \mathcal X)$  in $A\mbox{-}{\rm Mod}$ is {\it generalized projective}, if \ $\mathcal U\cap \mathcal X = \ _A\mathcal P$  and \ $\mathcal X$ is thick (see [Bec, 1.1.9]).
A generalized projective cotorsion pair (or in short, a gpctp) is always hereditary and not necessarily the projective cotorsion pair $(_A\mathcal P, \ A\mbox{-}{\rm Mod})$. Following [H2] and [Gil4],
an abelian model structure is {\it projective}, if
each object is fibrant, i.e., the corresponding Hovey triple is of the form $(\mathcal U, \ A\mbox{-}{\rm Mod}, \ \mathcal X)$. Note that gpctps and projective models are in one-one correspondence, i.e.,
\ $(\mathcal U, \ \mathcal X)$ is a gpctp
in $A\mbox{-}{\rm Mod}$  if and only if $(\mathcal U, \ A\mbox{-}{\rm Mod}, \ \mathcal X)$ is a Hovey triple. See Subsection 7.3.

\vskip5pt

Dually, one has the notion of {\it a generalized injective cotorsion pair} (or in short, {\it gictp}),  and {\it an injective model}.

\vskip5pt

The following result deals with the important  case of gpctps (gictps) in Theorem \ref{main72},
with stronger result, where the conditions
``$M\otimes_A\mathcal P \subseteq \mathcal Y\cap \mathcal W$" and ``$\Hom_A(N, \ _A\mathcal I)\subseteq \mathcal V'\cap \mathcal W$" in Theorem \ref{main72} can be dropped in these cases.

\vskip5pt

\begin{thm}\label{main73} {\rm(Theorem \ref{GHtriplegpgiB})} \ Assume that  \ $N_B$ is flat and  $_BM$ is projective.

\vskip5pt

$(1)$ \ Let \  $(\mathcal V, \ \mathcal Y)$ and $(\mathcal V', \ \mathcal Y')$ be compatible gpctps in $B$-{\rm Mod}, with Gillespie-Hovey triple
\ $(\mathcal V', \ \mathcal Y, \ \mathcal W)$. Then $$({\rm T}_A(_A\mathcal P)\oplus {\rm T}_B(\mathcal V), \ \left(\begin{smallmatrix} A\text{\rm\rm-Mod} \\ \mathcal Y\end{smallmatrix}\right))
\ \ \ \mbox{and} \ \ \
({\rm T}_A(_A\mathcal P)\oplus {\rm T}_B(\mathcal V'), \ \left(\begin{smallmatrix} A\text{\rm\rm-Mod} \\ \mathcal Y'\end{smallmatrix}\right))$$
are compatible gpctps in \ $\Lambda\mbox{-}{\rm Mod}$, with Gillespie-Hovey triple
$$({\rm T}_A(_A\mathcal P)\oplus {\rm T}_B(\mathcal V'), \ \left(\begin{smallmatrix} A\text{\rm\rm-Mod} \\ \mathcal Y\end{smallmatrix}\right), \ \left(\begin{smallmatrix} A\text{\rm\rm-Mod} \\ \mathcal W\end{smallmatrix}\right))$$
and \ ${\rm Ho}(\Lambda)\cong (\mathcal V'\cap \mathcal Y)/_B\mathcal P\cong {\rm Ho}(B)$.

\vskip5pt

$(2)$   \ Let \ $(\mathcal V, \mathcal Y)$ and \ $(\mathcal V', \mathcal Y')$ be compatible gictps, with Gillespie-Hovey triple
\ $(\mathcal V', \ \mathcal Y, \ \mathcal W)$.  Then
$$(\left(\begin{smallmatrix} A\mbox{-}{\rm Mod} \\ \mathcal V\end{smallmatrix}\right), \ {\rm H}_A(_A\mathcal I)\oplus {\rm H}_B(\mathcal Y)) \ \ \ \mbox{and} \ \ \ (\left(\begin{smallmatrix} A\mbox{-}{\rm Mod} \\ \mathcal V'\end{smallmatrix}\right), \ {\rm H}_A(_A\mathcal I)\oplus {\rm H}_B(\mathcal Y'))$$
are compatible gictps in $\Lambda$-{\rm Mod}, with Gillespie-Hovey triple
$$(\left(\begin{smallmatrix} A\mbox{-}{\rm Mod} \\ \mathcal V'\end{smallmatrix}\right), \ \ {\rm H}_A(_A\mathcal I)\oplus {\rm H}_B(\mathcal Y), \ \ \left(\begin{smallmatrix}A\text{\rm\rm-Mod}\\ \mathcal W\end{smallmatrix}\right))$$
and \ ${\rm Ho}(\Lambda) \cong (\mathcal V'\cap \mathcal Y)/ _B\mathcal I\cong {\rm Ho}(B)$.
\end{thm}

\vskip5pt

A gpctp $(\mathcal V, \ \mathcal Y)$ in $B$-{\rm Mod} gives compatible gpctps
$(_B\mathcal P, \ B\mbox{-}{\rm Mod})$ and $(\mathcal V, \ \mathcal Y)$.
A gictp $(\mathcal V, \ \mathcal Y)$ in $B$-{\rm Mod} gives compatible gictps
$(\mathcal V, \ \mathcal Y)$ and $(B\mbox{-}{\rm Mod}, \ _B\mathcal I)$.
Thus, by Theorem \ref{main73} one has

\vskip5pt

\begin{cor}\label{main74} \ {\rm (Corollaries \ref{projinjtripleB}, \ref{frobB})} \ Suppose that \ $N_B$ is flat and \ $_BM$ is projective.

\vskip5pt

$(1)$ \  Let \ $(\mathcal V, \mathcal Y)$ be a gpctp in $B$-{\rm Mod}. Then
$$({\rm T}_A(_A\mathcal P)\oplus {\rm T}_B(\mathcal V), \ \ \Lambda\text{\rm\rm-Mod}, \ \ \left(\begin{smallmatrix} A\text{\rm\rm-Mod} \\ \mathcal Y\end{smallmatrix}\right))$$
is a hereditary Hovey triple, with \ ${\rm Ho}(\Lambda)\cong\mathcal V/_B\mathcal P$.

\vskip5pt

In particular, if $B$ is quasi-Frobenius, then
\ $({\rm T}_A(_A\mathcal P)\oplus {\rm T}_B(B\text{\rm\rm-Mod}), \ \ \Lambda\text{\rm\rm-Mod}, \ \ \left(\begin{smallmatrix} A\text{\rm\rm-Mod} \\ _B\mathcal I\end{smallmatrix}\right))$ \
is a hereditary {\rm Hovey} triple with \ ${\rm Ho}(\Lambda) \cong B\mbox{-}\underline{{\rm Mod}}.$

\vskip5pt

$(2)$   \ Let \ $(\mathcal V, \mathcal Y)$ be a gictp in $B$-{\rm Mod}. Then
$$(\Lambda\text{\rm\rm-Mod}, \ \ {\rm H}_A(_A\mathcal I)\oplus {\rm H}_B(\mathcal Y), \ \ \left(\begin{smallmatrix}A\text{\rm\rm-Mod}\\ \mathcal V\end{smallmatrix}\right))$$
is a hereditary Hovey triple, with \ ${\rm Ho}(\Lambda)\cong \mathcal Y/_B\mathcal I$.

\vskip5pt

In particular, if $B$ is quasi-Frobenius, then \ $(\Lambda\text{\rm\rm-Mod}, \ \ {\rm H}_A(_A\mathcal I)\oplus {\rm H}_B(B\text{\rm\rm-Mod}), \ \ \left(\begin{smallmatrix}A\text{\rm\rm-Mod}\\ _B\mathcal P\end{smallmatrix}\right))$
is a hereditary Hovey triple with \ ${\rm Ho}(\Lambda)\cong B\mbox{-}\underline{{\rm Mod}}.$
\end{cor}

\vskip5pt

Similarly, starting from compatible gpctps (gictps) in $A$-{\rm Mod}, one has the corresponding results. See Theorem \ref{GHtriplegpgiA}, {\rm Corollaries \ref{projinjtripleA} and \ref{frobA}}.

\vskip5pt

We stress that all the abelian model structures obtained above are pairwise generally different and ``new".
See Proposition \ref{newmodel} for details.

\section{\bf Preliminaries}

\subsection{Notations}
 For a ring $R$, let $R$-Mod be the category of left $R$-modules,
 $_R\mathcal P$  (respectively,  $_R\mathcal I$) the full subcategory of $R$-Mod of projective (respectively, injective) modules;
 $_R\mathcal P^{<\infty}$  (respectively,  $_R\mathcal I^{<\infty}$) the full subcategory  of $R$-Mod modules of finite projective (respectively, injective) dimension. Denote by ${\rm GP}(R)$ (respectively, ${\rm GI}(R)$) the full subcategory of $R$-Mod of Gorenstein-projective (respectively, Gorenstein-injective) modules.

\vskip5pt

For a class $\mathcal C$ of objects in abelian category $\mathcal A$, let
\begin{align*}^\perp\mathcal C &= \{X\in \mathcal A \ \mid \ \Ext^1_\mathcal A(X,\mathcal C)=0\},  \ \ \ \ \  \ \ \ \ \ ^{\bot_{\ge 1}}\mathcal C = \{ X\in \mathcal A \  | \  \Ext^i_\mathcal A(X, \ \mathcal C) = 0, \ \forall \ i\ge 1\}, \\ \mathcal C^{\perp} & = \{X\in \mathcal A \ \mid \ \Ext^1_\mathcal A(\mathcal C, X)=0\}, \ \ \ \ \  \ \ \ \ \
\mathcal C^{\bot_{\ge 1}} = \{X\in \mathcal A \ |  \ \Ext^i_\mathcal A(\mathcal C, X)=0, \ \forall \ i\ge 1\}.\end{align*}
For classes $\mathcal C$ and $\mathcal D$ of objects in $\mathcal A$, by $\Hom_\mathcal A(\mathcal C, \mathcal D) = 0$  we mean
$\Hom_\mathcal A(C, D) = 0$ for all $C\in\mathcal C$ and for all $D\in\mathcal D$. Similarly for $\Ext^1_\mathcal A(\mathcal C, \mathcal D) = 0.$

\subsection{Morita rings} \ Let $A$ and $B$ be rings, $_BM_A$ a $B$-$A$-bimodule, \ $_AN_B$ an $A$-$B$-bimodule,  \ $\phi: M\otimes_AN\longrightarrow B$ a $B$-bimodule map,
and $\psi: N\otimes_BM\longrightarrow A$ an $A$-bimodule map, such that
$$m'\psi(n\otimes_Bm) = \phi(m'\otimes_An)m, \ \   n'\phi(m\otimes_An) = \psi(n'\otimes_Bm)n, \ \  \forall \ m, m'\in M, \ \ \forall \ n, n'\in N. \eqno(*)$$

A {\it Morita ring} is \ $\Lambda = \Lambda_{(\phi, \psi)}:=\left(\begin{smallmatrix} A & {}_AN_B \\
	{}_BM_A & B\end{smallmatrix}\right)$,  with componentwise addition, and multiplication
$$\left(\begin{smallmatrix} a & n \\ m & b \end{smallmatrix}\right) \left(\begin{smallmatrix} a' & n' \\ m' & b' \end{smallmatrix}\right)
=\left(\begin{smallmatrix} aa'+\psi(n\otimes_Bm') & an'+nb' \\ ma'+bm' & \phi(m\otimes_An')+bb'\end{smallmatrix}\right).$$
The assumptions $(*)$ guarantee the associativity of the multiplication (the converse is also true).
This construction is finally formulated in [Bas]. Throughout this paper, we will assume $\phi = 0 =\psi$. This contains triangular matrix rings (i.e., $M =0$ or $N = 0$).

\vskip5pt

Throughout this paper, we assume that the considered Morita rings $\Lambda$ are Artin algebras, i.e., $A$ and $B$ are Artin $R$-algebras, where $R$ is a commutative artinian ring,
$M$ and $N$ are finitely generated $R$-modules, such that $R$ acts centrally both on $M$ and $N$ (see [GrP, Proposition 2.2]).

\vskip5pt

\subsection{Two expressions of modules over Morita rings} \ Let $\mathcal M(\Lambda)$ be the category with objects $\left(\begin{smallmatrix} X \\ Y \end{smallmatrix}\right)_{f,g}$, where $X\in A\mbox{-Mod}$, \ $Y\in B\mbox{-Mod}$, \ $f\in \Hom_B(M\otimes_AX,Y)$ and  $g\in \Hom_A(N\otimes_BY, X)$,
satisfy  the conditions
$$g(n\otimes_Bf(m\otimes_Ax)) = \psi(n\otimes_Bm)x, \ \ \
f(m\otimes_Ag(n\otimes_Ay)) = \phi(m\otimes_An)y, \ \ \forall \ m\in M, \ n\in N, \ x\in X, \ y\in Y.$$
The maps $f$ and $g$ are called {\it the structure maps} of $\left(\begin{smallmatrix} X \\ Y \end{smallmatrix}\right)_{f,g}$.

\vskip5pt

For $\phi = 0 = \psi,$ the conditions are just \ $g(1_N\otimes f) = 0 = f(1_M\otimes g).$

\vskip5pt

A morphism  in $\mathcal M(\Lambda)$ is $\left(\begin{smallmatrix} a \\ b \end{smallmatrix}\right): \left(\begin{smallmatrix} X \\ Y\end{smallmatrix}\right)_{f,g}\longrightarrow \left(\begin{smallmatrix} X' \\ Y'\end{smallmatrix}\right)_{f',g'}$, where \ $a: X\rightarrow X'$ and  $b: Y\rightarrow Y'$ are respectively an $A$-map and a $B$-map, so that the following diagrams  commute:
$$\xymatrix@R= 0.7cm{M\otimes_AX\ar[d]_-f\ar[r]^-{1_M\otimes a} & M\otimes_AX' \ar[d]^-{f'} \\
	Y \ar[r]^-b & Y'}\qquad \qquad \qquad
\xymatrix@R= 0.7cm{N\otimes_BY\ar[d]_-g\ar[r]^-{1_N\otimes b} & N\otimes Y'\ar[d]^-{g'} \\
	X\ar[r]^-a & X'.}$$

Let $\eta_{X, Y}: \Hom_B(M\otimes_AX, Y)\cong \Hom_A(X, \ \Hom_B(M, Y))$
and $\eta'_{Y, X}: \Hom_A(N\otimes_BY, X)\cong \Hom_B(Y, \ \Hom_A(N, X))$ be the adjunction isomorphisms.
For $f\in \Hom_B(M\otimes_AX, Y)$ and
$g\in \Hom_A(N\otimes_BY, X)$, put $\widetilde{f} = \eta_{X, Y}(f)$ and $\widetilde{g} = \eta'_{X, Y}(g)$. Thus
$$\widetilde{f}(x) = ``m\mapsto f(m\otimes_Ax)", \ \forall \ x\in X; \ \ \ \ \ \widetilde{g}(y) = ``n\mapsto g(n\otimes_By)", \ \forall \ y\in Y.$$

Using the bi-functorial property of the  adjunction isomorphisms one knows that
$$fb = f'(1_M\otimes_Aa) \ \ \mbox{if and only if} \ \  (M, b)\widetilde{f} = \widetilde{f'}a$$ and
$$ ag = g'(1_N\otimes_Bb) \ \ \mbox{if and only if} \ \ (N, a)\widetilde{g} = \widetilde{g'}b.$$

Let $\mathcal M'(\Lambda)$ be the category with objects $\left(\begin{smallmatrix} X \\ Y \end{smallmatrix}\right)_{\widetilde{f}, \ \widetilde{g}}$, where $X\in A\mbox{-Mod}$, \ $Y\in B\mbox{-Mod}$, \ $\widetilde{f}\in \Hom_A(X, \ \Hom_B(M, Y))$ and  $\widetilde{g}\in \Hom_B(Y, \ \Hom_A(N, X))$, such that the following diagrams commute:
$$\xymatrix@R= 0.7cm {X \ar[d]_-{\widetilde{f}}\ar[r]^-{(\psi, X)h_{A, X}} & \Hom_A(N\otimes_BM, X)\ar[d]_-\cong^-{\eta'_{M, X}} \\
\Hom_B(M, Y) \ar[r]^-{(M, \widetilde{g})} & \Hom_A(M, \Hom_B(N, X))} \qquad
\xymatrix@R= 0.7cm{Y\ar[d]_-{\widetilde{g}}\ar[r]^-{(\phi, Y)h_{B, Y}} & \Hom_B(M\otimes_AN, Y)\ar[d]_-\cong
^-{\eta_{N, Y}} \\
\Hom_A(N, X) \ar[r]^-{(N, \widetilde{f})} & \Hom_A(N, \Hom_B(M, Y))}$$
where $h_{A, X}: X\rightarrow \Hom_A(A, X)$ and $h_{B, Y}: Y \rightarrow \Hom_B(B, Y)$ are the canonical isomorphisms.
The maps $\widetilde{f}$ and $\widetilde{g}$ are also called {\it the structure maps} of $\left(\begin{smallmatrix} X \\ Y \end{smallmatrix}\right)_{\widetilde{f}, \widetilde{g}}$.

\vskip5pt

For $\phi = 0 = \psi,$ the conditions are  just \ $(M, \widetilde{g}) \widetilde{f} = 0 = (N, \widetilde{f}) \widetilde{g}.$

\vskip5pt

A morphism  in $\mathcal M'(\Lambda)$ is $\left(\begin{smallmatrix} a \\ b \end{smallmatrix}\right): \left(\begin{smallmatrix} X \\ Y\end{smallmatrix}\right)_{\widetilde{f}, \widetilde{g}}\longrightarrow \left(\begin{smallmatrix} X' \\ Y'\end{smallmatrix}\right)_{\widetilde{f'}, \widetilde{g'}}$, where \ $a: X\rightarrow X'$ and  $b: Y\rightarrow Y'$ are respectively an $A$-map and a $B$-map, so that diagrams
$$\xymatrix@R= 0.7cm{X\ar[d]_-{\widetilde{f}}\ar[r]^-{a} & X' \ar[d]^-{\widetilde{f'}} \\
	\Hom_B(M, Y) \ar[r]^-{(M, b)} & \Hom_B(M, Y')}\qquad \qquad \qquad
\xymatrix@R= 0.7cm{Y\ar[d]_-{\widetilde{g}}\ar[r]^-b & Y'\ar[d]^-{\widetilde{g'}} \\
\Hom_A(N, X)\ar[r]^-{(N, a)} & \Hom_A(N, X')}$$
 commute. Then
 $$\left(\begin{smallmatrix} X \\ Y\end{smallmatrix}\right)_{f, g}\mapsto \left(\begin{smallmatrix} X \\ Y\end{smallmatrix}\right)_{\widetilde{f}, \widetilde{g}}, \ \ \ \ \ \ \ ``\left(\begin{smallmatrix} X \\ Y\end{smallmatrix}\right)_{f, g}\stackrel{\left(\begin{smallmatrix} a \\ b \end{smallmatrix}\right)}\longrightarrow \left(\begin{smallmatrix} X' \\ Y'\end{smallmatrix}\right)_{f', g'}"\mapsto ``\left(\begin{smallmatrix} X \\ Y\end{smallmatrix}\right)_{\widetilde{f}, \widetilde{g}}\stackrel{\left(\begin{smallmatrix} a \\ b \end{smallmatrix}\right)}\longrightarrow \left(\begin{smallmatrix} X' \\ Y'\end{smallmatrix}\right)_{\widetilde{f'}, \widetilde{g'}}"$$ gives an isomorphism $\mathcal M(\Lambda)\cong \mathcal M'(\Lambda)$ of categories.

\begin{thm}\label{modovermorita} {\rm (E. L. Green [G, 1.5])} \ Let $\Lambda =\left(\begin{smallmatrix} A & N \\
	M & B\end{smallmatrix}\right)$ be a Morita ring. Then $\Lambda\mbox{-}{\rm Mod}\cong \mathcal M(\Lambda)\cong \mathcal M'(\Lambda)$ as categories.
\end{thm}

Throughout we will identify   a  $\Lambda$-module with $\left(\begin{smallmatrix} X \\ Y \end{smallmatrix}\right)_{f,g}$.
We will also use  the expression $\left(\begin{smallmatrix} X \\ Y \end{smallmatrix}\right)_{\widetilde{f}, \widetilde{g}},$ when it is more convenient.
For convenience we will call $\left(\begin{smallmatrix} X \\ Y \end{smallmatrix}\right)_{\widetilde{f}, \widetilde{g}}$
the second expression of a $\Lambda$-module. A sequence of $\Lambda$-maps
 $$\left(\begin{smallmatrix} X_1 \\ Y_1\end{smallmatrix}\right)_{f_1,g_1}\stackrel{\binom{a_1}{b_1}}\longrightarrow \left(\begin{smallmatrix} X_2 \\ Y_2 \end{smallmatrix}\right)_{f_2,g_2}\stackrel{\binom{a_2}{b_2}}\longrightarrow
	\left(\begin{smallmatrix} X_3 \\ Y_3 \end{smallmatrix}\right)_{f_3,g_3}$$
is exact  if and only if both the sequences
\ $X_1\stackrel{a_1}\longrightarrow X_2\stackrel{a_2}\longrightarrow X_3$  and
\ $Y_1\stackrel{b_1}\longrightarrow Y_2\stackrel{b_2}\longrightarrow Y_3$
are exact. Also, in the second expressions of $\Lambda$-modules, a sequence of $\Lambda$-maps
 $$\left(\begin{smallmatrix} X_1 \\ Y_1\end{smallmatrix}\right)_{\widetilde{f_1},\widetilde{g_1}}\stackrel{\binom{a_1}{b_1}}\longrightarrow \left(\begin{smallmatrix} X_2 \\ Y_2 \end{smallmatrix}\right)_{\widetilde{f_2},\widetilde{g_2}}\stackrel{\binom{a_2}{b_2}}\longrightarrow
	\left(\begin{smallmatrix} X_3 \\ Y_3 \end{smallmatrix}\right)_{\widetilde{f_3},\widetilde{g_3}}$$
is exact  if and only if both the sequences
\ $X_1\stackrel{a_1}\longrightarrow X_2\stackrel{a_2}\longrightarrow X_3$  and
\ $Y_1\stackrel{b_1}\longrightarrow Y_2\stackrel{b_2}\longrightarrow Y_3$
are exact.

\subsection{Twelve functors and two recollements}
Denote by  $\Psi_X$ the composition \ $N\otimes_BM\otimes_AX \stackrel {\psi\otimes 1_X} \longrightarrow A\otimes_AX\stackrel {\cong} \rightarrow X$, and denote by $\Phi_Y$
the composition \ $M\otimes_AN\otimes_B Y\stackrel{1_M\otimes g} \longrightarrow B\otimes_BY \stackrel{\cong}\rightarrow Y$.

\vskip5pt

Let   $\epsilon:  M\otimes_A\Hom_B(M, -)\longrightarrow {\rm Id}_{B\text{-Mod}}$ be the counit,
and  $\delta: {\rm Id}_{A\text{-Mod}}\longrightarrow \Hom_B(M, M\otimes_A -)$ the unit, of the adjoint pair $(M\otimes_A-, \ \Hom_A(M, -))$. Let
\ $\epsilon': N\otimes_B\Hom_A(N, -)\longrightarrow {\rm Id}_{A\text{-Mod}}$ be the counit,
and \ $\delta': {\rm Id}_{B\text{-Mod}}\longrightarrow \Hom_A(N, N\otimes_B-)$ the unit, of the adjoint pair $(N\otimes_B-, \ \Hom_A(N, -))$.

\vskip5pt

Recall twelve functors involving $\Lambda$-Mod.

\vskip5pt

$\bullet$ \ ${\rm T}_A: A\mbox{-Mod} \longrightarrow \Lambda_{(\phi,\psi)}\mbox{-Mod}$, \quad $X \longmapsto\left(\begin{smallmatrix} X \\ M\otimes_A X\end{smallmatrix}\right)_{1_{M\otimes_A X}, \Psi_X}$.

\vskip5pt

$\bullet$ \ ${\rm T}_B: B\mbox{-Mod} \longrightarrow \Lambda_{(\phi,\psi)}\mbox{-Mod}$, \quad $Y \longmapsto\left(\begin{smallmatrix} N\otimes_BY \\ Y\end{smallmatrix}\right)_{\Phi_Y,1_{N\otimes_BY}}$.

\vskip5pt

If $\phi = \psi=0$, then \ ${\rm T}_AX= \left(\begin{smallmatrix} X \\ M\otimes_A X\end{smallmatrix}\right)_{1, 0}$ \ and  \
${\rm T}_BY= \left(\begin{smallmatrix} N\otimes_BY \\ Y\end{smallmatrix}\right)_{0, 1}.$

\vskip5pt

$\bullet$ \ ${\rm U}_A: \Lambda_{(\phi,\psi)}\mbox{-Mod} \longrightarrow A\mbox{-Mod}, \quad \left(\begin{smallmatrix} X \\ Y \end{smallmatrix}\right)_{f,g} \longmapsto X$.

\vskip5pt

$\bullet$ \ ${\rm U}_B: \Lambda_{(\phi,\psi)}\mbox{-Mod} \longrightarrow B\mbox{-Mod}, \quad \left(\begin{smallmatrix} X \\ Y \end{smallmatrix}\right)_{f,g} \longmapsto Y$.

\vskip5pt

$\bullet$ \ ${\rm H}_A: A\mbox{-Mod} \longrightarrow \Lambda_{(\phi,\psi)}\mbox{-Mod},
\quad X \longmapsto \left(\begin{smallmatrix} X \\ \Hom_A(N, X)\end{smallmatrix}\right)_{\widetilde{\Psi_X}, \ \epsilon'_X}$.

Note that \  $\widetilde{\Psi_X} = \Hom_A(N, \ \Psi_X)\circ\delta'_{M\otimes_AX}$; and
${\rm H}_AX =
\left(\begin{smallmatrix} X \\ \Hom_A(N, X)\end{smallmatrix}\right)_{\widetilde{\widetilde{\Psi_X}}, \ 1}$ in the second expression.

\vskip5pt

$\bullet$ \ ${\rm H}_B: B\mbox{-Mod} \longrightarrow \Lambda_{(\phi,\psi)}\mbox{-Mod},
\quad Y \longmapsto \left(\begin{smallmatrix} \Hom_B(M,Y) \\ Y\end{smallmatrix}\right)_{\epsilon_Y, \widetilde{\Phi_Y}}$.

Note that $\widetilde{\Phi_Y} = \Hom_B(M, \ \Phi_Y)\circ \delta_{N\otimes_BY};$ and ${\rm H}_BY =
\left(\begin{smallmatrix} \Hom_B(M,Y) \\ Y\end{smallmatrix}\right)_{1, \widetilde{\widetilde{\Phi_Y}}}$ in the second expression.

\vskip5pt

If $\phi = \psi=0$, then \ ${\rm H}_A X = \left(\begin{smallmatrix} X \\ \Hom_A(N, X)\end{smallmatrix}\right)_{0, \ \epsilon'_X},  \
{\rm H}_BY = \left(\begin{smallmatrix} \Hom_B(M, Y) \\ Y\end{smallmatrix}\right)_{\epsilon_Y, 0};$
and it is convenient to use the second expression:
${\rm H}_A X = \left(\begin{smallmatrix} X \\ \Hom_A(N, X)\end{smallmatrix}\right)_{0, 1},  \ \ \
{\rm H}_BY = \left(\begin{smallmatrix} \Hom_B(M, Y) \\ Y\end{smallmatrix}\right)_{1, 0}.$

\vskip5pt

$\bullet$ \ ${\rm C}_A: \Lambda_{(\phi, \psi)}\text{\rm-Mod}\longrightarrow A\text{\rm-Mod}$, \quad $\left(\begin{smallmatrix} X\\Y\end{smallmatrix}\right)_{f,g}\longmapsto \Coker g$.
	
\vskip5pt

$\bullet$ \ ${\rm C}_B: \Lambda_{(\phi, \psi)}\text{\rm-Mod}\longrightarrow B\text{\rm-Mod}$, \quad $\left(\begin{smallmatrix} X\\Y\end{smallmatrix}\right)_{f,g}\longmapsto \Coker f$.
	
\vskip5pt

$\bullet$ \ ${\rm K}_A: \Lambda_{(\phi, \psi)}\text{\rm-Mod}\longrightarrow A\text{\rm-Mod}$, \quad $\left(\begin{smallmatrix} X\\Y\end{smallmatrix}\right)_{f,g}\longmapsto \Ker \widetilde{f}$.
	
\vskip5pt
	
$\bullet$ \ ${\rm K}_B: \Lambda_{(\phi, \psi)}\text{\rm-Mod}\longrightarrow B\text{\rm-Mod}$, \quad $\left(\begin{smallmatrix} X\\Y\end{smallmatrix}\right)_{f,g}\longmapsto \Ker \widetilde{g}$.
	
\vskip10pt

If $\phi = \psi=0$, then $\left(\begin{smallmatrix} X \\ 0 \end{smallmatrix}\right)_{0,0}$ is a left $\Lambda$-module for any $A$-module $_AX$,
and $\left(\begin{smallmatrix} 0 \\ Y \end{smallmatrix}\right)_{0,0}$ is a left $\Lambda$-module for any $B$-module $_BY$. (In general, they are not left $\Lambda$-modules.) In this case,
one has extra functors:

\vskip5pt

$\bullet$ \ ${\rm Z}_A: A\mbox{-Mod}\longrightarrow \Lambda_{(0,0)}\mbox{-Mod}$, \quad $X\longmapsto \left(\begin{smallmatrix} X \\ 0 \end{smallmatrix}\right)_{0,0}$.

\vskip5pt

$\bullet$ \ ${\rm Z}_B: B\mbox{-Mod}\longrightarrow \Lambda_{(0,0)}\mbox{-Mod}$, \quad $Y\longmapsto \left(\begin{smallmatrix} 0 \\ Y \end{smallmatrix}\right)_{0,0}$.

\begin{thm} \label{recollments} {\rm ([GrP, 2.4])} \ There are recollements of abelian categories $($in the sense of {\rm [FP]}$)$$:$
$$\xymatrix{A\text{\rm-Mod}\ar[rr]^-{{\rm Z}_A} && \Lambda_{(0,0)}\text{\rm-Mod} \ar[rr]^-{{\rm U}_B}\ar@<-12pt>[ll]_-{{\rm C}_A}\ar@<12pt>[ll]_-{{\rm K}_A}
		&& B\text{\rm-Mod} \ar@<-12pt>[ll]_-{{\rm T}_B} \ar@<12pt>[ll]_-{{\rm H}_B}}$$
\noindent and
$$\xymatrix{B\text{\rm-Mod}\ar[rr]^-{{\rm Z}_B} && \Lambda_{(0,0)}\text{\rm-Mod} \ar[rr]^-{{\rm U}_A}\ar@<-13pt>[ll]_-{{\rm C}_B}\ar@<13pt>[ll]_-{{\rm K}_B}
		&& A\text{\rm-Mod}. \ar@<-13pt>[ll]_-{{\rm T}_A} \ar@<13pt>[ll]_-{{\rm H}_A}}$$
\end{thm}

\vskip10pt

\subsection{Projective (injective) modules}

A left $\Lambda_{(\phi, \psi)}$-module $\left(\begin{smallmatrix} L_1 \\ L_2 \end{smallmatrix}\right)_{f,g}$ is projective if and only if
\ $\left(\begin{smallmatrix} L_1 \\ L_2\end{smallmatrix}\right)_{f,g}\cong {\rm T}_AP \oplus {\rm T}_BQ$
\ for some $P\in \ _A\mathcal P$ \ and $Q\in \ _B\mathcal P$;  and it is injective if and only if
\ $\left(\begin{smallmatrix} L_1 \\ L_2\end{smallmatrix}\right)_{f,g}\cong {\rm H}_AI \oplus {\rm H}_BJ$
\ for some $I\in \ _A\mathcal I$ \ and $J\in \ _B\mathcal I$.

\vskip5pt
Thus, if $\phi = 0 = \phi$, a left $\Lambda_{(0, 0)}$-module $\left(\begin{smallmatrix} L_1 \\ L_2 \end{smallmatrix}\right)_{f,g}$ is projective if and only if
$$\left(\begin{smallmatrix} L_1 \\ L_2\end{smallmatrix}\right)_{f,g}\cong \left(\begin{smallmatrix} P \\ M\otimes_AP\end{smallmatrix}\right)_{1,0}\oplus \left(\begin{smallmatrix} N\otimes_BQ \\ Q \end{smallmatrix}\right)_{0,1}= \left(\begin{smallmatrix} P\oplus (N\otimes_BQ) \\ (M\otimes_AP)\oplus Q\end{smallmatrix}\right)_{\left(\begin{smallmatrix} 1 & 0 \\ 0 & 0 \end{smallmatrix}\right), \left(\begin{smallmatrix} 0 & 0 \\ 0 & 1 \end{smallmatrix}\right)}$$
for some $P\in \ _A\mathcal P$ and $Q\in \ _B\mathcal P$; and it is injective if and only if
$$\left(\begin{smallmatrix} L_1 \\ L_2\end{smallmatrix}\right)_{f,g}\cong \left(\begin{smallmatrix} I \\ \Hom_A(N,I)\end{smallmatrix}\right)_{0, \epsilon'_I}\oplus \left(\begin{smallmatrix} \Hom_B(M,J) \\ J \end{smallmatrix}\right)_{\epsilon_{_J},0}\cong \left(\begin{smallmatrix} I\oplus \Hom_B(M,J) \\ \Hom_A(N,I)\oplus J\end{smallmatrix}\right)_{\left(\begin{smallmatrix} 0 & 0 \\ 0 & \epsilon_{_J} \end{smallmatrix}\right), \left(\begin{smallmatrix} \epsilon'_{_I} & 0 \\ 0 & 0 \end{smallmatrix}\right)}$$
for some $I\in \ _A\mathcal I$ and $J\in \ _B\mathcal I$. Using the second expression of $\Lambda$-modules, a left
$\Lambda_{(0, 0)}$-module $\left(\begin{smallmatrix} L_1 \\ L_2 \end{smallmatrix}\right)_{\widetilde{f},\widetilde{g}}$ is injective if and only if
$$\left(\begin{smallmatrix} L_1 \\ L_2\end{smallmatrix}\right)_{\widetilde{f},\widetilde{g}}\cong \left(\begin{smallmatrix} I \\ \Hom_A(N,I)\end{smallmatrix}\right)_{0, 1}\oplus \left(\begin{smallmatrix} \Hom_B(M,J) \\ J \end{smallmatrix}\right)_{1,0}\cong \left(\begin{smallmatrix} I\oplus \Hom_B(M,J) \\ \Hom_A(N,I)\oplus J\end{smallmatrix}\right)_{\left(\begin{smallmatrix} 0 & 0 \\ 0 & 1 \end{smallmatrix}\right), \left(\begin{smallmatrix} 1 & 0 \\ 0 & 0 \end{smallmatrix}\right).}$$
See [GrP, 3.1].

\vskip10pt
	
\subsection{Cotorsion Pairs} \
Let $\mathcal A$ be an abelian category.   A pair \ $(\mathcal C, \ \mathcal F)$ of classes of objects of  $\mathcal A$ is
a {\it cotorsion pair} (see [S]), if \ $\mathcal C={}^\perp\mathcal F \ \ \mbox{and} \ \ \mathcal F = \mathcal C^{\perp}.$

\vskip5pt

A cotorsion pair $(\mathcal C, \mathcal F)$  is  {\it complete}, if for any object $X\in \mathcal A$, there are exact sequences
$$0\longrightarrow F\longrightarrow C\longrightarrow X\longrightarrow 0, \quad \text{and}\quad
0\longrightarrow X\longrightarrow F'\longrightarrow C'\longrightarrow 0,$$
with $C, \ C'\in \mathcal C$, and \ $F, \ F'\in \mathcal F$.

\begin{prop} \label{completenessandheredity} \ {\rm ([EJ, 7.17])} \  Let \ $\mathcal A$ be an abelian category with  enough projective objects and enough injective objects, and \ $(\mathcal C, \mathcal F)$ a cotorsion pair in
$\mathcal A$. Then the following are equivalent$:$

\vskip5pt

${\rm(i)}$ \   $(\mathcal C, \mathcal F)$ is complete$;$

\vskip5pt

${\rm(ii)}$ \  For any object $X\in \mathcal A$, there is an  exact sequence
$0\rightarrow F\rightarrow C\rightarrow X\rightarrow 0$ with $C\in \mathcal C$ and \ $F\in \mathcal F;$

\vskip5pt

${\rm(iii)}$ \ For an any object $X\in \mathcal A$, there is exact sequence
$0\rightarrow X\rightarrow F'\rightarrow C'\rightarrow 0$ with $C'\in \mathcal C$ and \ $F'\in \mathcal F$. \end{prop}

A cotorsion pair $(\mathcal C, \mathcal F)$ is {\it cogenerated by} a set $\mathcal S$,
if $\mathcal F = \mathcal S^\perp$. One should be careful with this terminology:
in some reference, e.g., in [GT, p.99],  it is also called ``{\it generated by}".

\begin{prop} \label{cogenerated}  \ Let \ $\mathcal A$ be a Grothendieck category with  enough projective objects.
Then any  cotorsion pair in $\mathcal A$ cogenerated by a set is complete.
\end{prop}

This  result is given in [ET2, Theorem 10] for the module category of a ring,
and has the generality by [SS] or [Bec, 1.2.2].

\vskip5pt

A cotorsion pair $(\mathcal C, \mathcal F)$  is  {\it hereditary}, if $\mathcal C$ is closed under the kernel of epimorphisms, and $\mathcal F$ is closed under the cokernel of monomorphisms.

\begin{prop} \label{heredity} \ {\rm ([GR, 1.2.10])}  \ Let \ $\mathcal A$ be an abelian category with  enough projective objects and enough injective objects, and \ $(\mathcal C, \mathcal F)$ a cotorsion pair in
$\mathcal A$. Then the following is equivalent$:$

\vskip5pt

${\rm(i)}$ \   $(\mathcal C, \mathcal F)$ is hereditary$;$

\vskip5pt

${\rm(ii)}$ \  $\mathcal C$ is closed under the kernel of epimorphisms$;$

\vskip5pt

${\rm(iii)}$ \  $\mathcal F$ is closed under the cokernel of monomorphisms$;$

\vskip5pt

${\rm(iv)}$ \  $\Ext^2_\mathcal A(\mathcal C, \mathcal F) = 0;$

\vskip5pt

${\rm(v)}$  \ $\Ext^i_\mathcal A(\mathcal C, \mathcal F) = 0$ for $i\ge 1$.
\end{prop}

The proof of Proposition \ref{heredity} really needs the assumption that abelian category $\mathcal A$ has enough projective objects and enough injective objects.

\subsection{\bf Model structures}  A closed model structure on a category and a model category are introduced by D. Quillen \ [Q1] (see also [Q2]). {\it A closed model structure} on a category $\mathcal M$ is a triple
\ $(\Cofib(\mathcal M)$, \ $\Fib(\mathcal M)$,  \ $\Weq(\mathcal M))$ of classes of morphisms,
where the morphisms in the three classes are respectively called cofibrations, fibrations, and weak equivalences, satisfying ${\rm (CM1) - (CM4)}$:

\vskip5pt

(CM1) \ Let $X\xlongrightarrow{f}Y\xlongrightarrow{g}Z$ be morphisms in $\mathcal M$. If two of the morphisms $f, \ g, \ gf$ are weak equivalences, then so is the third.

\vskip5pt

(CM2) \  If $f$ is a retract of $g$, and $g$ is a cofibration (fibration, weak equivalence), then so is $f$.

\vskip5pt

(CM3) \  Given a commutative square
$$\xymatrix{A\ar[r]^-a \ar[d]_-i & X \ar[d]^-p \\
B\ar[r]^-b \ar@{.>}[ru]^-s & Y }$$
where $i\in \Cofib(\mathcal M)$ and $p\in \Fib(\mathcal M)$, if either $i\in \Weq(\mathcal M)$ or $p\in \Weq(\mathcal M)$,
then there exists a morphism $s: B\longrightarrow X$ such that $a = si, \ \ b = ps$.

\vskip5pt

(CM4) \ Any morphism $f: X\longrightarrow Y$ has a factorizations \ $f=pi$ \ with  $i\in \Cofib(\mathcal M)\cap \Weq(\mathcal M)$ and $p\in \Fib(\mathcal M)$; and also \ $f=p'i'$ with $i'\in \Cofib(\mathcal M)$ and $p'\in \Fib(\mathcal M)\cap \Weq(\mathcal M)$.

\vskip5pt

Following [H1] (also [Hir]), we will call a closed model structure just {\it a model structure}. A category is {\it bicomplete} if it has an arbitrary small limits and colimits.
{\it A model category} is a bicomplete category equipped with a model structure ([H1, 1.1.4]).

\vskip5pt

For a model structure $(\Cofib(\mathcal M)$, \ $\Fib(\mathcal M)$,  \ $\Weq(\mathcal M))$ on category $\mathcal M$ with zero object,
an object $X$ is {\it trivial} if $0 \longrightarrow X $ is a weak equivalence, or, equivalently,  $X\longrightarrow 0$ is a weak equivalence. It is {\it cofibrant} if $0\longrightarrow X$ is a cofibration, and it is {\it fibrant} if $X\longrightarrow 0$ is a fibration.
For a model structure  on category $\mathcal M$ with zero object ($\mathcal M$  is not necessarily a model category), Quillen's homotopy category is the localization $\mathcal M[\Weq(\mathcal M)^{-1}]$, and is denoted by ${\rm Ho}(\mathcal M)$.

\vskip5pt

A model structure on an abelian category is {\it an abelian model structure}, provided that
cofibrations are exactly monomorphisms with cofibrant cokernel, and that fibrations are exactly
epimorphisms with fibrant kernel. This is equivalent to the original definition ([H2, 2.1, 4.2]), see also [Bec, 1.1.3].
{\it An abelian model category} is a bicomplete abelian category equipped with an abelian model structure.

\subsection{\bf Hovey triples} \ Let $\mathcal A$ be an abelian category.
A triple \ $(\mathcal C, \mathcal F, \mathcal W)$ \ of classes of objects is {\it a Hovey triple} in $\mathcal A$ (see [H2]), if it satisfies the conditions:

\vskip5pt

(i) \ The class \  $\mathcal W$ is {\it thick}, i.e., $\mathcal W$ is closed under direct summands,
and if two out of three terms in a short exact sequence are in $\mathcal W$, then so is the third;

\vskip5pt

(ii) \ $(\mathcal C \cap \mathcal W, \ \mathcal F)$ and \ $(\mathcal C, \ \mathcal F \cap \mathcal W)$ \  are complete cotorsion pairs.

\begin{thm} \label{hoveycorrespondence} {\rm (Hovey correspondence) \ ([H2, Theorem 2.2];   also [BR, VIII 3.5, 3.6])} \ Let $\mathcal A$ be an abelian category.
Then there is a one-to-one correspondence between abelian model structures and Hovey triples in \ $\mathcal A$, given by
$$({\rm Cofib}(\mathcal{A}), \ {\rm Fib}(\mathcal{A}), \ {\rm Weq}(\mathcal{A}))\mapsto (\mathcal{C}, \ \mathcal{F}, \ \mathcal W)$$
where \ $\mathcal C  = \{\mbox{cofibrant objects}\}, \ \
\mathcal F  = \{\mbox{fibrant objects}\}, \ \
\mathcal W  = \{\mbox {trivial objects}\}$, with inverse $$(\mathcal{C}, \ \mathcal{F}, \ \mathcal W) \mapsto ({\rm Cofib}(\mathcal{A}), \ {\rm Fib}(\mathcal{A}), \ {\rm Weq}(\mathcal{A}))$$ where
\begin{align*} &{\rm Cofib}(\mathcal{A}) = \{\mbox{monomorphisms with cokernel in} \ \mathcal{C}\}, \ \ \
{\rm Fib}(\mathcal{A})  = \{\mbox{epimorphisms with kernel in} \ \mathcal{F} \}, \\
& {\rm Weq}(\mathcal{A})  = \{pi \ \mid \ i \ \mbox{is monic,} \ \Coker i\in \mathcal{C}\cap \mathcal W, \ p \ \mbox{is epic,} \ \Ker p\in \mathcal{F}\cap \mathcal W\}.\end{align*}
\end{thm}

We stress that, in Theorem \ref{hoveycorrespondence}, \ $\mathcal A$ is not necessarily to be bicomplete: although this is assumed in [H2, Theorem 2.2],
however, the proof given there does not use the assumption ``bicomplete".
(In fact, one can also read this from lines of [Gil2] and [Gil3].)

\vskip5pt

A cofibrantly generated model category has been introduced in [H1, 2.1.17].
Let $\mathcal A$ be a Grothendieck category with enough projective objects.
A Hovey triple $(\mathcal C, \mathcal F, \mathcal W)$  in $\mathcal A$
will be called {\it cofibrantly generated}, if cotorsion pairs $(\mathcal C\cap \mathcal W, \mathcal F)$ and $(\mathcal C, \mathcal F\cap\mathcal W)$ are cogenerated by sets.
Note that a Grothendieck category is always bicomplete (see e.g. [KS, 8.3.27]).

\begin{prop} \label{cofibrantly} {\rm ([Bec, 1.2.7; 1.2.2])} \ Let $\mathcal A$ be a Grothendieck category with enough projective objects.
Then a Hovey triple $(\mathcal C, \mathcal F, \mathcal W)$  in $\mathcal A$
is cofibrantly generated if and only if the corresponding abelian model category $\mathcal A$ is cofibrantly generated. \end{prop}

\subsection{Hereditary Hovey triples} \ A Hovey triple  \ $(\mathcal C, \mathcal F, \mathcal W)$ \  is {\it hereditary}, if both
\ $(\mathcal C \cap \mathcal W, \mathcal F)$ and \ $(\mathcal C, \mathcal F \cap \mathcal W)$ \ are hereditary cotorsion pairs.
Hereditary Hovey triples enjoy the following pleasant property.

\vskip5pt

\begin{thm} \label{Ho} {\rm ([Bec, 1.1.14]; [BR, VIII 4.2]; [Gil4, 4.3])} \ Let  \ $(\mathcal C, \mathcal F, \mathcal W)$  be a  hereditary Hovey triple in abelian category \ $\mathcal A$.  Then \ $\mathcal C \cap \mathcal F$ is a Frobenius category $($with the canonical exact structure$)$, with $\mathcal C \cap \mathcal F\cap \mathcal W$ as
the class of projective-injective objects. The composition $\mathcal C \cap \mathcal F \hookrightarrow \mathcal A \longrightarrow {\rm Ho}(\mathcal A)$ induces a triangle equivalence \ ${\rm Ho}(\mathcal A)\cong (\mathcal C \cap \mathcal F)/(\mathcal C \cap \mathcal F\cap \mathcal W),$
where $(\mathcal C \cap \mathcal F)/(\mathcal C \cap \mathcal F\cap \mathcal W)$ is the stable category of \ $\mathcal C \cap \mathcal F$  modulo \ $\mathcal C \cap \mathcal F\cap \mathcal W$.
\end{thm}

Note that the definition of  ${\rm Ho}(\mathcal A)$ does not need that $\mathcal A$ is bicomplete.
By this result, hereditary Hovey triples $(\mathcal{C}, \ \mathcal{F}, \ \mathcal W)$ with $(\mathcal C \cap \mathcal F) \nsubseteqq \mathcal W$ are of special interest.

\vskip5pt

 Two cotorsion pairs \ $(\Theta, \ \Theta^\perp)$ and \ $(^\perp\Upsilon, \ \Upsilon)$ are {\it compatible}
(see [Gil3]), if \ $\Ext_\mathcal A^1(\Theta, \ \Upsilon)$ $=0$ and  \
$\Theta\cap \Theta^\perp ={}^\perp\Upsilon\cap \Upsilon$. The compatibility depends on the order of two cotorsion pairs. This terminology of
compatible is taken from [HJ].

\vskip5pt

J. Gillespie  gives the following approach to construct all the hereditary Hovey triples.

\vskip5pt

\begin{thm} \label{GHtriple} \ {\rm (Gillespie Theorem) \ ([Gil3, 1.1])} \  Let $\mathcal A$ be an abelian category, and \ $(\Theta, \ \Theta^\perp)$ and \ $(^\perp\Upsilon, \ \Upsilon)$ complete hereditary cotorsion pairs in $\mathcal A$.
If \ $(\Theta, \ \Theta^\perp)$ and $({}^\perp\Upsilon, \ \Upsilon)$ are compatible, then
\ $(^\perp\Upsilon, \ \Theta^\perp, \ \mathcal W)$ is a hereditary Hovey triple, where
\begin{align*} \mathcal W & =
    \{W\in \mathcal A \ \mid \ \exists \ \text{ an exact sequence} \ 0\rightarrow P\rightarrow F\rightarrow W\rightarrow 0 \ \mbox{with} \ F\in \Theta, \ P\in \Upsilon\} \\ & =
    \{W\in \mathcal A \ \mid \ \exists \ \text{an exact sequence} \ 0\rightarrow W\rightarrow P'\rightarrow F'\rightarrow 0 \ \mbox{with} \ P'\in \Upsilon, \ F'\in \Theta\}.
\end{align*}
Conversely, any hereditary Hovey triple in $\mathcal A$ arises in this way. \end{thm}

For later applications, we will call the hereditary Hovey triple \ $(^\perp\Upsilon, \ \Theta^\perp, \ \mathcal W)$ in Theorem \ref{GHtriple}
{\it the Gillespie-Hovey triple},
induced by compatible complete hereditary cotorsion pairs \ $(\Theta, \ \Theta^\perp)$ and \ $(^\perp\Upsilon, \ \Upsilon)$.  Thus,
the Gillespie-Hovey triples are exactly hereditary Hovey triples.

\subsection{Gorenstein rings} A noetherian ring $R$ is {\it a Iwanaga-Gorenstein ring}, or {\it a Gorenstein ring}, if \ ${\rm inj.dim} _RR < \infty$ and \ ${\rm inj.dim} R_R < \infty$.
In this case, it is well-known that

\vskip5pt

$\bullet$ \ ${\rm inj.dim} _RR = {\rm inj.dim} R_R$ \ and  \ $_R \mathcal P^{<\infty} = \ _R \mathcal I^{<\infty}$;

\vskip5pt

$\bullet$ \ ([EJ, p. 211]) \ If \ ${\rm inj.dim} _RR \le n$, then \ $_R \mathcal P^{< \infty} = \ _R \mathcal P^{\le n} = \ _R \mathcal I^{\le n} = \ _R \mathcal I^{<\infty}$, where \ $_R \mathcal P^{\le n}$  \ ($_R \mathcal I^{\le n}$, respectively) is the full subcategory of $R$-Mod consisting of modules  $X$ with ${\rm proj.dim} X \le n$ (${\rm inj.dim} X \le n$, respectively).

\vskip5pt

$\bullet$ \ ([EJ, 11.5.3]) \  ${\rm GP}(R) = \ ^{\bot_{\ge 1}} \ _R \mathcal P = \ ^{\bot_{\ge 1}} \ _R \mathcal P^{< \infty}$, and \ ${\rm GI}(R) = {}_R \mathcal I^{\bot_{\ge 1}} ={}(_R \mathcal I^{< \infty})^{\bot_{\ge 1}}$;

\vskip5pt

$\bullet$ \ ([H2, 8.6]) \  $({\rm GP}(R), \  R\mbox{-Mod}, \ _R\mathcal P^{<\infty})$  and $(R\mbox{-Mod},  \ {\rm GI}(R), \ _R\mathcal P^{<\infty})$ are hereditary Hovey triples in $R$-Mod. In particular,
\ $({\rm GP}(R), \  _R\mathcal P^{<\infty})$  and $({\rm GI}(R), \ _R\mathcal P^{<\infty})$ are complete hereditary  cotorsion pairs.

\section{\bf (Hereditary) cotorsion pairs in Morita rings}

\subsection{Three classes of modules over  a Morita ring} \ Let $\Lambda = \left(\begin{smallmatrix} A & N \\ M & B\end{smallmatrix}\right)$ be a Morita ring. For a class $\mathcal X$ of $A$-modules and a class $\mathcal Y$ of $B$-modules,  define
\begin{align*}\left(\begin{smallmatrix}\mathcal X\\ \mathcal Y\end{smallmatrix}\right): & =  \ \{\left(\begin{smallmatrix} X \\ Y\end{smallmatrix}\right)_{f, g}\in \Lambda\mbox{-}{\rm Mod} \ \mid \   X\in \mathcal X, \ \ Y\in\mathcal Y\};\\[7pt]
\Delta(\mathcal X, \ \mathcal Y): & = \ \{\left(\begin{smallmatrix} L_1 \\ L_2\end{smallmatrix}\right)_{f,g}\in \Lambda\mbox{-}{\rm Mod} \ \mid \ f: M\otimes_AL_1 \longrightarrow L_2  \ \ \text{and}
\ \ g: N\otimes_B L_2 \longrightarrow L_1
\\ & \ \ \ \ \ \ \  \ \text{are monomorphisms}, \ \ \Coker f\in \mathcal Y, \ \ \Coker g\in \mathcal X \}; \\[7pt] \nabla(\mathcal X,\ \mathcal Y): & = \ \{\left(\begin{smallmatrix} L_1 \\ L_2\end{smallmatrix}\right)_{f,g}\in \Lambda\mbox{-}{\rm Mod} \ \mid \ \widetilde{f}: L_1\longrightarrow \Hom_B(M, L_2) \ \text{and} \ \widetilde{g}:  L_2\longrightarrow \Hom_A(N, L_1) \\ &  \ \ \ \ \ \ \  \ \text{are epimorphisms}, \ \Ker\widetilde{f}\in \mathcal X, \ \Ker \widetilde{g}\in \mathcal Y\}.
\end{align*}
In particular, we put  \begin{align*}{\rm Mon}(\Lambda): & = \Delta(A\mbox{\rm-Mod},  \ B\mbox{\rm-Mod}) = \ \{\left(\begin{smallmatrix} L_1 \\ L_2\end{smallmatrix}\right)_{f,g}\in \Lambda\mbox{-}{\rm Mod} \ \mid \ f \ \mbox{and} \ g \ \text{are monomorphisms}\};
\\[7pt] {\rm Epi}(\Lambda): & = \nabla(A\mbox{\rm-Mod},  \ B\mbox{\rm-Mod}) = \ \{\left(\begin{smallmatrix} L_1 \\ L_2\end{smallmatrix}\right)_{f,g}\in \Lambda\mbox{-}{\rm Mod} \ \mid \ \widetilde{f} \ \mbox{and} \ \widetilde{g} \ \text{are epimorphisms}\}.
\end{align*}
They will be called {\it the monomorphism category} and {\it the epimorphism category} of $\Lambda$, respectively.

\vskip5pt

It is clear that if $M\otimes_A N = 0 = N\otimes_BM$, then
$$\Delta(_A\mathcal P,  \ _B\mathcal P) = \ _\Lambda\mathcal P \ \ \ \ \mbox{and} \ \ \ \ \nabla(_A\mathcal I, \ _B\mathcal I) =  \ _\Lambda\mathcal I.$$

\subsection{Constructions on (hereditary) cotorsion pairs}

\begin{thm}\label{ctp1} \ Let \ $\Lambda = \left(\begin{smallmatrix} A & N \\ M & B\end{smallmatrix}\right)$  be a Morita ring  with $\phi = 0=\psi$. Let \ $(\mathcal U, \ \mathcal X)$ and \ $(\mathcal V, \ \mathcal Y)$  be cotorsion pairs in $A\mbox{-}{\rm Mod}$ and $B$\mbox{\rm-Mod}, respectively.

\vskip5pt

$(1)$ \ If \ $\Tor^A_1(M, \ \mathcal U)=0 = \Tor^B_1(N, \ \mathcal V)$,
then \ $({}^\perp\left(\begin{smallmatrix}\mathcal X\\ \mathcal Y\end{smallmatrix}\right), \ \left(\begin{smallmatrix}\mathcal X\\ \mathcal Y\end{smallmatrix}\right))$ is a cotorsion pair in $\Lambda${\rm-Mod}$;$
and moreover, it is hereditary if and only if so are \ $(\mathcal U, \ \mathcal X)$  and  $(\mathcal V, \ \mathcal Y)$.

\vskip10pt

$(2)$ \ If \ $\Ext_A^1(N, \ \mathcal X) =0 = \Ext_B^1(M, \ \mathcal Y)$,
then \ $(\left(\begin{smallmatrix}\mathcal U\\ \mathcal V\end{smallmatrix}\right), \ \left(\begin{smallmatrix}\mathcal U\\ \mathcal V\end{smallmatrix}\right)^\perp)$ is a cotorsion pair in $\Lambda${\rm-Mod}$;$ and moreover, it is hereditary if and only if so are \ $(\mathcal U, \ \mathcal X)$  and  $(\mathcal V, \ \mathcal Y)$.
\end{thm}

\vskip5pt

\begin{thm}\label{ctp6} \ Let $\Lambda = \left(\begin{smallmatrix} A & N \\	M & B\end{smallmatrix}\right)$ be a Morita ring  with $M\otimes_A N = 0 = N\otimes_BM$. Let \ $(\mathcal U, \ \mathcal X)$ and $(\mathcal V, \ \mathcal Y)$  be cotorsion pairs in $A\mbox{-}{\rm Mod}$ and  $B$\mbox{\rm-Mod}, respectively.  Then

\vskip5pt

$(1)$ \ $(\Delta(\mathcal U, \ \mathcal V), \ \Delta(\mathcal U, \ \mathcal V)^\perp)$ is a cotorsion pair in $\Lambda${\rm-Mod}.

\vskip5pt

Moreover, if \ $M_A$ and \ $N_B$ are flat, then \ $(\Delta(\mathcal U, \ \mathcal V),$ \ $\Delta(\mathcal U, \ \mathcal V)^\perp)$  is hereditary if and only if so are \ $(\mathcal U, \ \mathcal X)$  and \ $(\mathcal V, \ \mathcal Y)$.

\vskip10pt

$(2)$ \  $(^{\perp}\nabla(\mathcal X, \ \mathcal Y), \ \nabla(\mathcal X, \ \mathcal Y))$ is a cotorsion pair in $\Lambda${\rm-Mod}.

\vskip5pt

Moreover, if \ $_BM$ and \ $_AN$ are projective, then
\ $(^{\perp}\nabla(\mathcal X, \ \mathcal Y),$ \ $\nabla(\mathcal X, \ \mathcal Y))$  is hereditary if and only if
so are \ $(\mathcal U, \ \mathcal X)$  and  $(\mathcal V, \ \mathcal Y)$.
\end{thm}

\begin{Notation} \label{not} \ For convenience, we will call the cotorsion pairs in Theorem {\rm\ref{ctp1}}
{\it the cotorsion pairs in Series} {\rm I}$;$
and the ones in Theorem {\rm\ref{ctp6}}  {\it the cotorsion pairs in Series} {\rm II}.\end{Notation}

\begin{exm}\label{irem1}  {\it The condition \ ``$M\otimes_A N = 0 = N\otimes_BM$" in {\rm Theorem \ref{ctp6}} can not be weakened as \ ``$\phi = 0 = \psi$", in general.

\vskip5pt

For example, taking \ $\Lambda = \left(\begin{smallmatrix} A & A \\A & A\end{smallmatrix}\right)$ with $A\ne 0$ and $\phi = 0=\psi$. Then for any class \ $\mathcal U \subseteq A\mbox{-}{\rm Mod}$ and any class \  $\mathcal V\subseteq B\mbox{-}{\rm Mod}$,  one has \ $\Delta(\mathcal U,  \mathcal V)= \{0\}$. In fact, if \  $\left(\begin{smallmatrix} L_1 \\ L_2\end{smallmatrix}\right)_{f, g}\in \Delta(\mathcal U,  \mathcal V),$ then $fg = 0=gf$. However,
$f: L_1 \longrightarrow L_2$ and $g: L_2 \longrightarrow L_1$ are monomorphisms.
Thus $L_1 = 0 = L_2$.

\vskip5pt

But $\{0\}$  can not occur in any cotorsion pair $($since $\Lambda\ne 0)$.}
\end{exm}

We will compare the cotorsion pairs in Series I with the corresponding cotorsion pairs in Series II. Comparing cotorsion pair \ $({}^\perp\left(\begin{smallmatrix}\mathcal X\\ \mathcal Y\end{smallmatrix}\right), \ \left(\begin{smallmatrix}\mathcal X\\ \mathcal Y\end{smallmatrix}\right))$  in Theorem \ref{ctp1}(1)
with  cotorsion pair \ $(\Delta(\mathcal U, \ \mathcal V), \ \Delta(\mathcal U, \ \mathcal V)^\perp)$ in Theorem \ref{ctp6}(1), we get the assertion (1) below;
comparing cotorsion pair \ $(\left(\begin{smallmatrix} \mathcal U\\ \mathcal V\end{smallmatrix}\right), \ \left(\begin{smallmatrix} \mathcal U\\ \mathcal V\end{smallmatrix}\right)^\perp)$ in Theorem \ref{ctp1}(2) with
cotorsion pair \ $(^{\perp}\nabla(\mathcal X, \ \mathcal Y), \ \nabla(\mathcal X, \ \mathcal Y))$ in Theorem \ref{ctp6}(2),
we get the assertion (2) below.

\begin{thm}\label{compare} \ Let $\Lambda = \left(\begin{smallmatrix} A & N \\
	M & B\end{smallmatrix}\right)$ be a Morita ring  with $M\otimes_A N = 0 = N\otimes_BM$. Suppose that $(\mathcal U, \ \mathcal X)$ and $(\mathcal V, \ \mathcal Y)$ are cotorsion pairs in $A\mbox{-}{\rm Mod}$ and $B$\mbox{\rm-Mod}, respectively.

\vskip5pt

$(1)$ \ If \ $\Tor^A_1(M, \ \mathcal U) =0 = \Tor^B_1(N, \ \mathcal V)$, then the  cotorsion pairs \ $$({}^\perp\left(\begin{smallmatrix}\mathcal X\\ \mathcal Y\end{smallmatrix}\right), \ \left(\begin{smallmatrix}\mathcal X\\ \mathcal Y\end{smallmatrix}\right)) \ \ \mbox{and} \ \ (\Delta(\mathcal U, \ \mathcal V), \ \Delta(\mathcal U, \ \mathcal V)^\perp)$$ in $\Lambda$-{\rm Mod}  have a relation   \ $\Delta(\mathcal U, \ \mathcal V)^{\bot} \subseteq \left(\begin{smallmatrix}\mathcal X\\ \mathcal Y\end{smallmatrix}\right),$ \  or equivalently,
\ \ $^\bot\left(\begin{smallmatrix}\mathcal X\\ \mathcal Y\end{smallmatrix}\right)\subseteq \Delta(\mathcal U, \ \mathcal V)$.

\vskip10pt

$(2)$ \ If \ $\Ext_A^1(N, \ \mathcal X)=0 = \Ext_B^1(M, \ \mathcal Y)$, then the  cotorsion pairs  $$(\left(\begin{smallmatrix} \mathcal U\\ \mathcal V\end{smallmatrix}\right), \ \left(\begin{smallmatrix} \mathcal U\\ \mathcal V\end{smallmatrix}\right)^\perp) \ \ \mbox{and} \ (^{\perp}\nabla(\mathcal X, \ \mathcal Y), \ \nabla(\mathcal X, \ \mathcal Y))$$ in $\Lambda$-{\rm Mod} have a relation  \ $^{\bot}\nabla(\mathcal X, \ \mathcal Y) \subseteq \left(\begin{smallmatrix}\mathcal U\\ \mathcal V\end{smallmatrix}\right)$, \  or equivalently,
\ \ $\left(\begin{smallmatrix}\mathcal U\\ \mathcal V\end{smallmatrix}\right)^\bot \subseteq \nabla(\mathcal X, \ \mathcal Y).$
\end{thm}

\begin{rem}\label{irem2} \ {\it If  $M=0$ or $N = 0$, then {\rm Theorems \ref{ctp1},  \ref{ctp6} and  \ref{compare}} have been obtained by {\rm R. M. Zhu}, {\rm Y. Y. Peng} and {\rm N. Q. Ding} {\rm [ZPD]}. In particular, in that case one has $$({}^\perp\left(\begin{smallmatrix}\mathcal X\\ \mathcal Y\end{smallmatrix}\right), \ \left(\begin{smallmatrix}\mathcal X\\ \mathcal Y\end{smallmatrix}\right)) = (\Delta(\mathcal U, \ \mathcal V), \ \Delta(\mathcal U, \ \mathcal V)^\perp)$$  and $$(\left(\begin{smallmatrix} \mathcal U\\ \mathcal V\end{smallmatrix}\right), \ \left(\begin{smallmatrix} \mathcal U\\ \mathcal V\end{smallmatrix}\right)^\perp) =(^{\perp}\nabla(\mathcal X, \ \mathcal Y), \ \nabla(\mathcal X, \ \mathcal Y)).$$ See {\rm [ZPD, Proposition 3.7]}. But, in general,
{\bf they are not true!} See {\rm Example \ref{ie}}.} \end{rem}

\subsection{Induced isomorphisms between $\Ext^1$} To prove Theorems \ref{ctp1}, we need some preparations.
In the following lemma, functors ${\rm F}$ and ${\rm G}$ are not required to be exact. This is important for applications.

\begin{lem} \label{adj} \ Let $R$ and $S$ be rings, $({\rm F}, \ {\rm G})$ an adjoint pair with ${\rm F}: R\mbox{\rm -Mod}\longrightarrow S\mbox{\rm -Mod}$.

\vskip5pt

$(1)$ \ For an $X\in R$-{\rm Mod}, if \ $0 \rightarrow K \rightarrow P \rightarrow X \rightarrow 0$ is exact with $P$  projective,
such that \ $0\rightarrow {\rm F}K\rightarrow {\rm F}P\rightarrow {\rm F}X\rightarrow 0$ is exact with ${\rm F}P$ projective, then \
$\Ext_S^1({\rm F}X, \ Y)\cong \Ext_R^1(X, \ {\rm G}Y), \ \ \forall \ Y\in S\text{\rm-Mod}.$

\vskip5pt

$(2)$ \ For a $Y\in S$-{\rm Mod}, if \ $0 \rightarrow Y \rightarrow I \rightarrow C \rightarrow 0$ is exact with $I$ injective,
such that \ $0\rightarrow {\rm G}Y\rightarrow {\rm G}I\rightarrow {\rm G}C\rightarrow 0$ is exact with ${\rm G}I$ injective, then  \ $\Ext_S^1({\rm F}X, \ Y)\cong \Ext_R^1(X, \ {\rm G}Y), \ \ \forall \ X\in R\text{\rm-Mod}.$
\end{lem}
\begin{proof} \ (1) \  Applying $\Hom_R(-, {\rm G}Y)$ to \ $0 \rightarrow K \rightarrow P \rightarrow X \rightarrow 0$
and applying $\Hom_S(-, Y)$ to \ $0\rightarrow {\rm F}K\rightarrow {\rm F}P\rightarrow {\rm F}X\rightarrow 0$, one gets a commutative diagram with exact rows
$$\xymatrix@C=.5cm{\Hom_S({\rm F}P, \ Y) \ar[r]\ar[d]^-\cong & \Hom_S({\rm F}K, \ Y)\ar[r]\ar[d]^-\cong & \Ext_S^1({\rm F}X, \ Y)\ar[r]\ar@{-->}[d] & 0 \\
\Hom_R(P, \ {\rm G}Y)\ar[r] & \Hom_R(K, \ {\rm G}Y)\ar[r] & \Ext_R^1(X, \ {\rm G}Y)\ar[r] & 0.}$$
Then the assertion follows from the Five Lemma.

\vskip5pt

The assertion $(2)$ is the dual of (1).
\end{proof}

\begin{lem}\label{extadj1} \ Let $\Lambda = \left(\begin{smallmatrix} A & N \\
	M & B\end{smallmatrix}\right)$ be a Morita ring  with $\phi = 0=\psi$, \  $X\in A\mbox{-}{\rm Mod}$ and $Y\in B\mbox{-}{\rm Mod}$. Then for any $L =\left(\begin{smallmatrix} L_1\\ L_2\end{smallmatrix}\right)_{f,g}\in \Lambda\mbox{-}{\rm Mod}$ one has

\vskip5pt

$(1)$ \ If \ $\Tor_1^A(M, \ X)=0$, then $\Ext_\Lambda^1({\rm T}_AX, \ L)\cong \Ext_A^1(X, \ {\rm U}_AL).$

\vskip5pt

$(2)$ \ If \ $\Tor_1^B(N, \ Y)=0$, then \ $\Ext_\Lambda^1({\rm T}_BY, \ L)\cong \Ext_B^1(Y, \ {\rm U}_BL).$

\vskip5pt

$(3)$ \ If \ $\Ext_A^1({}N, \ X)=0$, then \ $\Ext_A^1({\rm U}_AL, \ X)\cong \Ext_\Lambda^1(L, \ {\rm H}_AX).$

\vskip5pt

$(4)$ \ If \ $\Ext_B^1(M, \ Y)=0$, then \ $\Ext_B^1({\rm U}_BL, \ Y)\cong \Ext_\Lambda^1(L, \ {\rm H}_BY).$

\end{lem}

\begin{proof} We only justify (1) and (3). The assertions (2) and (4) can be similarly proved.

\vskip5pt

(1) \ Take an exact sequence \ $0\rightarrow K\rightarrow P\rightarrow X\rightarrow 0$ with $P$ projective.
Since by assumption \ $\Tor_1^A(M, \ X)=0$, one has an exact sequence of $B$-modules
$$0\longrightarrow M\otimes_AK \longrightarrow M\otimes_AP \longrightarrow M\otimes_AX \longrightarrow 0.$$
Applying ${\rm T}_A$ (note that \ ${\rm T}_A$ is not an exact functor), one gets an exact sequence of $\Lambda$-modules
$$0\longrightarrow \left(\begin{smallmatrix} K \\ M\otimes_AK\end{smallmatrix}\right)_{1, 0}\longrightarrow \left(\begin{smallmatrix} P \\ M\otimes_AP\end{smallmatrix}\right)_{1, 0}
\longrightarrow \left(\begin{smallmatrix} X \\ M\otimes_AX\end{smallmatrix}\right)_{1, 0}\longrightarrow 0$$
where $\left(\begin{smallmatrix} P \\ M\otimes_AP\end{smallmatrix}\right)_{1, 0}$ \ is a projective $\Lambda$-module.
Consider  adjoint pair $({\rm T}_A, \ {\rm U}_A)$ between $A$-Mod and $\Lambda$-Mod.
Applying Lemma \ref{adj}$(1)$ to  $X$,  one gets $\Ext_\Lambda^1({\rm T}_AX, \ L)\cong \Ext_A^1(X, \ {\rm U}_AL).$

\vskip5pt

(3) \ Take an exact sequence \ $0\rightarrow X\rightarrow I\rightarrow C\rightarrow 0$ with $I$ injective.
Since by assumption \ $\Ext_A^1(N, \ X)=0$, one has an exact sequence of $B$-modules
$$0\longrightarrow \Hom_A(N, X) \longrightarrow \Hom_A(N, I) \longrightarrow \Hom_A(N, C) \longrightarrow 0.$$
Applying ${\rm H}_A$ (note that \ ${\rm H}_A$ is also not an exact functor) one gets an exact sequence of $\Lambda$-modules
$$0\longrightarrow \left(\begin{smallmatrix} X \\ \Hom_A(N, X)\end{smallmatrix}\right)_{0, \epsilon_X} \longrightarrow \left(\begin{smallmatrix} I \\ \Hom_A(N, I)\end{smallmatrix}\right)_{0, \epsilon_I}
\longrightarrow \left(\begin{smallmatrix} C \\ \Hom_A(N, C) \end{smallmatrix}\right)_{0, \epsilon_C}\longrightarrow 0$$
where $\left(\begin{smallmatrix} I \\ \Hom_A(N, I)\end{smallmatrix}\right)_{0, \epsilon_I}$ is an injective $\Lambda$-module.
Consider  adjoint pair $({\rm U}_A, \ {\rm H}_A)$ between $\Lambda$-Mod and $A$-Mod.
Applying Lemma \ref{adj}$(2)$ to  $X$,  one gets
$\Ext_A^1({\rm U}_AL, \ X)\cong \Ext_\Lambda^1(L, \ {\rm H}_AX).$
\end{proof}

\begin{lem}\label{destheta} \ Let $\Lambda = \left(\begin{smallmatrix} A & N \\
	M & B\end{smallmatrix}\right)$ be a Morita ring  with $\phi = 0=\psi$, \ $\mathcal X \subseteq A\mbox{-}{\rm Mod}$, and $\mathcal Y\subseteq B\mbox{-}{\rm Mod}$.

\vskip5pt

$(1)$  \ If \ $\Tor_1^A(M, \ \mathcal X)=0 = \Tor_1^B(N, \ \mathcal Y)$, then
\ $\binom{\mathcal X^\perp}{\mathcal Y^\perp} = {\rm T}_A(\mathcal X)^\perp \ \cap \ {\rm T}_B(\mathcal Y)^\perp.$

\vskip5pt

$(2)$ \ If \  $\Ext_A^1(N, \ \mathcal X)=0=\Ext_B^1(M, \ \mathcal Y)$, then \ $\binom{^\perp\mathcal X}{^\perp\mathcal Y} = \ ^\perp{\rm H}_A(\mathcal X) \ \cap \ ^\perp{\rm H}_B(\mathcal Y).$
\end{lem}
\begin{proof}
(1) \ By definition \ $L=\left(\begin{smallmatrix} L_1 \\ L_2\end{smallmatrix}\right)_{f,g}\in \binom{\mathcal X^\perp}{\mathcal Y^\perp}$ if and only if
$L_1\in \mathcal X^\perp$ and $L_2\in \mathcal Y^\perp$, or equivalently,
$\Ext_A^1(\mathcal X, \ L_1)=0 =\Ext_B^1(\mathcal Y, \ L_2)$.
Since by assumption  $\Tor_1^A(M, \ \mathcal X)=0 =\Tor_1^B(N, \ \mathcal Y)$, it follows from Lemma \ref{extadj1}(1) and (2) that
$\Ext_A^1(\mathcal X, \ L_1)\cong \Ext_\Lambda^1({\rm T}_A(\mathcal X), \ L)$ and \ $\Ext_B^1(\mathcal Y, \ L_2) \cong \Ext_\Lambda^1({\rm T}_B(\mathcal Y), \ L).$
Thus,
$L=\left(\begin{smallmatrix} L_1 \\ L_2\end{smallmatrix}\right)_{f,g}\in \binom{\mathcal X^\perp}{\mathcal Y^\perp}$ if and only if
$$\Ext_\Lambda^1({\rm T}_A(\mathcal X), \ L) = 0 = \Ext_\Lambda^1({\rm T}_B(\mathcal Y), \ L)$$
i.e., $L\in {\rm T}_A(\mathcal X)^\perp \ \cap \ {\rm T}_B(\mathcal Y)^\perp$.

\vskip5pt

(2) \ Similarly,  $L=\left(\begin{smallmatrix} L_1 \\ L_2\end{smallmatrix}\right)_{f,g}\in \binom{{}^\perp\mathcal X}{^\perp\mathcal Y}$
if and only if $\Ext_A^1(L_1, \ \mathcal X)=0$ and $\Ext_B^1(L_2, \ \mathcal Y)=0$.
Since \ $\Ext_A^1(N, \ \mathcal X)=0$ and \ $\Ext_B^1(M, \ \mathcal Y)=0$,  by Lemma \ref{extadj1}(3) and (4),
$\Ext_A^1(L_1, \ \mathcal X)\cong \Ext_\Lambda^1(L, \ {\rm H}_A(\mathcal X))$ and \ $\Ext_B^1(L_2, \ \mathcal Y) \cong \Ext_\Lambda^1(L, \ {\rm H}_B(\mathcal Y)).$
Thus,
$L\in \left(\begin{smallmatrix}^\perp\mathcal X\\ ^\perp\mathcal Y\end{smallmatrix}\right)$ if and only if $L\in \ ^\perp{\rm H}_A(\mathcal X) \ \cap \ ^\perp{\rm H}_B(\mathcal Y)$.
\end{proof}

\subsection{Proof of Theorem \ref{ctp1}} $(1)$ \ To prove that \ $({}^\perp\binom{\mathcal X}{\mathcal Y}, \ \left(\begin{smallmatrix}\mathcal X\\ \mathcal Y\end{smallmatrix}\right))$ is a cotorsion pair,
it suffices to show  \ $\left(\begin{smallmatrix}\mathcal X\\ \mathcal Y\end{smallmatrix}\right) = (^\perp\left(\begin{smallmatrix}\mathcal X\\ \mathcal Y\end{smallmatrix}\right))^\perp.$
\vskip5pt
In fact, since \ $(\mathcal U, \ \mathcal X)$ and \ $(\mathcal V, \ \mathcal Y)$ are cotorsion pairs,
it follows that \ $\left(\begin{smallmatrix}\mathcal X\\ \mathcal Y\end{smallmatrix}\right) = \binom{\mathcal U^\perp}{\mathcal V^\perp}$.
Since by assumption \ $\Tor^A_1(M, \ \mathcal U)=0 = \Tor^B_1(N, \ \mathcal V)$, it follows from
Lemma \ref{destheta}(1) that
$$\left(\begin{smallmatrix}\mathcal U^\perp\\\mathcal V^\perp\end{smallmatrix}\right) = {\rm T}_A(\mathcal U)^\perp \ \cap \ {\rm T}_B(\mathcal V)^\perp =
({\rm T}_A(\mathcal U) \ \cup \ {\rm T}_B(\mathcal V))^\perp.$$
Thus
\begin{align*} (^\perp\left(\begin{smallmatrix}\mathcal X\\ \mathcal Y\end{smallmatrix}\right))^\perp & = (^\perp\left(\begin{smallmatrix}\mathcal U^\perp\\ \mathcal V^\perp\end{smallmatrix}\right))^\perp
= \{^\perp [({\rm T}_A(\mathcal U) \ \cup \ {\rm T}_B(\mathcal V))^\perp]\}^\perp \\ &
= ({\rm T}_A(\mathcal U) \ \cup \ {\rm T}_B(\mathcal V))^\perp = \left(\begin{smallmatrix}\mathcal X\\ \mathcal Y\end{smallmatrix}\right)
\end{align*}
here one uses the fact \ $(^\bot(\mathcal S^\bot))^\bot = \mathcal S^\bot,$  for any class $\mathcal S$ of modules.

\vskip5pt

If $(\mathcal U, \ \mathcal X)$  and  $(\mathcal V, \ \mathcal Y)$ are hereditary, then $\mathcal X$ and $\mathcal Y$ are closed under taking the cokernels of monomorphisms.
By the construction of $\left(\begin{smallmatrix}\mathcal X\\ \mathcal Y\end{smallmatrix}\right)$, it is clear that $\left(\begin{smallmatrix}\mathcal X\\ \mathcal Y\end{smallmatrix}\right)$ is also closed under taking the cokernels of monomorphisms, i.e.,
$({}^\perp\left(\begin{smallmatrix}\mathcal X\\ \mathcal Y\end{smallmatrix}\right), \ \left(\begin{smallmatrix}\mathcal X\\ \mathcal Y\end{smallmatrix}\right))$ is hereditary.

\vskip5pt

Conversely, let $({}^\perp\left(\begin{smallmatrix}\mathcal X\\ \mathcal Y\end{smallmatrix}\right), \ \left(\begin{smallmatrix}\mathcal X\\ \mathcal Y\end{smallmatrix}\right))$ be hereditary.
Using functors
\ ${\rm Z}_A$  and \ ${\rm Z}_B$, one sees that $\mathcal X$ and $\mathcal Y$ are closed under taking the cokernels of monomorphisms, i.e., $(\mathcal U, \ \mathcal X)$  and  $(\mathcal V, \ \mathcal Y)$ are hereditary.

\vskip5pt

$(2)$   \ Similarly, it suffices to show \ $\left(\begin{smallmatrix}\mathcal U\\ \mathcal V\end{smallmatrix}\right) = \ ^\perp(\left(\begin{smallmatrix}\mathcal U\\ \mathcal V\end{smallmatrix}\right)^\perp).$
In fact, by Lemma \ref{destheta}(2) one has
$$\left(\begin{smallmatrix}\mathcal U\\ \mathcal V\end{smallmatrix}\right) = \left(\begin{smallmatrix}^\perp\mathcal X\\ ^\perp\mathcal Y\end{smallmatrix}\right) = \ ^\perp{\rm H}_A(\mathcal X) \ \cap \ ^\perp{\rm H}_B(\mathcal Y) =
\ ^\perp ({\rm H}_A(\mathcal X) \ \cup \ {\rm H}_B(\mathcal Y)).$$
Thus
\begin{align*} \ ^\perp(\left(\begin{smallmatrix}\mathcal U\\ \mathcal V\end{smallmatrix}\right)^\perp) & = \ ^\perp(\left(\begin{smallmatrix}^\perp\mathcal X\\ ^\perp\mathcal Y\end{smallmatrix}\right)^\perp)
= \ ^\perp\{ [^\perp ({\rm H}_A(\mathcal X) \ \cup \ {\rm H}_B(\mathcal Y))]^\perp\} \\ &
= \ ^\perp ({\rm H}_A(\mathcal X) \ \cup \ {\rm H}_B(\mathcal Y)) =\left(\begin{smallmatrix}\mathcal U\\ \mathcal V\end{smallmatrix}\right)
\end{align*}
here one uses the fact \ $^\bot((^\bot\mathcal S)^\bot) = \ ^\bot\mathcal S,$ for any class $\mathcal S$ of modules.

\vskip5pt

If $(\mathcal U, \ \mathcal X)$  and  $(\mathcal V, \ \mathcal Y)$ are hereditary, then $\mathcal U$ and $\mathcal V$ are closed under taking the kernels of epimorphisms.
By construction,  $\left(\begin{smallmatrix}\mathcal U\\ \mathcal V\end{smallmatrix}\right)$ is also closed under taking the kernels of epimorphisms, i.e.,
\ $(\left(\begin{smallmatrix}\mathcal U\\ \mathcal V\end{smallmatrix}\right), \ \left(\begin{smallmatrix}\mathcal U\\\mathcal V\end{smallmatrix}\right)^\perp)$ is hereditary. One can see the converse, by using functors
\ ${\rm Z}_A$  and \ ${\rm Z}_B$. \hfill $\square$

\subsection{Induced isomorphisms between $\Ext^1$ (continued)}

\begin{lem}\label{extadj2}  \ Let $\Lambda = \left(\begin{smallmatrix} A & N \\
	M & B\end{smallmatrix}\right)$ be a Morita ring  with $\phi = 0=\psi$,  \   $L =\left(\begin{smallmatrix} L_1\\ L_2\end{smallmatrix}\right)_{f,g}$  a $\Lambda$-module.

\vskip5pt

$(1)$ \ If $g$ is a monomorphism, then \ $\Ext_A^1({\rm C}_AL, \ X)\cong \Ext_\Lambda^1(L, \ {\rm Z}_AX)$, \ $\forall \ X\in A\mbox{\rm -Mod}$.

\vskip5pt

$(2)$ \ If $f$ is a monomorphism, then \ $\Ext_B^1({\rm C}_BL, \ Y)\cong \Ext_\Lambda^1(L, \ {\rm Z}_BY)$, \ $\forall \ Y\in B\mbox{\rm -Mod}$.

\vskip5pt

$(3)$ \ If $\widetilde{f}$ is an epimorphism, then \ $\Ext_\Lambda^1({\rm Z}_AX, \ L)\cong \Ext_A^1(X, \ {\rm K}_AL)$, \ $\forall \ X\in A\mbox{\rm -Mod}$.

\vskip5pt

$(4)$ \ If $\widetilde{g}$ is an epimorphism, then \ $\Ext_\Lambda^1({\rm Z}_BY, \ L)\cong \Ext_B^1(Y, \ {\rm K}_BL)$, \ $\forall \ Y\in B\mbox{\rm -Mod}$.
\end{lem}

\begin{proof} We only prove (1) and (3). The assertions (2) and (4) can be similarly proved.

\vskip5pt

(1) \ Taking an exact sequence
$$0\longrightarrow \left(\begin{smallmatrix} K_1 \\ K_2\end{smallmatrix}\right)_{s,t}\xlongrightarrow{\left(\begin{smallmatrix} i_1 \\ i_2 \end{smallmatrix}\right)} \left(\begin{smallmatrix} P_1 \\ P_2\end{smallmatrix}\right)_{u,v}\xlongrightarrow{\left(\begin{smallmatrix}p_1 \\ p_2\end{smallmatrix}\right)} \left(\begin{smallmatrix} L_1 \\ L_2\end{smallmatrix}\right)_{f,g}\longrightarrow 0$$
with $\left(\begin{smallmatrix} P_1 \\ P_2\end{smallmatrix}\right)_{u,v}$ a projective module, one gets a commutative diagram with exact rows:
$$\xymatrix{& N\otimes_B K_2 \ar[r]^-{1\otimes i_2}\ar[d]_-t & N\otimes_B P_2\ar[d]_-v\ar[r]^-{1\otimes p_2} & N\otimes_B L_2\ar[d]^-g \ar[r] & 0 \\
0\ar[r] & K_1\ar[r]^-{i_1} & P_1\ar[r]^-{p_1} & L_1\ar[r] & 0}$$
Since $g$ is a monomorphism, by Snake Lemma
$$0\longrightarrow \Coker t\longrightarrow \Coker v\longrightarrow \Coker g\longrightarrow 0$$
is  exact. Since $P = \left(\begin{smallmatrix} P_1 \\ P_2\end{smallmatrix}\right)_{u,v}$ is projective,
 $\Coker v = {\rm C}_A P$ is a projective $A$-module.

\vskip5pt

Consider  adjoint pair $({\rm C}_A, \ {\rm Z}_A)$ between $\Lambda$-Mod and $A$-Mod. (Note that  ${\rm C}_A$ is not exact.)
Applying Lemma \ref{adj}(1) to  $L =\left(\begin{smallmatrix} L_1\\ L_2\end{smallmatrix}\right)_{f,g}$,  one gets
$$\Ext_A^1({\rm C}_AL, \ X) \cong \Ext_\Lambda^1(L, \ {\rm Z}_AX), \ \forall \ X\in A\mbox{\rm -Mod}.$$

\vskip5pt

(3) \ Similarly, taking an exact sequence of $\Lambda$-modules
$$0\longrightarrow \left(\begin{smallmatrix} L_1 \\ L_2\end{smallmatrix}\right)_{f,g}\xlongrightarrow{\left(\begin{smallmatrix} \sigma_1 \\ \sigma_2 \end{smallmatrix}\right)}
\left(\begin{smallmatrix} I_1 \\ I_2 \end{smallmatrix}\right)_{u,v}\xlongrightarrow{\left(\begin{smallmatrix}\pi_1 \\ \pi_2\end{smallmatrix}\right)}
\left(\begin{smallmatrix} C_1 \\ C_2 \end{smallmatrix}\right)_{s,t}\longrightarrow 0$$
with $\left(\begin{smallmatrix} I_1 \\ I_2 \end{smallmatrix}\right)_{u,v}$  an injective module, one gets a commutative diagram with exact rows:
$$\xymatrix{& M\otimes_A L_1\ar[r]^-{1\otimes \sigma_1}\ar[d]_-{f} &  M\otimes_A I_1 \ar[r]^-{1\otimes \pi_1}\ar[d]_-{u} &  M\otimes_A C_1\ar[r]\ar[d]^-{s} & 0 \\
0\ar[r] & L_2\ar[r]^-{\sigma_2} & I_2 \ar[r]^-{\pi_2} & C_2\ar[r] & 0}$$
Using adjoint isomorphism, one gets a commutative diagram with exact rows:
$$\xymatrix{0\ar[r] & L_1\ar[r]^-{\sigma_1}\ar[d]_-{\widetilde{f}} & I_1 \ar[r]^-{\pi_1}\ar[d]_-{\widetilde{u}} & C_1\ar[r]\ar[d]^-{\widetilde{s}} & 0 \\
0\ar[r] & \Hom_B(M, L_2)\ar[r]^-{(M, \sigma_2)} & \Hom_B(M, I_2) \ar[r]^-{(M, \pi_2)} & \Hom_B(M, C_2)}$$
Since $\widetilde{f}$ is an epimorphism, by Snake Lemma that
$$0\longrightarrow \Ker\widetilde{f}\longrightarrow \Ker\widetilde{u} \longrightarrow \Ker\widetilde{s}\longrightarrow 0$$
is exact. Since $I = \left(\begin{smallmatrix} I_1 \\ I_2\end{smallmatrix}\right)_{u, v}$ is injective,
 $\Ker \widetilde{u} = {\rm K}_A I$ is an injective $A$-module.

\vskip5pt

Consider adjoint pair $({\rm Z}_A, \ {\rm K}_A)$ between $A$-Mod and $\Lambda$-Mod. Applying Lemma \ref{adj}$(2)$ to  $L =\left(\begin{smallmatrix} L_1\\ L_2\end{smallmatrix}\right)_{f,g}$,  one gets
$\Ext_\Lambda^1({\rm Z}_AX, \ L)\cong \Ext_A^1(X, \ {\rm K}_AL), \ \forall \ X\in A\mbox{\rm -Mod}.$
\end{proof}

\subsection{Key lemmas for Theorem \ref{ctp6}} \ The following lemma will play an important role in the proof of Theorem \ref{ctp6}.

\begin{lem}\label{desdelta} \ Let $\Lambda = \left(\begin{smallmatrix} A & N \\
	M & B\end{smallmatrix}\right)$ be a Morita ring  with $M\otimes_A N = 0 = N\otimes_BM$, \ \ $\mathcal X \subseteq A\mbox{-}{\rm Mod}$, and \ $\mathcal Y\subseteq B\mbox{-}{\rm Mod}$.

\vskip5pt

$(1)$ \ If \ $\mathcal X \supseteq \ _A \mathcal I$  \ and \  $\mathcal Y\supseteq  \ _A\mathcal I $,  \ then
\ $\Delta({}^\perp\mathcal X, \ {}^\perp\mathcal Y) = \ ^\perp{\rm Z}_A(\mathcal X) \ \cap \ ^\perp{\rm Z}_B(\mathcal Y).$

\vskip5pt

$(2)$ \ If \ $\mathcal X \supseteq \ _A \mathcal P$ \ and \ $\mathcal Y\supseteq \ _B \mathcal P$,  \ then \ $\nabla(\mathcal X^\perp, \ \mathcal Y^\perp) = {\rm Z}_A(\mathcal X)^\perp \ \cap \ {\rm Z}_B(\mathcal Y)^\perp.$
\end{lem}
\begin{proof} (1) \ Let $L=\left(\begin{smallmatrix} L_1 \\ L_2\end{smallmatrix}\right)_{f,g}\in \Delta({}^\perp\mathcal X, \ {}^\perp\mathcal Y)$.
By definition $f$ and $g$ are monomorphisms, and $\Coker f\in {}^\perp\mathcal Y$ and $\Coker g\in {}^\perp\mathcal X$.
Since $g$ is a monomorphism and \ $\Ext_B^1({\rm C}_AL, \  \mathcal X) = \Ext_A^1(\Coker g,  \ \mathcal X) = 0$, it follows from Lemma \ref{extadj2}(1) that
$\Ext_\Lambda^1(L, \ {\rm Z}_A(\mathcal X)) = 0$, i.e., $L\in \ ^\perp{\rm Z}_A(\mathcal X)$.
Similarly, since $f$ is a monomorphism and \ $\Ext_B^1({\rm C}_BL, \ \mathcal Y) = \Ext_B^1(\Coker f, \ \mathcal Y) = 0$, by Lemma \ref{extadj2}(2),
$\Ext_\Lambda^1(L, \ {\rm Z}_B(\mathcal Y)) = 0$, i.e., $L\in \ ^\perp{\rm Z}_B(\mathcal Y)$. Thus, $L\in \ ^\perp{\rm Z}_A(\mathcal X) \ \cap \ ^\perp{\rm Z}_B(\mathcal Y)$.

\vskip5pt

Conversely, let $L=\left(\begin{smallmatrix} L_1 \\ L_2\end{smallmatrix}\right)_{f,g}\in \ ^\perp{\rm Z}_A(\mathcal X) \ \cap \ ^\perp{\rm Z}_B(\mathcal Y)$, i.e.,
$\Ext_\Lambda^1(L, \ {\rm Z}_A(\mathcal X)) = 0 = \Ext_\Lambda^1(L, \ {\rm Z}_B(\mathcal Y))$.

\vskip5pt

{\bf Claim 1:} \ $\Hom_A(g, \ X): \Hom_A(L_1, \ X)\longrightarrow \Hom_A(N\otimes_BL_2, \ X)$ is an epimorphism,
for any module $X\in \mathcal X$. In fact, for any $A$-map $u: N\otimes_B L_2\longrightarrow X$,
consider $A$-map $g' =\left(\begin{smallmatrix}u\\ g\end{smallmatrix}\right): N\otimes_BL_2\longrightarrow X\oplus L_1$ and
the exact sequence of $A$-modules
$$0\longrightarrow X\xlongrightarrow{\left(\begin{smallmatrix} 1\\ 0\end{smallmatrix}\right)}X\oplus L_1\xlongrightarrow{(0,1)}L_1\longrightarrow 0.$$
Put \ $f' = (0, f): M\otimes_A(X\oplus L_1)\longrightarrow L_2.$ Then $\left(\begin{smallmatrix} X\oplus L_1 \\ L_2\end{smallmatrix}\right)_{f',g'}$ is indeed a $\Lambda$-module.
We stress that this is a place where one needs the assumption \ $M\otimes_A N = 0 = N\otimes_BM$:
Given any $U\in A\mbox{-Mod}$ and $V\in B\mbox{-Mod}$,  for arbitrary $u\in \Hom_B(M\otimes_AU, V)$ and $v\in \Hom_A(N\otimes_BV, U)$,
\ $\left(\begin{smallmatrix} U \\ V \end{smallmatrix}\right)_{u, v}$ is always a left $\Lambda$-module, since the conditions
\ $v(1_N\otimes u) = 0$ and  $u(1_M\otimes v) = 0$ automatically hold.

\vskip5pt

Then one can check that
$$
0\longrightarrow\left(\begin{smallmatrix} X\\ 0\end{smallmatrix}\right)_{0,0}\xlongrightarrow{\left(\begin{smallmatrix}\left(\begin{smallmatrix} 1 \\ 0 \end{smallmatrix}\right) \\ 0 \end{smallmatrix}\right)}\left(\begin{smallmatrix} X\oplus L_1 \\ L_2\end{smallmatrix}\right)_{f',g'}\xlongrightarrow{\left(\begin{smallmatrix} (0,1) \\ 1 \end{smallmatrix}\right)} \left(\begin{smallmatrix} L_1\\ L_2\end{smallmatrix}\right)_{f,g}\longrightarrow 0$$
is an exact sequence of $\Lambda$-modules. Since $\left(\begin{smallmatrix} X\\ 0\end{smallmatrix}\right)_{0,0} = {\rm Z}_A X\in {\rm Z}_A(\mathcal X)$ and   $L\in \ ^\perp{\rm Z}_A(\mathcal X)\cap \ ^\perp{\rm Z}_B(\mathcal Y)$,
this exact sequence  splits. Thus there is a $\Lambda$-map  $$\left(\begin{smallmatrix}\left(\begin{smallmatrix} a \\ b\end{smallmatrix}\right) \\ \beta\end{smallmatrix}\right):\left(\begin{smallmatrix} L_1 \\ L_2\end{smallmatrix}\right)_{f,g}\longrightarrow \left(\begin{smallmatrix} X\oplus L_1 \\ L_2\end{smallmatrix}\right)_{f',g'}$$ such that
$\left(\begin{smallmatrix} (0,1) \\ 1 \end{smallmatrix}\right)\left(\begin{smallmatrix}\left(\begin{smallmatrix} a \\ b\end{smallmatrix}\right) \\ \beta\end{smallmatrix}\right)={\rm Id}_L$.
So $b={\rm Id}_{L_1}$ and $\beta ={\rm Id}_{L_2}$.  Thus one gets a commutative diagram
$$\xymatrix@C=2cm{N\otimes_BL_2\ar[d]_-g\ar@{=}[r] & N\otimes_BL_2\ar[d]^-{g'=\left(\begin{smallmatrix} u \\ g \end{smallmatrix}\right)} \\ L_1\ar[r]^-{\left(\begin{smallmatrix} a \\ 1 \end{smallmatrix}\right)} & X\oplus L_1}$$
and hence $u=ag$. This proves {\bf Claim 1}.

\vskip5pt

{\bf Claim 2:} \  $g$ is a monomorphism.  In fact, embedding  $N\otimes_BL_2$ into an injective $A$-module one has a monomorphism $i: N\otimes_BL_2\hookrightarrow I$. By assumption $I\in \mathcal X$, hence
$\Hom_A(g, \ I): \Hom_A(L_1, \ I)\longrightarrow \Hom_A(N\otimes_BL_2, \ I)$ is an epimorphism, by {\bf Claim 1}. Hence there is an $A$-map $v: L_1\longrightarrow I$ such that $vg = i.$
Thus, $g$ is a monomorphism.

\vskip5pt

Similar as {\bf Claim 1}, one has

\vskip5pt

{\bf Claim 3:} \ $\Hom_A(f, \ Y): \Hom_B(L_2, \ Y)\longrightarrow \Hom_B(M\otimes_AL_1, \ Y)$ is an epimorphism, for any module $Y\in \mathcal Y$.

\vskip5pt

Similar as {\bf Claim 2}, one has

\vskip5pt

{\bf Claim 4:} \ $f$ is a monomorphism.

\vskip5pt

We omit the similar proof of {\bf Claim 3} and {\bf Claim 4}.

\vskip5pt

Now, since $g$ and $f$ are monic, by Lemma \ref{extadj2}(1) and (2) one has
$$\Ext_A^1(\Coker g, \ \mathcal X) = \Ext_A^1({\rm C}_AL, \ \mathcal X) \cong  \Ext_\Lambda^1(L, \ {\rm Z}_A(\mathcal X)) = 0,$$
$$\Ext_B^1(\Coker f, \ \mathcal Y) = \Ext_B^1({\rm C}_BL, \ \mathcal Y)\cong  \Ext_\Lambda^1(L, \ {\rm Z}_B(\mathcal Y)) = 0$$
\noindent By definition, $L=\left(\begin{smallmatrix} L_1 \\ L_2\end{smallmatrix}\right)_{f,g}\in \Delta({}^\perp\mathcal X, \ {}^\perp\mathcal Y)$. This completes the proof of $(1)$.

\vskip5pt

(2) \ This can be similarly proved, however, it is difficult to say that it is the dual of (1), thus we include a justification. It will be much convenient to use the second expression of a
$\Lambda$-module.

\vskip5pt

Let $L=\left(\begin{smallmatrix} L_1 \\ L_2\end{smallmatrix}\right)_{\widetilde{f}, \widetilde{g}}\in \nabla(\mathcal X^\perp, \ \mathcal Y^\perp)$, i.e.,  $\widetilde{f}$ and $\widetilde{g}$ are epimorphisms, and $\Ker \widetilde{f}\in \mathcal X^\perp$ and $\Ker \widetilde{g}\in \mathcal Y^\perp$.
Since $\widetilde{f}$ is an epimorphism and \ $\Ext_B^1(\mathcal X, \ {\rm K}_AL) = \Ext_A^1(\mathcal X, \ \Ker\widetilde{f}) = 0$, by Lemma \ref{extadj2}(3), $L\in {\rm Z}_A(\mathcal X)^\perp$.
Similarly, since $\widetilde{g}$ is an epimorphism and \ $\Ext_B^1(\mathcal Y, \ {\rm K}_BL) = \Ext_B^1(\mathcal Y, \ \Ker\widetilde{g}) = 0$, by Lemma \ref{extadj2}(4),
 $L\in {\rm Z}_B(\mathcal Y)^\perp$. Thus, $L\in {\rm Z}_A(\mathcal X)^\perp \ \cap \ {\rm Z}_B(\mathcal Y)^\perp$.

\vskip5pt

Conversely, let $L=\left(\begin{smallmatrix} L_1 \\ L_2\end{smallmatrix}\right)_{\widetilde{f}, \widetilde{g}}\in \ {\rm Z}_A(\mathcal X)^\perp \ \cap \ {\rm Z}_B(\mathcal Y)^\perp$.

\vskip5pt

{\bf Claim 1:}  $\Hom_B(Y, \ \widetilde{g}): \Hom_B(Y, \ L_2)\longrightarrow \Hom_B(Y, \ \Hom_A(N, L_1))$ is an epimorphism, for any module $Y\in \mathcal Y$.
In fact, $\forall \ u\in \Hom_B(Y, \ \Hom_A(N,L_1))$,
consider $B$-map $\widetilde{g'}:=(u,\widetilde{g}): Y\oplus L_2\longrightarrow \Hom_A(N, \ L_1)$. Thus \ $g'\in \Hom_A(N\otimes_B(Y\oplus L_2), \ L_1)$.
Put \ $\widetilde{f'}= \binom{0}{\widetilde{f}}: L_1\longrightarrow \Hom_B(M,Y)\oplus \Hom_B(M,L_2).$ Thus \ $f'\in \Hom_B(M\otimes_AL_1, \ Y\oplus L_2)$.
Since \ $M\otimes_A N = 0 = N\otimes_BM$, \ $\left(\begin{smallmatrix} L_1 \\ Y\oplus L_2\end{smallmatrix}\right)_{\widetilde{f'}, \widetilde{g'}}$ is indeed a $\Lambda$-module.

\vskip5pt

Then one has the exact sequence
$$
0\longrightarrow\left(\begin{smallmatrix} L_1\\ L_2\end{smallmatrix}\right)_{\widetilde{f}, \widetilde{g}}\xlongrightarrow{\binom{1}{\binom{0}{1}}}
\left(\begin{smallmatrix} L_1 \\ Y\oplus L_2\end{smallmatrix}\right)_{\widetilde{f'}, \widetilde{g'}}\xlongrightarrow{\left(\begin{smallmatrix} 0 \\ (1,0) \end{smallmatrix}\right)} \left(\begin{smallmatrix} 0\\ Y\end{smallmatrix}\right)_{0,0}\longrightarrow 0.$$
 (We stress that it is much convenient to use the second expression of
$\Lambda$-modules. Otherwise, say, it is not direct to see that ${\binom{1}{\binom{0}{1}}}$ is a $\Lambda$-map.)

\vskip5pt

\noindent Since $\left(\begin{smallmatrix} 0\\ Y\end{smallmatrix}\right)_{0,0} = {\rm Z}_B Y\in {\rm Z}_B(\mathcal Y)$ and   $L\in \ {\rm Z}_A(\mathcal X)^\perp\cap \ {\rm Z}_B(\mathcal Y)^\perp$,
this exact sequence splits, i.e., there is a $\Lambda$-map  $$\left(\begin{smallmatrix}\alpha \\ (a,b)\end{smallmatrix}\right):\left(\begin{smallmatrix} L_1 \\ Y\oplus L_2\end{smallmatrix}\right)_{\widetilde{f'},\widetilde{g'}}\longrightarrow \left(\begin{smallmatrix} L_1 \\ L_2\end{smallmatrix}\right)_{\widetilde{f},\widetilde{g}}$$ such that
$\left(\begin{smallmatrix} \alpha \\ (a,b) \end{smallmatrix}\right)\left(\begin{smallmatrix} 1 \\ \binom{0}{1}\end{smallmatrix}\right)={\rm Id}_L$.
So $\alpha=\Id_{L_1}$ and $b={\rm Id}_{L_2}$.
This gives the commutative diagram
$$\xymatrix@C=2cm{Y\oplus L_2\ar[d]_-{\widetilde{g'}=(u,\widetilde{g})}\ar[r]^-{(a,1)} & L_2\ar[d]^-{\widetilde{g}} \\
\Hom_A(N, L_1)\ar@{=}[r] & \Hom_A(N,L_1)}$$

\noindent commutes. Hence $u=\widetilde{g}a$. This proves {\bf Claim 1}.

\vskip5pt

{\bf Claim 2:} \  $\widetilde{g}$ is an epimorphism.  In fact, taking a $B$-epimorphism  $q: Q\longrightarrow \Hom_A(N, L_1)$ with $Q$ projective. Then $Q\in \mathcal Y$, hence
$\Hom_B(Q, \ \widetilde{g}): \Hom_B(Q, \ L_2)\longrightarrow \Hom_B(Q, \ \Hom_A(N, \ L_1))$ is an epimorphism. So there is a $B$-map $v: Q\longrightarrow L_2$ with $q = \widetilde{g}v.$ This proves {\bf Claim 2}.

\vskip5pt

Similarly, \ $\Hom_A(X, \ \widetilde{f}): \Hom_A(X, \ L_1)\longrightarrow \Hom_A(X, \ \Hom_B(M,L_2))$ is an epimorphism for any  $X\in \mathcal X$; and
\ $\widetilde{f}$ is an epimorphism.

\vskip5pt

It follows from Lemma \ref{extadj2}(3) and (4) that \ $$\Ext_A^1(\mathcal X, \ \Ker\widetilde{f}) = \Ext_A^1(\mathcal X, \ {\rm K}_AL)\cong  \Ext_\Lambda^1({\rm Z}_A(\mathcal X), \ L) = 0$$
and that
$$ \Ext_B^1(\mathcal Y, \ \Ker\widetilde{g}) = \Ext_B^1(\mathcal Y, \ {\rm K}_BL) \cong  \Ext_\Lambda^1({\rm Z}_B(\mathcal Y), \ L) = 0.$$
By definition, $L=\left(\begin{smallmatrix} L_1 \\ L_2\end{smallmatrix}\right)_{\widetilde{f},\widetilde{g}}\in \nabla(\mathcal X^\perp, \ \mathcal Y^\perp)$. This completes the proof.
\end{proof}

\vskip5pt

\begin{lem}\label{deltaher} \ \ Let $\Lambda = \left(\begin{smallmatrix} A & N \\
	M & B\end{smallmatrix}\right)$ be a Morita ring with $M\otimes_A N = 0 = N\otimes_BM$.

\vskip5pt

$(1)$ \ Assume that $M_A$ and $N_B$ are flat modules.
Then  \ $\Delta(\mathcal U, \ \mathcal V)$ is closed under  the kernels of epimorphisms if and only if \ $\mathcal U$ and \ $\mathcal V$ are closed under  the kernels of epimorphisms.

\vskip5pt

$(2)$ \  Assume that  $_BM$ and  $_AN$ are projective.
Then $\nabla(\mathcal X, \ \mathcal Y)$ is closed under the cokernels of monomorphisms if and only if \ $\mathcal X$ and \ $\mathcal Y$ are closed under  the cokernels of monomorphisms.

\end{lem}
\begin{proof} \ $(1)$ \ Assume that \ $\mathcal U$ and \ $\mathcal V$ are closed under  the kernels of epimorphisms.
Let \ $0\longrightarrow \left(\begin{smallmatrix} L_1 \\ L_2\end{smallmatrix}\right)_{f, g}
\longrightarrow \left(\begin{smallmatrix} M_1 \\ M_2 \end{smallmatrix}\right)_{u, v}\longrightarrow
\left(\begin{smallmatrix} N_1 \\ N_2 \end{smallmatrix}\right)_{s, t}\longrightarrow 0$ \
be an exact sequence with  $\left(\begin{smallmatrix} M_1 \\ M_2 \end{smallmatrix}\right)_{u, v}, \ \left(\begin{smallmatrix} N_1 \\ N_2 \end{smallmatrix}\right)_{s, t}\in \Delta(\mathcal U, \ \mathcal V)$.
Thus $u, v, s, t$ are monomorphisms, $\Coker u \in \mathcal V,  \ \Coker v\in \mathcal U, \ \Coker s\in \mathcal V$,  and  \  $\Coker t\in \mathcal U.$
Since $M_A$ is flat, one has the commutative diagram with exact rows:
$$\xymatrix@R=0.5cm{0\ar[r] & M\otimes_A L_1\ar[r]^-{1\otimes \alpha}\ar[d]_-{f} &  M\otimes_A M_1 \ar[r]\ar[d]_-{u} &  M\otimes_A N_1\ar[r]\ar[d]^-{s} & 0 \\
0\ar[r] & L_2\ar[r] & M_2 \ar[r]& N_2\ar[r] & 0.}$$
Since $1\otimes \alpha$ and $u$ are monomorphisms, so is $f$. By Snake Lemma and the assumption
that $\mathcal V$ is closed under the kernels of epimorphisms,
one knows that $\Coker f \in \mathcal V$. Similarly,  $g$ is a monomorphism and  $\Coker g \in \ \mathcal U$. By definition $\left(\begin{smallmatrix} L_1 \\ L_2 \end{smallmatrix}\right)_{f, g}\in \Delta(\mathcal U, \ \mathcal V)$.
This proves that \ $\Delta(\mathcal U, \ \mathcal V)$ is closed under  the kernels of epimorphisms.

\vskip5pt

Conversely, using functors \ ${\rm T}_A$  and \ ${\rm T}_B$, one sees that
\ $\mathcal U$ and \ $\mathcal V$ are closed under  the kernels of epimorphisms.

\vskip5pt

$(2)$ \ Assume that \ $\mathcal X$ and \ $\mathcal Y$ are closed under  the cokernels of monomorphisms. Let \ $0\longrightarrow \left(\begin{smallmatrix} L_1 \\ L_2\end{smallmatrix}\right)_{f, g}
\longrightarrow \left(\begin{smallmatrix} M_1 \\ M_2 \end{smallmatrix}\right)_{u, v}\longrightarrow
\left(\begin{smallmatrix} N_1 \\ N_2 \end{smallmatrix}\right)_{s, t}\longrightarrow 0$
be an exact sequence of $\Lambda$-modules with $\left(\begin{smallmatrix} L_1 \\ L_2 \end{smallmatrix}\right)_{f, g}\in \nabla(\mathcal X, \ \mathcal Y)$ and \ $\left(\begin{smallmatrix} M_1 \\ M_2 \end{smallmatrix}\right)_{u, v}\in \nabla(\mathcal X, \ \mathcal Y)$.
Thus \ $\widetilde{f}, \ \widetilde{g}, \ \widetilde{u}, \ \widetilde{v}$ are epimorphisms, $\Ker \widetilde{f} \in \mathcal X,  \ \Ker \widetilde{g}\in \mathcal Y, \ \Ker \widetilde{u} \in \mathcal X$,   and \ $\Ker \widetilde{v}\in \mathcal Y.$
Since \ $_BM$ is projective, one has the commutative diagram with exact rows
$$\xymatrix@R=0.5cm{0\ar[r] & L_1\ar[r]\ar[d]_-{\widetilde{f}} & M_1 \ar[r]\ar[d]_-{\widetilde{u}} & N_1\ar[r]\ar[d]^-{\widetilde{s}} & 0 \\
0\ar[r] & \Hom_B(M, L_2)\ar[r] & \Hom_B(M, M_2) \ar[r]^-{(M, \beta)} & \Hom_B(M, N_2)\ar[r] & 0.}$$
Since $\widetilde{u}$ and $(M, \beta)$ are epimorphisms, so is $\widetilde{s}$. By Snake Lemma and the assumption that $\mathcal X$ is closed under taking the cokernels of monomorphisms,
one knows that $\Ker \widetilde{s} \in \mathcal X$. Similarly,  $\widetilde{t}$ is an epimorphism and \ $\Ker \widetilde{t} \in \mathcal Y$. By definition $\left(\begin{smallmatrix} N_1 \\ N_2 \end{smallmatrix}\right)_{s, t}
\in \nabla(\mathcal X, \ \mathcal Y)$. This proves that \ $\nabla(\mathcal X, \ \mathcal Y)$ is closed under the cokernels of monomorphisms.

\vskip5pt

Conversely, using functors \ ${\rm H}_A$  and \ ${\rm H}_B$, one sees that
\ $\mathcal X$ and \ $\mathcal Y$ are closed under the cokernels of monomorphisms.
\end{proof}

\subsection{Proof of Theorem \ref{ctp6}}  $(1)$ \  It suffices  to prove \ $\Delta(\mathcal U, \ \mathcal V) = \ ^\perp(\Delta(\mathcal U, \ \mathcal V)^\perp).$
In fact, $\Delta(\mathcal U, \ \mathcal V) = \Delta(^\perp\mathcal X, \ ^\perp\mathcal Y)$.
Since \ $\mathcal X$ contains all the injective $A$-modules and \ $\mathcal Y$ contains all the injective $B$-modules, it follows from
Lemma \ref{desdelta}(1) that
$$\Delta(^\perp\mathcal X, \ ^\perp\mathcal Y) = \ ^\perp{\rm Z}_A(\mathcal X) \ \cap \ ^\perp{\rm Z}_B(\mathcal Y) =
\ ^\perp({\rm Z}_A(\mathcal X) \ \cup \ {\rm Z}_B(\mathcal Y)).$$
Thus
\begin{align*} ^\perp(\Delta(\mathcal U, \ \mathcal V)^\perp) & = \ ^\perp(\Delta(^\perp\mathcal X, \ ^\perp\mathcal Y)^\perp)
= \ ^\perp \{[^\perp({\rm Z}_A(\mathcal X) \ \cup \ {\rm Z}_B(\mathcal Y))]^\perp\} \\ &
= \ ^\perp ({\rm Z}_A(\mathcal X) \ \cup \ {\rm Z}_B(\mathcal Y)) = \Delta(\mathcal U, \ \mathcal V).
\end{align*}

\vskip5pt

By Lemma \ref{deltaher}(1), $\Delta(\mathcal U, \ \mathcal V)$ is closed under  the kernels of epimorphisms if and only if
$\mathcal U$ and $\mathcal V$ are closed under  the kernels of epimorphisms.
That is, \ $(\Delta(\mathcal U, \ \mathcal V), \ \Delta(\mathcal U, \ \mathcal V)^\perp)$  is hereditary if and only if  $(\mathcal U, \ \mathcal X)$  and  $(\mathcal V, \ \mathcal Y)$ are hereditary.

\vskip5pt

$(2)$ \ Similarly, it suffices to show \ $\nabla(\mathcal X, \ \mathcal Y) = \ (^{\perp}\nabla(\mathcal X, \ \mathcal Y))^\perp.$
In fact, $\nabla(\mathcal X, \ \mathcal Y) = \nabla(\mathcal U^\perp, \ \mathcal V^\perp)$.
Since \ $\mathcal U$ contains all the projective $A$-modules and \ $\mathcal V$ contains all the projective $B$-modules, it follows from
Lemma \ref{desdelta}(2) that
$$\nabla(\mathcal U^\perp, \ \mathcal V^\perp) = \ {\rm Z}_A(\mathcal U)^\perp \ \cap \ {\rm Z}_B(\mathcal V)^\perp =
 ({\rm Z}_A(\mathcal U) \ \cup \ {\rm Z}_B(\mathcal V))^\perp.$$
Thus
\begin{align*} (^{\perp}\nabla(\mathcal X, \ \mathcal Y))^\perp & = (^{\perp}\nabla(\mathcal U^\perp, \ \mathcal V^\perp))^\perp
=  \{^\perp[({\rm Z}_A(\mathcal U) \ \cup \ {\rm Z}_B(\mathcal V))^\perp]\}^\perp  \\ &
= ({\rm Z}_A(\mathcal U) \ \cup \ {\rm Z}_B(\mathcal V))^\perp = \nabla(\mathcal X, \ \mathcal Y).
\end{align*}

\vskip5pt

By Lemma \ref{deltaher}(2), \  $\nabla(\mathcal X, \ \mathcal Y)$  is
closed under the cokernels of monomorphisms if and only if $\mathcal X$ and $\mathcal Y$ are closed under the cokernels of monomorphisms. That is, \  $(^{\perp}\nabla(\mathcal X, \ \mathcal Y), \ \nabla(\mathcal X, \ \mathcal Y))$  is hereditary if and only if \ $(\mathcal U, \ \mathcal X)$  and  $(\mathcal V, \ \mathcal Y)$ are hereditary. \hfill $\square$

\subsection{Proof of Theorem \ref{compare}}
$(1)$ \ By Theorem \ref{ctp1}(1), one has cotorsion pair \ $({}^\perp\left(\begin{smallmatrix}\mathcal X\\ \mathcal Y\end{smallmatrix}\right), \ \left(\begin{smallmatrix}\mathcal X\\ \mathcal Y\end{smallmatrix}\right))$;
and by Theorem \ref{ctp6}(1), one has cotorsion pair \ $(\Delta(\mathcal U, \ \mathcal V), \ \Delta(\mathcal U, \ \mathcal V)^\perp)$.  We will prove $\Delta(\mathcal U, \ \mathcal V)^{\bot} \subseteq \left(\begin{smallmatrix}\mathcal X\\ \mathcal Y\end{smallmatrix}\right).$ By Lemma \ref{desdelta}(1) one has
$$\Delta(\mathcal U, \ \mathcal V)^{\bot} = [\Delta({}^\perp\mathcal X, \ {}^\perp\mathcal Y)]^{\bot} = [^\perp{\rm Z}_A(\mathcal X) \ \cap \ ^\perp{\rm Z}_B(\mathcal Y)]^{\bot}.$$
Since by assumption \ $\Tor^A_1(M, \ \mathcal U) =0 = \Tor^B_1(N, \ \mathcal V)$, it follows from Lemma \ref{destheta}(1) that
$$\left(\begin{smallmatrix} \mathcal X\\ \mathcal Y\end{smallmatrix}\right) = \left(\begin{smallmatrix} \mathcal U^\perp \\ \mathcal V^\perp \end{smallmatrix}\right)
= {\rm T}_A(\mathcal U)^\perp \ \cap \ {\rm T}_B(\mathcal V)^\perp = ({\rm T}_A(\mathcal U) \ \cup \ {\rm T}_B(\mathcal V))^\perp.$$
Thus, to show \ $\Delta(\mathcal U, \ \mathcal V)^{\bot} \subseteq \left(\begin{smallmatrix}\mathcal X\\ \mathcal Y\end{smallmatrix}\right),$ it suffices to show
$${\rm T}_A(\mathcal U) \ \cup \ {\rm T}_B(\mathcal V) \subseteq \ ^\perp{\rm Z}_A(\mathcal X)\ \cap \ ^\perp{\rm Z}_B(\mathcal Y).$$

In fact, since $N\otimes_BM = 0$, the structure map $g = 0$ of any $\Lambda$-module in ${\rm T}_A(\mathcal U)$ is a monomorphism,
it follows from Lemma \ref{extadj2}(1) that
$$\Ext^1_{\Lambda}({\rm T}_A(\mathcal U), \ {\rm Z}_A(\mathcal X)) \cong \Ext^1_A({\rm C}_A{\rm T}_A(\mathcal U), \ \mathcal X) = \Ext^1_A(\mathcal U, \ \mathcal X) = 0.$$
By Lemma \ref{extadj2}(2) one has
$$\Ext^1_{\Lambda}({\rm T}_A(\mathcal U), \ {\rm Z}_B(\mathcal Y)) \cong \Ext^1_A({\rm C}_B{\rm T}_A(\mathcal U), \ \mathcal Y) = 0$$
since ${\rm C}_B{\rm T}_A = 0$.
So ${\rm T}_A(\mathcal U)\subseteq  \ ^\perp{\rm Z}_A(\mathcal X)\ \cap \ ^\perp{\rm Z}_B(\mathcal Y)$.

\vskip5pt

Similarly, by Lemma \ref{extadj2}(1) one has
$$\Ext^1_{\Lambda}({\rm T}_B(\mathcal V), \ {\rm Z}_A(\mathcal X)) \cong \Ext^1_A({\rm C}_A{\rm T}_B(\mathcal V), \ \mathcal X) = 0$$
since ${\rm C}_A{\rm T}_B = 0.$ Since $M\otimes_AN = 0$, the structure map $f=0$ of any $\Lambda$-module in ${\rm T}_B(\mathcal V)$ is a monomorphism,
it follows from Lemma \ref{extadj2}(2) that
$$\Ext^1_{\Lambda}({\rm T}_B(\mathcal V), \ {\rm Z}_B(\mathcal Y)) \cong \Ext^1_A({\rm C}_B{\rm T}_B(\mathcal V), \ \mathcal Y) = \Ext^1_A(\mathcal V, \ \mathcal Y) = 0.$$
So ${\rm T}_B(\mathcal V)\subseteq \ ^\perp{\rm Z}_A(\mathcal X)\ \cap \ ^\perp{\rm Z}_B(\mathcal Y)$. This completes the proof of (1).

\vskip5pt

(2) \  Comparing cotorsion pair \ $(\left(\begin{smallmatrix} \mathcal U\\ \mathcal V\end{smallmatrix}\right), \ \left(\begin{smallmatrix} \mathcal U\\ \mathcal V\end{smallmatrix}\right)^\perp)$ in Theorem \ref{ctp1}(2)
with  \ $(^{\perp}\nabla(\mathcal X, \ \mathcal Y), \ \nabla(\mathcal X, \ \mathcal Y))$ in Theorem \ref{ctp6}(2), we will prove \ $^{\bot}\nabla(\mathcal X, \ \mathcal Y) \subseteq \left(\begin{smallmatrix}\mathcal U\\ \mathcal V\end{smallmatrix}\right).$  This can be similarly done as (1). For convenience we include a brief justification. By Lemma \ref{desdelta}(2) one has
$$^{\bot}\nabla(\mathcal X, \ \mathcal Y) = \ ^{\bot}\nabla(\mathcal U^\perp, \ \mathcal V^\perp) = \ ^{\bot}[{\rm Z}_A(\mathcal U)^\perp \ \cap \ {\rm Z}_B(\mathcal V)^\perp].$$
By Lemma \ref{destheta}(2), one has
$$\left(\begin{smallmatrix} \mathcal U\\ \mathcal V\end{smallmatrix}\right) = \left(\begin{smallmatrix} ^{\bot}\mathcal X \\ ^{\bot} \mathcal Y\end{smallmatrix}\right) = \ ^{\bot}[{\rm H}_A(\mathcal X) \ \cup \ {\rm H}_B(\mathcal Y)].$$
So, it suffices to show \ ${\rm H}_A(\mathcal X) \ \cup \ {\rm H}_B(\mathcal Y) \subseteq {\rm Z}_A(\mathcal U)^\perp \ \cap \ {\rm Z}_B(\mathcal V)^\perp.$

\vskip5pt

In fact, since \ $\Ext_A^1(N, \ \mathcal X)=0$, it follows from Lemma \ref{extadj1}(3) that
$$\Ext^1_{\Lambda}({\rm Z}_A(\mathcal U), \ {\rm H}_A(\mathcal X)) \cong \Ext^1_A({\rm U}_A{\rm Z}_A(\mathcal U), \ \mathcal X) = \Ext^1_A(\mathcal U, \ \mathcal X) = 0$$
and
$$\Ext^1_{\Lambda}({\rm Z}_B(\mathcal V), \ {\rm H}_A(\mathcal X)) \cong \Ext^1_A({\rm U}_A{\rm Z}_B(\mathcal V), \ \mathcal X) = 0.$$
Thus ${\rm H}_A(\mathcal X)\subseteq {\rm Z}_A(\mathcal U)^\perp \ \cap \ {\rm Z}_B(\mathcal V)^\perp$.

\vskip5pt

Since \ $\Ext_B^1(M, \ \mathcal Y)=0$, it follows from Lemma \ref{extadj1}(4) that
$$\Ext^1_{\Lambda}({\rm Z}_A(\mathcal U), \ {\rm H}_B(\mathcal Y)) \cong \Ext^1_A({\rm U}_B{\rm Z}_A(\mathcal U), \ \mathcal Y) = 0$$
and
$$\Ext^1_{\Lambda}({\rm Z}_B(\mathcal V), \ {\rm H}_B(\mathcal Y)) \cong \Ext^1_A({\rm U}_B{\rm Z}_B(\mathcal V), \ \mathcal Y) = \Ext^1_A(\mathcal V, \ \mathcal Y) = 0,$$
which show ${\rm H}_B(\mathcal Y)\subseteq {\rm Z}_A(\mathcal U)^\perp \ \cap \ {\rm Z}_B(\mathcal V)^\perp$.
This completes the proof. \hfill $\square$

\section{\bf Identifications}

We will prove that the four constructions of cotorsion pairs, given in Theorem \ref{ctp1} and Theorem \ref{ctp6},
are pairwise generally different; and on the other hand, study the problem of identifications, i.e.,
we will show that, in many important cases, the cotorsion pairs in Series I coincide with the corresponding ones in Series II, and then we will get only two cotorsion pairs
$$(\Delta(\mathcal U, \ \mathcal V), \ \left(\begin{smallmatrix}\mathcal X\\ \mathcal Y\end{smallmatrix}\right))
\ \ \ \ \mbox{and} \ \ \ \ (\left(\begin{smallmatrix} \mathcal U\\ \mathcal V\end{smallmatrix}\right), \ \nabla(\mathcal X, \ \mathcal Y)).$$
Since the both cotorsion pairs are explicitly given,
they can be used in finding Hovey triples, i.e., the abelian model structures on Morita rings.

\subsection{Generally different cotorsion pairs}

For use in Section 6, we introduce the following notion.

\begin{defn} \label{difference} \ Let \ $\Omega$ be a class of Morita rings, $(\mathcal X, \ \mathcal Y)$ and \ $(\mathcal X', \ \mathcal Y')$  cotorsion pairs defined in $\Lambda\mbox{-}{\rm Mod}$, for arbitrary
Morita rings $\Lambda\in \Omega$.
We say that \ $(\mathcal X, \ \mathcal Y)$ and \ $(\mathcal X', \ \mathcal Y')$ are generally different, provided that there exist $\Lambda\in \Omega$, such that \ $(\mathcal X, \ \mathcal Y) \ne (\mathcal X', \ \mathcal Y')$ in $\Lambda\mbox{-}{\rm Mod}$.
\end{defn}

\begin{exm} \label{gdsame}  Generally different cotorsion pairs  could be the same for some special Morita rings, as the following example shows.

\vskip5pt

Let \ $\Omega = \{ \mbox{Morita ring} \ \Lambda = \left(\begin{smallmatrix}A  & N \\ M & B\end{smallmatrix}\right) \ | \ \phi= 0  = \psi, \ _BN \ \mbox{and} \ _AM \ \mbox{are projective}\}.$ Then
$(_\Lambda\mathcal P, \ \Lambda\mbox{\rm-Mod})$ and $(\binom{_A\mathcal P}{_B\mathcal P}, \ \binom{_A\mathcal P}{_B\mathcal P}^\perp)$ are cotorsion pairs in $\Lambda\mbox{-}{\rm Mod}$,  \ $\forall \ \Lambda\in \Omega$. See Theorem \ref{ctp1}(2).

\vskip5pt

If $M \ne 0$, then \ $\binom{A}{0}_{0, 0}\notin \ _\Lambda\mathcal P$. Thus \ $_\Lambda\mathcal P\ne \binom{_A\mathcal P}{_B\mathcal P}$ for
$\Lambda\in \Omega$ with $M\ne 0$. Hence
$(_\Lambda\mathcal P, \ \Lambda\mbox{\rm-Mod})$ and $(\binom{_A\mathcal P}{_B\mathcal P}, \ \binom{_A\mathcal P}{_B\mathcal P}^\perp)$ are generally different cotorsion pairs.
But they are the same for $\Lambda\in \Omega$ with  $M = 0 = N$.
\end{exm}

\subsection{The four cotorsion pairs are pairwise generally different} \ By Theorems \ref{ctp1} and \ref{ctp6},  the cotorsion pairs
$$(^\perp\left(\begin{smallmatrix}\mathcal X\\ \mathcal Y\end{smallmatrix}\right), \ \left(\begin{smallmatrix}\mathcal X\\ \mathcal Y\end{smallmatrix}\right)), \ \  \  \
(\Delta(\mathcal U, \ \mathcal V), \ \Delta(\mathcal U, \ \mathcal V)^\perp)$$
and
$$(\left(\begin{smallmatrix} \mathcal U\\ \mathcal V\end{smallmatrix}\right), \ \left(\begin{smallmatrix} \mathcal U\\ \mathcal V\end{smallmatrix}\right)^\perp),  \ \  \ \   (^{\perp}\nabla(\mathcal X, \ \mathcal Y), \ \nabla(\mathcal X, \ \mathcal Y))$$
are defined in $\Lambda$-Mod, \ $\forall \ \Lambda\in \Omega$, where $$\Omega = \{\Lambda = \left(\begin{smallmatrix} A & N \\
	M & A\end{smallmatrix}\right) \ | \ M\otimes_AN = 0 = N\otimes_BM, \ M_A \ \mbox{and} \ N_B \ \mbox{are flat}, \ _BM \ \mbox{and} \ _AN \ \mbox{are projective}\}.$$
We will show that the four cotorsion pairs are pairwise generally  different. For convenience, we will call the cotorsion pairs above the first, the second, the third, and the fourth cotorsion pairs.

\begin{exm} \label{ie} Let $A = B$ be the path algebra $k(1 \longrightarrow 2)$, where ${\rm char} \ k\ne 2$. Write the conjunction of paths from right to left. Thus $e_1Ae_2 = 0$ and $e_2Ae_1 \cong k$. Take $M = N = Ae_2\otimes_ke_1A$. Then $M\otimes_AN = 0 = N\otimes_AM$. Let  $\Lambda$ be the Morita ring $\left(\begin{smallmatrix} A & N \\
	M & A\end{smallmatrix}\right).$ Then $\Lambda\in \Omega.$

\vskip5pt

Note that  \  $_AM = \ _AN$ is isomorphic to the simple projective left $A$-module $Ae_2 = S_2$, and that \ $M_A = N_A$ is isomorphic to the simple projective right $A$-module $e_1A$.  Then \ $M\otimes_AAe_1\cong Ae_2\otimes_k (e_1A\otimes_AAe_1)  \cong S_2$. To see the left $A$-module structure on $\Hom_A(M, Ae_1)$, note that
$\Hom_A(M, Ae_1)\cong \Hom_A(Ae_2, Ae_1) \cong e_2A e_1 \cong k$ as $k$-spaces. For $f\in \Hom_A(M, Ae_1)$ given by $f(e_2\otimes_k e_1) = e_1$,  one has $e_1f = f$. Thus $\Hom_A(M, Ae_1)\cong S_1$ as left $A$-modules. The Auslander-Reiten quiver of $A$ is

$$\xymatrix@R=0.3cm@C=0.6cm{& Ae_1\ar[dr]^-\pi
\\ S_2\ar[ur]^-\sigma & & S_1}$$

\vskip5pt

Take \ $(\mathcal U, \mathcal X) = (A\mbox{-}{\rm Mod}, \ _A\mathcal I) = (\mathcal V, \ \mathcal Y)$. Note that $M\otimes_A\mathcal U \nsubseteq \mathcal Y, \ N\otimes_B\mathcal V \nsubseteq \mathcal X.$
Take \ $L = \left(\begin{smallmatrix}Ae_1\\ Ae_1\end{smallmatrix}\right)_{\sigma, \sigma}$.
Then $L\in {\rm Mon(\Lambda)} = \Delta(A\mbox{-}{\rm Mod}, \ A\mbox{-}{\rm Mod}) = \Delta(\mathcal U, \ \mathcal V)$ and $L\in \left(\begin{smallmatrix}_A\mathcal I\\ _A\mathcal I\end{smallmatrix}\right) =
\left(\begin{smallmatrix}\mathcal X\\ \mathcal Y\end{smallmatrix}\right).$
Consider the exact sequence of $\Lambda$-modules
$$0\longrightarrow \left(\begin{smallmatrix}Ae_1\\ Ae_1\end{smallmatrix}\right)_{\sigma, \sigma}
\stackrel{\left(\begin{smallmatrix} \binom{1}{0}\\ \binom{1}{0}\end{smallmatrix}\right)}\longrightarrow \left(\begin{smallmatrix} Ae_1\oplus Ae_1\\ Ae_1\oplus Ae_1\end{smallmatrix}\right)_{\left(\begin{smallmatrix}\sigma & \sigma\\ 0 & \sigma\end{smallmatrix}\right), \left(\begin{smallmatrix}\sigma & \sigma\\ 0 & \sigma\end{smallmatrix}\right)}
\stackrel{\left(\begin{smallmatrix} (0, 1) \\ (0,1)\end{smallmatrix}\right)}
\longrightarrow \left(\begin{smallmatrix}Ae_1\\ Ae_1\end{smallmatrix}\right)_{\sigma, \sigma}\longrightarrow 0.$$
This exact sequence does not split. In fact, if it splits, then
there is a $\Lambda$-map $\left(\begin{smallmatrix} (a,b) \\ (c,d)\end{smallmatrix}\right): \left(\begin{smallmatrix} Ae_1\oplus Ae_1\\ Ae_1\oplus Ae_1\end{smallmatrix}\right)_{\left(\begin{smallmatrix}\sigma & \sigma\\ 0 & \sigma\end{smallmatrix}\right), \left(\begin{smallmatrix}\sigma & \sigma\\ 0 & \sigma\end{smallmatrix}\right)}\longrightarrow \left(\begin{smallmatrix}Ae_1\\ Ae_1\end{smallmatrix}\right)_{\sigma, \sigma}$ such that
$\left(\begin{smallmatrix} (a,b) \\ (c,d)\end{smallmatrix}\right) \left(\begin{smallmatrix} \binom{1}{0}\\ \binom{1}{0}\end{smallmatrix}\right) = \left(\begin{smallmatrix} 1 \\ 1\end{smallmatrix}\right),$
i.e., $a = 1=c.$ Since the following diagrams
$$\xymatrix@R= 0.7cm{S_2\oplus S_2\ar[d]_-{\left(\begin{smallmatrix}\sigma & \sigma\\ 0 & \sigma\end{smallmatrix}\right)}\ar[r]^-{(1, b)} & S_2 \ar[d]^-{\sigma} \\
	Ae_1\oplus Ae_1 \ar[r]^-{(1, d)} & Ae_1}\qquad \qquad \qquad
\xymatrix@R= 0.7cm{S_2\oplus S_2\ar[d]_-{\left(\begin{smallmatrix}\sigma & \sigma\\ 0 & \sigma\end{smallmatrix}\right)}\ar[r]^-{(1, d)} & S_2 \ar[d]^-{\sigma} \\
	Ae_1\oplus Ae_1 \ar[r]^-{(1, b)} & Ae_1}$$
commute, \  $d+1 = b$ and  $b+1 = d$, which is a contradiction, since ${\rm char} \ k\ne 2$.

\vskip5pt

Thus \ ${\rm Ext}^1_\Lambda(L, L)\ne 0$. This means  \ $L\notin \ ^\perp\left(\begin{smallmatrix}_A\mathcal I\\ _A\mathcal I\end{smallmatrix}\right)$. Since $L\in {\rm Mon(\Lambda)}$,
\ $^\perp\left(\begin{smallmatrix}_A\mathcal I\\ _A\mathcal I\end{smallmatrix}\right) \ne {\rm Mon(\Lambda)}$.
Thus,  the first cotorsion pair
is not equal to the second one, i.e.,
$$(^\perp\left(\begin{smallmatrix}\mathcal X\\ \mathcal Y\end{smallmatrix}\right), \ \left(\begin{smallmatrix}\mathcal X\\ \mathcal Y\end{smallmatrix}\right)) =  ({}^\perp\left(\begin{smallmatrix}_A\mathcal I\\ _A\mathcal I\end{smallmatrix}\right), \ \left(\begin{smallmatrix}_A\mathcal I\\ _A\mathcal I\end{smallmatrix}\right)) \ne
({\rm Mon}(\Lambda), \ {\rm Mon}(\Lambda)^\bot) = (\Delta(\mathcal U, \ \mathcal V), \ \Delta(\mathcal U, \ \mathcal V)^\perp).$$

\vskip5pt

Since $\left(\begin{smallmatrix} \mathcal U\\ \mathcal V\end{smallmatrix}\right) = \left(\begin{smallmatrix} A\mbox{-}{\rm Mod}\\ A\mbox{-}{\rm Mod}\end{smallmatrix}\right) = \Lambda\mbox{-}{\rm Mod}$,
it follows that  $(\left(\begin{smallmatrix} \mathcal U\\ \mathcal V\end{smallmatrix}\right), \ \left(\begin{smallmatrix} \mathcal U\\ \mathcal V\end{smallmatrix}\right)^\perp) = (\Lambda\mbox{-}{\rm Mod}, \ _\Lambda\mathcal I).$
Since \ $L\notin \ ^\perp\left(\begin{smallmatrix}_A\mathcal I\\ _A\mathcal I\end{smallmatrix}\right)$,  the first cotorsion pair
is not equal to the third one:
$$(^\perp\left(\begin{smallmatrix}\mathcal X\\ \mathcal Y\end{smallmatrix}\right), \ \left(\begin{smallmatrix}\mathcal X\\ \mathcal Y\end{smallmatrix}\right)) =  ({}^\perp\left(\begin{smallmatrix}_A\mathcal I\\ _A\mathcal I\end{smallmatrix}\right), \ \left(\begin{smallmatrix}_A\mathcal I\\ _A\mathcal I\end{smallmatrix}\right)) \ne
(\Lambda\mbox{-}{\rm Mod}, \ _\Lambda\mathcal I) = (\left(\begin{smallmatrix} \mathcal U\\ \mathcal V\end{smallmatrix}\right), \ \left(\begin{smallmatrix} \mathcal U\\ \mathcal V\end{smallmatrix}\right)^\perp).$$

\vskip5pt

Since $\left(\begin{smallmatrix} Ae_1\\ 0\end{smallmatrix}\right)\notin {\rm Mon}(\Lambda) = \Delta(\mathcal U, \ \mathcal V)$, the second cotorsion pair
is not equal to the third one:
$$(\Delta(\mathcal U, \ \mathcal V), \ \Delta(\mathcal U, \ \mathcal V)^\perp)
\ne
(\Lambda\mbox{-}{\rm Mod}, \ _\Lambda\mathcal I) = (\left(\begin{smallmatrix} \mathcal U\\ \mathcal V\end{smallmatrix}\right), \ \left(\begin{smallmatrix} \mathcal U\\ \mathcal V\end{smallmatrix}\right)^\perp).$$

\vskip5pt

By definition $\nabla(\mathcal X, \ \mathcal Y) = \nabla(_A\mathcal I, \ _A\mathcal I) = \ _\Lambda \mathcal I$. Thus
$(^{\perp}\nabla(\mathcal X, \ \mathcal Y), \ \nabla(\mathcal X, \ \mathcal Y)) = (\Lambda\mbox{-}{\rm Mod}, \ _\Lambda \mathcal I)$, i.e., the fourth cotorsion pair is exactly
third cotorsion pair. Therefore, the first cotorsion pair is not equal to the fourth one, and  the second cotorsion pair
is not equal to the fourth one.

\vskip5pt

Finally, to see  the third cotorsion pair is not equal to the fourth one, namely,
$$(\left(\begin{smallmatrix} \mathcal U\\ \mathcal V\end{smallmatrix}\right), \ \left(\begin{smallmatrix} \mathcal U\\ \mathcal V\end{smallmatrix}\right)^\perp)\ne  (^{\perp}\nabla(\mathcal X, \ \mathcal Y), \ \nabla(\mathcal X, \ \mathcal Y))$$
we take $(\mathcal U, \mathcal X) = (_A\mathcal P, \ A\mbox{-}{\rm Mod}) = (\mathcal V, \ \mathcal Y)$.
Note that  \
$\Hom_B(M, \ \mathcal Y)\nsubseteq \mathcal U, \ \Hom_A(N, \ \mathcal X) \nsubseteq \mathcal V.$

\vskip5pt

Take $L = \left(\begin{smallmatrix}Ae_1\\ Ae_1\end{smallmatrix}\right)_{\sigma, \sigma}$ as above.
Then $L\in \left(\begin{smallmatrix}_A\mathcal P\\ _A\mathcal P\end{smallmatrix}\right) =
\left(\begin{smallmatrix}\mathcal U\\ \mathcal V\end{smallmatrix}\right).$
Since $\widetilde{\sigma}: Ae_1 \longrightarrow \Hom_A(M, Ae_1)\cong S_1$ is exactly the epimorphism $\pi: Ae_1\longrightarrow S_1$, by definition
$L\in {\rm Epi(\Lambda)} = \nabla (A\mbox{-}{\rm Mod}, \ A\mbox{-}{\rm Mod}) = \nabla(\mathcal X, \ \mathcal Y)$.
Since \ ${\rm Ext}^1_\Lambda(L, L)\ne 0$, \ $L\notin  \left(\begin{smallmatrix}_A\mathcal P\\ _A\mathcal P\end{smallmatrix}\right)^\perp$. Thus
$\nabla(\mathcal X, \ \mathcal Y)\ne \left(\begin{smallmatrix}_A\mathcal P\\ _A\mathcal P\end{smallmatrix}\right)^\perp$, and hence  the third cotorsion pair is not equal to the fourth one:
$$(\left(\begin{smallmatrix} \mathcal U\\ \mathcal V\end{smallmatrix}\right), \ \left(\begin{smallmatrix} \mathcal U\\ \mathcal V\end{smallmatrix}\right)^\perp) = (\left(\begin{smallmatrix}_A\mathcal P\\ _A\mathcal P\end{smallmatrix}\right), \ \left(\begin{smallmatrix}_A\mathcal P\\ _A\mathcal P\end{smallmatrix}\right)^\perp)\ne  (^{\perp}\nabla(\mathcal X, \ \mathcal Y), \ \nabla(\mathcal X, \ \mathcal Y)).$$

All together, we have proved that the four cotorsion pairs are pairwise generally different.
In fact, we have found an example $\Lambda$, such that the four constructions of cotorsion pairs in $\Lambda$-Mod
are pairwise different.
\end{exm}

\subsection{Main results on identification}

\begin{thm}\label{identify1} Let \ $\Lambda = \left(\begin{smallmatrix} A & N \\ M & B \end{smallmatrix}\right)$ be a Morita ring with \ $M\otimes_A N = 0 = N\otimes_BM$, \ $(\mathcal U, \ \mathcal X)$ and $(\mathcal V, \ \mathcal Y)$ be cotorsion pairs in $A\mbox{-}{\rm Mod}$ and in $B$\mbox{\rm-Mod}, respectively.

\vskip5pt

$(1)$ \ Assume that \ $\Tor^A_1(M, \ \mathcal U) =0 = \Tor^B_1(N, \ \mathcal V)$. If
$M\otimes_A\mathcal U \subseteq \mathcal Y$ or \ $N\otimes_B\mathcal V \subseteq \mathcal X$, then
$$(\Delta(\mathcal U, \ \mathcal V), \ \Delta(\mathcal U, \ \mathcal V)^\bot) =
(^\perp\left(\begin{smallmatrix} \mathcal X\\ \mathcal Y\end{smallmatrix}\right), \ \left(\begin{smallmatrix} \mathcal X \\ \mathcal Y\end{smallmatrix}\right)).$$
Thus \ \ $(\Delta(\mathcal U, \ \mathcal{V}), \ \left(\begin{smallmatrix} \mathcal X\\ \mathcal{Y}\end{smallmatrix}\right))$ \  is a cotorsion pair in $\Lambda\mbox{-}{\rm Mod}$.

\vskip5pt

Moreover, if $M\otimes_A\mathcal U \subseteq \mathcal Y$ and \ $N\otimes_B\mathcal V \subseteq \mathcal X$, then
\ $\Delta(\mathcal U, \ \mathcal V) = {\rm T}_A(\mathcal U)\oplus {\rm T}_B(\mathcal V).$

\vskip10pt

$(2)$ \ Assume that \ $\Ext_B^1(M, \ \mathcal Y) =0 = \Ext_A^1(N, \ \mathcal X)$. If \
$\Hom_B(M, \ \mathcal Y)\subseteq \mathcal U$ or \ $\Hom_A(N, \ \mathcal X) \subseteq \mathcal V$, then
$$(^\perp\nabla(\mathcal X, \ \mathcal Y), \ \nabla(\mathcal X, \ \mathcal Y)) =
(\left(\begin{smallmatrix} \mathcal U\\ \mathcal V\end{smallmatrix}\right), \ \left(\begin{smallmatrix} \mathcal U \\ \mathcal V\end{smallmatrix}\right)^\perp).$$
Thus \ \ $(\left(\begin{smallmatrix} \mathcal U \\ \mathcal V\end{smallmatrix}\right), \ \nabla(\mathcal X, \ \mathcal Y))$ \  is a cotorsion pair in $\Lambda\mbox{-}{\rm Mod}$.

\vskip5pt

Moreover, if \ $\Hom_B(M, \ \mathcal Y)\subseteq \mathcal U$ and \ $\Hom_A(N, \ \mathcal X) \subseteq \mathcal V$, then
\ $\nabla(\mathcal X, \ \mathcal Y)
={\rm H}_A(\mathcal X)\oplus {\rm H}_B(\mathcal Y).$
\end{thm}

\subsection{Applications} In Theorem \ref{identify1}, taking one of $(\mathcal U, \ \mathcal X)$  and
$(\mathcal V, \ \mathcal Y)$ being the projective cotorsion pair or the injective cotorsion pair,
and another being an arbitrary cotorsion pair, one has

\begin{cor}\label{identification1} \ Let \ $\Lambda = \left(\begin{smallmatrix} A & N \\ M & B \end{smallmatrix}\right)$ be a Morita ring with \ $M\otimes_A N = 0 = N\otimes_BM$.

\vskip10pt

$(1)$ \ If  \ $N_B$ is flat, then for any
cotorsion  pair \ $(\mathcal V, \ \mathcal Y)$ in $B\mbox{-}{\rm Mod}$
one has
$$(\Delta(_A\mathcal P, \ \mathcal V), \ \Delta(_A\mathcal P, \ \mathcal V)^\bot) =
(^\perp\left(\begin{smallmatrix} A\text{\rm\rm-Mod}\\ \mathcal Y\end{smallmatrix}\right), \ \left(\begin{smallmatrix} A\text{\rm-Mod} \\ \mathcal Y\end{smallmatrix}\right)).$$
Thus \ \ $(\Delta(_A\mathcal P, \ \mathcal{V}), \ \left(\begin{smallmatrix} A\text{\rm-Mod}\\ \mathcal{Y}\end{smallmatrix}\right))$ \  is a cotorsion pair in $\Lambda\mbox{-}{\rm Mod}$.

\vskip5pt

Moreover, if \ $M\otimes_A \mathcal P \subseteq \mathcal Y$  $($e.g., this is the case if $_BM$ is injective$)$,
then
\ $\Delta(_A\mathcal P, \ \mathcal V) = {\rm T}_A(_A\mathcal P)\oplus {\rm T}_B(\mathcal V).$

\vskip5pt

$(2)$ \ If $M_A$ is flat, then for any
cotorsion  pair \  $(\mathcal U, \ \mathcal X)$ in $A\mbox{-}{\rm Mod}$ one has
$$(\Delta(\mathcal U, \ _B\mathcal P), \ \Delta(\mathcal U, \ _B\mathcal P)^\bot) =
({}^\perp\left(\begin{smallmatrix}\mathcal X\\ B\mbox{-}{\rm Mod}\end{smallmatrix}\right),  \ \left(\begin{smallmatrix}\mathcal X\\ B\mbox{-}{\rm Mod}\end{smallmatrix}\right)).$$
Thus \ $(\Delta(\mathcal{U}, \ _B\mathcal P), \ \binom{\mathcal{X}}{B\text{\rm-Mod}})$ is a cotorsion pair in $\Lambda\mbox{-}{\rm Mod}$.

\vskip5pt
Moreover, if \ $N\otimes_B \mathcal P\subseteq \mathcal X$ $($e.g., this is the case if $_AN$ is injective$)$, then
\ $\Delta(\mathcal U, \ _B\mathcal P)= {\rm T}_A(\mathcal U)\oplus {\rm T}_B(_B\mathcal P).$

\vskip5pt

$(3)$ \ If \ $_BM$ is projective, then for any cotorsion pair $(\mathcal V, \ \mathcal Y)$ in \ $B\mbox{-}{\rm Mod}$ one has
$$(^\bot\nabla(_A\mathcal I, \ \mathcal Y), \ \nabla(_A\mathcal I, \ \mathcal Y))
= (\left(\begin{smallmatrix} A\mbox{-}{\rm Mod}\\ \mathcal V\end{smallmatrix}\right),  \ \left(\begin{smallmatrix}A\mbox{-}{\rm Mod}\\ \mathcal V\end{smallmatrix}\right)^\perp).$$
Thus \  $(\binom{A\text{\rm-Mod}}{_B\mathcal V}, \ \nabla(_A\mathcal I, \ \mathcal Y))$ is a cotorsion pair in $\Lambda\mbox{-}{\rm Mod}$.

\vskip5pt

Moreover, if \ $\Hom_A(N, \ _A\mathcal I)\subseteq \mathcal V$ \ $($e.g., this is the case if \ $B$ is quasi-Frobenius and $N_B$ is flat$)$, then
 \ $\nabla(_A\mathcal I, \ \mathcal Y) ={\rm H}_A(_A\mathcal I)\oplus {\rm H}_B(\mathcal Y).$

\vskip5pt

$(4)$  \  If $_AN$ is projective, then for any cotorsion pair
$(\mathcal U, \ \mathcal X)$ in $A\mbox{-}{\rm Mod}$ one has
$$(^\bot\nabla(\mathcal X, \ _B\mathcal I), \ \nabla(\mathcal X, \ _B\mathcal I)) = (\left(\begin{smallmatrix}\mathcal U\\ B\mbox{-}{\rm Mod}\end{smallmatrix}\right),  \ \left(\begin{smallmatrix}\mathcal U\\ B\mbox{-}{\rm Mod}\end{smallmatrix}\right)^\perp).$$
Thus $(\binom{\mathcal U}{B\text{\rm-Mod}}, \ \nabla(_A\mathcal X, \ _B\mathcal I))$ is a cotorsion pair in $\Lambda\mbox{-}{\rm Mod}$.

\vskip5pt

Moreover, if \ $\Hom_B(M, \ _B\mathcal I)\subseteq \mathcal U$ \ $($e.g., this is the case if \ $A$ is quasi-Frobenius and $M_A$ is flat$)$,  then \ $\nabla(\mathcal X, \ _B\mathcal I) = {\rm H}_A(\mathcal X)\oplus {\rm H}_B(_B\mathcal I).$
\end{cor}
\begin{proof}
(1) \ Taking \ $(\mathcal U, \ \mathcal X) = (_A\mathcal P, \ A\mbox{-}{\rm Mod})$ in Theorem \ref{identify1}(1). Then
\ $N\otimes_B\mathcal V \subseteq A\mbox{-}{\rm Mod} = \mathcal X$. By Theorem \ref{identify1}(1) one has
\ $(\Delta(_A\mathcal P, \ \mathcal V), \ \Delta(_A\mathcal P, \ \mathcal V)^\bot) =
(^\perp\left(\begin{smallmatrix} A\text{\rm\rm-Mod}\\ \mathcal Y\end{smallmatrix}\right), \ \left(\begin{smallmatrix} A\text{\rm-Mod} \\ \mathcal Y\end{smallmatrix}\right)).$

\vskip5pt

If  $M\otimes_A \mathcal P \subseteq \mathcal Y$, i.e.,
$M\otimes_A\mathcal U = M\otimes_A\mathcal P \subseteq \mathcal Y$,
then by Theorem \ref{identify1}(1),
$\Delta(_A\mathcal P, \ \mathcal V) = {\rm T}_A(_A\mathcal P)\oplus {\rm T}_B(\mathcal V).$

\vskip5pt

Assume that $_BM$ is injective.  For any $P\in \ _A\mathcal P$,
as a left $B$-module, \ $M\otimes_AP$ is a direct summand of
a direct sum of copies of $_BM$. Since $_BM$ is injective and $B$ is left noetherian (note that $\Lambda$ is assumed to be an Artin algebra, hence $B$ is an Artin algebra),
$M\otimes_AP$ is an injective left $B$-module, and hence it is in \ $\mathcal Y$. Thus
$M\otimes_A\mathcal P \subseteq \mathcal Y$, and hence
$\Delta(_A\mathcal P, \ \mathcal V) = {\rm T}_A(_A\mathcal P)\oplus {\rm T}_B(\mathcal V),$ by Theorem \ref{identify1}(1).

\vskip10pt
    	
(2) \ Taking \ $(\mathcal V, \ \mathcal Y) = (_B\mathcal P, \ B\mbox{-}{\rm Mod})$ in Theorem \ref{identify1}(1). Then
\ $M\otimes_A\mathcal U \subseteq B\mbox{-}{\rm Mod} = \mathcal Y$. By Theorem \ref{identify1}(1) one has
\ $(\Delta(\mathcal U, \ _B\mathcal P), \ \Delta(\mathcal U, \ _B\mathcal P)^\bot) =
({}^\perp\left(\begin{smallmatrix}\mathcal X\\ B\mbox{-}{\rm Mod}\end{smallmatrix}\right),  \ \left(\begin{smallmatrix}\mathcal X\\ B\mbox{-}{\rm Mod}\end{smallmatrix}\right)).$

\vskip5pt

If  $N\otimes_B\mathcal P\subseteq \mathcal X$, then by Theorem \ref{identify1}(1),
$\Delta(\mathcal U, \ _B\mathcal P)= {\rm T}_A(\mathcal U)\oplus {\rm T}_B(_B\mathcal P).$

\vskip5pt

Assume that \ $A$ is  $_AN$ is injective. Then
\ $N\otimes_B\mathcal P \subseteq \ _A\mathcal I \subseteq \mathcal X,$ and hence
$\Delta(\mathcal U, \ _B\mathcal P)= {\rm T}_A(\mathcal U)\oplus {\rm T}_B(_B\mathcal P).$

\vskip10pt

(3) \ Taking \ $(\mathcal U, \ \mathcal X) = (A\mbox{-}{\rm Mod}, \ _A\mathcal I)$ in Theorem \ref{identify1}(2). Then
\ $\Hom_B(M, \ \mathcal Y) \subseteq A\mbox{-}{\rm Mod} = \mathcal U.$
 By Theorem \ref{identify1}(2) one has
\ $(^\bot\nabla(_A\mathcal I, \ \mathcal Y), \ \nabla(_A\mathcal I, \ \mathcal Y))
= (\left(\begin{smallmatrix} A\mbox{-}{\rm Mod}\\ \mathcal V\end{smallmatrix}\right),  \ \left(\begin{smallmatrix}A\mbox{-}{\rm Mod}\\ \mathcal V\end{smallmatrix}\right)^\perp).$

\vskip5pt If  \ $\Hom_A(N, \ _A\mathcal I)\subseteq \mathcal V$,
then by Theorem \ref{identify1}(2), $\nabla(_A\mathcal I, \ \mathcal Y) ={\rm H}_A(_A\mathcal I)\oplus {\rm H}_B(\mathcal Y)$.

\vskip5pt

Assume that $B$ is quasi-Frobenius and $N_B$ is flat. Then
\ $\Hom_A(N, \ _A\mathcal I)\subseteq \ _B\mathcal I = \ _B\mathcal P\subseteq \mathcal V$, and thus
\ $\nabla(_A\mathcal I, \ \mathcal Y)
={\rm H}_A(_A\mathcal I)\oplus {\rm H}_B(\mathcal Y)$,  by Theorem \ref{identify1}(2).

\vskip10pt

(4) \ Taking $(\mathcal V, \ \mathcal Y)  = (B\mbox{-}{\rm Mod}, \ _B\mathcal I)$ in Theorem \ref{identify1}(2). Then
 \ $\Hom_A(N, \ \mathcal X) \subseteq B \mbox{-}{\rm Mod} = \mathcal V.$
 By Theorem \ref{identify1}(2) one has
\ $(^\bot\nabla(\mathcal X, \ _B\mathcal I), \ \nabla(\mathcal X, \ _B\mathcal I)) = (\left(\begin{smallmatrix}\mathcal U\\ B\mbox{-}{\rm Mod}\end{smallmatrix}\right),  \ \left(\begin{smallmatrix}\mathcal U\\ B\mbox{-}{\rm Mod}\end{smallmatrix}\right)^\perp).$

\vskip5pt

If  \ $\Hom_B(M, \ _B\mathcal I)\subseteq \mathcal U$,
then by Theorem \ref{identify1}(2), \  $\nabla(\mathcal X, \ _B\mathcal I)={\rm H}_A(\mathcal X)\oplus {\rm H}_B(_B\mathcal I)$.

\vskip5pt

Assume that  $A$ is quasi-Frobenius and $M_A$ is flat.     	
Then \ $\Hom_B(M, \ _B\mathcal I)\subseteq  \ _A\mathcal I = \ _A\mathcal P\subseteq \mathcal U,$
and hence  \  $\nabla(\mathcal X, \ _B\mathcal I)={\rm H}_A(\mathcal X)\oplus {\rm H}_B(_B\mathcal I)$.
\end{proof}

\vskip5pt

\subsection{Proof of Theorem \ref{identify1}} \ (1) \ By Theorem \ref{compare}(1),  one has cotorsion pairs
$$(^\perp\left(\begin{smallmatrix} \mathcal X\\ \mathcal Y\end{smallmatrix}\right), \ \left(\begin{smallmatrix} \mathcal X \\ \mathcal Y\end{smallmatrix}\right)), \ \ \ \ \ (\Delta(\mathcal U, \ \mathcal V), \ \Delta(\mathcal U, \ \mathcal V)^\perp)$$
with \ $^\bot\left(\begin{smallmatrix}\mathcal X \\ \mathcal Y\end{smallmatrix}\right)\subseteq \Delta(\mathcal U, \ \mathcal V)$.
To see that they are equal, it remains to prove $\Delta(\mathcal U, \ \mathcal V)\subseteq \ ^\perp\left(\begin{smallmatrix} \mathcal X \\ \mathcal Y\end{smallmatrix}\right)$.
 	
\vskip5pt

Let \ $L = \left(\begin{smallmatrix} L_1 \\ L_2\end{smallmatrix}\right)_{f, g}\in \Delta(\mathcal U, \ \mathcal V)$. By definition there are exact sequences
$$0\rightarrow M\otimes_AL_1\xlongrightarrow{f}L_2\xlongrightarrow{p_1}\Coker f\rightarrow 0 \ \ \ \mbox{and} \ \ \
0\rightarrow N\otimes_BL_2\xlongrightarrow{g}L_1\xlongrightarrow {p_2}\Coker g\rightarrow 0$$
with $\Coker f\in \ \mathcal V$ and $\Coker g\in \ \mathcal U$. Since $N\otimes_BM = 0 = M\otimes_AN$, it follows that
$$1\otimes p_1: N\otimes_B L_2 \longrightarrow N\otimes_B \Coker f \ \ \ \ \mbox{and} \ \ \
1\otimes p_2: M\otimes_A L_1 \longrightarrow M\otimes_A \Coker g$$ are isomorphisms.

\vskip5pt

{\bf Case I:} \ Assume that \ $M\otimes_A\mathcal U \subseteq \mathcal Y$. Then
$$M\otimes_A L_1\cong M\otimes_A\Coker g\in M\otimes_A\mathcal U \subseteq \mathcal Y.$$
Since $(\mathcal V, \ \mathcal Y)$ is a cotorsion pair, the exact sequence
$$0\longrightarrow M\otimes_AL_1\xlongrightarrow{f}L_2\xlongrightarrow{p_1}\Coker f\longrightarrow 0$$
splits. Thus there are $B$-maps $f': L_2\longrightarrow M\otimes_AL_1$ and $\sigma_1: \Coker f \longrightarrow L_2$
such that
$$f'f = 1_{M\otimes_AL_1}, \ \ \ p_1\sigma_1 = 1_{\Coker f}, \ \ \ ff' + \sigma_1p_1 = 1_{L_2}, \ \ f'\sigma_1 = 0.$$
Thus  \ $\left(\begin{smallmatrix} f' \\ p_1\end{smallmatrix}\right): L_2 \longrightarrow (M\otimes_AL_1)\oplus \Coker f$ is a $B$-isomorphism, and $$\left(\begin{smallmatrix} 1 \\ \left(\begin{smallmatrix} f' \\ p_1\end{smallmatrix}\right)\end{smallmatrix}\right): \ L = \left(\begin{smallmatrix} L_1 \\ L_2\end{smallmatrix}\right)_{f, g} \cong \left(\begin{smallmatrix} L_1 \\  (M\otimes_AL_1)\oplus \Coker f \end{smallmatrix}\right)_{\binom{1}{0}, \ g(1\otimes p_1)^{-1}}$$
is a $\Lambda$-isomorphism,  and  	
$$0\rightarrow \left(\begin{smallmatrix} N\otimes_B\Coker f \\ \Coker f\end{smallmatrix}\right)_{0,1}\xlongrightarrow{\left(\begin{smallmatrix} g(1\otimes p_1)^{-1} \\ \binom{0}{1}\end{smallmatrix}\right)}
\left(\begin{smallmatrix} L_1 \\ (M\otimes_AL_1)\oplus \Coker f \end{smallmatrix}\right)_{\binom{1}{0}, \ g(1\otimes  p_1)^{-1}}\xlongrightarrow{\left(\begin{smallmatrix} p_2 \\ (1\otimes p_2,0)\end{smallmatrix}\right)}
\left(\begin{smallmatrix} \Coker g \\ M\otimes_A\Coker g\end{smallmatrix}\right)_{1,0}\rightarrow 0$$
is an exact sequence of $\Lambda$-modules, i.e.,
$$0\longrightarrow {\rm T}_B\Coker f\longrightarrow \left(\begin{smallmatrix} L_1 \\ (M\otimes_AL_1)\oplus \Coker f \end{smallmatrix}\right)_{\binom{1}{0}, \ g(1\otimes  p_1)^{-1}}\longrightarrow {\rm T}_A\Coker g\longrightarrow 0$$
is exact.

\vskip5pt

Since $\Coker f\in \ \mathcal V$ and $(\mathcal V, \ \mathcal Y)$ is a cotorsion pair,
by Lemma \ref{extadj1}(2), ${\rm T}_B\Coker f\in \ ^\perp \binom{\mathcal{X}}{\mathcal Y}$.
Since $\Coker g \in \mathcal U$ and $(\mathcal U, \ \mathcal X)$ is a cotorsion pair, by Lemma \ref{extadj1}(1), ${\rm T}_A\Coker g\in \ ^\perp \binom{\mathcal{X}}{\mathcal Y}$.
Thus  $L\cong \left(\begin{smallmatrix} L_1 \\ (M\otimes_AL_1)\oplus \Coker f \end{smallmatrix}\right)_{\binom{1}{0}, \ g(1\otimes  p_1)^{-1}}\in \ ^\perp \binom{\mathcal{X}}{\mathcal Y}$.

\vskip5pt
    	
{\bf Case II:} \ Assume that \ $N\otimes_B\mathcal V \subseteq \mathcal X$.  \ This is similar to {\bf Case I}.
We include the main steps. Since $N\otimes_B L_2\cong N\otimes_B\Coker f\in N\otimes_B\mathcal V \subseteq \mathcal X,$
the exact sequence
$$0\longrightarrow N\otimes_BL_2\xlongrightarrow{g}L_1\xlongrightarrow {p_2}\Coker g\longrightarrow 0$$
splits. Then $L = \left(\begin{smallmatrix} L_1 \\ L_2\end{smallmatrix}\right)_{f, g} \cong \left(\begin{smallmatrix}(N\otimes_BL_2)\oplus \Coker g\\ L_2 \end{smallmatrix}\right)_{f(1\otimes p_2)^{-1}, \binom{1}{0}}$
and
$$0\rightarrow \left(\begin{smallmatrix} \Coker g \\ M\otimes_A\Coker g\end{smallmatrix}\right)_{1,0}\xlongrightarrow{\left(\begin{smallmatrix} \binom{0}{1} \\ f(1\otimes p_2)^{-1} \end{smallmatrix}\right)}
\left(\begin{smallmatrix}(N\otimes_BL_2)\oplus \Coker g\\ L_2 \end{smallmatrix}\right)_{f(1\otimes p_2)^{-1}, \binom{1}{0}} \xlongrightarrow{\left(\begin{smallmatrix} (1\otimes p_1,0) \\ p_1 \end{smallmatrix}\right)}
\left(\begin{smallmatrix} N\otimes_B\Coker f \\ \Coker f\end{smallmatrix}\right)_{0,1}\rightarrow 0$$
is an exact sequence of $\Lambda$-modules, i.e.,
$$0\longrightarrow {\rm T}_A\Coker g \longrightarrow \left(\begin{smallmatrix}(N\otimes_BL_2)\oplus \Coker g\\ L_2 \end{smallmatrix}\right)_{f(1\otimes p_2)^{-1}, \binom{1}{0}}\longrightarrow {\rm T}_B\Coker f\longrightarrow 0$$
is exact.
By Lemma \ref{extadj1}(1), ${\rm T}_A\Coker g \in  \ ^\perp\binom{\mathcal{X}}{\mathcal{Y}}$; and
by Lemma \ref{extadj1}(2), ${\rm T}_B\Coker f \in \ ^\perp\binom{\mathcal{Y}}{\mathcal{Y}}.$
Thus $L\in \ ^\perp\binom{\mathcal{Y}}{_B\mathcal{Y}}$.

\vskip5pt

Finally, assume that $M\otimes_A\mathcal U \subseteq \mathcal Y$ and \ $N\otimes_B\mathcal V \subseteq \mathcal X$.
Then from the proof above one sees that both \ $0\rightarrow M\otimes_AL_1\xlongrightarrow{f}L_2\xlongrightarrow{p_1}\Coker f\rightarrow 0$ \ and
\ $0\rightarrow N\otimes_BL_2\xlongrightarrow{g}L_1\xlongrightarrow {p_2}\Coker g\rightarrow 0$ split, and
$$L\cong\left(\begin{smallmatrix} \Coker g \\ M\otimes_A\Coker g\end{smallmatrix}\right)_{1, 0}
\oplus \left(\begin{smallmatrix} N\otimes_B \Coker f \\ \Coker f\end{smallmatrix}\right)_{0, 1} = {\rm T}_A \Coker g\oplus {\rm T}_B\Coker f.$$
Thus  \ $\Delta(\mathcal U, \ \mathcal V) \subseteq {\rm T}_A(\mathcal U)\oplus {\rm T}_B(\mathcal V).$ The inclusion
\ $ {\rm T}_A(\mathcal U)\oplus {\rm T}_B(\mathcal V)\subseteq \Delta(\mathcal U, \ \mathcal V)$ is clear. Thus shows
$\Delta(\mathcal U, \ \mathcal V) = {\rm T}_A(\mathcal U)\oplus {\rm T}_B(\mathcal V).$

\vskip10pt   	

(2) \ By Theorem \ref{compare}(2),  one has cotorsion pairs
$$(\left(\begin{smallmatrix} \mathcal U\\ \mathcal V\end{smallmatrix}\right), \ \left(\begin{smallmatrix} \mathcal U \\ \mathcal V\end{smallmatrix}\right)^\perp), \ \ \ \ \ (^\perp\nabla(\mathcal X, \ \mathcal Y), \ \nabla(\mathcal X, \ \mathcal Y))$$
with \ $\left(\begin{smallmatrix}\mathcal U \\ \mathcal V\end{smallmatrix}\right)^\bot \subseteq \nabla(\mathcal X, \ \mathcal Y)$.
To see that they are equal, it remains to prove $\nabla(\mathcal X, \ \mathcal Y)\subseteq \left(\begin{smallmatrix} \mathcal U \\ \mathcal V\end{smallmatrix}\right)^\bot$.

\vskip5pt

Here it is much more convenient to use the second expression of $\Lambda$-modules. Thus, let \ $L = \left(\begin{smallmatrix} L_1 \\ L_2\end{smallmatrix}\right)_{\widetilde{f}, \widetilde{g}}\in \nabla(\mathcal X, \ \mathcal Y)$, where $\widetilde{f}\in \Hom_A(X, \ \Hom_B(M, Y))$ and
$\widetilde{g}\in \Hom_B(Y, \ \Hom_A(N, X))$. By definition there are exact sequences
$$0\rightarrow \Ker\widetilde{f}\xlongrightarrow{i_1}L_1\xlongrightarrow{\widetilde{f}}\Hom_B(M, L_2)\rightarrow 0 \ \ \ \mbox{and} \ \ \ 0\rightarrow \Ker\widetilde{g}\xlongrightarrow{i_2}L_2\xlongrightarrow{\widetilde{g}}\Hom_A(N, L_1)\rightarrow 0$$
with $\Ker \widetilde{f}\in \ \mathcal X$ and $\Ker \widetilde{g}\in \ \mathcal Y$. Since $M\otimes_AN = 0 = N\otimes_BM$, it follows that
$$(N, \ i_1): \Hom_A(N, \Ker\widetilde{f})\cong \Hom_A(N, L_1) \ \ \ \mbox{and} \ \ \ (M, \ i_2): \Hom_B(M, \Ker\widetilde{g})\cong \Hom_B(M, L_2).$$

{\bf Case I:} \ Assume that \ $\Hom_B(M, \mathcal Y)\subseteq \mathcal U$. Then
$$\Hom_B(M, \ L_2)\cong \Hom_B(M, \ \Ker\widetilde{g})\in \Hom_B(M, \mathcal Y)\subseteq \mathcal U.$$
Since $(\mathcal U, \ \mathcal X)$ is a cotorsion pair, the exact sequence
$$0\longrightarrow \Ker\widetilde{f}\xlongrightarrow{i_1}L_1\xlongrightarrow{\widetilde{f}}\Hom_B(M, L_2)\longrightarrow 0$$
splits. Thus there are $A$-maps $\alpha: \Hom_B(M,\  L_2)\longrightarrow L_1$ and $\pi_1: L_1 \longrightarrow \Ker\widetilde{f}$
such that
$$\pi_1i_1 = 1_{\Ker\widetilde{f}}, \ \ \ \widetilde{f}\alpha = 1_{\Hom_B(M,L_2)}, \ \ \ \alpha\widetilde{f} + i_1\pi_1 = 1_{L_1}, \ \ \ \pi_1\alpha = 0.$$
Hence \ $\left(\begin{smallmatrix} \pi_1 \\ \widetilde{f}\end{smallmatrix}\right): L_1 \cong \Ker\widetilde{f}\oplus \Hom_B(M, L_2)$ and
$$\left(\begin{smallmatrix} \binom{\pi_1}{\widetilde{f}} \\ 1 \end{smallmatrix}\right): \ L = \left(\begin{smallmatrix} L_1 \\ L_2\end{smallmatrix}\right)_{\widetilde{f}, \widetilde{g}} \cong \left(\begin{smallmatrix} \Ker\widetilde{f}\oplus \Hom_B(M, \ L_2) \\ L_2\end{smallmatrix}\right)_{(0, \ 1), \ \binom{(N, \pi_1)\widetilde{g}}{0}}.$$
Moreover,   	
$$0 \rightarrow \left(\begin{smallmatrix} (M, \ \Ker\widetilde{g}) \\ \Ker\widetilde{g}\end{smallmatrix}\right)_{1, 0}
\xrightarrow{\left(\begin{smallmatrix} \binom{0}{(M, i_2)} \\ i_2\end{smallmatrix}\right)}
\left(\begin{smallmatrix} \Ker\widetilde{f}\oplus (M, \ L_2) \\ L_2\end{smallmatrix}\right)_{(0, 1), \ \binom{(N, \pi_1)\widetilde{g}}{0}}\xlongrightarrow{\left(\begin{smallmatrix} (1,0) \\ (N, \pi_1)\widetilde{g}\end{smallmatrix}\right)}
\left(\begin{smallmatrix} \Ker\widetilde{f} \\ (N, \ \Ker\widetilde{f}) \end{smallmatrix}\right)_{0, 1}\rightarrow 0$$
is an exact sequence of $\Lambda$-modules, i.e.,
$$0\longrightarrow {\rm H}_B\Ker \widetilde{g}\longrightarrow
\left(\begin{smallmatrix} \Ker\widetilde{f}\oplus \Hom_B(M, \ L_2) \\ L_2\end{smallmatrix}\right)_{(0, \ 1), \ \binom{(N, \pi_1)\widetilde{g}}{0}}
\longrightarrow {\rm H}_A\Ker \widetilde{f}\longrightarrow 0$$
is exact. (We stress that all the $\Lambda$-modules are in the second expression.)

\vskip5pt

Since $\Ker\widetilde{g}\in \ \mathcal Y$ and $(\mathcal V, \ \mathcal Y)$ is a cotorsion pair,
by Lemma \ref{extadj1}(4), ${\rm H}_B\Ker\widetilde{g}\in \binom{\mathcal{U}}{\mathcal V}^\bot$.
Since $\Ker\widetilde{f} \in \mathcal X$ and $(\mathcal U, \ \mathcal X)$ is a cotorsion pair, by Lemma \ref{extadj1}(3), ${\rm H}_A\Ker\widetilde{f}\in \binom{\mathcal{U}}{\mathcal V}^\perp$.
Thus  $L\cong \ \left(\begin{smallmatrix} \Ker\widetilde{f}\oplus \Hom_B(M,L_2) \\ L_2\end{smallmatrix}\right)_{(0, \ 1), \ \binom{(N, \pi_1)\widetilde{g}}{0}}\in \binom{\mathcal{U}}{\mathcal V}^\perp$.

\vskip5pt
    	
{\bf Case II:} \ Assume that \ $\Hom_A(N, \ \mathcal X)\subseteq \mathcal V$.  \ This is similar to {\bf Case I}.
Since $\Hom_A(N, \ L_1)\cong \Hom_A(N, \ \Ker\widetilde{f})\in \Hom_A(N, \ \mathcal X)\subseteq \mathcal V,$
the exact sequence
$$0\longrightarrow \Ker\widetilde{g}\xlongrightarrow{i_2}L_2\xlongrightarrow{\widetilde{g}}\Hom_A(N, L_1)\longrightarrow 0$$
splits. Thus $L = \left(\begin{smallmatrix} L_1 \\ L_2\end{smallmatrix}\right)_{\widetilde{f}, \widetilde{g}} \cong \left(\begin{smallmatrix}L_1 \\ \Ker\widetilde{g}\oplus \Hom_A(N, L_1) \end{smallmatrix}\right)_{\binom{(M, i_2)^{-1}\widetilde{f}}{0}, (0, 1)}$
and
$$0 \rightarrow \left(\begin{smallmatrix} \Ker\widetilde{f} \\ (N, \Ker\widetilde{f})\end{smallmatrix}\right)_{0,1}
\xlongrightarrow{\left(\begin{smallmatrix} i_1 \\ \binom{0}{(N, i_1)}\end{smallmatrix}\right)}
\left(\begin{smallmatrix} L_1 \\ \Ker\widetilde{g}\oplus (N, L_1) \end{smallmatrix}\right)_{\binom{(M, i_2)^{-1}\widetilde{f}}{0}, (0, 1)}
\xlongrightarrow{\left(\begin{smallmatrix} (M, i_2)^{-1}\widetilde{f} \\ (1,0) \end{smallmatrix}\right)}
\left(\begin{smallmatrix} (M, \Ker\widetilde{g}) \\ \Ker\widetilde{g} \end{smallmatrix}\right)_{1, 0}\rightarrow 0$$
is exact, i.e.,
$$0\longrightarrow {\rm H}_A\Ker \widetilde{f}\longrightarrow \left(\begin{smallmatrix} L_1 \\ \Ker\widetilde{g}\oplus \Hom_A(N, L_1) \end{smallmatrix}\right)_{\binom{(M, i_2)^{-1}\widetilde{f}}{0}, (0, 1)}
\longrightarrow {\rm H}_B\Ker \widetilde{g}\longrightarrow 0$$
is exact. By Lemma \ref{extadj1}(3), ${\rm H}_A\Ker\widetilde{f} \in  \binom{\mathcal{U}}{\mathcal{V}}^\perp$; and by  Lemma \ref{extadj1}(4), ${\rm H}_B\Ker\widetilde{g} \in \binom{\mathcal{U}}{\mathcal{V}}^\perp.$
Thus $L\in \binom{\mathcal{U}}{\mathcal{V}}^\perp$.

\vskip10pt

Finally, assume that \ $\Hom_B(M, \ \mathcal Y)\subseteq \mathcal U$ and \ $\Hom_A(N, \ \mathcal X) \subseteq \mathcal V$. Then both
\ $0\rightarrow \Ker\widetilde{f}\xlongrightarrow{i_1}L_1\xlongrightarrow{\widetilde{f}}\Hom_B(M, L_2)\rightarrow 0$  \ and \ $0\rightarrow \Ker\widetilde{g}\xlongrightarrow{i_2}L_2\xlongrightarrow{\widetilde{g}}\Hom_A(N, L_1)\rightarrow 0$ splits, and
$$L=\left(\begin{smallmatrix}L_1\\ L_2\end{smallmatrix}\right)_{\widetilde{f},\widetilde{g}}\cong \left(\begin{smallmatrix} \Ker\widetilde{f}\\ \Hom_A(N, \Ker\widetilde{f})\end{smallmatrix}\right)_{0, 1}\oplus
\left(\begin{smallmatrix} \Hom_B(M, \Ker \widetilde{g}) \\ \Ker \widetilde{g}\end{smallmatrix}\right)_{1,0}= {\rm H}_A\Ker\widetilde{f}\oplus {\rm H}_B\Ker \widetilde{g}
\in {\rm H}_A(\mathcal X)\oplus {\rm H}_B(\mathcal Y).$$
Conversely, it is clear that
${\rm H}_A(\mathcal X)\oplus {\rm H}_B(\mathcal Y)\subseteq \nabla(\mathcal X, \ \mathcal Y)$.
Thus \  $\nabla(\mathcal X, \ \mathcal Y)
={\rm H}_A(\mathcal X)\oplus {\rm H}_B(\mathcal Y)$.
\hfill $\square$

\subsection{Remark} In Theorem \ref{compare}, taking one of $(\mathcal U, \ \mathcal X)$  and
$(\mathcal V, \ \mathcal Y)$ being the projective cotorsion pair or the injective cotorsion pair,
and another being an arbitrary cotorsion pair, we conclude as follows: where ``$=$" follows from Corollary \ref{identification1}, and
``$\ne$" follows from Example \ref{ie}.

\vskip5pt

$(1)$ \  If \ $(\mathcal U, \ \mathcal X) = (_A\mathcal P, \ A\mbox{-}{\rm Mod})$, and \ $(\mathcal V, \ \mathcal Y)$
is an arbitrary cotorsion pair  in $B\mbox{-}{\rm Mod}$, then
$$(\Delta(_A\mathcal P, \ \mathcal V), \ \Delta(_A\mathcal P, \ \mathcal V)^\bot) = ({}^\perp\left(\begin{smallmatrix}A\mbox{-}{\rm Mod}\\ \mathcal Y\end{smallmatrix}\right), \ \left(\begin{smallmatrix}A\mbox{-}{\rm Mod}\\ \mathcal Y\end{smallmatrix}\right)).$$
But, in general \ $(^\bot\nabla(A\mbox{-}{\rm Mod}, \ \mathcal Y), \ \nabla(A\mbox{-}{\rm Mod}, \ \mathcal Y))\ne (\left(\begin{smallmatrix}_A\mathcal P\\ \mathcal V\end{smallmatrix}\right),  \ \left(\begin{smallmatrix}_A\mathcal P\\ \mathcal V\end{smallmatrix}\right)^\perp).$

\vskip5pt

$(2)$ \  If \ $(\mathcal U, \ \mathcal X) = (A\mbox{-}{\rm Mod}, \ _A\mathcal I)$, and \ $(\mathcal V, \ \mathcal Y)$ an arbitrary cotorsion pair  in $B\mbox{-}{\rm Mod}$,
then in general \ $(\Delta(A\mbox{-}{\rm Mod}, \ \mathcal V), \ \Delta(A\mbox{-}{\rm Mod}, \ \mathcal V)^\bot)\ne ({}^\perp\left(\begin{smallmatrix}_A\mathcal I\\ \mathcal Y\end{smallmatrix}\right), \ \left(\begin{smallmatrix}_A\mathcal I\\ \mathcal Y\end{smallmatrix}\right)).$
However, one has
$$(^\bot\nabla(_A\mathcal I, \ \mathcal Y), \ \nabla(_A\mathcal I, \ \mathcal Y)) =  (\left(\begin{smallmatrix}A\mbox{-}{\rm Mod}\\ \mathcal V\end{smallmatrix}\right), \ \left(\begin{smallmatrix}A\mbox{-}{\rm Mod}\\ \mathcal V\end{smallmatrix}\right)^\perp).$$

$(3)$ \  If \ $(\mathcal U, \ \mathcal X)$ is an arbitrary cotorsion pair  in $A\mbox{-}{\rm Mod}$, and
\ $(\mathcal V, \ \mathcal Y)  = (_B\mathcal P, \ B\mbox{-}{\rm Mod})$, then
$$(\Delta(\mathcal U, \ _B\mathcal P), \ \Delta(\mathcal U, \ _B\mathcal P)^\bot)= (^\perp\left(\begin{smallmatrix}\mathcal X\\ B\mbox{-}{\rm Mod}\end{smallmatrix}\right),  \ \left(\begin{smallmatrix}\mathcal X\\ B\mbox{-}{\rm Mod}\end{smallmatrix}\right)).$$
But, in general \ $(^\bot\nabla(\mathcal X, \ B\mbox{-}{\rm Mod}), \ \nabla(\mathcal X, \ B\mbox{-}{\rm Mod}))\ne (\left(\begin{smallmatrix}\mathcal U\\ _B\mathcal P\end{smallmatrix}\right),  \ \left(\begin{smallmatrix}\mathcal U\\ _B\mathcal P\end{smallmatrix}\right)^\perp).$

\vskip5pt

$(4)$ \ If \ $(\mathcal U, \ \mathcal X)$ is an arbitrary cotorsion pair in $A\mbox{-}{\rm Mod}$, and
\ $(\mathcal V, \ \mathcal Y)  = (B\mbox{-}{\rm Mod}, \ _B\mathcal I)$, then in general
\ $(\Delta(\mathcal U, \ B\mbox{-}{\rm Mod}), \ \Delta(\mathcal U, \ B\mbox{-}{\rm Mod})^\bot)\ne ({}^\perp\left(\begin{smallmatrix}\mathcal X\\ _B\mathcal I\end{smallmatrix}\right),  \ \left(\begin{smallmatrix}\mathcal X\\ _B\mathcal I\end{smallmatrix}\right)).$
However, one has
$$(^\perp\nabla(\mathcal X, \ _B\mathcal I), \ \nabla(\mathcal X, \ _B\mathcal I)) = (\left(\begin{smallmatrix}\mathcal U\\ B\mbox{-}{\rm Mod}\end{smallmatrix}\right), \ \left(\begin{smallmatrix}\mathcal U\\ B\mbox{-}{\rm Mod}\end{smallmatrix}\right)^\perp).$$

\vskip5pt

The above information is listed in Table 1 below, where
$$\mathcal A: = A\mbox{-}{\rm Mod}, \ \ \ \ \mathcal B: = B\mbox{-}{\rm Mod}, \ \ \ \ \mbox{proj.}: = \mbox{projective}.$$

\vskip10pt

\centerline{\bf Table 1: \ Cotorsion pairs in $\Lambda$-Mod}

\vspace{-10pt}
$${\tiny\begin{tabular}{|c|c|c|c|c|}
\hline
\phantom{\LARGE 0} & \multicolumn{2}{c|} {\tabincell{c}{Cotorsion pairs in Series I\\[3pt] $\varphi=0=\psi$}}
& \multicolumn{2}{c|} {\tabincell{c}{Cotorsion pairs in Series II\\[3pt] $M\otimes_AN=0 = N\otimes_BM$}}
\\[8pt]\hline
\tabincell{c}{$(_A\mathcal U, \ _A\mathcal X)$ \\ $(_B\mathcal V, \ _B\mathcal Y)$}
& \tabincell{c}{$\Tor_1(M,\mathcal U)=0$ \\ [3pt] $\Tor_1(N,\mathcal V)=0$: \\[3pt] $(^\perp\binom{\mathcal X}{\mathcal Y},\binom{\mathcal X}{\mathcal Y})$}
& \tabincell{c}{$\Ext^1(N, \mathcal X)=0$ \\ [3pt]$\Ext^1(M, \mathcal Y)=0$: \\[3pt] $(\binom{\mathcal U}{\mathcal V}, \binom{\mathcal U}{\mathcal V}^\perp)$}
& \tabincell{c}{$(\Delta(\mathcal U, \mathcal V), \ \Delta(\mathcal U, \mathcal V)^\bot)$}
& \tabincell{c}{$(^\bot\nabla(\mathcal X, \mathcal Y), \ \nabla(\mathcal X, \mathcal Y))$}
\\[15pt] \hline
\tabincell{c}{$(\mathcal P, \mathcal A$) \\ $(\mathcal V, \mathcal Y)$}
&\tabincell{c}{$N_B$ flat \\ [3pt] $(^\perp\binom{\mathcal A}{\mathcal Y}, \binom{\mathcal A}{\mathcal Y})$}
& \tabincell{c}{$_AN$, $_BM$  proj.: \\[3pt] $(\binom{\mathcal P}{\mathcal V}, \ \binom{\mathcal P}{\mathcal V}^\perp)$}
& \tabincell{c}{$(\Delta(\mathcal P, \mathcal V), \ \Delta(\mathcal P, \mathcal V)^\bot)$. \\[3pt] If $N_B$ flat then it is \\ [3pt]$(^\perp\binom{\mathcal A}{\mathcal Y}, \binom{\mathcal A}{\mathcal Y})$ \\[3pt] thus it is \\ [3pt]$(\Delta(\mathcal P, \mathcal V), \binom{\mathcal A}{\mathcal Y})$}
& \tabincell{c}{$(^\bot\nabla(\mathcal A, \mathcal Y), \ \nabla(\mathcal A, \mathcal Y))$. \\[3pt] Even if $_AN$, $_BM$  proj.,   \\[3pt]it  $\ne (\binom{\mathcal P}{\mathcal V}, \ \binom{\mathcal P}{\mathcal V}^\perp)$\\[3pt] in general.}
\\[30pt] \hline
\tabincell{c}{$(\mathcal A, \mathcal I$) \\ $(\mathcal V, \mathcal Y)$}
& \tabincell{c}{$M_A$, $N_B$ flat: \\ [3pt] $(^\perp\binom{\mathcal I}{\mathcal Y}, \ \binom{\mathcal I}{\mathcal Y})$}
& \tabincell{c}{$_BM$ proj.: \\[3pt] $(\binom{\mathcal A}{\mathcal V}, \ \binom{\mathcal A}{\mathcal V}^\perp)$}
& \tabincell{c}{$(\Delta(\mathcal A, \mathcal V), (\Delta(\mathcal A, \mathcal V)^\bot).$ \\ [3pt] Even if $M_A$, $N_B$ flat \\ [3pt] it $\ne (^\perp\binom{\mathcal I}{\mathcal Y}, \ \binom{\mathcal I}{\mathcal Y})$\\ [3pt] in general}
& \tabincell{c}{$(^\bot\nabla(\mathcal I,  \mathcal Y), \ \nabla(\mathcal I,  \mathcal Y)).$ \\ [3pt] If $_BM$ proj. then it is \\ [3pt]$(\binom{\mathcal A}{\mathcal V}, \ \binom{\mathcal A}{\mathcal V}^\perp)$\\[3pt]thus it is \\[3pt]$(\binom{\mathcal A}{\mathcal V}, \ \nabla(\mathcal I, \mathcal Y))$}
\\[30pt] \hline
\tabincell{c}{$(\mathcal U, \mathcal X$) \\ $(\mathcal P, \mathcal B)$}
& \tabincell{c}{$M_A$ flat: \\[3pt] $(^\perp\binom{\mathcal X}{\mathcal B}, \ \binom{\mathcal X}{\mathcal B})$}
& \tabincell{c}{$_BM, \ _AN$ proj.: \\[3pt] $(\binom{\mathcal U}{\mathcal P}, \ \binom{\mathcal U}{\mathcal P}^\perp)$}
& \tabincell{c}{$(\Delta(\mathcal U, \mathcal P), \Delta(\mathcal U, \mathcal P)^\bot).$ \\[3pt] If $M_A$ flat then it is \\[3pt]$(^\perp\binom{\mathcal X}{\mathcal B}, \binom{\mathcal X}{\mathcal B})$\\[3pt] thus it is \\ [3pt]$((\Delta(\mathcal U, \mathcal P), \ \binom{\mathcal X}{\mathcal B})$}
& \tabincell{c}{$(^\bot\nabla(\mathcal X, \mathcal B), \nabla(\mathcal X, \mathcal B))$. \\[3pt] Even if $_BM, \ _AN$ proj., \\[3pt] it $\ne (\binom{\mathcal U}{\mathcal P}, \ \binom{\mathcal U}{\mathcal P}^\perp)$\\[3pt]in general }
\\[30pt] \hline
\tabincell{c}{$(\mathcal U, \mathcal X)$ \\ $(\mathcal B, \mathcal I)$}
& \tabincell{c}{$M_A$, $N_B$ flat: \\[3pt] $(^\perp\binom{\mathcal X}{\mathcal I}, \ \binom{\mathcal X}{\mathcal I})$}
& \tabincell{c}{$_AN$ proj.: \\[3pt] $(\binom{\mathcal U}{\mathcal B}, \ \binom{\mathcal U}{\mathcal B}^\bot)$}
& \tabincell{c}{$(\Delta(\mathcal U, \mathcal B), \ \Delta(\mathcal U, \mathcal B)^\bot)$.\\[3pt]Even if $M_A$, $N_B$ flat,\\ [3pt]it $\ne (^\perp\binom{\mathcal X}{\mathcal I}, \ \binom{\mathcal X}{\mathcal I})$\\ [3pt]in general}
& \tabincell{c}{$(^\bot\nabla(\mathcal X,  \mathcal I), \ \nabla(\mathcal X,  \mathcal I)).$ \\ [3pt] If $_AN$ proj. then it is \\ [3pt]$(\binom{\mathcal U}{\mathcal B}, \ \binom{\mathcal U}{\mathcal B}^\perp)$\\[3pt]thus it is \\[3pt]$(\binom{\mathcal U}{\mathcal B}, \ \nabla(\mathcal X, \mathcal I))$}
\\[30pt] \hline
\end{tabular}}$$

\vskip5pt

\subsection{Monomorphism categories and epimorphism categories} \ Even if in the case of ``$\ne$" in general, the two cotorsion pairs can be the same, in some special cases.

\vskip5pt

If $_AN$ and $_BM$ are projective, then cotorsion pairs  \ $(\left(\begin{smallmatrix}_A\mathcal P\\ _B\mathcal P\end{smallmatrix}\right), \ \left(\begin{smallmatrix}_A\mathcal P\\ _B\mathcal P\end{smallmatrix}\right)^\perp)$  and $(^\bot{\rm Epi}(\Lambda), \ {\rm Epi}(\Lambda)) = (^\perp\nabla (A\mbox{-}{\rm Mod}, \ B\mbox{-}{\rm Mod}), \ \nabla (A\mbox{-}{\rm Mod}, \ B\mbox{-}{\rm Mod}))$
are not equal in general (cf. Example \ref{ie}); but the following result claims that they can be the same in some special cases.

\vskip5pt

Also, if $M_A$ and $N_B$ are flat, then \ $({}^\perp\left(\begin{smallmatrix}_A\mathcal I\\ _B\mathcal I\end{smallmatrix}\right), \ \left(\begin{smallmatrix}_A\mathcal I\\ _B\mathcal I\end{smallmatrix}\right))
\ne ({\rm Mon}(\Lambda), \ {\rm Mon}(\Lambda)^\bot)$ in general
(cf. Example \ref{ie}); but the following result claims that they can be the same in some special cases.

\vskip5pt

\begin{thm}\label{ctp4} \ Let \ $\Lambda = \left(\begin{smallmatrix} A & N \\ M & B \end{smallmatrix}\right)$ be a Morita ring which is an Artin algebra with $M\otimes_A N = 0 = N\otimes_BM$.
Assume that $A$ and $B$ are quasi-Frobenius rings, that $_AN$ and $_BM$ are projective, and that $M_A$ and $N_B$ are flat. Then

\vskip5pt

${\rm (1)}$ \ \ $\Lambda$ is a Gorenstein ring with ${\rm inj.dim} _\Lambda\Lambda \le 1$, and
$_\Lambda \mathcal P^{<\infty} = \ _\Lambda \mathcal P^{\le 1} = \binom{_A\mathcal P}{_B\mathcal P} = \binom{_A\mathcal I}{_B\mathcal I} = \ _\Lambda \mathcal I^{\le 1} = \ _\Lambda \mathcal I^{<\infty}.$

\vskip5pt

${\rm (2)}$ \ \ The cotorsion pair $(^\perp\binom{_A\mathcal I}{_B\mathcal I}, \ \binom{_A\mathcal I}{_B\mathcal I})$ coincides with \ $({\rm Mon}(\Lambda), \ {\rm Mon}(\Lambda)^\bot);$ and it is
exactly the Gorenstein-projective cotorsion pair  $({\rm GP}(\Lambda), \ _\Lambda\mathcal P^{\le 1})$.  So, it is complete and hereditary, and
$${\rm GP}(\Lambda) = {\rm Mon}(\Lambda) = {}^\perp \ _\Lambda\mathcal P, \ \ \ \ {\rm Mon}(\Lambda)^\bot = \ _\Lambda \mathcal P^{\le 1}.$$

\vskip5pt

${\rm (2)'}$ \ \ The cotorsion pair $(\binom{_A\mathcal P}{_B\mathcal P}, \ \binom{_A\mathcal P}{_B\mathcal P}^\perp)$ coincides with \ $(^\bot{\rm Epi}(\Lambda), \ {\rm Epi}(\Lambda));$ and it
is exactly the Gorenstein-injective cotorsion pair \ $(_\Lambda \mathcal P^{\le 1}, \ {\rm GI}(\Lambda))$.  So, it is  complete and hereditary, and
$$ {\rm GI}(\Lambda) = {\rm Epi}(\Lambda) = \ _\Lambda\mathcal I{}^\perp, \ \ \ \ ^\bot{\rm Epi}(\Lambda) = \ _\Lambda \mathcal P^{\le 1}.$$
\end{thm}

\vskip5pt

\begin{exm}\label{examctp4} $(1)$ \ {\it We give an example to justify the existence of the assumptions in {\rm Theorem \ref{ctp4}}.
 Let $Q$ be the quiver
$$\xymatrix{1\ar[r] & 2\ar[r] & 3\ar[r] & \cdots \ar[r] & n \ar@/_20pt/[llll]}$$

\vskip5pt

\noindent and $A=kQ/J^h$, where $J$ is the ideal of path algebra $kQ$ generated by all the arrows, and $2\le h\le n$. Then $A$ is a self-injective algebra, in particular, a quasi-Frobenius ring. Let $e=e_i$, $e'=e_j$, where $1\le i < j\le n$,  satisfying $j-i\ge h$. Then \ $e'Ae=e_jAe_i=0$.
Put $M: = Ae\otimes_k e'A$. Then $_AM$ and $M_A$ are projective, and  $M\otimes_AM  = (Ae\otimes_k e'A)\otimes_A (Ae\otimes_k e'A) = Ae\otimes_k(e'A\otimes_A Ae)\otimes_k e'A = 0.$

\vskip5pt

Take \ $\Lambda = \left(\begin{smallmatrix} A & M \\ M & A \end{smallmatrix}\right)$. Then $\Lambda$ satisfies
all the conditions in {\rm Theorem \ref{ctp4}}.}\end{exm}

\begin{rem}\label{remctp4} {\rm $(1)$} \ {\it Non-zero Morita rings $\Lambda$ in {\rm Theorem \ref{ctp4}} do not satisfy the sufficient condition for self-injective algebras in {\rm [GrP, Proposition 3.7]}.  In fact,
$\Lambda$ can not be quasi-Frobenius$:$ otherwise ${\rm Mon}(\Lambda) = {\rm GP}(\Lambda) = \Lambda$-{\rm Mod}, which is absurd $!$

\vskip5pt

{\rm $(2)$} \ Although {\rm Theorem \ref{ctp4}} does not give new cotorsion pairs, in the sense that
they are just the Gorenstein-projective $($respectively, Gorenstein-injective$)$ cotorsion pairs,
however, \ ${\rm GP}(\Lambda) = {\rm Mon}(\Lambda)$  is a new result. In the special case of triangular matrix rings, this is known, by
{\rm [LiZ, Thm. 1.1], [XZ, Cor.1.5], [Z2, Thm.1.4], [LuoZ1, Thm.4.1], [ECIT, Thm.3.5]}. For more
relations between monomorphism categories and the Gorenstein-projective modules, we refer to {\rm [Z1], [GrP], [LuoZ2], [GaP], [ZX], [HLXZ]}.}
\end{rem}

\subsection{Modules $\binom{_A\mathcal P}{_B\mathcal P}$ and $\binom{_A\mathcal I}{_B\mathcal I}$} To prove  Theorem \ref{ctp4}, we need the following fact, which is of independent interest.

\vskip5pt

\begin{lem}\label{proj-injdim} \ Let \ $\Lambda = \left(\begin{smallmatrix} A & N \\ M & B \end{smallmatrix}\right)$ be a Morita ring with $M\otimes_A N = 0 = N\otimes_BM.$

\vskip5pt

$(1)$ \  Assume that $_AN$ and $_BM$ are projective modules. Let $\binom{P}{Q}_{f, g}\in \binom{_A\mathcal P}{_B\mathcal P}$. Then
$$0\rightarrow
\left(\begin{smallmatrix} N\otimes_BQ\\M\otimes_AP \end{smallmatrix}\right)_{0,0}
\stackrel{\left(\begin{smallmatrix} -\binom{g}{1} \\ \binom{1}{f}\end{smallmatrix}\right)} \longrightarrow
\left(\begin{smallmatrix} P\\ M\otimes_AP\end{smallmatrix}\right)_{1, 0}
\oplus \left(\begin{smallmatrix} N\otimes_BQ \\ Q\end{smallmatrix}\right)
_{0, 1}\stackrel{\left(\begin{smallmatrix}	\binom{1}{f}, &  \binom{g}{1}\end{smallmatrix}\right)}\longrightarrow
\left(\begin{smallmatrix} P \\ Q\end{smallmatrix}\right)_{f,g}\rightarrow 0$$
is a projective resolution of $\binom{P}{Q}_{f, g}$. In particular,   ${\rm proj.dim} \binom{P}{Q}_{f, g} \le 1,
\ \forall \ \binom{P}{Q}_{f, g}\in \binom{_A\mathcal P}{_B\mathcal P}$.

\vskip5pt

$(2)$ \ Assume that $M_A$ and $N_B$ are flat modules. Let $\binom{I}{J}_{f, g}\in \binom{_A\mathcal I}{_B\mathcal I}$. Then
$$0\rightarrow
\left(\begin{smallmatrix}I\\ J\end{smallmatrix}\right)_{f, g}\stackrel {\left(\begin{smallmatrix} \binom {1} {\widetilde{g}} \\ \binom {\widetilde{f}}{1}\end{smallmatrix}\right)} \longrightarrow
\left(\begin{smallmatrix} I\\ \Hom_A(N, I)\end{smallmatrix}\right)_{0, \epsilon'_{_I}}\oplus \left(\begin{smallmatrix} \Hom_B(M, J)\\ J\end{smallmatrix}\right)_{\epsilon_{_J}, 0}
\stackrel {\left(\begin{smallmatrix}\binom{\widetilde{f}}{1}, &  -\binom{1}{\widetilde{g}}\end{smallmatrix}\right)} \longrightarrow
\left(\begin{smallmatrix}\Hom_B(M, J)\\ \Hom_A(N, I)\end{smallmatrix}\right)_{0, 0}\rightarrow 0$$
is an injective resolution of $\binom{I}{J}_{f, g}$. In particular,   ${\rm inj.dim} \binom{I}{J}_{f, g} \le 1, \ \forall \ \binom{I}{J}_{f, g}\in \binom{_A\mathcal I}{_B\mathcal I}$.\end{lem}

\begin{rem} {\it The condition $M\otimes_A N = 0 = N\otimes_BM$ can not be relaxed to $\phi = 0 = \psi$. Otherwise,  for example in {\rm (1)}, $\Ker \left(\begin{smallmatrix}	\binom{1}{f}, &  \binom{g}{1}\end{smallmatrix}\right)=\left(\begin{smallmatrix} N\otimes_BQ\\M\otimes_AP \end{smallmatrix}\right)_{-(1_M\otimes_Ag),-(1_N\otimes_Bf)}$, which is no longer a projective left $\Lambda$-module.}
The similar remark for (2).
\end{rem}

\vskip5pt \noindent {\bf Proof of Lemma \ref{proj-injdim}}  (1) \ Thanks to the assumption $M\otimes_A N = 0 = N\otimes_BM,$ the given maps are $\Lambda$-maps (otherwise $\left(\begin{smallmatrix} -\binom{g}{1} \\ \binom{1}{f}\end{smallmatrix}\right)$ is not necessarily a $\Lambda$-map in general, even if $\phi = 0 = \psi$). We omit the details.
The given sequence of $\Lambda$-modules is exact, since
$$\xymatrix {0\ar[r] &  N\otimes_BQ \ar[r]^-{\binom{-g}{1}} & P \oplus (N\otimes_BQ) \ar[rr]^-{(1, g)} && P \ar[r] & 0}$$
and
$$\xymatrix {0\ar[r] &  M\otimes_AP \ar[r]^-{\binom{-1}{f}} & (M\otimes_AP) \oplus Q \ar[rr]^-{(f, 1)} &&  Q \ar[r] & 0}$$
are exact.

\vskip5pt

We claim that  $\left(\begin{smallmatrix} N\otimes_BQ\\M\otimes_AP \end{smallmatrix}\right)_{0,0}$ is a projective left $\Lambda$-module.
In fact, since $_AN$ and $_BQ$ are projective, $N\otimes_BQ$ is a projective left $A$-module. Since $M\otimes_AN = 0$, it follows that
$$\left(\begin{smallmatrix} N\otimes_BQ \\ 0\end{smallmatrix}\right)_{0, 0} = \left(\begin{smallmatrix} N\otimes_BQ\\M\otimes_AN\otimes_BQ\end{smallmatrix}\right)_{0, 0}$$
is a projective left $\Lambda$-module. Similarly,  $\left(\begin{smallmatrix} 0 \\ M\otimes_AP\end{smallmatrix}\right)_{0, 0}$ is a projective left $\Lambda$-module. Thus, $\left(\begin{smallmatrix} N\otimes_BQ\\M\otimes_AP \end{smallmatrix}\right)_{0,0} = \binom{N\otimes_BQ}{0}_{0, 0}\oplus \binom{0}{M\otimes_AP}_{0, 0}$ is a projective left $\Lambda$-module.

\vskip10pt

(2) \ This can be similarly proved as (1). Since \  $N\otimes_BM = 0= M\otimes_A N,$ the given sequence  is an exact sequence of $\Lambda$-maps.
Since $M_A$ is flat and $_BJ$ is injective, \  $\Hom_B(M, J)$ is an injective left $A$-module, and hence
$$\left(\begin{smallmatrix}\Hom_B(M, J)\\ 0\end{smallmatrix}\right)_{0, 0} = \left(\begin{smallmatrix}\Hom_B(M, J)\\ \Hom_A(N, \ \Hom_B(M, J))\end{smallmatrix}\right)_{0, 0}$$
is an injective left $\Lambda$-module. Similarly,  $\binom{0}{\Hom_A(N, I)}_{0, 0}$ is an injective left $\Lambda$-module.
Thus, $\binom{\Hom_B(M, J)}{\Hom_A(N, I)}_{0, 0} = \binom{\Hom_B(M, J)}{0}_{0, 0}\oplus \binom{0}{\Hom_A(N, I)}_{0, 0}$ is an injective left $\Lambda$-module.
\hfill $\square$

\vskip5pt

\subsection {Proof of Theorem \ref{ctp4}.} \ (1) \  Since $A$ is quasi-Frobenius, $_AA\in \ _A\mathcal I$.
Since $B$ is quasi-Frobenius and $_BM$ is projective,  \ $_BM\in \ _B\mathcal I$.
Since \ $M_A$ and $N_B$ are flat, it follows from Lemma \ref{proj-injdim}(2) that \ ${\rm inj.dim} _\Lambda \binom{A}{M}_{1, 0} \le 1$.

\vskip5pt

Similarly,   since $A$ is quasi-Frobenius and $_AN$ is projective, $_AN\in \ _A\mathcal I$.
Since $B$ is quasi-Frobenius,  \ $_BB\in \ _B\mathcal I$.
Since \ $M_A$ and $N_B$ are flat, ${\rm inj.dim} _\Lambda \binom{N}{B}_{0, 1} \le 1,$
by Lemma \ref{proj-injdim}(2).

\vskip5pt

Thus, ${\rm inj.dim} _\Lambda\Lambda \le 1$.
By the right module version of Lemma \ref{proj-injdim}(2) one knows  \ ${\rm inj.dim} \Lambda_\Lambda \le 1.$
Thus $\Lambda$ is a Gorenstein ring.

\vskip5pt

Since $\Lambda$ is Gorenstein with  ${\rm inj.dim} _\Lambda\Lambda \le 1$,  it is well-known that $_\Lambda \mathcal P^{<\infty} = \ _\Lambda \mathcal P^{\le 1} = \ _\Lambda \mathcal I^{\le 1} = \ _\Lambda \mathcal I^{<\infty}.$

\vskip5pt

Since $_AN$ and $_BM$ are projective modules, it follows from Lemma \ref{proj-injdim}(1) that \ $\binom{_A\mathcal P}{_A\mathcal P} \subseteq \ _\Lambda \mathcal P^{\le 1}$. On the other hand, for any $\left(\begin{smallmatrix}X\\ Y\end{smallmatrix}\right)_{f, g}\in \ _\Lambda \mathcal P^{\le 1},$
let $0 \longrightarrow \binom{P_{11}}{P_{12}} \longrightarrow \binom{P_{01}}{P_{02}} \longrightarrow \left(\begin{smallmatrix}X\\ Y\end{smallmatrix}\right)_{f, g} \longrightarrow 0$ be a projective resolution of $\left(\begin{smallmatrix}X\\ Y\end{smallmatrix}\right)_{f, g}$. Then
one has exact sequence $0\longrightarrow P_{11} \longrightarrow P_{01} \longrightarrow X \longrightarrow 0$. Since $_AN$ is projective, $P_{11}$ and $P_{01}$ are projective (cf. Subsection 2.5), and hence injective. Thus the exact sequence splits and hence $X$ is projective. Similarly, $Y$ is projective. This shows $\left(\begin{smallmatrix}X\\ Y\end{smallmatrix}\right)_{f, g}\in \binom{_A\mathcal P}{_A\mathcal P}$. Hence $\ _\Lambda \mathcal P^{\le 1} = \binom{_A\mathcal P}{_A\mathcal P} = \binom{_A\mathcal I}{_A\mathcal I}$.

\vskip5pt

(2) \ \  By (1), $\Lambda$ is Gorenstein and $\binom{_A\mathcal I}{_B\mathcal I} = \ _\Lambda \mathcal P^{<\infty}$. Thus, ${}^\perp\binom{_A\mathcal I}{_B\mathcal I} = {\rm GP}(\Lambda)$
and $({}^\perp\binom{_A\mathcal I}{_B\mathcal I}, \ \binom{_A\mathcal I}{_B\mathcal I})$ is just the Gorenstein-projective cotorsion pair, so it is complete and hereditary.

\vskip5pt

By Theorem \ref{compare}(1), ${\rm Mon}(\Lambda)^\bot \subseteq \binom{_A\mathcal I}{_B\mathcal I} = \ _\Lambda \mathcal P^{<\infty}$.
Thus, to see $({\rm Mon}(\Lambda), \ {\rm Mon}(\Lambda)^\bot) = ({\rm GP}(\Lambda), \ _\Lambda \mathcal P^{<\infty}),$ it suffices to show \ ${\rm Mon}(\Lambda)\subseteq {\rm GP}(\Lambda)$.
Since $\Lambda$ is Gorenstein, ${\rm GP}(\Lambda) = \ ^{\bot_{\ge 1}} \ _\Lambda \mathcal P$. See Subsection 2.9.
While ${\rm inj.dim} _\Lambda\Lambda \le 1$, each projective $\Lambda$-module is of injective dimension $\le 1$.
It follows that \ $^\bot \ _\Lambda \mathcal P = \ ^{\bot_{\ge 1}} \ _\Lambda \mathcal P.$
Thus, it suffices to show \ ${\rm Mon}(\Lambda)\subseteq \ ^\bot \ _\Lambda \mathcal P$, namely,
it suffices  to show
$$\Ext_\Lambda^1({\rm Mon}(\Lambda), \ {\rm T}_A(_A\mathcal P)\oplus {\rm T}_B(_B\mathcal P))=0.$$
This is indeed true. In fact, let \ $\left(\begin{smallmatrix} X \\ Y\end{smallmatrix}\right)_{f,g}\in {\rm Mon}(\Lambda).$ For any $P\in \ _A\mathcal P$,  there is an exact sequence
$$0\longrightarrow \left(\begin{smallmatrix} 0 \\ M\otimes_AP\end{smallmatrix}\right)_{0,0}\xlongrightarrow{\binom{0}{1}}
{\rm T}_A P = \left(\begin{smallmatrix} P \\ M\otimes_AP\end{smallmatrix}\right)_{1,0}\xlongrightarrow{\binom{1}{0}}
\left(\begin{smallmatrix} P \\ 0 \end{smallmatrix}\right)\longrightarrow 0.$$
By Lemma \ref{extadj2}(2) one has
$$\Ext_\Lambda^1(\left(\begin{smallmatrix} X \\ Y\end{smallmatrix}\right)_{f,g}, \ \left(\begin{smallmatrix} 0 \\ M\otimes_AP\end{smallmatrix}\right)_{0,0})=
\Ext_B^1(\Coker f, \ M\otimes_AP) = 0$$
since $M\otimes_AP$ is projective as a left $B$-module  (and hence injective). By Lemma \ref{extadj2}(1), one has
$$\Ext_\Lambda^1(\left(\begin{smallmatrix} X \\ Y\end{smallmatrix}\right)_{f,g}, \ \left(\begin{smallmatrix} P \\ 0 \end{smallmatrix}\right)_{0,0})=
\Ext_A^1(\Coker g, \ P)=0.$$
Thus $\Ext_\Lambda^1(\left(\begin{smallmatrix} X \\ Y\end{smallmatrix}\right)_{f,g}, \ {\rm T}_A P )= 0$. This shows $\Ext_\Lambda^1({\rm Mon}(\Lambda), \ {\rm T}_A(_A\mathcal P))=0.$

\vskip5pt

Similarly, $\Ext_\Lambda^1({\rm Mon}(\Lambda), \ {\rm T}_B(_B\mathcal P))=0.$ Thus \ $({\rm Mon}(\Lambda), \ {\rm Mon}(\Lambda)^\bot) = ({\rm GP}(\Lambda), \ _\Lambda \mathcal P^{<\infty}),$
in particular,
 \ ${\rm Mon}(\Lambda) = \mathcal {GP}(\Lambda) = \ {}^\perp \ _\Lambda\mathcal P, \ \ \ \ {\rm Mon}(\Lambda)^\bot = \ _\Lambda \mathcal P^{\le 1}.$

\vskip5pt

The assertion \ ${\rm (2)'}$ \ is the dual of (2). \hfill $\square$

\section{\bf Completeness}

To study abelian model structures on Morita rings,  a key step  is to know the completeness of
cotorsion pairs in Morita rings.

\subsection{Completeness via cogenerations by sets}

First, by [ET2, Theorem 10], one has

\vskip5pt

\begin{prop}\label{generatingcomplete} \ Let  $\Lambda = \left(\begin{smallmatrix} A & N \\ M & B\end{smallmatrix}\right)$  be a Morita ring  with $\phi = 0=\psi$,  $(\mathcal U,  \mathcal X)$ and  $(\mathcal V,  \mathcal Y)$  cotorsion pairs in $A\mbox{-}{\rm Mod}$ and $B\mbox{-}{\rm Mod}$, cogenerated by sets $S_1$ and $S_2$, respectively.

\vskip5pt

$(1)$ \ If  $\Tor^A_1(M,  \mathcal U)=0 = \Tor^B_1(N,  \mathcal V)$,  then cotorsion pair  $(^\perp\binom{\mathcal X}{\mathcal Y},  \binom{\mathcal X}{\mathcal Y})$ is cogenerated by  ${\rm T}_A(S_1) \cup {\rm T}_B(S_2),$ and hence complete.

\vskip5pt

{\rm (2)} \ If \ $M\otimes_A N = 0 = N\otimes_BM$, then cotorsion pair \ $(^\perp\nabla(\mathcal X, \ \mathcal Y), \ \nabla(\mathcal X, \ \mathcal Y))$  is cogenerated by  ${\rm Z}_A(S_1) \cup {\rm Z}_B(S_2)$, and hence complete.

\end{prop}
\begin{proof} \ (1) \ By Theorem \ref{ctp1}(1), \ $(^\perp\binom{\mathcal X}{\mathcal Y}, \ \binom{\mathcal X}{\mathcal Y})$ is a cotorsion pair in $\Lambda$-Mod. By Lemma \ref{destheta}(1),
\ $\left (\begin{smallmatrix} \mathcal X \\ \mathcal Y\end{smallmatrix}\right) = \left (\begin{smallmatrix} S_1^\perp \\ S_2^\perp\end{smallmatrix}\right) = ({\rm T}_A(S_1) \cup {\rm T}_B(S_2))^\perp.$ Thus,
\ $(^\perp\binom{\mathcal X}{\mathcal Y}, \ \binom{\mathcal X}{\mathcal Y})$ is complete, by Proposition \ref{cogenerated}.

\vskip5pt

(2) \ Without loss of generality, one may assume that $S_1 \supseteq \ _A\mathcal P$ and
$S_2 \supseteq \ _B\mathcal P$. Then by Lemma \ref{desdelta}(2) one has \
$\nabla(\mathcal X, \ \mathcal Y) = \nabla(S_1^\perp, \ S_2^\perp) = ({\rm Z}_A(S_1) \cup {\rm Z}_B(S_2))^\perp.$
\end{proof}

\vskip5pt

Proposition \ref{generatingcomplete} gives some information on the completeness of the cotorsion pairs in Morita rings.
However, since Proposition \ref{cogenerated} has no dual versions in general, there are no results on the completeness of $(\left(\begin{smallmatrix}\mathcal U\\ \mathcal V\end{smallmatrix}\right),  \left(\begin{smallmatrix}\mathcal U\\ \mathcal V\end{smallmatrix}\right)^\perp)$  and  $(\Delta(\mathcal U,  \mathcal V)$,  $\Delta(\mathcal U, \ \mathcal V)^{\bot})$; moreover, it is more natural to study the completeness of the cotorsion pairs given in
Theorems \ref{ctp1} and \ref{ctp6}, directly from the completeness of $(\mathcal U,  \mathcal X)$ and  $(\mathcal V,  \mathcal Y)$, rather than requiring that
they are cogenerated by sets.
Thus, we need  module-theoretical methods to the completeness of the cotorsion pairs in Morita rings.

\vskip5pt

Such a general investigation is difficult.
We will deal with this question, by assuming that one of \ $(\mathcal U, \ \mathcal X)$ and \ $(\mathcal V, \ \mathcal Y)$
is arbitrary, and that another is the projective or injective cotorsion pair. In view of Section 4, we only consider cotorsion pairs in Theorem \ref{ctp1}.

\subsection{Main results on completeness} Take  $(\mathcal V,  \mathcal Y)$ to be an arbitrary complete cotorsion pair in $B$-Mod.
For cotorsion pair  $({}^\perp\left(\begin{smallmatrix}\mathcal X\\ \mathcal Y\end{smallmatrix}\right),  \left(\begin{smallmatrix}\mathcal X\\ \mathcal Y\end{smallmatrix}\right))$ in Theorem \ref{ctp1}(1),
taking  $(\mathcal U, \ \mathcal X) = (_A\mathcal P,  A\mbox{-}{\rm Mod})$, we have assertion (1) below;
for cotorsion pair  $(\left(\begin{smallmatrix}\mathcal U\\ \mathcal V\end{smallmatrix}\right),  \left(\begin{smallmatrix}\mathcal U\\ \mathcal V\end{smallmatrix}\right)^\perp)$ in Theorem \ref{ctp1}(2),
taking  $(\mathcal U,  \mathcal X) = (A\mbox{-}{\rm Mod}, \ _A\mathcal I)$, we have assertion (2) below.

\begin{thm} \label{ctp2}  \ Let \ $\Lambda = \left(\begin{smallmatrix} A & N \\ M & B \end{smallmatrix}\right)$ be a Morita ring with $\phi = 0= \psi$,
and \  $(\mathcal V, \ \mathcal Y)$ a complete cotorsion pair in $B\mbox{-}{\rm Mod}$. Suppose that \ $N_B$ is flat and \ $_BM$ is projective.
	
\vskip5pt
	
$(1)$  \  If \ $M\otimes_A\mathcal P\subseteq \mathcal Y$,
then \ $({}^\perp\binom{A\mbox{-}{\rm Mod}}{\mathcal Y}, \ \binom{A\mbox{-}{\rm Mod}}{\mathcal Y})$  is a complete cotorsion pair in $\Lambda\mbox{-}{\rm Mod};$
and it is hereditary if \  $(\mathcal V, \ \mathcal Y)$ is hereditary.
	
\vskip5pt

Moreover, if \ $M\otimes_AN = 0= N\otimes_BM$, then \ \ ${}^\perp\left (\begin{smallmatrix} A\mbox{-}{\rm Mod} \\ \mathcal Y\end{smallmatrix}\right) = {\rm T}_A(_A\mathcal P)\oplus {\rm T}_B(\mathcal V),$  \ and hence
$$({\rm T}_A(_A\mathcal P)\oplus {\rm T}_B(\mathcal V), \ \left (\begin{smallmatrix} A\mbox{-}{\rm Mod} \\ \mathcal Y\end{smallmatrix}\right))$$ is a complete cotorsion pair$;$ and it is hereditary if \  $(\mathcal V, \ \mathcal Y)$ is hereditary.

\vskip10pt

$(2)$ \ If \ $\Hom_A(N, \ _A\mathcal I) \subseteq \mathcal V$, then \ $(\binom{A\mbox{-}{\rm Mod}}{\mathcal V},  \ \binom{A\mbox{-}{\rm Mod}}{\mathcal V}^\perp)$
is a complete cotorsion pair in $\Lambda\mbox{-}{\rm Mod};$ and it is hereditary if \  $(\mathcal V, \ \mathcal Y)$ is hereditary.
	
\vskip5pt

Moreover, if \  $M\otimes_AN = 0= N\otimes_BM$, then \ \ $\binom{A\mbox{-}{\rm Mod}}{\mathcal V}^\perp = {\rm H}_A(_A\mathcal I)\oplus {\rm H}_B(\mathcal Y)$, and hence
$$(\left(\begin{smallmatrix} A\mbox{-}{\rm Mod} \\ \mathcal V\end{smallmatrix}\right), \ {\rm H}_A(_A\mathcal I)\oplus {\rm H}_B(\mathcal Y))$$ is a complete cotorsion pair$;$
and it is hereditary if \  $(\mathcal V, \ \mathcal Y)$ is hereditary. \end{thm}

\begin{rem}\label{remcomplete1}  (1) \ {\it $_BM$ is injective, then \ $M\otimes_A\mathcal P\subseteq \mathcal Y$ always
holds.}

\vskip5pt

$(2)$ \ {\it If \ $B$ is quasi-Frobenius and  \ $N_B$ is flat, then \ $\Hom_A(N, \ _A\mathcal I) \subseteq \mathcal V$ always holds.}
\end{rem}

\vskip5pt

Take  $(\mathcal U,  \mathcal X)$ to be an arbitrary complete cotorsion pair in $A$-Mod.
For cotorsion pair  $({}^\perp\left(\begin{smallmatrix}\mathcal X\\ \mathcal Y\end{smallmatrix}\right),  \left(\begin{smallmatrix}\mathcal X\\ \mathcal Y\end{smallmatrix}\right))$ in Theorem \ref{ctp1}(1),
taking $(\mathcal V,  \mathcal Y)  = (_B\mathcal P,  B\mbox{-}{\rm Mod})$, we have assertion (1) below;
for  cotorsion pair  $(\left(\begin{smallmatrix}\mathcal U\\ \mathcal V\end{smallmatrix}\right),  \left(\begin{smallmatrix}\mathcal U\\ \mathcal V\end{smallmatrix}\right)^\perp)$ in Theorem \ref{ctp1}(2),
taking $(\mathcal V,  \mathcal Y) = (B\mbox{-}{\rm Mod}, \ _B\mathcal I)$, we have assertion (2) below.

\begin{thm} \label{ctp3} \ Let \ $\Lambda = \left(\begin{smallmatrix} A & N \\ M & B \end{smallmatrix}\right)$ be a Morita ring with $\phi = 0= \psi$, and \ $(\mathcal U, \ \mathcal X)$ a complete cotorsion pair in $A\mbox{-}{\rm Mod}$.
Suppose that \ $M_A$ is flat and \ $_AN$ is projective.
	
\vskip5pt

$(1)$ \   If \ $N\otimes_B  \mathcal P\subseteq \mathcal X$,
then \ $({}^\perp\binom{\mathcal X}{B\mbox{-}{\rm Mod}},  \ \binom{\mathcal X}{B\mbox{-}{\rm Mod}})$  is a complete cotorsion pair in $\Lambda\mbox{-}{\rm Mod};$ and it is hereditary if \  $(\mathcal U, \ \mathcal X)$ is hereditary.

\vskip5pt

Moreover, if $M\otimes_AN = 0= N\otimes_BM$, then \ \ ${}^\perp\left(\begin{smallmatrix}\mathcal X\\ B\mbox{-}{\rm Mod}\end{smallmatrix}\right) = {\rm T}_A(\mathcal U)\oplus {\rm T}_B(_B\mathcal P)$, \ and hence
$$ ({\rm T}_A(\mathcal U)\oplus {\rm T}_B(_B\mathcal P),\ \left(\begin{smallmatrix}\mathcal X\\ B\mbox{-}{\rm Mod}\end{smallmatrix}\right))$$ is a complete cotorsion pair$;$ and it is hereditary if \  $(\mathcal U, \ \mathcal X)$ is hereditary.
	
\vskip10pt

$(2)$ \ If \ $\Hom_B(M, \ _B\mathcal I) \subseteq \mathcal U$, then
\ $(\binom{\mathcal U}{B\mbox{-}{\rm Mod}},  \ \binom{\mathcal U}{B\mbox{-}{\rm Mod}}^\perp)$  is a complete  cotorsion pair in $\Lambda\mbox{-}{\rm Mod};$ and it is hereditary if \  $(\mathcal U, \ \mathcal X)$ is hereditary.

\vskip5pt

Moreover, if  $M\otimes_AN = 0= N\otimes_BM$, then \ $\left(\begin{smallmatrix}\mathcal U\\ B\mbox{-}{\rm Mod}\end{smallmatrix}\right)^\perp = {\rm H}_A(\mathcal X)\oplus {\rm H}_B(_B\mathcal I)$, \ and hence
$$(\left(\begin{smallmatrix}\mathcal U\\ B\mbox{-}{\rm Mod}\end{smallmatrix}\right),  \ {\rm H}_A(\mathcal X)\oplus {\rm H}_B(_B\mathcal I))$$ is a complete cotorsion pair$;$ and it is hereditary if \  $(\mathcal U, \ \mathcal X)$ is hereditary.
\end{thm}

\begin{rem}\label{remcomplete2} $(1)$  \ {\it If \ $_AN$ is injective, then \ $N\otimes_B\mathcal P\subseteq \mathcal X$ always holds.}

\vskip5pt

(2) \ {\it If \ $A$ is quasi-Frobenius and  \ $M_A$ is flat, then \ $\Hom_B(M, \ _B\mathcal I) \subseteq \mathcal U$ always holds.}
\end{rem}

\subsection{Lemmas for Theorem \ref{ctp2}} \ To prove Theorem \ref{ctp2}(1), we need

\vskip5pt

\begin{lem}\label{completeness1} \ Let \ $\Lambda = \left(\begin{smallmatrix} A & N \\ M & B \end{smallmatrix}\right)$ be a Morita ring with $\phi = 0= \psi$.
Suppose $_BM$ is projective. For a $\Lambda$-module $\left(\begin{smallmatrix}L_1\\ L_2\end{smallmatrix}\right)_{f, g}$, let $\pi: P \longrightarrow L_1$ be an epimorphism with $_AP$ projective,
and \ $0\longrightarrow Y\stackrel \sigma \longrightarrow V\stackrel {\pi'}\longrightarrow L_2\longrightarrow 0$ an exact sequence.
Then there is an exact sequence of the form$:$
$$0\rightarrow {\left(\begin{smallmatrix}K \\(M\otimes P)\oplus Y \end{smallmatrix}\right)_{\alpha,\beta}} \longrightarrow
 {\left(\begin{smallmatrix} P \\ M\otimes P\end{smallmatrix}\right)_{1, 0}}\oplus
{\left(\begin{smallmatrix} N\otimes V \\ V\end{smallmatrix}\right)_{0, 1}}
\xlongrightarrow {(\left(\begin{smallmatrix} \pi \\ f(1\otimes \pi)\end{smallmatrix}\right), \left(\begin{smallmatrix} g(1\otimes \pi')\\ \pi'\end{smallmatrix}\right))}
{\left(\begin{smallmatrix} L_1 \\ L_2\end{smallmatrix}\right)_{f,g}} \rightarrow 0.$$
\end{lem}
\begin{proof} \ For convenience, rewrite the sequence as
 $$0\rightarrow \left(\begin{smallmatrix}K\\(M\otimes_AP)\oplus Y \end{smallmatrix}\right)_{\alpha, \beta} \longrightarrow\left(\begin{smallmatrix} P\oplus (N\otimes_BV) \\ (M\otimes_AP)\oplus V\end{smallmatrix}\right)_{\left(\begin{smallmatrix} 1_{M\otimes P} & 0 \\ 0 & 0 \end{smallmatrix}\right), \left(\begin{smallmatrix} 0 & 0 \\ 0 & 1_{N\otimes V} \end{smallmatrix}\right)}\xlongrightarrow{\left(\begin{smallmatrix}
	(\pi, \ g(1_N\otimes \pi')) \\ (f(1_M\otimes \pi), \ \pi')\end{smallmatrix}\right)}
	\left(\begin{smallmatrix} L_1 \\ L_2\end{smallmatrix}\right)_{f,g}\rightarrow 0.$$

\vskip5pt

We claim that \ $\left(\begin{smallmatrix}
	(\pi, g(1_N\otimes \pi')) \\ (f(1_M\otimes \pi), \ \pi')\end{smallmatrix}\right)$ is a $\Lambda$-epimorphism.
In fact, by $\phi = 0 = \psi$, \ $f(1_M\otimes g) = 0 = g(1_N\otimes f)$.  Hence
$$f(1_M\otimes g(1_N\otimes \pi')) = 0: \ M\otimes_AN\otimes_B V\longrightarrow L_2$$
and $$g(1_N\otimes f(1_M\otimes \pi)) = 0: \ N\otimes_B M\otimes P\longrightarrow L_1.$$
Thus, the following diagrams commute:
$$\xymatrix@R=0.7cm@C=1.2cm{(M\otimes_AP)\oplus (M\otimes_AN\otimes_B V)
\ar[d]_-{\left(\begin{smallmatrix} 1_{M\otimes P} &0 \\ 0&0 \end{smallmatrix}\right)}
\ar[rr]^-{(1_M \otimes \pi, 1_M\otimes g(1_N\otimes \pi'))} && M\otimes_AL_1\ar[d]^-f \\
(M\otimes_AP)\oplus V \ar[rr]^-{(f(1_M\otimes \pi), \ \pi')}  && L_2}$$

$$\xymatrix@R=0.7cm@C=1.2cm{(N\otimes_BM\otimes_A P)\oplus (N\otimes_BV) \ar[d]_-{\left(\begin{smallmatrix} 0& 0 \\ 0& 1_{N\otimes V} \end{smallmatrix}\right)}
\ar[rr]^-{(1_N\otimes f(1_M\otimes \pi), 1_N\otimes \pi')}&& N\otimes_BL_2  \ar[d]^-g \\
P\oplus (N\otimes_BV) \ar[rr]^-{(\pi, \ g(1_N\otimes \pi'))} && L_1}$$
i.e.,   $\left(\begin{smallmatrix}
(\pi, g(1_N\otimes \pi') \\ (f(1_M\otimes \pi), \ \pi')\end{smallmatrix}\right)$ is a $\Lambda$-map. Clearly, it is an epimorphism.
	
\vskip5pt
	
It remains to see that \ $\Ker \left(\begin{smallmatrix}
(\pi, g(1_N\otimes \pi') \\ (f(1_M\otimes \pi), \ \pi')\end{smallmatrix}\right)$ is of the form $\left(\begin{smallmatrix}K\\ (M\otimes_A P)\oplus Y\end{smallmatrix}\right)_{\alpha, \beta}$.

\vskip5pt

In fact, as a $\Lambda$-module, $\Ker \left(\begin{smallmatrix}
(\pi, g(1_N\otimes \pi') \\ (f(1_M\otimes \pi), \ \pi')\end{smallmatrix}\right)$ is of the form
$\left(\begin{smallmatrix} K\\ K'\end{smallmatrix}\right)_{\alpha,\beta}$,
where $K' = \Ker (f(1_M\otimes \pi), \ \pi')$.
Thus, it suffices to show  \ $\Ker (f(1_M\otimes \pi), \ \pi') \cong (M\otimes_AP)\oplus Y.$

\vskip5pt

Since \ $_AP$ is projective, \ $M\otimes_AP$ is a direct summand of copies of \ $_BM$, as a left $B$-module.
While  by assumption \ $_BM$ is projective, it follows that
\ $M\otimes_AP$ is a projective left $B$-module. Thus,  there is a $B$-map $h$ such that the following diagram commutes:
$$\xymatrix{& & M\otimes_AP\ar@{-->}[d]_{h} \ar[r]^-{1_M\otimes \pi} & M\otimes L_1\ar[d]^f \ar[r] & 0 \\
0\ar [r] & Y \ar[r]^-\sigma & V \ar[r]^{\pi'} & L_2 \ar[r] & 0.}$$
Then it is clear that
$$\xymatrix{0 \ar[r] &  (M\otimes_AP)\oplus Y \ar[rr]^-{\left(\begin{smallmatrix} 1 & 0\\ -h & \sigma\end{smallmatrix}\right)} &&  (M\otimes_AP)\oplus V\ar[rr]^-{(f(1_M\otimes \pi), \ \pi')}
&& L_2 \ar[r] & 0}$$
is exact. This completes the proof.
\end{proof}

\vskip5pt

To prove Theorem \ref{ctp2}(2), we need the following lemma, in which it is more convenient to write
a $\Lambda$-module in the second expression.

\begin{lem}\label{completeness3} \ Let \ $\Lambda = \left(\begin{smallmatrix} A & N \\ M & B \end{smallmatrix}\right)$ be a Morita ring with $\phi = 0= \psi$. Suppose $N_B$ is flat. For  $\Lambda$-module $\left(\begin{smallmatrix}L_1\\ L_2\end{smallmatrix}\right)_{\widetilde{f}, \widetilde{g}}$, let $\sigma: L_1 \longrightarrow I$ be a monomorphism with $_AI$ injective,
and \ $0\longrightarrow L_2 \stackrel {\sigma'} \longrightarrow Y\stackrel {\pi}\longrightarrow V\longrightarrow 0$ an exact sequence.
Then there is an exact sequence of the form$:$
$$0\rightarrow \left(\begin{smallmatrix} L_1 \\ L_2\end{smallmatrix}\right)_{\widetilde{f},\widetilde{g}}
\xlongrightarrow{({\left(\begin{smallmatrix} \sigma \\ (N, \sigma)\widetilde{g} \end{smallmatrix}\right)} \\ {\left(\begin{smallmatrix} (M, \sigma')\widetilde{f} \\ \sigma' \end{smallmatrix}\right)})}
\left(\begin{smallmatrix} I \\ \Hom_A(N, I)\end{smallmatrix}\right)_{0, 1} \oplus \left(\begin{smallmatrix} \Hom_B(M, Y) \\ Y\end{smallmatrix}\right)_{1, 0}
\longrightarrow
\left(\begin{smallmatrix} C \\ \Hom_A(N, I)\oplus V\end{smallmatrix}\right)_{\widetilde{\alpha},\widetilde{\beta}}\rightarrow 0.$$
\end{lem}

\begin{proof} Rewrite the sequence as
 $$0\rightarrow \left(\begin{smallmatrix} L_1 \\ L_2\end{smallmatrix}\right)_{\widetilde{f},\widetilde{g}}
\xlongrightarrow{\left(\begin{smallmatrix} {\left(\begin{smallmatrix} \sigma \\ (M, \sigma')\widetilde{f} \end{smallmatrix}\right)} \\ {\left(\begin{smallmatrix} (N, \sigma)\widetilde{g} \\ \sigma' \end{smallmatrix}\right)} \end{smallmatrix}\right)} \left(\begin{smallmatrix} I\oplus \Hom_B(M, Y) \\ \Hom_A(N, I)\oplus Y\end{smallmatrix}\right)_{{\left(\begin{smallmatrix} 0 & 0 \\ 0 & 1\end{smallmatrix}\right)}, {\left(\begin{smallmatrix} 1 & 0 \\ 0 & 0 \end{smallmatrix}\right)}}
\longrightarrow
\left(\begin{smallmatrix} C \\ \Hom_A(N, I)\oplus V\end{smallmatrix}\right)_{\widetilde{\alpha}, \widetilde{\beta}}\rightarrow 0.$$

Since $\phi = 0 = \psi,$ \ $(M, \widetilde{g}) \widetilde{f} = 0 = (N, \widetilde{f}) \widetilde{g},$ and hence
$(M, (N, \sigma)\widetilde{g}) \widetilde{f} = 0 = (N, (M, \sigma')\widetilde{f}) \widetilde{g}.$ Thus
the following diagrams commute:
$$\xymatrix@R=1cm{L_1 \ar[d]_-{\widetilde{f}}\ar[rr]^-{\left(\begin{smallmatrix} \sigma \\ (M, \sigma')\widetilde{f} \end{smallmatrix}\right)} && I\oplus \Hom_A(M,Y)\ar[d]^-{{\left(\begin{smallmatrix} 0 & 0 \\ 0 & 1\end{smallmatrix}\right)}} \\
\Hom_B(M, L_2)\ar[rr]^-{\left(\begin{smallmatrix} (M, (N, \sigma)\widetilde{g}) \\ (M, \sigma')\end{smallmatrix}\right)}&& \Hom_B(M, \Hom_A(N, I))\oplus \Hom_A(M,Y)}$$
$$\xymatrix@R=1cm{L_2 \ar[d]_{\widetilde{g}}\ar[rr]^-{\left(\begin{smallmatrix} (N, \sigma)\widetilde{g} \\ \sigma'\end{smallmatrix}\right)}
&& \Hom_A(N, I)\oplus Y\ar[d]^-{\left(\begin{smallmatrix} 1 & 0 \\ 0 & 0 \end{smallmatrix}\right)}  \\
\Hom_A(N, L_1)\ar[rr]^-{\left(\begin{smallmatrix} (N, \sigma)\\ (N, (M, \sigma')\widetilde{f})\end{smallmatrix}\right)} && \Hom_A(N, I)\oplus \Hom_A(N, \Hom_B(M, Y))}.$$
Therefore the map $\left(\begin{smallmatrix} {\left(\begin{smallmatrix} \sigma \\ (M, \sigma')\widetilde{f} \end{smallmatrix}\right)} \\ {\left(\begin{smallmatrix} (N, \sigma)\widetilde{g} \\ \sigma' \end{smallmatrix}\right)} \end{smallmatrix}\right)$ is a $\Lambda$-map. Clearly, it is  a monomorphism.

\vskip10pt

Write $\Coker \left(\begin{smallmatrix} {\left(\begin{smallmatrix} \sigma \\ (M, \sigma')\widetilde{f} \end{smallmatrix}\right)} \\ {\left(\begin{smallmatrix} (N, \sigma)\widetilde{g} \\ \sigma' \end{smallmatrix}\right)} \end{smallmatrix}\right)$
as
$\left(\begin{smallmatrix} C\\ C'\end{smallmatrix}\right)_{\widetilde{\alpha},\widetilde{\beta}}$.
Then $C'$ is the cokernel of $B$-monomorphism $\left(\begin{smallmatrix} (N, \sigma)\widetilde{g} \\ \sigma' \end{smallmatrix}\right)$.
Since $N_B$ is flat and $_AI$ is injective, it follows that $\Hom_A(N, I)$ is an injective left $B$-module. Thus there is a $B$-map $h$ such that the diagram
$$\xymatrix{0\ar[r] & L_2\ar[d]_-{\widetilde{g}} \ar[r]^-{\sigma'} & Y\ar@{-->}[d]^-h \ar[r]^-{\pi}& V\ar[r] & 0 \\
0\ar[r] & \Hom_A(N, L_1) \ar[r]^-{(N, \sigma)} & \Hom_A(N, I) & }$$
commutes. Therefore
$$\xymatrix{0 \ar[r] &  L_2\ar[rr]^-{\left(\begin{smallmatrix} (N, \sigma)\widetilde{g} \\ \sigma' \end{smallmatrix}\right)} &&
\Hom_A(N, I)\oplus Y\ar[rrr]^-{\left(\begin{smallmatrix} 1_{(N, I)} & -h \\ 0 & \pi \end{smallmatrix}\right)}
&&& \Hom_A(N, I)\oplus V \ar[r] & 0}$$
is an exact sequence. It follows that \ $C' \cong \Hom_A(N, I)\oplus V.$ This completes the proof. \end{proof}

\subsection{Proof of Theorem \ref{ctp2}}

\vskip5pt

(1) \ Since \ $N_B$ is flat, by Theorem \ref{ctp1}(1), \ \ $(^\perp\binom{A\mbox{-}{\rm Mod}}{\mathcal Y}, \ \binom{A\mbox{-}{\rm Mod}}{\mathcal Y})$ is a cotorsion pair;
and it is hereditary if \  $(\mathcal V, \ \mathcal Y)$ is hereditary.

\vskip5pt

Since \ $(\mathcal V, \ \mathcal Y)$ is complete, for any $\Lambda$-module $\left(\begin{smallmatrix}L_1\\ L_2\end{smallmatrix}\right)_{f, g}$,
there is an exact sequence $$0\longrightarrow Y \longrightarrow V\longrightarrow L_2\longrightarrow 0$$
with \ $V\in \mathcal V$ and \ $Y\in\mathcal Y$. Since \ $_BM$ is projective, by Lemma \ref{completeness1},
there is an exact sequence of $\Lambda$-modules of the form$:$
$$0\longrightarrow \left(\begin{smallmatrix}K\\ (M\otimes_AP)\oplus Y\end{smallmatrix}\right)_{\alpha,\beta} \longrightarrow \left(\begin{smallmatrix} P \\ M\otimes_AP\end{smallmatrix}\right)_{1, 0}
\oplus
\left(\begin{smallmatrix} N\otimes_BV \\ V\end{smallmatrix}\right)_{0, \ 1} \longrightarrow
\left(\begin{smallmatrix} L_1 \\ L_2\end{smallmatrix}\right)_{f,g}\longrightarrow 0$$
where $_AP$ is projective. Since by assumption \ $M\otimes_A\mathcal P\subseteq \mathcal Y$,
$(M\otimes_AP)\oplus Y\in \mathcal Y$, and hence $\left(\begin{smallmatrix}K\\ (M\otimes_AP)\oplus Y \end{smallmatrix}\right)_{\alpha, \beta}\in
\binom{A\mbox{-}{\rm Mod}}{\mathcal Y}.$

\vskip5pt

On the other hand, $\left(\begin{smallmatrix} P \\ M\otimes_AP\end{smallmatrix}\right)_{1, 0}
= {\rm T}_AP$ is a projective $\Lambda$-module, so it is in  $\ ^\perp\binom{A\mbox{-}{\rm Mod}}{\ _B\mathcal Y}.$
Also, $\left(\begin{smallmatrix} N\otimes_BV \\ V\end{smallmatrix}\right)_{0, 1} = {\rm T}_BV\in {\rm T}_B(\mathcal V)$.
Since  $N_B$ is flat and \ $\Ext^1_B(\mathcal V, \ \mathcal Y) = 0$, by Lemma \ref{extadj1}(2),  $\left(\begin{smallmatrix} N\otimes_BV \\ V\end{smallmatrix}\right)_{0, 1} = {\rm T}_BV\in
\ ^\perp\binom{A\mbox{-}{\rm Mod}}{\mathcal Y}.$  This shows the completeness of \ $(^\perp\binom{A\mbox{-}{\rm Mod}}{\mathcal Y}, \ \binom{A\mbox{-}{\rm Mod}}{\mathcal Y})$.

\vskip5pt

Finally, if  $M\otimes_AN = 0= N\otimes_BM$, then by Corollary \ref{identification1}(1) one has
${}^\perp\left (\begin{smallmatrix} A\mbox{-}{\rm Mod} \\ \mathcal Y\end{smallmatrix}\right) = \Delta(_A\mathcal P, \mathcal V) = {\rm T}_A(_A\mathcal P)\oplus {\rm T}_B(\mathcal V)$.

\vskip10pt

(2) \ Since \ $_BM$ is projective, by Theorem \ref{ctp1}(2), $(\binom{A\mbox{-}{\rm Mod}}{\mathcal V},  \ \binom{A\mbox{-}{\rm Mod}}{\mathcal V}^\perp)$  is a cotorsion pair; and it is hereditary if \  $(\mathcal V, \ \mathcal Y)$ is hereditary.

\vskip5pt

Since \ $(\mathcal V, \ \mathcal Y)$ is complete, for any $\Lambda$-module $\left(\begin{smallmatrix}L_1\\ L_2\end{smallmatrix}\right)_{\widetilde{f}, \widetilde{g}}$,
there is an exact sequence \ $0\longrightarrow L_2 \longrightarrow Y\longrightarrow V\longrightarrow 0$
with \ $Y\in \mathcal Y$ and \ $V\in\mathcal V$. Since  \ $N_B$ is flat, by Lemma \ref{completeness3}, there is an exact sequence of $\Lambda$-modules of the form$:$
$$0\rightarrow \left(\begin{smallmatrix} L_1 \\ L_2\end{smallmatrix}\right)_{\widetilde{f},\widetilde{g}}
\longrightarrow
\left(\begin{smallmatrix} I \\ \Hom_A(N, I)\end{smallmatrix}\right)_{0, 1} \oplus \left(\begin{smallmatrix} \Hom_B(M, Y) \\ Y\end{smallmatrix}\right)_{1, 0}
\longrightarrow
\left(\begin{smallmatrix} C \\ \Hom_A(N, I)\oplus V\end{smallmatrix}\right)_{\widetilde{\alpha},\widetilde{\beta}}\rightarrow 0$$
where $_AI$ is injective. Since $\left(\begin{smallmatrix} I \\ \Hom_A(N, I)\end{smallmatrix}\right)_{0, 1}$ is an injective $\Lambda$-module, it is in $\binom{A\mbox{-}{\rm Mod}}{\mathcal V}^\perp$.
Since \ $_BM$ is projective and $\Ext^1_B(\mathcal V, \mathcal Y) = 0$, it follows from
Lemma \ref{extadj1}(4) that \ $\left(\begin{smallmatrix} \Hom_B(M, Y) \\ Y\end{smallmatrix}\right)_{1, 0} = {\rm H}_BY\in
\ \binom{A\mbox{-}{\rm Mod}}{\mathcal V}^\perp.$

\vskip5pt

Since by assumption \ $\Hom_A(N, \ _A\mathcal I)\subseteq \mathcal V$, \ $\Hom_A(N, I) \in \mathcal V$, and hence  $\left(\begin{smallmatrix} C \\ \Hom_A(N, I)\oplus V\end{smallmatrix}\right)_{\widetilde{\alpha},\widetilde{\beta}}\in
\binom{A\mbox{-}{\rm Mod}}{\mathcal V}.$ This proves  the completeness of \ $(\binom{A\mbox{-}{\rm Mod}}{\mathcal V},  \ \binom{A\mbox{-}{\rm Mod}}{\mathcal V}^\perp)$.

\vskip5pt

Finally, if $M\otimes_AN = 0= N\otimes_BM$, then by Corollary \ref{identification1}(3) one has \ $\binom{A\mbox{-}{\rm Mod}}{\mathcal V}^\perp = \nabla(_A\mathcal I, \ \mathcal Y) ={\rm H}_A(_A\mathcal I)\oplus {\rm H}_B(\mathcal Y)$. \hfill $\square$

\vskip5pt

\subsection{Lemmas for Theorem \ref{ctp3}} To see Theorem \ref{ctp3}(1), we need

\begin{lem}\label{completeness2} \ Let \ $\Lambda = \left(\begin{smallmatrix} A & N \\ M & B \end{smallmatrix}\right)$ be a Morita ring with $\phi = 0= \psi$.
Suppose $_AN$ is projective. For a $\Lambda$-module $\left(\begin{smallmatrix}L_1\\ L_2\end{smallmatrix}\right)_{f, g}$,
let $\pi: Q \longrightarrow L_2$ be an epimorphism with $_BQ$ projective, and
\ $0\longrightarrow X\stackrel \sigma \longrightarrow U\stackrel {\pi'}\longrightarrow L_1\longrightarrow 0$ an exact sequence.
Then there is an exact sequence of the form$:$
$$0\rightarrow \left(\begin{smallmatrix}X\oplus (N\otimes Q)\\ K\end{smallmatrix}\right)_{\alpha,\beta}
\longrightarrow
\left(\begin{smallmatrix} U\\ M\otimes U\end{smallmatrix}\right)_{1, 0}
\oplus \left(\begin{smallmatrix} N\otimes Q\\ Q\end{smallmatrix}\right)_{0, 1}
\xlongrightarrow{(\left(\begin{smallmatrix} \pi' \\ f(1\otimes \pi')\end{smallmatrix}\right), \left(\begin{smallmatrix}
g(1\otimes \pi) \\ \pi \end{smallmatrix}\right))}
\left(\begin{smallmatrix} L_1 \\ L_2\end{smallmatrix}\right)_{f,g}\rightarrow 0.$$
\end{lem}

\begin{proof} The proof is similar to Lemma \ref{completeness1}. We include the points.
Rewrite the sequence as
$$0\rightarrow \left(\begin{smallmatrix}X\oplus (N\otimes_BQ)\\ K\end{smallmatrix}\right)_{\alpha,\beta}
\longrightarrow\left(\begin{smallmatrix} U\oplus (N\otimes_BQ) \\ (M\otimes_AU)\oplus Q\end{smallmatrix}\right)_{\left(\begin{smallmatrix} 1 & 0 \\ 0 & 0 \end{smallmatrix}\right), \left(\begin{smallmatrix} 0 & 0 \\ 0 & 1 \end{smallmatrix}\right)}\xlongrightarrow{\left(\begin{smallmatrix}
(\pi', g(1_N\otimes \pi)) \\ (f(1_M\otimes \pi'), \pi)\end{smallmatrix}\right)}
\left(\begin{smallmatrix} L_1 \\ L_2\end{smallmatrix}\right)_{f,g}\rightarrow 0.$$
The map \ $\left(\begin{smallmatrix}
(\pi', g(1_N\otimes \pi)) \\ (f(1_M\otimes \pi'), \pi)\end{smallmatrix}\right)$ is a $\Lambda$-epimorphism, since the diagrams commute:
$$\xymatrix@R=0.6cm@C=1.2cm{(M\otimes_AU)\oplus (M\otimes_AN\otimes_B Q) \ar[d]_-{\left(\begin{smallmatrix} 1 &0 \\ 0&0 \end{smallmatrix}\right)}
\ar[rr]^-{(1\otimes \pi', 1_M\otimes g(1_N\otimes \pi))} && M\otimes_AL_1 \ar[d]^-f \\
(M\otimes_AU)\oplus Q \ar[rr]^-{(f(1\otimes \pi'), \pi)} && L_2}$$

$$\xymatrix@R=0.6cm@C=1.2cm{(N\otimes_BM\otimes_A U)\oplus (N\otimes_BQ) \ar[d]_-{\left(\begin{smallmatrix} 0& 0 \\ 0& 1 \end{smallmatrix}\right)}
\ar[rr]^-{(1_N\otimes f(1\otimes \pi'), 1_N\otimes \pi)}&& N\otimes_BL_2
\ar[d]^-g \\
U\oplus (N\otimes_BQ)\ar[rr]^-{(\pi', \ g(1_N\otimes \pi))} && L_1.}$$

\vskip5pt

\noindent It remains to prove \ $\Ker (\pi', \ g(1_N\otimes \pi))\cong X\oplus (N\otimes_BQ)$.
Since \ $_AN$ is projective, \ $N\otimes_BQ$ is a projective left $A$-module. Thus,  there is an $A$-map $h$ such that the diagram
$$\xymatrix{& & N\otimes_BQ\ar@{-->}[d]_{h} \ar[r]^-{1_N\otimes \pi} & N\otimes L_2\ar[d]^g \ar[r] & 0 \\
0\ar [r] & X \ar[r]^-\sigma & U \ar[r]^{\pi'} & L_1 \ar[r] & 0.}$$
commutes. Then
$$\xymatrix{0 \ar[r] & X\oplus(N\otimes_BQ)\ar[rrr]^-{\left(\begin{smallmatrix} \sigma &  -h \\ 0 & 1\end{smallmatrix}\right)} &&&  U\oplus (N\otimes_BQ)\ar[rr]^-{(\pi', \ g(1_N\otimes \pi))}
&& L_1 \ar[r] & 0}$$
is exact. This completes the proof.\end{proof}

\vskip5pt

 To prove Theorem \ref{ctp3}(2), we need the following lemma, in which
the second expression of a $\Lambda$-module is more convenient.

\begin{lem}\label{completeness4} \ Let \ $\Lambda = \left(\begin{smallmatrix} A & N \\ M & B \end{smallmatrix}\right)$ be a Morita ring with $\phi = 0= \psi$. Suppose $M_A$ is flat.
For $\Lambda$-module $\left(\begin{smallmatrix}L_1\\ L_2\end{smallmatrix}\right)_{\widetilde{f}, \widetilde{g}}$, let $\sigma: L_2 \longrightarrow J$ be a monomorphism with $_BJ$ injective, and
\ $0\longrightarrow L_1 \stackrel {\sigma'} \longrightarrow X\stackrel {\pi}\longrightarrow U\longrightarrow 0$ an exact sequence.
Then there is an exact sequence  of the form$:$
$$0\rightarrow \left(\begin{smallmatrix} L_1 \\ L_2\end{smallmatrix}\right)_{\widetilde{f},\widetilde{g}}
\xlongrightarrow{({\left(\begin{smallmatrix} \sigma' \\ (N, \sigma')\widetilde{g} \end{smallmatrix}\right)} \\ {\left(\begin{smallmatrix} (M, \sigma)\widetilde{f} \\ \sigma \end{smallmatrix}\right)})}
\left(\begin{smallmatrix} X \\ \Hom_A(N, X)\end{smallmatrix}\right)_{0, 1} \oplus \left(\begin{smallmatrix} \Hom_B(M, J) \\ J\end{smallmatrix}\right)_{1, 0}
\longrightarrow
\left(\begin{smallmatrix} U\oplus \Hom_B(M, J) \\ C\end{smallmatrix}\right)_{\widetilde{\alpha},\widetilde{\beta}}\rightarrow 0.$$
\end{lem}

\begin{proof}  The proof is similar to Lemma \ref{completeness3}. We include the points.
First, as in the proof of Lemma \ref{completeness3}, one can show that
the map $$\left(\begin{smallmatrix} {\left(\begin{smallmatrix} \sigma' \\ (M, \sigma)\widetilde{f} \end{smallmatrix}\right)} \\ {\left(\begin{smallmatrix} (N, \sigma')\widetilde{g} \\ \sigma \end{smallmatrix}\right)} \end{smallmatrix}\right): \left(\begin{smallmatrix} L_1 \\ L_2\end{smallmatrix}\right)_{\widetilde{f},\widetilde{g}}
\longrightarrow \left(\begin{smallmatrix} X\oplus \Hom_B(M, J) \\ \Hom_A(N, X)\oplus J\end{smallmatrix}\right)_{{\left(\begin{smallmatrix} 0 & 0 \\ 0 & 1\end{smallmatrix}\right)}, {\left(\begin{smallmatrix} 1 & 0 \\ 0 & 0 \end{smallmatrix}\right)}}$$ is a $\Lambda$-monomorphism. We omit the details.

\vskip5pt
	
Write $\Coker \left(\begin{smallmatrix} {\left(\begin{smallmatrix} \sigma' \\ (M, \sigma)\widetilde{f} \end{smallmatrix}\right)} \\ {\left(\begin{smallmatrix} (N, \sigma')\widetilde{g} \\ \sigma \end{smallmatrix}\right)} \end{smallmatrix}\right)$
as
$\left(\begin{smallmatrix} C'\\ C\end{smallmatrix}\right)_{\widetilde{\alpha},\widetilde{\beta}}$.
Then \ $C' \cong \Coker\left(\begin{smallmatrix} \sigma' \\ (M, \sigma)\widetilde{f} \end{smallmatrix}\right)$.
Since $M_A$ is flat and $_BJ$ is injective, it follows that $\Hom_B(M, J)$ is an injective left $A$-module. Thus there is an $A$-map $h$ such that the diagram
$$\xymatrix{0\ar[r] & L_1\ar[d]_-{\widetilde{f}} \ar[r]^-{\sigma'} & X\ar@{-->}[d]^-h \ar[r]^-{\pi}& U\ar[r] & 0 \\
0\ar[r] & \Hom_B(M, L_2) \ar[r]^-{(M, \sigma')} & \Hom_B(M, J) & }$$
commutes. Therefore  $$\xymatrix{0 \ar[r] & L_1\ar[rr]^-{\left(\begin{smallmatrix} \sigma' \\ (M, \sigma)\widetilde{f} \end{smallmatrix}\right)} &&  X\oplus\Hom_B(M, J)\ar[rrr]^-
{\left(\begin{smallmatrix} \pi & 0 \\ -h & 1_{(M, J)} \end{smallmatrix}\right)}
&&& U\oplus \Hom_B(M, J) \ar[r] & 0}$$
is exact,  and hence \ $C' \cong U\oplus \Hom_B(M, J).$  \end{proof}

\subsection{Proof of Theorem \ref{ctp3}} \ The proof is similar to Theorem \ref{ctp2}.

\vskip5pt

(1) \ Since \ $M_A$ is flat, by Theorem \ref{ctp1}(1), \ $({}^\perp\binom{\mathcal X}{B\mbox{-}{\rm Mod}},  \ \binom{\mathcal X}{B\mbox{-}{\rm Mod}})$  is a cotorsion pair;
and it is hereditary if \  $(\mathcal U, \ \mathcal X)$ is hereditary.

\vskip5pt

For any $\Lambda$-module $\left(\begin{smallmatrix}L_1\\ L_2\end{smallmatrix}\right)_{f, g}$, \ since \ $(\mathcal U, \ \mathcal X)$ is complete,
there is an exact sequence $$0\longrightarrow X \longrightarrow U\longrightarrow L_1\longrightarrow 0$$
with \ $U\in \mathcal U$, $X\in\mathcal X$.
Since \ $_AN$ is projective, by Lemma \ref{completeness2},  there is an exact sequence of $\Lambda$-modules of the form$:$
$$0\longrightarrow \left(\begin{smallmatrix}X\oplus(N\otimes Q)\\ K\end{smallmatrix}\right)_{\alpha,\beta} \longrightarrow \left(\begin{smallmatrix} U \\ M\otimes U\end{smallmatrix}\right)_{1, 0}
\oplus \left(\begin{smallmatrix} N\otimes Q \\ Q\end{smallmatrix}\right)_{0, 1} \longrightarrow
\left(\begin{smallmatrix} L_1 \\ L_2\end{smallmatrix}\right)_{f,g}\longrightarrow 0$$
where $_BQ$ is projective. Since by assumption $N\otimes_B\mathcal P\subseteq \mathcal X$, it follows that $X\oplus(N\otimes_BQ)\in \mathcal X$, and hence $\left(\begin{smallmatrix}X\oplus(N\otimes Q)\\ K\end{smallmatrix}\right)_{\alpha, \beta}\in
\binom{\mathcal X}{B\mbox{-}{\rm Mod}}.$

\vskip5pt

Since $\left(\begin{smallmatrix} N\otimes Q \\ Q\end{smallmatrix}\right)_{0, 1}\in \ _\Lambda\mathcal P$,  it is in \ $^\perp\binom{_A\mathcal X}{B\mbox{-}{\rm Mod}}.$
Since \ $M_A$ is flat and \ $\Ext^1_A(\mathcal U, \mathcal X) = 0$, by Lemma \ref{extadj1}(1), \ $\left(\begin{smallmatrix} U \\ M\otimes U\end{smallmatrix}\right)_{1, 0}
= {\rm T}_AU\in
\ ^\perp\binom{\mathcal X}{B\mbox{-}{\rm Mod}}.$ Thus,  $({}^\perp\binom{\mathcal X}{B\mbox{-}{\rm Mod}},  \ \binom{\mathcal X}{B\mbox{-}{\rm Mod}})$ is complete.

\vskip5pt

Finally if $M\otimes_AN = 0= N\otimes_BM$, then by Corollary \ref{identification1}(2) one has
\ \ ${}^\perp\left (\begin{smallmatrix} A\mbox{-}{\rm Mod} \\ \mathcal Y\end{smallmatrix}\right) = \Delta(_A\mathcal P, \mathcal V)= {\rm T}_A(_A\mathcal P)\oplus {\rm T}_B(\mathcal V)$.

\vskip10pt

(2) \ Since \ $_AN$ is projective, by Theorem \ref{ctp1}(2), $(\binom{_A\mathcal U}{B\mbox{-}{\rm Mod}},  \ \binom{_A\mathcal U}{B\mbox{-}{\rm Mod}}^\perp)$  is a cotorsion pair;
and it is hereditary if \  $(\mathcal U, \ \mathcal X)$ is hereditary.

\vskip5pt For any $\Lambda$-module $\left(\begin{smallmatrix}X\\ Y\end{smallmatrix}\right)_{\widetilde{f}, \widetilde{g}}$, \ since \ $(\mathcal U, \ \mathcal X)$ is complete,
there is an exact sequence  \ $0\longrightarrow L_1 \longrightarrow X \longrightarrow U\longrightarrow 0$
with \ $X\in\mathcal X$, \ $U\in \mathcal U$. Since \ $M_A$ is flat,
by Lemma \ref{completeness4}, there is an exact sequence$:$
$$0\rightarrow \left(\begin{smallmatrix} L_1 \\ L_2\end{smallmatrix}\right)_{\widetilde{f},\widetilde{g}}
\longrightarrow \left(\begin{smallmatrix} X \\ \Hom_A(N, X)\end{smallmatrix}\right)_{0, 1} \oplus \left(\begin{smallmatrix} \Hom_B(M, J) \\ J\end{smallmatrix}\right)_{1, 0}
\longrightarrow
\left(\begin{smallmatrix} U\oplus \Hom_B(M, J) \\ C\end{smallmatrix}\right)_{\widetilde{\alpha}, \widetilde{\beta}}\rightarrow 0$$
where $_BJ$ is injective. Since \ $\left(\begin{smallmatrix} \Hom_B(M, J) \\ J\end{smallmatrix}\right)_{1, 0}$ is an injective $\Lambda$-module, it is in
$\binom{\ _A\mathcal U}{B\mbox{-}{\rm Mod}}^\perp.$ Since \ $_AN$ is projective and $\Ext^1_A(\mathcal U, \mathcal X) = 0$, by Lemma \ref{extadj1}(3),  $\left(\begin{smallmatrix} X \\ \Hom_A(N, X)\end{smallmatrix}\right)_{0, 1}
= {\rm H}_AX\in
\ \binom{\ _A\mathcal U}{B\mbox{-}{\rm Mod}}^\perp.$

\vskip5pt

Since by assumption \ $\Hom_B(M, \ _B\mathcal I)\subseteq \mathcal U$, \ $\Hom_B(M, J)\in\mathcal U$,  and  hence \ $\left(\begin{smallmatrix} U\oplus \Hom_B(M, J) \\ C \end{smallmatrix}\right)_{\widetilde{\alpha},\widetilde{\beta}}\in
\binom{\ _A\mathcal U}{B\mbox{-}{\rm Mod}}.$ So \
$(\binom{_A\mathcal U}{B\mbox{-}{\rm Mod}},  \ \binom{_A\mathcal U}{B\mbox{-}{\rm Mod}}^\perp)$  is complete.

\vskip5pt

Finally if $M\otimes_AN = 0= N\otimes_BM$, then by Corollary \ref{identification1}(4) one has
$\binom{_A\mathcal U}{B\mbox{-}{\rm Mod}}^\perp = \nabla(\mathcal X, \ _B\mathcal I) = {\rm H}_A(\mathcal X)\oplus {\rm H}_B(_B\mathcal I).$ \hfill $\square$

\vskip10pt

\subsection{Remark} Under the framework of one of $(\mathcal U, \ \mathcal X)$  and
$(\mathcal V, \ \mathcal Y)$ being an arbitrary complete cotorsion pair, and another being the projective or the injective one, the careful reader will find that the completeness of the following cotorsion pairs
\begin{align*}&({}^\perp\left(\begin{smallmatrix}_A\mathcal I\\ \mathcal Y\end{smallmatrix}\right), \ \left(\begin{smallmatrix}_A\mathcal I\\ \mathcal Y\end{smallmatrix}\right)),
\ \ \ \ \ \ ({}^\perp\left(\begin{smallmatrix}\mathcal X\\ _B\mathcal I\end{smallmatrix}\right),  \ \left(\begin{smallmatrix}\mathcal X\\ _B\mathcal I\end{smallmatrix}\right)) \ \ \ \ \ \
(\mbox{if} \ M_A \ \mbox{and} \  N_B \ \mbox{are flat})\\ &
(\left(\begin{smallmatrix}_A\mathcal P\\ \mathcal V\end{smallmatrix}\right), \ \left(\begin{smallmatrix}_A\mathcal P\\ \mathcal V\end{smallmatrix}\right)^\perp),
\ \ \ \ \ \ (\left(\begin{smallmatrix}\mathcal U\\ _B\mathcal P\end{smallmatrix}\right),  \ \left(\begin{smallmatrix}\mathcal U\\ _B\mathcal P\end{smallmatrix}\right)^\perp) \ \ \ \ \ \
(\mbox{if} \ _BM \ \mbox{and} \  _AN \ \mbox{are projective})\end{align*}
have not been discussed (also they will be not used in constructing Hovey triples in Section 7).
An interesting special cases of $(\left(\begin{smallmatrix}_A\mathcal P\\ _B\mathcal P\end{smallmatrix}\right), \ \left(\begin{smallmatrix}_A\mathcal P\\ _B\mathcal P\end{smallmatrix}\right)^\perp)$ and
$({}^\perp\left(\begin{smallmatrix}_A\mathcal I\\ _B\mathcal I\end{smallmatrix}\right), \ \left(\begin{smallmatrix}_A\mathcal I\\ _B\mathcal I\end{smallmatrix}\right))$
have been treated in Theorem \ref{ctp4}.

\subsection{Triangular matrix rings} \ For the case of $M = 0$ one has

\begin{prop}\label{triangular}
Let $\Lambda=\left(\begin{smallmatrix} A & N \\ 0 & B \end{smallmatrix}\right)$ be an upper triangular matrix ring.
Suppose that \ $(\mathcal U, \ \mathcal X)$ and $(\mathcal V, \ \mathcal Y)$ are complete cotorsion pairs
in $A$-{\rm Mod} and $B$-{\rm Mod}, respectively.

\vskip5pt

$(1)$ \ Assume that $\Tor_1^B(N, \ \mathcal V)=0$. If
\ $N\otimes_B\mathcal V\subseteq \mathcal X$, then the cotorsion pair
$$(\Delta(\mathcal U, \ \mathcal V), \ \left(\begin{smallmatrix} \mathcal X \\ \mathcal Y\end{smallmatrix}\right))
=({\rm T}_A(\mathcal U)\oplus {\rm T}_B(\mathcal V), \ \left(\begin{smallmatrix} \mathcal X \\ \mathcal Y\end{smallmatrix}\right))$$
is complete.

\vskip5pt

$(2)$ \ Assume that $\Ext_A^1(N, \ \mathcal X)=0$. If \ $\Hom_A(N, \ \mathcal X) \subseteq \mathcal V$, then the cotorsion pair
$$(\left(\begin{smallmatrix} \mathcal U \\ \mathcal V\end{smallmatrix}\right), \ \nabla(\mathcal X, \ \mathcal Y))
=(\left(\begin{smallmatrix} \mathcal U \\ \mathcal V\end{smallmatrix}\right), \ {\rm H}_A(\mathcal X)\oplus {\rm H}_B(\mathcal Y))$$
is complete.
\end{prop}

For lower triangle matrix rings (i.e., $N = 0$) one has the similar results. We omit the details. For proof of Proposition \ref{triangular} we need

\begin{lem}\label{horseshoe} {\rm ([AA, 3.1])}
Let $\mathcal C$ be an abelian category with enough projectives and injectives.
Assume that $(\mathcal A, \mathcal B)$ be a hereditary cotorsion pair in $\mathcal C$,
and $0\longrightarrow X\xlongrightarrow{f}Y\xlongrightarrow{g}Z\longrightarrow 0$
be an exact sequence in $\mathcal C$.

\vskip5pt

$(1)$ \ Assume that $X$ and $Z$ have special right $\mathcal A$-approximation, i.e., there are exact sequences:
$$0\longrightarrow B_1\longrightarrow A_1\longrightarrow X\longrightarrow 0, \qquad
0\longrightarrow B_2\longrightarrow A_2\longrightarrow Z\longrightarrow 0,$$
with $A_i\in \mathcal A$, $B_i\in \mathcal B$, $i=1, 2$. Then $Y$ has a special right $\mathcal A$-approximation.

\vskip5pt

$(2)$ \ Assume that $X$ and $Z$ have special left $\mathcal B$-approximation, i.e., there are exact sequences
$$0\longrightarrow X\longrightarrow B_1\longrightarrow A_1\longrightarrow 0, \qquad
0\longrightarrow Z\longrightarrow B_2\longrightarrow A_2\longrightarrow 0$$
with $B_i\in \mathcal B$, $A_i\in \mathcal A$, $i=1, 2$.
Then $Y$ has a special left $\mathcal B$-approximation.
\end{lem}

\vskip5pt

\noindent {\bf Proof of Proposition \ref{triangular}.}
(1) \ By Theorem \ref{identify1}(1), $(\Delta(\mathcal U, \ \mathcal V), \ \left(\begin{smallmatrix} \mathcal X \\ \mathcal Y\end{smallmatrix}\right))
=({\rm T}_A(\mathcal U)\oplus {\rm T}_B(\mathcal V), \ \left(\begin{smallmatrix} \mathcal X \\ \mathcal Y\end{smallmatrix}\right))$ is a cotorsion pair. Let $\binom{L_1}{L_2}_g$ be a $\Lambda$-module. By the completeness of $(\mathcal U, \ \mathcal X)$ and $(\mathcal V, \ \mathcal Y)$, one has the exact sequences
$$0\longrightarrow X\xlongrightarrow{\sigma_1}U\xlongrightarrow{\pi_1}L_1\longrightarrow 0, \qquad
0\longrightarrow Y\xlongrightarrow{\sigma_2}V\xlongrightarrow{\pi_2}L_2\longrightarrow 0$$
in $A$-Mod and $B$-Mod respectively, with $U\in \mathcal U$, \ $X\in \mathcal X$, \ $V\in \mathcal V$, and \ $Y\in \mathcal Y$.
Then
$$0\longrightarrow \left(\begin{smallmatrix} X \\ 0 \end{smallmatrix}\right)_0\xlongrightarrow{\binom{\sigma_1}{0}}
\left(\begin{smallmatrix} U \\ 0 \end{smallmatrix}\right)_0\xlongrightarrow{\binom{\pi_1}{0}}
\left(\begin{smallmatrix} L_1 \\ 0 \end{smallmatrix}\right)_0\longrightarrow 0$$
is the special right $\Delta(\mathcal U, \ \mathcal V)$-approximation of $\binom{L_1}{0}_0$. Also, since
$N\otimes_B\mathcal V\subseteq \mathcal X$,
$$0\longrightarrow \left(\begin{smallmatrix} N\otimes_BV \\ Y \end{smallmatrix}\right)_{1\otimes\sigma_2} \xlongrightarrow{\binom{1}{\sigma_2}}
\left(\begin{smallmatrix} N\otimes_BV \\ V \end{smallmatrix}\right)_1\xlongrightarrow{\binom{0}{\pi_2}}
\left(\begin{smallmatrix} 0 \\ L_2 \end{smallmatrix}\right)_0\longrightarrow 0
$$
is the special right $\Delta(\mathcal U, \ \mathcal V)$-approximation of $\binom{0}{L_2}_0$.
Since  $$0\longrightarrow \left(\begin{smallmatrix} L_1 \\ 0 \end{smallmatrix}\right) \longrightarrow
\left(\begin{smallmatrix} L_1 \\ L_2 \end{smallmatrix}\right)_g \longrightarrow
\left(\begin{smallmatrix} 0 \\ L_2 \end{smallmatrix}\right) \longrightarrow 0$$
is exact, it follows from Lemma \ref{horseshoe}(1) that
$\Lambda$-module $\binom{L_1}{L_2}_g$ has a special right $\Delta(\mathcal U, \ \mathcal V)$-approximation
$$0\longrightarrow \left(\begin{smallmatrix} X' \\ Y' \end{smallmatrix}\right)_s \longrightarrow
\left(\begin{smallmatrix} U' \\ V' \end{smallmatrix}\right)_h \longrightarrow
\left(\begin{smallmatrix} L_1 \\ L_2 \end{smallmatrix}\right)_g \longrightarrow 0$$
with $\binom{U'}{V'}_h\in \Delta(\mathcal U, \ \mathcal V)$ and $\binom{X'}{Y'}_s\in \binom{\mathcal X}{\mathcal Y}$.
This proves the completeness.

\vskip5pt

The assertion (2) can be similarly proved.
\hfill $\square$

\vskip5pt

Theorems \ref{ctp2}, \ref{ctp3}, and Proposition \ref{triangular} are new, even when $M = 0$ or $N =0$.

\section{\bf Realizations}
In Table 1, taking \ $(\mathcal U, \ \mathcal X)$ and \ $(\mathcal V, \ \mathcal Y)$ to be the projective cotorsion pair or the injective cotorsion pair, we get Table 2 below.
This section is to show that these cotorsion pairs in Table 2 are pairwise generally different and ``new" in some sense. For details see Definitions \ref{difference} and \ref{new}, Propositions \ref{different}, \ref{newI}, \ref{different2} and  \ref{newII}.
All these results are new, even for $M= 0$ or $N = 0$. Thus, it turns out that Morita rings are rich in producing ``new" cotorsion pairs.

\subsection{Cotorsion pairs in Series I in Table 2 are pairwise generally different}  To save the space, in Table 2 we use
\ $\mathcal A: = A\mbox{-}{\rm Mod}, \ \ \mathcal B: = B\mbox{-}{\rm Mod}, \ \ \mbox{proj.}: = \mbox{projective},$ \
$\mathcal M: = {\rm Mon}(\Lambda) = \Delta(A\mbox{\rm-Mod},  \ B\mbox{\rm-Mod})$ \ and \ $\mathcal E: = {\rm Epi}(\Lambda) = \nabla(A\mbox{\rm-Mod},  \ B\mbox{\rm-Mod}).$

\vskip5pt

{\bf About Table 2:}  (i)  \ It is clear that (see also Subsection 3.1)
$$(\Delta(_A\mathcal P,  \ _B\mathcal P), \ \Delta(_A\mathcal P,  \ _B\mathcal P)^\bot) = (_\Lambda\mathcal P, \ \Lambda\mbox{\rm-Mod}); \ \ \ (^\bot\nabla(_A\mathcal I, \ _B\mathcal I), \ \nabla(_A\mathcal I, \ _B\mathcal I)) = (\Lambda\mbox{-}{\rm Mod}, \ _\Lambda\mathcal I).$$

\vskip5pt

(ii) \  The cotorsion pairs in columns 2 and 3 in Table 2 are
cotorsion pairs in Series I,
and the ones in columns 4 and 5 are the cotorsion pairs in Series II. See Notation \ref{not}.

\vskip5pt

\vskip10pt

\centerline{\bf Table 2: \ Cotorsion pairs in $\Lambda$-Mod}

\vspace{-10pt}
$${\tiny\begin{tabular}{|c|c|c|c|c|}
\hline
\phantom{\LARGE 0} & \multicolumn{2}{c|} {\tabincell{c}{Hereditary cotorsion pairs in Series I\\[3pt] $\varphi=0=\psi$}}
& \multicolumn{2}{c|} {\tabincell{c}{Cotorsion pairs in Series II\\[3pt] $M\otimes_AN=0 = N\otimes_BM$}}
\\[8pt]\hline
\tabincell{c}{$(_A\mathcal U, \ _A\mathcal X)$ \\ $(_B\mathcal V, \ _B\mathcal Y)$}
& \tabincell{c}{$\Tor_1(M,\mathcal U)=0$ \\ [3pt] $\Tor_1(N,\mathcal V)=0$: \\[3pt] $(^\perp\binom{\mathcal X}{\mathcal Y},\binom{\mathcal X}{\mathcal Y})$}
& \tabincell{c}{$\Ext^1(N,\mathcal X)=0$ \\ [3pt]$\Ext^1(M,\mathcal Y)=0$: \\[3pt] $(\binom{\mathcal U}{\mathcal V},\binom{\mathcal U}{\mathcal V}^\perp)$}
& \tabincell{c}{$(\Delta(\mathcal U,\mathcal V), \ \Delta(\mathcal U, \mathcal V)^\bot)$}
& \tabincell{c}{$(^\bot\nabla(\mathcal X,\mathcal Y), \ \nabla(\mathcal X,\mathcal Y))$}
\\[15pt] \hline
\tabincell{c}{$(\mathcal P, \mathcal A$) \\ $(\mathcal P, \mathcal B)$}
&\tabincell{c}{$(_\Lambda\mathcal P, \Lambda\mbox{\rm-Mod})$}
& \tabincell{c}{$_AN$, $_BM$  proj.: \\[3pt] $(\binom{\mathcal P}{\mathcal P}, \ \binom{\mathcal P}{\mathcal P}^\perp)$}
& \tabincell{c}{$(_\Lambda\mathcal P, \Lambda\mbox{\rm-Mod})$}
& \tabincell{c}{$(^\bot{\rm Epi}(\Lambda), \ {\rm Epi}(\Lambda)).$ \\[3pt] Even if $_AN$, $_BM$  proj.,   \\[3pt]$(^\bot\mathcal E, \mathcal E)\ne (\binom{\mathcal P}{\mathcal P}, \ \binom{\mathcal P}{\mathcal P}^\perp)$\\[3pt] in general.}
\\[20pt] \hline
\tabincell{c}{$(\mathcal P, \mathcal A$) \\ $(\mathcal B, \mathcal I)$}
& \tabincell{c}{$N_B$ flat: \\ [3pt] $(^\perp\binom{\mathcal A}{\mathcal I}, \ \binom{\mathcal A}{\mathcal I})$}
& \tabincell{c}{$_AN$ proj.: \\[3pt] $(\binom{\mathcal P}{\mathcal B}, \ \binom{\mathcal P}{\mathcal B}^\perp)$}
& \tabincell{c}{$(\Delta(\mathcal P, \mathcal B), (\Delta(\mathcal P, \mathcal B)^\bot).$ \\ [3pt] If  $N_B$ flat then it is \\ [3pt] $(^\perp\binom{\mathcal A}{\mathcal I}, \ \binom{\mathcal A}{\mathcal I})$\\ [3pt] thus it is \\ [3pt] $(\Delta(\mathcal P, \mathcal B), \ \binom{\mathcal A}{\mathcal I})$}
& \tabincell{c}{$(^\perp\nabla(\mathcal A,  \mathcal I),  \nabla(\mathcal A,  \mathcal I)).$ \\ [3pt] If $_AN$ proj. then it is \\ [3pt]$(\binom{\mathcal P}{\mathcal B}, \ \binom{\mathcal P}{\mathcal B}^\perp)$\\[3pt]thus it is \\[3pt]$(\binom{\mathcal P}{\mathcal B}, \ \nabla(\mathcal A, \mathcal I))$}
\\[30pt] \hline
\tabincell{c}{$(\mathcal A, \mathcal I$) \\ $(\mathcal P, \mathcal B)$}
& \tabincell{c}{$M_A$ flat: \\[3pt] $(^\perp\binom{\mathcal I}{\mathcal B}, \ \binom{\mathcal I}{\mathcal B})$}
& \tabincell{c}{$_BM$ proj.: \\[3pt] $(\binom{\mathcal A}{\mathcal P}, \ \binom{\mathcal A}{\mathcal P}^\perp)$}
& \tabincell{c}{$(\Delta(\mathcal A, \mathcal P), \Delta(\mathcal A, \mathcal P)^\bot).$ \\[3pt] If $M_A$ flat then it is \\[3pt]$(^\perp\binom{\mathcal I}{\mathcal B}, \binom{\mathcal I}{\mathcal B})$\\[3pt] thus it is \\ [3pt]$((\Delta(\mathcal A, \mathcal P), \ \binom{\mathcal I}{\mathcal B})$}
& \tabincell{c}{$(^\bot\nabla(\mathcal I, \mathcal B), \nabla(\mathcal I, \mathcal B))$. \\[3pt] If $_BM$ proj. then it is\\[3pt]$(\binom{\mathcal A}{\mathcal P}, \ \binom{\mathcal A}{\mathcal P}^\perp)$\\[3pt]thus it is \\[3pt] $(\binom{\mathcal A}{\mathcal P}, \ \nabla(\mathcal I, \mathcal B))$}
\\[30pt] \hline
\tabincell{c}{$(\mathcal A, \mathcal I)$ \\ $(\mathcal B, \mathcal I)$}
& \tabincell{c}{$M_A$, $N_B$ flat: \\[3pt] $(^\perp\binom{\mathcal I}{\mathcal I}, \ \binom{\mathcal I}{\mathcal I})$}
& \tabincell{c}{$(\Lambda\mbox{-}{\rm Mod}, \ _\Lambda\mathcal I)$}
& \tabincell{c}{$({\rm Mon}(\Lambda), \ {\rm Mon}(\Lambda)^\bot)$.\\[3pt]Even if $M_A$, $N_B$ flat,\\ [3pt] $(\mathcal M, \ \mathcal M^\bot)\ne (^\perp\binom{\mathcal I}{\mathcal I}, \ \binom{\mathcal I}{\mathcal I})$\\ [3pt]in general}
& \tabincell{c} {$(\Lambda\mbox{-}{\rm Mod}, \ _\Lambda\mathcal I)$}
\\[20pt]\hline
\end{tabular}}$$

\vskip10pt

From the proof of Proposition \ref{different} we will see that, in the most cases,
the eight hereditary cotorsion pairs in Series I in Table 2 are pairwise different.

\vskip5pt

\begin{prop}\label{different}  \ Let $\Lambda = \left(\begin{smallmatrix} A & N \\
	M & B\end{smallmatrix}\right)$ be a Morita ring  with $\phi = 0=\psi$.
Then the eight hereditary cotorsion pairs in {\rm Series  I} in {\rm Table 2} are pairwise generally different, in the sense of {\rm Definition \ref{difference}}.\end{prop}

\begin{proof} \ All together there are $\binom{8}{2} = 28$ situations.

\vskip5pt {\bf Step 1.} \ If \ $A$ and $B$ are not semisimple,
then the cotorsion pairs in {\rm Series  I} in the same columns are pairwise different. This occupies $2\binom{4}{2} = 12$ situations.

\vskip5pt

For example, since $A$ is not semisimple, $A\mbox{-}{\rm Mod}\ne \ _A\mathcal I.$
Thus \ $\left(\begin{smallmatrix} A\mbox{-}{\rm Mod}\\ _B\mathcal I\end{smallmatrix}\right)
\ne \left(\begin{smallmatrix} _A\mathcal I\\ _B\mathcal I\end{smallmatrix}\right),$
and hence
$$(^\perp\left(\begin{smallmatrix} A\mbox{-}{\rm Mod}\\ _B\mathcal I\end{smallmatrix}\right), \ \left(\begin{smallmatrix} A\mbox{-}{\rm Mod}\\ _B\mathcal I\end{smallmatrix}\right))
\ne (^\perp\left(\begin{smallmatrix} _A\mathcal I\\ _B\mathcal I\end{smallmatrix}\right), \ \left(\begin{smallmatrix} _A\mathcal I\\ _B\mathcal I\end{smallmatrix}\right)).$$

{\bf Step 2.} \ The projective cotorsion pair $(_\Lambda\mathcal P, \ \Lambda\mbox{\rm-Mod})$
is generally different from all other seven cotorsion pairs in Series I. This occupies $4$ situations.

\vskip5pt

In fact, taking \ $_AN = A = B = \ _BM \ne 0$, then \
$$\left(\begin{smallmatrix} N\\ 0\end{smallmatrix}\right)_{0, 0}\in \left(\begin{smallmatrix}_A\mathcal P\\ _B\mathcal P\end{smallmatrix}\right), \ \ \  \
\left(\begin{smallmatrix} N\\ 0\end{smallmatrix}\right)_{0, 0}\in \left(\begin{smallmatrix}_A\mathcal P\\ B\mbox{-}{\rm Mod}\end{smallmatrix}\right), \ \  \ \
\left(\begin{smallmatrix} N\\ 0\end{smallmatrix}\right)_{0, 0}\in \left(\begin{smallmatrix} A\mbox{-}{\rm Mod}\\ _B\mathcal P\end{smallmatrix}\right), \ \ \ \ \left(\begin{smallmatrix} N\\ 0\end{smallmatrix}\right)_{0, 0}\in \Lambda\mbox{-}{\rm Mod}$$
but \  $\binom{N}{0}_{0, 0}\notin \ _\Lambda\mathcal P$.
Thus \ $(_\Lambda\mathcal P, \ \Lambda\mbox{\rm-Mod})$ is generally different from
$$(\left(\begin{smallmatrix}_A\mathcal P\\ _B\mathcal P\end{smallmatrix}\right), \ \left(\begin{smallmatrix}_A\mathcal P\\ _B\mathcal P\end{smallmatrix}\right)^\perp),  \ \ \ (\left(\begin{smallmatrix}_A\mathcal P\\ B\mbox{-}{\rm Mod}\end{smallmatrix}\right), \ \left(\begin{smallmatrix}_A\mathcal P\\ B\mbox{-}{\rm Mod}\end{smallmatrix}\right)^\perp), \ \ \ (\left(\begin{smallmatrix}A\mbox{-}{\rm Mod}\\ _B\mathcal P\end{smallmatrix}\right), \ \left(\begin{smallmatrix}A\mbox{-}{\rm Mod}\\ _B\mathcal P\end{smallmatrix}\right)^\perp), \ \ \ (\Lambda\mbox{\rm-Mod}, \ _\Lambda\mathcal I).$$

{\bf Step 3.} \ Similarly, the injective cotorsion pair $(\Lambda\mbox{\rm-Mod}, \ _\Lambda\mathcal I)$
is generally different from all other seven cotorsion pairs in Series I. This occupies $3$ situations.

\vskip5pt

{\bf Step 4.} \ For convenience, denote by $R_{\mathcal X, \ \mathcal Y}$ the cotorsion pair where \
$\left(\begin{smallmatrix}\mathcal X\\ \mathcal Y\end{smallmatrix}\right)$ is at the right hand side, i.e.,
$R_{\mathcal X, \ \mathcal Y} = ({}^\perp\left(\begin{smallmatrix}\mathcal X\\ \mathcal Y\end{smallmatrix}\right), \ \left(\begin{smallmatrix}\mathcal X\\ \mathcal Y\end{smallmatrix}\right))$.
Similarly, \ $L_{\mathcal U, \ \mathcal V} = (\left(\begin{smallmatrix}\mathcal U\\ \mathcal V\end{smallmatrix}\right), \ \left(\begin{smallmatrix}\mathcal U\\ \mathcal V\end{smallmatrix}\right)^\perp)$.

\vskip5pt

Assume that \ $A$ and $B$ are not semisimple. Under some extra conditions we will show the following remaining $9$ cases (listed in the order of comparing each cotorsion pair with the ones after):
\begin{align*}& L_{_A\mathcal P, \ _B\mathcal P} \ne R_{A\mbox{-}{\rm Mod}, \ _B\mathcal I}; \ \ \ \ \ \ \ \ L_{_A\mathcal P, \ _B\mathcal P}\ne R_{_A\mathcal I, \ B\mbox{-}{\rm Mod}};
\ \ \ \ \ \ \ \ \ \ L_{_A\mathcal P, \ _B\mathcal P}\ne R_{_A\mathcal I, \ _B\mathcal I};
\\ &
R_{A\mbox{-}{\rm Mod}, \ _B\mathcal I} \ne L_{_A\mathcal P, \ B\mbox{-}{\rm Mod}}; \ \ \ \  R_{A\mbox{-}{\rm Mod}, \ _B\mathcal I} \ne L_{A\mbox{-}{\rm Mod}, \ _B\mathcal P};
\ \ \ \ \ \ L_{_A\mathcal P, B\mbox{-}{\rm Mod}}\ne R_{_A\mathcal I, \ B\mbox{-}{\rm Mod}};
\\ &
L_{_A\mathcal P, B\mbox{-}{\rm Mod}}\ne R_{_A\mathcal I, \ _B\mathcal I}; \ \ \ \ \ \ \ \ \
R_{_A\mathcal I, \ B\mbox{-}{\rm Mod}}\ne  L_{A\mbox{-}{\rm Mod}, \ _B\mathcal P}; \ \ \ \ \ \
L_{A\mbox{-}{\rm Mod}, \ _B\mathcal P}\ne R_{_A\mathcal I, \ _B\mathcal I}.
\end{align*}

To see the  inequalities involving $L_{_A\mathcal P, \ _B\mathcal P}= (\binom{_A\mathcal P}{_B\mathcal P}, \ \binom{_A\mathcal P}{_B\mathcal P}^\perp)$, it suffices to show
$$\left(\begin{smallmatrix}_A\mathcal P\\ _B\mathcal P\end {smallmatrix}\right) \ne \ ^\perp\left(\begin{smallmatrix}A\mbox{-}{\rm Mod} \\ _B\mathcal I\end {smallmatrix}\right), \ \ \ \
\left(\begin{smallmatrix}_A\mathcal P\\ _B\mathcal P\end {smallmatrix}\right) \ne \ ^\perp\left(\begin{smallmatrix}_A\mathcal I\\ B\mbox{-}{\rm Mod}\end {smallmatrix}\right), \ \ \ \
\left(\begin{smallmatrix}_A\mathcal P\\ _B\mathcal P\end {smallmatrix}\right) \ne \ ^\perp\left(\begin{smallmatrix}_A\mathcal I\\_B\mathcal I\end {smallmatrix}\right).$$

Since $B$ is not semisimple, there is a non-projective $B$-module $Y$. By Lemma \ref{extadj1}(2),
\ $ {\rm T}_BY = \left(\begin{smallmatrix}N\otimes_BY\\ Y\end {smallmatrix}\right)_{0, 1}\in \ ^\perp\left(\begin{smallmatrix}A\mbox{-}{\rm Mod}\\ _B\mathcal I\end {smallmatrix}\right),$  but
$\left(\begin{smallmatrix}N\otimes_BY\\ Y\end {smallmatrix}\right)_{0,1}\notin \left(\begin{smallmatrix}_A\mathcal P\\ _B\mathcal P\end {smallmatrix}\right)$.
This shows \ $\left(\begin{smallmatrix}_A\mathcal P\\ _B\mathcal P\end {smallmatrix}\right) \ne \ ^\perp\left(\begin{smallmatrix}A\mbox{-}{\rm Mod}\\ _B\mathcal I\end {smallmatrix}\right)$.

Since $A$ is not semisimple, there is a non-projective $A$-module $X$. By Lemma \ref{extadj1}(1),
\ ${\rm T}_A X= \left(\begin{smallmatrix}X\\ M\otimes_AX\end {smallmatrix}\right)_{1,0}\in \ ^\perp\left(\begin{smallmatrix}_A\mathcal I\\ B\mbox{-}{\rm Mod}\end {smallmatrix}\right)\cap
\ \ ^\perp\left(\begin{smallmatrix}_A\mathcal I\\ _B\mathcal I\end {smallmatrix}\right)$.   But $\left(\begin{smallmatrix}X\\ M\otimes_AX\end {smallmatrix}\right)_{1,0}
\notin \left(\begin{smallmatrix}_A\mathcal P\\ _B\mathcal P\end {smallmatrix}\right)$.
This shows \ $\left(\begin{smallmatrix}_A\mathcal P\\ _B\mathcal P\end {smallmatrix}\right) \ne \ ^\perp\left(\begin{smallmatrix}_A\mathcal I\\ B\mbox{-}{\rm Mod}\end {smallmatrix}\right)$ and
\ $\binom{_A\mathcal P}{_B\mathcal P} \ne
\ ^\perp\left(\begin{smallmatrix}_A\mathcal I\\ _B\mathcal I\end {smallmatrix}\right).$

\vskip10pt

For the next inequalities involving  \ $R_{A\mbox{-}{\rm Mod}, \ _B\mathcal I}= (^\perp\binom{A\mbox{-}{\rm Mod}}{_B\mathcal I}, \ \binom{A\mbox{-}{\rm Mod}}{_B\mathcal I}),$
we need to find conditions such that
$$^\perp\left(\begin{smallmatrix}A\mbox{-}{\rm Mod}\\ _B\mathcal I\end {smallmatrix}\right) \ne \left(\begin{smallmatrix}_A\mathcal P\\ B\mbox{-}{\rm Mod}\end {smallmatrix}\right), \ \ \ \ \ \ \ ^\perp\left(\begin{smallmatrix}A\mbox{-}{\rm Mod}\\_B\mathcal I\end {smallmatrix}\right) \ne \left(\begin{smallmatrix}A\mbox{-}{\rm Mod}\\ _B\mathcal P\end {smallmatrix}\right).$$
Taking $A = B = M = N\ne 0$ and a non-projective $B$-module $Y$,
then ${\rm T}_BY = \left(\begin{smallmatrix}N\otimes_BY\\ Y\end {smallmatrix}\right)_{0, 1} = \left(\begin{smallmatrix}Y\\ Y\end {smallmatrix}\right)_{0, 1} \in \ ^\perp\left(\begin{smallmatrix}A\mbox{-}{\rm Mod}\\_B\mathcal I\end {smallmatrix}\right)$ by Lemma \ref{extadj1}(2),
but \ $\left(\begin{smallmatrix}Y\\ Y\end {smallmatrix}\right)_{0,1}\notin \left(\begin{smallmatrix}_A\mathcal P\\ B\mbox{-}{\rm Mod}\end {smallmatrix}\right)$ \ and \ $\left(\begin{smallmatrix}Y\\ Y\end {smallmatrix}\right)_{0,1}\notin \left(\begin{smallmatrix}A\mbox{-}{\rm Mod}\\ _B\mathcal P\end {smallmatrix}\right)$.

\vskip10pt

To see $L_{_A\mathcal P, \ B\mbox{-}{\rm Mod}}\ne R_{_A\mathcal I, \ B\mbox{-}{\rm Mod}}$ and \ $L_{_A\mathcal P, \ B\mbox{-}{\rm Mod}}\ne R_{_A\mathcal I, \ _B\mathcal I}$,
it suffices to show
$$\left(\begin{smallmatrix}_A\mathcal P\\ B\mbox{-}{\rm Mod}\end {smallmatrix}\right) \ne \ ^\perp\left(\begin{smallmatrix}_A\mathcal I\\ B\mbox{-}{\rm Mod}\end {smallmatrix}\right),
\ \ \ \ \ \ \ \left(\begin{smallmatrix}_A\mathcal P\\ B\mbox{-}{\rm Mod}\end {smallmatrix}\right) \ne \ ^\perp\left(\begin{smallmatrix}_A\mathcal I\\ _B\mathcal I\end {smallmatrix}\right).$$
For a non-projective $A$-module $X$, \ $ {\rm T}_AX = \binom{X}{M\otimes_AX}_{1, 0}\in \ ^\perp\left(\begin{smallmatrix}_A\mathcal I\\ B\mbox{-}{\rm Mod}\end {smallmatrix}\right)\cap
\ ^\perp\binom{_A\mathcal I}{_B\mathcal I}$,
but $\left(\begin{smallmatrix}X\\ M\otimes_AX \end {smallmatrix}\right)_{1, 0}\notin \binom{_A\mathcal P}{B\mbox{-}{\rm Mod}}$.

\vskip10pt

Finally, we show that \ $L_{A\mbox{-}{\rm Mod}, \ _B\mathcal P}$ is generally different from \ $R_{_A\mathcal I, \ B\mbox{-}{\rm Mod}}$ \ and \ $R_{_A\mathcal I, \ _A\mathcal I}$.
Taking $A = B = M = N\ne 0$ and a non-projective $A$-module $X$,  it suffices to see
$${}^\perp\left(\begin{smallmatrix}_A\mathcal I\\ B\mbox{-}{\rm Mod}\end {smallmatrix}\right)\ne  \left(\begin{smallmatrix} A\mbox{-}{\rm Mod}\\_B\mathcal P\end {smallmatrix}\right), \ \ \ \ \ \
\left(\begin{smallmatrix} A\mbox{-}{\rm Mod}\\ _B\mathcal P\end {smallmatrix}\right) \ne \ ^\perp\left(\begin{smallmatrix}_A\mathcal I\\_B\mathcal I\end {smallmatrix}\right).$$
In fact, by Lemma \ref{extadj1}(2), \ ${\rm T}_AX = \left(\begin{smallmatrix}X\\ M\otimes_AX\end {smallmatrix}\right)_{1, 0}= \left(\begin{smallmatrix}X\\ X\end {smallmatrix}\right)_{1, 0}\in \ ^\perp\left(\begin{smallmatrix}_A\mathcal I\\ B\mbox{-}{\rm Mod}\end {smallmatrix}\right)\cap \ ^\perp\left(\begin{smallmatrix}_A\mathcal I\\ _B\mathcal I\end {smallmatrix}\right)$;
but $\left(\begin{smallmatrix}X\\ X\end {smallmatrix}\right)_{1, 0}\notin \left(\begin{smallmatrix}A\mbox{-}{\rm Mod}\\ _B\mathcal P\end {smallmatrix}\right)$.

\vskip5pt

This completes the proof. \end{proof}

\subsection{``New" cotorsion pairs in Series I in Table 2} Keep the notations $R_{\mathcal X, \ \mathcal Y}$ and $L_{\mathcal U, \ \mathcal V}$ as in the {\bf Step 4} of the proof of Proposition \ref{different}.
Taking off the projective cotorsion pair and
the injective one from Series I of Table 2,
the remaining six hereditary cotorsion pairs
\begin{align*}& R_{_A\mathcal I, \ _B\mathcal I} =
(^\perp\left(\begin{smallmatrix}_A\mathcal I\\ _B\mathcal I\end {smallmatrix}\right), \ \left(\begin{smallmatrix}_A\mathcal I\\ _B\mathcal I\end {smallmatrix}\right)),
\ \ \ \ \ \ \ \ \ \ \ \ \ \ \ \ \ \ L_{_A\mathcal P, \ _B\mathcal P}   =
(\left(\begin{smallmatrix}_A\mathcal P\\ _B\mathcal P\end {smallmatrix}\right), \ \left(\begin{smallmatrix}_A\mathcal P\\ _B\mathcal P\end {smallmatrix}\right)^\perp),
\\ & R_{A\mbox{-}{\rm Mod}, \ _B\mathcal I}= (^\perp\left(\begin{smallmatrix}A\mbox{-}{\rm Mod}\\ _B\mathcal I\end {smallmatrix}\right), \ \left(\begin{smallmatrix}A\mbox{-}{\rm Mod}\\ _B\mathcal I\end {smallmatrix}\right)),
\ \ \ \ \ L_{_A\mathcal P, \ B\mbox{-}{\rm Mod}} = (\left(\begin{smallmatrix}_A\mathcal P\\ B\mbox{-}{\rm Mod}\end {smallmatrix}\right),  \
\left(\begin{smallmatrix}_A\mathcal P\\ B\mbox{-}{\rm Mod}\end {smallmatrix}\right)^\perp),
\\ & R_{_A\mathcal I, \ B\mbox{-}{\rm Mod}} = (^\perp\left(\begin{smallmatrix}_A\mathcal I\\ B\mbox{-}{\rm Mod}\end {smallmatrix}\right), \ \left(\begin{smallmatrix}_A\mathcal I\\B\mbox{-}{\rm Mod}\end {smallmatrix}\right)),
\ \ \ \ \ L_{A\mbox{-}{\rm Mod}, \ _B\mathcal P} = (\left(\begin{smallmatrix}A\mbox{-}{\rm Mod}\\ _B\mathcal P\end {smallmatrix}\right), \ (\left(\begin{smallmatrix}A\mbox{-}{\rm Mod}\\ _B\mathcal P\end {smallmatrix}\right)^\perp)\end{align*}
are ``new", in the following sense.

\begin{defn} \label{new} \ A cotorsion pair in $\Lambda\mbox{-}{\rm Mod}$ is said to be ``new", provided that it is generally different from all of the following cotorsion pairs:

\vskip5pt

$\bullet$ \ the projective cotorsion pair \ $(_\Lambda\mathcal P, \ \Lambda\mbox{-}{\rm Mod});$

$\bullet$ \ the injective cotorsion pair \ $(\Lambda\mbox{-}{\rm Mod},  \ _\Lambda\mathcal I);$

$\bullet$ \ the Gorenstein-projective cotorsion pair \ $({\rm GP}(\Lambda),  \ _\Lambda\mathcal P^{<\infty})$, if \ $\Lambda$ is a Gorenstein ring;

$\bullet$ \ the Gorenstein-injective cotorsion pair \ $(_\Lambda\mathcal P^{<\infty}, \ {\rm GI}(\Lambda))$, if \ $\Lambda$ is a Gorenstein ring;

$\bullet$ \ the flat cotorsion pair \ $(_\Lambda{\rm F}, \ _\Lambda{\rm C})$, where $_\Lambda{\rm F}$ is the class of flat $\Lambda$-modules, and
$_\Lambda{\rm C}$ is the class of cotorsion $\Lambda$-modules. See [FJ, Lemma 7.1.4].
\end{defn}

\begin{prop}\label{newI} \  Let $\Lambda = \left(\begin{smallmatrix} A & N \\ M & B \end{smallmatrix}\right)$ be a Morita ring with $\phi = 0= \phi$. Then the following six cotorsion pairs
$$R_{_A\mathcal I, \ _B\mathcal I}, \ \ \ L_{_A\mathcal P, \ _B\mathcal P},  \ \ \ R_{A\mbox{-}{\rm Mod}, \ _B\mathcal I}, \ \ \ L_{_A\mathcal P, \ B\mbox{-}{\rm Mod}}, \ \ \ R_{_A\mathcal I, \ B\mbox{-}{\rm Mod}}, \ \ \ L_{A\mbox{-}{\rm Mod}, \ _B\mathcal P}$$ are ``new", in the sense of {\rm Definition \ref{new}}.
\end{prop}

To prove Proposition \ref{newI} we need some preparations.

\begin{lem} \label{agoralg} \ {\rm ([GaP, 4.15])} \ Let $\Lambda = \left(\begin{smallmatrix} A & N \\ N & A \end{smallmatrix}\right)$ be a Morita ring with $N\otimes_AN = 0$.
Assume that \ $_AN$ and $N_A$ are projective. If $A$ is a Gorenstein ring, then so is $\Lambda$.
\end{lem}

\begin{lem}\label{notgor1} \ Let $\Lambda = \left(\begin{smallmatrix} A & N \\ M & B \end{smallmatrix}\right)$ be a Morita ring which is an Artin algebra and a Gorenstein ring, with $\phi = 0= \phi$. Then the cotorsion pairs
\ $R_{_A\mathcal I, \ _B\mathcal I}$ and $L_{_A\mathcal P, \ _B\mathcal P}$ are generally different from the Gorenstein-projective cotorsion pair and
the Gorenstein-injective cotorsion pair.
\end{lem}

\begin{proof} \ Take \ $\Lambda$ to be the Morita ring $\Lambda = \left(\begin{smallmatrix} A & N \\ N & A \end{smallmatrix}\right)$, constructed in Example \ref{ie}. Thus
 $A$ is the path algebra $k(1 \longrightarrow 2)$ with ${\rm char} \ k\ne 2$,  $N = Ae_2\otimes_ke_1A$,  and $N\otimes_AN = 0$. By Lemma \ref{agoralg},  $\Lambda$ is a Gorenstein algebra.

\vskip5pt

{\bf Claim 1.} \ $R_{_A\mathcal I, \ _B\mathcal I} = (^\perp\left(\begin{smallmatrix}_A\mathcal I\\ _B\mathcal I\end {smallmatrix}\right), \ \left(\begin{smallmatrix}_A\mathcal I\\ _B\mathcal I\end {smallmatrix}\right))$
is generally different from $({\rm GP}(\Lambda), \ _\Lambda\mathcal P^{<\infty})$.

\vskip5pt

In fact, since  \ $N\otimes_A S_2 = Ae_2\otimes_k(e_1Ae_2) =0$,  \ $\left(\begin{smallmatrix}S_2\\ 0\end{smallmatrix}\right)_{0, 0} = {\rm T}_AS_2$ is a projective $\Lambda$-module, thus
$\left(\begin{smallmatrix}S_2\\ 0\end{smallmatrix}\right)_{0, 0}\in \ _\Lambda\mathcal P^{<\infty}$, but \ $\left(\begin{smallmatrix}S_2\\ 0\end{smallmatrix}\right)_{0, 0}\notin \left(\begin{smallmatrix}_A\mathcal I\\ _A\mathcal I\end{smallmatrix}\right)$. Thus $\left(\begin{smallmatrix}_A\mathcal I\\ _A\mathcal I\end{smallmatrix}\right)\ne \ _\Lambda\mathcal P^{<\infty}$, and hence
$$R_{_A\mathcal I, \ _B\mathcal I} = (^\perp\left(\begin{smallmatrix}_A\mathcal I\\ _B\mathcal I\end {smallmatrix}\right), \ \left(\begin{smallmatrix}_A\mathcal I\\ _B\mathcal I\end {smallmatrix}\right))
\ne ({\rm GP}(\Lambda), \ _\Lambda\mathcal P^{<\infty}).$$

\vskip5pt

{\bf Claim 2.} \ $R_{_A\mathcal I, \ _B\mathcal I} = (^\perp\left(\begin{smallmatrix}_A\mathcal I\\ _B\mathcal I\end {smallmatrix}\right), \ \left(\begin{smallmatrix}_A\mathcal I\\ _B\mathcal I\end {smallmatrix}\right))$
is generally different from  $(_\Lambda\mathcal P^{<\infty}, \ {\rm GI}(\Lambda))$.

\vskip5pt

In fact, by Example \ref{ie} one knows
 $L = \left(\begin{smallmatrix}Ae_1\\ Ae_1\end{smallmatrix}\right)_{\sigma, \sigma}\notin \ ^\perp\left(\begin{smallmatrix}_A\mathcal I\\ _A\mathcal I\end{smallmatrix}\right).$
The following $\Lambda$-projective resolution of $_\Lambda L$
$$0\longrightarrow \left(\begin{smallmatrix}S_2\\ S_2\end{smallmatrix}\right)_{0, 0}
\stackrel{\left(\begin{smallmatrix} \binom{\sigma}{-1}\\ \binom{1}{-\sigma}\end{smallmatrix}\right)}\longrightarrow \left(\begin{smallmatrix} Ae_1\oplus S_2\\ S_2 \oplus Ae_1\end{smallmatrix}\right)_{\left(\begin{smallmatrix}1 & 0 \\ 0 & 0\end{smallmatrix}\right), \left(\begin{smallmatrix}0 & 0\\ 0 & 1\end{smallmatrix}\right)}
\stackrel{\left(\begin{smallmatrix} (1, \sigma) \\ (\sigma,1)\end{smallmatrix}\right)}
\longrightarrow \left(\begin{smallmatrix}Ae_1\\ Ae_1\end{smallmatrix}\right)_{\sigma, \sigma}\longrightarrow 0$$
shows that ${\rm proj.dim} _\Lambda L = 1$.
So $L\in \  _\Lambda\mathcal P^{<\infty}$, and hence $^\perp\left(\begin{smallmatrix}_A\mathcal I\\ _B\mathcal I\end {smallmatrix}\right)\ne \  _\Lambda\mathcal P^{<\infty}$. Thus
$$R_{_A\mathcal I, \ _B\mathcal I} = (^\perp\left(\begin{smallmatrix}_A\mathcal I\\ _B\mathcal I\end {smallmatrix}\right), \ \left(\begin{smallmatrix}_A\mathcal I\\ _B\mathcal I\end {smallmatrix}\right))\ne (_\Lambda\mathcal P^{<\infty}, \ {\rm GI}(\Lambda)).$$

\vskip5pt

{\bf Claim 3.} \ $L_{_A\mathcal P, \ _B\mathcal P}   =
(\left(\begin{smallmatrix}_A\mathcal P\\ _B\mathcal P\end {smallmatrix}\right), \ \left(\begin{smallmatrix}_A\mathcal P\\ _B\mathcal P\end {smallmatrix}\right)^\perp)$
is generally different from $({\rm GP}(\Lambda), \ _\Lambda\mathcal P^{<\infty})$.

\vskip5pt

In fact, by Example \ref{ie} one knows
 $L = \left(\begin{smallmatrix}Ae_1\\ Ae_1\end{smallmatrix}\right)_{\sigma, \sigma}\notin  \left(\begin{smallmatrix}_A\mathcal P\\ _A\mathcal P\end{smallmatrix}\right)^\perp.$
By {\bf Claim 2},  $L\in \  _\Lambda\mathcal P^{<\infty}$. Thus
$\left(\begin{smallmatrix}_A\mathcal P\\ _A\mathcal P\end{smallmatrix}\right)^\perp \ne \  _\Lambda\mathcal P^{<\infty}$, and hence
$$L_{_A\mathcal P, \ _B\mathcal P}   =
(\left(\begin{smallmatrix}_A\mathcal P\\ _B\mathcal P\end {smallmatrix}\right), \ \left(\begin{smallmatrix}_A\mathcal P\\ _B\mathcal P\end {smallmatrix}\right)^\perp)
\ne ({\rm GP}(\Lambda), \ _\Lambda\mathcal P^{<\infty}).$$

\vskip5pt

{\bf Claim 4.} \ $L_{_A\mathcal P, \ _B\mathcal P}   =
(\left(\begin{smallmatrix}_A\mathcal P\\ _B\mathcal P\end {smallmatrix}\right), \ \left(\begin{smallmatrix}_A\mathcal P\\ _B\mathcal P\end {smallmatrix}\right)^\perp)$
is generally different from  $(_\Lambda\mathcal P^{<\infty}, \ {\rm GI}(\Lambda))$.

\vskip5pt
In fact, since  \ $\Hom_A(N, S_1) = 0$,  \ $\left(\begin{smallmatrix} 0\\ S_1\end{smallmatrix}\right)_{0, 0} = {\rm H}_BS_1$ is an injective $\Lambda$-module, thus
$\left(\begin{smallmatrix}0\\ S_1\end{smallmatrix}\right)_{0, 0}\in \ _\Lambda\mathcal I^{<\infty} = \ _\Lambda\mathcal P^{<\infty}$,
but \ $\left(\begin{smallmatrix}0\\ S_1\end{smallmatrix}\right)_{0, 0}\notin \left(\begin{smallmatrix}_A\mathcal P\\ _A\mathcal P\end{smallmatrix}\right)$.
Thus $\left(\begin{smallmatrix}_A\mathcal P\\ _A\mathcal P\end{smallmatrix}\right)\ne \ _\Lambda\mathcal P^{<\infty}$, and hence
$$L_{_A\mathcal P, \ _B\mathcal P}   =
(\left(\begin{smallmatrix}_A\mathcal P\\ _B\mathcal P\end {smallmatrix}\right), \ \left(\begin{smallmatrix}_A\mathcal P\\ _B\mathcal P\end {smallmatrix}\right)^\perp)
\ne (_\Lambda\mathcal P^{<\infty}, \ {\rm GI}(\Lambda)).$$

This completes the proof. \end{proof}

\begin{lem}\label{notgor2} \ Let $\Lambda = \left(\begin{smallmatrix} A & N \\ M & B \end{smallmatrix}\right)$ be a Morita ring which is an Artin algebra and a Gorenstein ring, with $\phi = 0= \phi$. Then the cotorsion pairs
$$R_{A\mbox{-}{\rm Mod}, \ _B\mathcal I}, \ \ \ L_{_A\mathcal P, \ B\mbox{-}{\rm Mod}}, \ \ \ R_{_A\mathcal I, \ B\mbox{-}{\rm Mod}}, \ \ \ L_{A\mbox{-}{\rm Mod}, \ _B\mathcal P}$$ are generally different from the Gorenstein-projective cotorsion pair and
the Gorenstein-injective cotorsion pair.
\end{lem}

\begin{proof} \ Choose quasi-Frobenius rings $A$ and $B$, bimodules \ $_BM_A$ and $_AN_B$, satisfying the following conditions (i), (ii), (iii), (iv):

\vskip5pt

(i) \ \ \ $A$ and $B$ are quasi-Frobenius and not semisimple;

(ii) \ \ $_AN$ and $_BM$ are non-zero projective modules, and $M_A$ and $N_B$ are flat;

(iii) \ \ $\Lambda = \left(\begin{smallmatrix} A & N \\ M & B \end{smallmatrix}\right)$ is a Morita ring with $M\otimes_A N = 0 = N\otimes_BM;$

(iv) \ \  $\Lambda$ is a noetherian ring.

\vskip5pt

By Remark \ref{examctp4}, such $\Lambda$'s exist! By Theorem \ref{ctp4}, \ $\Lambda$ is a Gorenstein ring with ${\rm inj.dim} \Lambda \le 1$, the Gorenstein-projective cotorsion pair \ $({\rm GP}(\Lambda), \ \mathcal P^{\le 1})$ is exactly
\ $({}^\perp\binom{_A\mathcal I}{_B\mathcal I}, \ \binom{_A\mathcal I}{_B\mathcal I})$, and the Gorenstein-injective cotorsion pair \ $(_\Lambda \mathcal P^{\le 1}, \ {\rm GI}(\Lambda))$ is exactly \
$(\left(\begin{smallmatrix}_A\mathcal P\\ _B\mathcal P\end {smallmatrix}\right), \ \left(\begin{smallmatrix}_A\mathcal P\\ _B\mathcal P\end {smallmatrix}\right)^\perp)$.

\vskip5pt

{\bf Claim 1.} \ $R_{A\mbox{-}{\rm Mod}, \ _B\mathcal I}$ and \ $R_{_A\mathcal I, \ B\mbox{-}{\rm Mod}}$
are generally different from the Gorenstein-projective cotorsion pair.

Since $A$ and $B$ are not semisimple, $A\mbox{-}{\rm Mod}\ne \ _A\mathcal I$ \ and \ $B\mbox{-}{\rm Mod}\ne \ _B\mathcal I$. Thus \
$\left(\begin{smallmatrix}A\mbox{-}{\rm Mod}\\ _B\mathcal I\end {smallmatrix}\right)\ne \left(\begin{smallmatrix}_A\mathcal I\\ _B\mathcal I\end {smallmatrix}\right)$ and \ $\left(\begin{smallmatrix}_A\mathcal I\\B\mbox{-}{\rm Mod}\end {smallmatrix}\right)\ne \left(\begin{smallmatrix}_A\mathcal I\\ _B\mathcal I\end {smallmatrix}\right)$, and hence
$$R_{A\mbox{-}{\rm Mod}, \ _B\mathcal I}
= (^\perp\left(\begin{smallmatrix}A\mbox{-}{\rm Mod}\\ _B\mathcal I\end {smallmatrix}\right), \ \left(\begin{smallmatrix}A\mbox{-}{\rm Mod}\\ _B\mathcal I\end {smallmatrix}\right))\ne ({}^\perp\left(\begin{smallmatrix}_A\mathcal I\\ _B\mathcal I\end {smallmatrix}\right), \ \left(\begin{smallmatrix}_A\mathcal I\\ _B\mathcal I\end {smallmatrix}\right)) = ({\rm GP}(\Lambda), \ _\Lambda\mathcal P^{\le 1})$$
and
$$R_{_A\mathcal I, \ B\mbox{-}{\rm Mod}} = (^\perp\left(\begin{smallmatrix}_A\mathcal I\\ B\mbox{-}{\rm Mod}\end {smallmatrix}\right), \ \left(\begin{smallmatrix}_A\mathcal I\\B\mbox{-}{\rm Mod}\end {smallmatrix}\right))\ne ({}^\perp\left(\begin{smallmatrix}_A\mathcal I\\ _B\mathcal I\end {smallmatrix}\right), \ \left(\begin{smallmatrix}_A\mathcal I\\ _B\mathcal I\end {smallmatrix}\right)) = ({\rm GP}(\Lambda), \ _\Lambda\mathcal P^{\le 1}).$$

\vskip5pt

{\bf Claim 2.} \ $L_{_A\mathcal P, \ B\mbox{-}{\rm Mod}}$ and \ $L_{A\mbox{-}{\rm Mod}, \ _B\mathcal P}$
are generally different from the Gorenstein-projective cotorsion pair.

Since $A$ is not semisimple, there is a non-projective $A$-module $X$. By Lemma \ref{extadj1}(1),
\ ${\rm T}_AX = \binom{X}{M\otimes_AX}_{1, 0}\in \ ^\perp\binom{_A\mathcal I}{_B\mathcal I}$,
but $\left(\begin{smallmatrix}X\\ M\otimes_AX \end {smallmatrix}\right)_{1, 0} \notin \binom{_A\mathcal P}{B\mbox{-}{\rm Mod}}$,
which shows \ $\binom{_A\mathcal P}{B\mbox{-}{\rm Mod}} \ne \ ^\perp\binom{_A\mathcal I}{_B\mathcal I}.$
Hence
$$L_{_A\mathcal P, \ B\mbox{-}{\rm Mod}} = (\left(\begin{smallmatrix}_A\mathcal P\\ B\mbox{-}{\rm Mod}\end {smallmatrix}\right),  \ \left(\begin{smallmatrix}_A\mathcal P\\ B\mbox{-}{\rm Mod}\end {smallmatrix}\right)^\perp)\ne
({}^\perp\left(\begin{smallmatrix}_A\mathcal I\\ _B\mathcal I\end {smallmatrix}\right), \ \left(\begin{smallmatrix}_A\mathcal I\\ _B\mathcal I\end {smallmatrix}\right))= ({\rm GP}(\Lambda), \ _\Lambda\mathcal P^{\le 1}).$$

Similarly, \ $L_{A\mbox{-}{\rm Mod}, \ _B\mathcal P} = (\left(\begin{smallmatrix}A\mbox{-}{\rm Mod}\\ _B\mathcal P\end {smallmatrix}\right), \ (\left(\begin{smallmatrix}A\mbox{-}{\rm Mod}\\ _B\mathcal P\end {smallmatrix}\right)^\perp)\ne
({}^\perp\left(\begin{smallmatrix}_A\mathcal I\\ _B\mathcal I\end {smallmatrix}\right), \ \left(\begin{smallmatrix}_A\mathcal I\\ _B\mathcal I\end {smallmatrix}\right))= ({\rm GP}(\Lambda), \ _\Lambda\mathcal P^{\le 1})
.$

\vskip5pt

{\bf Claim 3.} \ $R_{A\mbox{-}{\rm Mod}, \ _B\mathcal I}$ and \ $R_{_A\mathcal I, \ B\mbox{-}{\rm Mod}}$ are generally different from the Gorenstein-injective cotorsion pair.

Since $B$ is not semisimple, there is a non-projective $B$-module $Y$. Then
\ $ {\rm T}_BY = \left(\begin{smallmatrix}N\otimes_BY\\ Y\end {smallmatrix}\right)_{0, 1}\in \ ^\perp\left(\begin{smallmatrix}A\mbox{-}{\rm Mod}\\ _B\mathcal I\end {smallmatrix}\right)$ by Lemma \ref{extadj1}(2), but
$\left(\begin{smallmatrix}N\otimes_BY\\ Y\end {smallmatrix}\right)_{0,1}\notin \left(\begin{smallmatrix}_A\mathcal P\\ _B\mathcal P\end {smallmatrix}\right)$.
This shows \ $^\perp\left(\begin{smallmatrix}A\mbox{-}{\rm Mod}\\ _B\mathcal I\end {smallmatrix}\right)\ne \left(\begin{smallmatrix}_A\mathcal P\\ _B\mathcal P\end {smallmatrix}\right)$.
Thus
$$R_{A\mbox{-}{\rm Mod}, \ _B\mathcal I}
= (^\perp\left(\begin{smallmatrix}A\mbox{-}{\rm Mod}\\ _B\mathcal I\end {smallmatrix}\right), \ \left(\begin{smallmatrix}A\mbox{-}{\rm Mod}\\ _B\mathcal I\end {smallmatrix}\right))\ne (\left(\begin{smallmatrix}_A\mathcal P\\ _B\mathcal P\end {smallmatrix}\right), \ \left(\begin{smallmatrix}_A\mathcal P\\ _B\mathcal P\end {smallmatrix}\right)^\perp)= (_\Lambda \mathcal P^{\le 1}, \ {\rm GI}(\Lambda)).$$

Similarly, \ $R_{_A\mathcal I, \ B\mbox{-}{\rm Mod}}= (^\perp\left(\begin{smallmatrix}_A\mathcal I\\ B\mbox{-}{\rm Mod}\end {smallmatrix}\right), \ \left(\begin{smallmatrix}_A\mathcal I\\ B\mbox{-}{\rm Mod}\end {smallmatrix}\right))\ne (\left(\begin{smallmatrix}_A\mathcal P\\ _B\mathcal P\end {smallmatrix}\right), \ \left(\begin{smallmatrix}_A\mathcal P\\ _B\mathcal P\end {smallmatrix}\right)^\perp)= (_\Lambda \mathcal P^{\le 1}, \ {\rm GI}(\Lambda)).$

\vskip5pt

{\bf Claim 4.} \ $L_{_A\mathcal P, \ B\mbox{-}{\rm Mod}}$ and \ $L_{A\mbox{-}{\rm Mod}, \ _B\mathcal P}$
are generally different from the Gorenstein-injective cotorsion pair.

Since $B$ is not semisimple, $B\mbox{-}{\rm Mod}\ne \ _B\mathcal P$. Thus \
$\left(\begin{smallmatrix}_A\mathcal P \\ B\mbox{-}{\rm Mod}\end {smallmatrix}\right)\ne \left(\begin{smallmatrix}_A\mathcal P\\ _B\mathcal P\end {smallmatrix}\right)$, and hence $$L_{_A\mathcal P, \ B\mbox{-}{\rm Mod}} = (\left(\begin{smallmatrix}_A\mathcal P\\ B\mbox{-}{\rm Mod}\end {smallmatrix}\right),  \
\left(\begin{smallmatrix}_A\mathcal P\\ B\mbox{-}{\rm Mod}\end {smallmatrix}\right)^\perp)\ne (\left(\begin{smallmatrix}_A\mathcal P\\ _B\mathcal P\end {smallmatrix}\right), \ \left(\begin{smallmatrix}_A\mathcal P\\ _B\mathcal P\end {smallmatrix}\right)^\perp) = (_\Lambda \mathcal P^{\le 1}, \ {\rm GI}(\Lambda)).$$

Similarly,  \ $L_{A\mbox{-}{\rm Mod}, \ _B\mathcal P} = (\left(\begin{smallmatrix}A\mbox{-}{\rm Mod}\\ _B\mathcal P\end {smallmatrix}\right), \ (\left(\begin{smallmatrix}A\mbox{-}{\rm Mod}\\ _B\mathcal P\end {smallmatrix}\right)^\perp)\ne (\left(\begin{smallmatrix}_A\mathcal P\\ _B\mathcal P\end {smallmatrix}\right), \ \left(\begin{smallmatrix}_A\mathcal P\\ _B\mathcal P\end {smallmatrix}\right)^\perp) = (_\Lambda \mathcal P^{\le 1}, \ {\rm GI}(\Lambda)).$

\vskip5pt

This completes the proof. \end{proof}

\vskip5pt
	
We also need the following result due to P. A. Krylov and E. Yu. Yardykov [KY].

\begin{lem}\label{flat} {\rm ([KY, Corollary 2.5])}
	Let $L=\left(\begin{smallmatrix}X\\ Y\end{smallmatrix}\right)_{f,g}$ be a flat $\Lambda$-module. Then $\Coker g$ is a flat $A$-module and $\Coker f$ is a flat $B$-module.
\end{lem}

\vskip5pt

\noindent {\bf Proof of Proposition \ref{newI}.} \ By Proposition \ref{different}, these six cotorsion pairs
are generally different from the projective cotorsion pair and the injective one; and
they are generally different from the Gorenstein-projective cotorsion pair and the Gorenstein-injective one, by Lemmas \ref{notgor1} and \ref{notgor2}. It remains to show that they are generally different from the flat cotorsion pair.

\vskip5pt

In fact, choose rings $A$ and $B$ such that they admit  non flat modules (such a ring $A$ of course exists! For example, just take a finite-dimensional algebra $A$ which is not semisimple. Then
$A$ has a finitely generated module $M$ which is not projective, and $M$ is not flat).  Taking non-flat modules \ $_AX$ and  $_BY$, by Lemma \ref{flat}, all the following  $\Lambda$-modules are not flat:
$$\left(\begin{smallmatrix} X\\ 0\end{smallmatrix}\right)_{0, 0}, \ \ \left(\begin{smallmatrix} 0\\ Y\end{smallmatrix}\right)_{0,0},
\ \ {\rm T}_A X = \left(\begin{smallmatrix}X\\ M\otimes_AX\end{smallmatrix}\right)_{1,0}, \ \ {\rm T}_BY = \left(\begin{smallmatrix} N\otimes_BY\\ Y\end{smallmatrix}\right)_{0,1}.
$$

\vskip5pt

\noindent However,

\vskip5pt

$\bullet$ \ For the cotorsion pair \ $R_{_A\mathcal I, \ _B\mathcal I} = (^\perp\left(\begin{smallmatrix}_A\mathcal I\\ _B\mathcal I\end {smallmatrix}\right), \ \left(\begin{smallmatrix}_A\mathcal I\\ _B\mathcal I\end {smallmatrix}\right))$,  one has \
${\rm T}_A X = \binom{X}{M\otimes_AX}_{1,0}\in {}^\perp\binom{_A\mathcal I}{_B\mathcal I}$, by Lemma \ref{extadj1}(1).

\vskip5pt

$\bullet$ \ For the cotorsion pair \ $R_{A\mbox{-}{\rm Mod}, \ _B\mathcal I} = (^\perp\binom{A\mbox{-}{\rm Mod}}{_B\mathcal I}, \ \binom{A\mbox{-}{\rm Mod}}{_B\mathcal I})$,
one has \ ${\rm T}_BY = \binom{N\otimes_BY}{Y}_{0,1}\in \ ^\perp\binom{A\mbox{-}{\rm Mod}}{_B\mathcal I}$, by Lemma \ref{extadj1}(2).

\vskip5pt

$\bullet$ \ For the cotorsion pair \ $L_{_A\mathcal P, \ B\mbox{-}{\rm Mod}} = (\binom{_A\mathcal P}{B\mbox{-}{\rm Mod}}, \ \binom{_A\mathcal P}{B\mbox{-}{\rm Mod}}^\perp),$ one has
\ $\binom{0}{Y}_{0,0}\in \binom{_A\mathcal P}{B\mbox{-}{\rm Mod}}$.

\vskip5pt

$\bullet$ \ For the cotorsion pair \ $R_{_A\mathcal I, \ B\mbox{-}{\rm Mod}} = ({}^\perp\binom{_A\mathcal I}{B\mbox{-}{\rm Mod}}, \ \binom{_A\mathcal I}{B\mbox{-}{\rm Mod}})$, one has \
${\rm T}_A X = \binom{X}{M\otimes_AX}_{1,0}\in {}^\perp\binom{_A\mathcal I}{B\mbox{-}{\rm Mod}}$, by Lemma \ref{extadj1}(1).

\vskip5pt

$\bullet$ \ For the cotorsion pair \ $L_{A\mbox{-}{\rm Mod}, \ _B\mathcal P} = (\binom{A\mbox{-}{\rm Mod}}{_B\mathcal P}, \ \binom{A\mbox{-}{\rm Mod}}{_B\mathcal P}^\perp)$, one has
\ $\binom{X}{0}_{0, 0}\in \binom{A\mbox{-}{\rm Mod}}{_B\mathcal P}$.

\vskip5pt

\noindent
In conclusion, the five cotorsion pairs $R_{_A\mathcal I, \ _B\mathcal I}, \ \ R_{A\mbox{-}{\rm Mod}, \ _B\mathcal I},  \ \ L_{_A\mathcal P, \ B\mbox{-}{\rm Mod}},  \ \ R_{_A\mathcal I, \ B\mbox{-}{\rm Mod}},  \ \ L_{A\mbox{-}{\rm Mod}, \ _B\mathcal P}$ are generally different from the flat cotorsion pair.

\vskip5pt
Finally,  \ for the cotorsion pair \ $L_{_A\mathcal P, \ _B\mathcal P}   =
(\left(\begin{smallmatrix}_A\mathcal P\\ _B\mathcal P\end {smallmatrix}\right), \ \left(\begin{smallmatrix}_A\mathcal P\\ _B\mathcal P\end {smallmatrix}\right)^\perp)$,
we take $\Lambda$ to be the Morita ring and $L = \binom{Ae_1}{Ae_1}_{\sigma, \sigma}$, as given in Example \ref{ie}. Then $L\in \left(\begin{smallmatrix}_A\mathcal P\\ _B\mathcal P\end {smallmatrix}\right)$.
But $L$ is not a flat $\Lambda$-module (otherwise, since $L$ is finitely generated, $L$ is projective, which is absurd).

\vskip5pt

This completes the proof.  \hfill $\square$

\subsection{Cotorsion pairs in Series II in Table 2 are pairwise generally different} Also, in the most cases, the eight cotorsion pairs in Series II in Table 2 are pairwise different.

\vskip5pt

\begin{lem}\label{different1} \ Let $\Lambda = \left(\begin{smallmatrix} A & N \\
	M & B\end{smallmatrix}\right)$ be a Morita ring  with $M\otimes_A N=0=N\otimes_BM$, \ $(\mathcal U, \ \mathcal X)$ and \ $(\mathcal U', \ \mathcal X')$ cotorsion pairs in $A\mbox{-}{\rm Mod}$, and  \ $(\mathcal V, \ \mathcal Y)$ and \ $(\mathcal V', \ \mathcal Y')$ cotorsion pairs in $B\mbox{-}{\rm Mod}$.
Then

\vskip5pt

$(1)$ \ $\Delta(\mathcal U, \ \mathcal V)=\Delta(\mathcal U', \ \mathcal V')$ if and only if \ $\mathcal U=\mathcal U'$ and \ $\mathcal V=\mathcal V'.$

\vskip5pt

$(2)$ \ $\nabla(\mathcal X, \ \mathcal Y) = \nabla(\mathcal X', \ \mathcal Y')$ if and only if \
$\mathcal X = \mathcal X'$ and \ $\mathcal Y = \mathcal Y'$.
\end{lem}

\begin{proof} \ $(1)$ \ This follows from the fact $${\rm T}_A U = \left(\begin{smallmatrix}U\\ M\otimes_AU\end {smallmatrix}\right)_{1,0} \in
\Delta(\mathcal U, \ \mathcal V), \ \ \forall \ U\in \mathcal U; \ \
\ \ \ {\rm T}_B V = \left(\begin{smallmatrix}N\otimes_BV \\ V\end {smallmatrix}\right)_{0,1} \in
\Delta(\mathcal U, \ \mathcal V), \ \ \forall \ V\in \mathcal V.$$

$(2)$ \  This follows from the fact \ ${\rm H}_A X = \left(\begin{smallmatrix} X \\ \Hom_A(N, X)\end{smallmatrix}\right)_{0, 1}
\in \nabla(\mathcal X, \ \mathcal Y), \ \ \forall \ X\in \mathcal X,$ and \ $
{\rm H}_BY = \left(\begin{smallmatrix} \Hom_B(M, Y) \\ Y\end{smallmatrix}\right)_{1, 0} \in
\nabla(\mathcal X, \ \mathcal Y), \ \ \forall \ Y\in \mathcal Y,$  here we use the second expression of $\Lambda$-modules. \end{proof}

\begin{prop}\label{different2}
Let $\Lambda = \left(\begin{smallmatrix} A & N \\
	M & B\end{smallmatrix}\right)$ be a Morita ring  with $M\otimes_A N=0=N\otimes_BM$.
Then the eight cotorsion pairs in {\rm Series II} in {\rm Table 2} are pairwise generally different.\end{prop}
\begin{proof}   \ All together there are $\binom{8}{2} = 28$ situations.

\vskip5pt

{\bf Step 1.} \ By Lemma \ref{different1}, the cotorsion pairs in {\rm Series II} in the same columns of Table 2 are pairwise different. This occupies $2\binom{4}{2} = 12$ situations.

\vskip5pt

{\bf Step 2.} \ $(_\Lambda\mathcal P, \ \Lambda\mbox{\rm-Mod})$
is generally different from all other cotorsion pairs in Series II in Table 2. This occupies $4$ situations.

\vskip5pt

In fact, taking $A$ to be a non semisimple ring and $N$ a non injective $A$-module. Then $\binom{N}{0}_{0,0}\in \Lambda$-Mod.
Since the map $0\longrightarrow \Hom_A(N,N)$ is not epic, it follows that
$$\left(\begin{smallmatrix}N\\ 0\end{smallmatrix}\right)_{0,0}\notin {\rm Epi}(\Lambda),  \ \ \ \left(\begin{smallmatrix}N\\ 0\end{smallmatrix}\right)_{0,0}\notin \nabla(A\mbox{\rm-Mod}, \ _B\mathcal I), \ \ \
\left(\begin{smallmatrix}N\\ 0\end{smallmatrix}\right)_{0,0}\notin\nabla(_A\mathcal I, \ B\mbox{\rm-Mod}), \ \ \ \left(\begin{smallmatrix}N\\ 0\end{smallmatrix}\right)_{0,0}\notin \ _\Lambda\mathcal I.$$
So $(_\Lambda\mathcal P, \ \Lambda\mbox{\rm-Mod})$ is generally different from \ \
$(^\perp{\rm Epi}(\Lambda), \ {\rm Epi}(\Lambda))$, \ \ \ $(^\perp\nabla(A\mbox{\rm-Mod}, \ _B\mathcal I), \ \nabla(A\mbox{\rm-Mod}, \ _B\mathcal I))$, \ \ \  $(^\perp\nabla(\mathcal I, \ B\mbox{\rm-Mod}), \ \nabla(\mathcal I, \ B\mbox{\rm-Mod}))$, \ and
\  $(\Lambda\mbox{\rm-Mod}, \ _\Lambda\mathcal I)$.

\vskip5pt

{\bf Step 3.} \ Similarly,  \ $(\Lambda\mbox{\rm-Mod}, \ _\Lambda\mathcal I)$
is generally different from other cotorsion pairs in Series II. This occupies $3$ situations.

\vskip5pt

{\bf Step 4.} \ Assume that  $M \neq 0\ne N$. It remains to show the following $9$ cases:
\begin{align*}& {\rm Epi}(\Lambda)\ne \Delta(_A\mathcal P, \ B\mbox{-}{\rm Mod})^\perp; \ \  \ \ \ \  {\rm Epi}(\Lambda)\ne \Delta (A\mbox{-}{\rm Mod}, \ _B\mathcal P)^\perp ;
\ \ \ \ \ \ \ \  {\rm Epi}(\Lambda)\ne {\rm Mon}(\Lambda)^\perp;
\\ &
\Delta(_A\mathcal P, \ B\mbox{-}{\rm Mod})^\perp \ne \nabla(A\mbox{-}{\rm Mod}, \ _B\mathcal I); \ \ \ \ \ \ \ \ \ \ \ \ \  \Delta(_A\mathcal P, \ B\mbox{-}{\rm Mod})^\perp \ne \nabla(_A\mathcal I, \ B\mbox{-}{\rm Mod});
\\ &\nabla(A\mbox{-}{\rm Mod}, \ _B\mathcal I) \ne \Delta(A\mbox{-}{\rm Mod}, \ _B\mathcal P)^\perp;
\ \ \ \ \ \ \ \ \ \ \ \ \
\nabla(A\mbox{-}{\rm Mod}, \ _B\mathcal I)\ne {\rm Mon}(\Lambda)^\perp; \\ &
\Delta(A\mbox{-}{\rm Mod}, \ _B\mathcal P)^\perp\ne \nabla(_A\mathcal I, \ B\mbox{-}{\rm Mod}); \ \ \ \ \ \ \ \ \ \ \ \ \
\nabla(_A\mathcal I, \ B\mbox{-}{\rm Mod})\ne {\rm Mon}(\Lambda)^\perp.
\end{align*}

First, we see the inequalities involving ${\rm Epi}(\Lambda) = \nabla(A\mbox{-}{\rm Mod}, \ B\mbox{-}{\rm Mod})$.
Let $_AI$ be the injective envelope of $_AN$. By Lemma \ref{extadj2}(1) one has
$${\rm Z}_AI= \left(\begin{smallmatrix}I\\ 0\end{smallmatrix}\right )_{0,0}\in \Delta(_A\mathcal P, \ B\mbox{-}{\rm Mod})^\perp\cap \Delta(A\mbox{-}{\rm Mod}, \ _B\mathcal P)^\perp \cap \Delta(A\mbox{-}{\rm Mod}, \ B\mbox{-}{\rm Mod})^\perp .$$
But $\widetilde{g}: 0\longrightarrow \Hom_A(N, I)$ is not epic, so $\binom{I}{0}_{0,0}\notin \nabla(A\mbox{-}{\rm Mod}, \ B\mbox{-}{\rm Mod}) = {\rm Epi}(\Lambda)$.
This shows ${\rm Epi}(\Lambda)\ne \Delta(_A\mathcal P, \ B\mbox{-}{\rm Mod})^\perp$, \ \ ${\rm Epi}(\Lambda)\ne \Delta(A\mbox{-}{\rm Mod}, \ _B\mathcal P)^\perp$ and \ ${\rm Epi}(\Lambda)\ne {\rm Mon}(\Lambda)^\perp$.

\vskip10pt

Next, we see the two inequalities involving  \ $\Delta(_A\mathcal P, \ B\mbox{-}{\rm Mod})^\perp.$
By Lemma \ref{extadj2}(1),
\ ${\rm Z}_AN=\binom{N}{0}_{0,0}\in \Delta(_A\mathcal P, \ B\mbox{-}{\rm Mod})^\perp$. But $\widetilde{g}: 0\longrightarrow \Hom_A(N,N)\ne 0$ is not epic, so $\binom{N}{0}_{0,0}$ is not in $\nabla(A\mbox{-}{\rm Mod}, \  _B\mathcal I)$ and $\nabla(_A\mathcal I, \ B\mbox{-}{\rm Mod})$.
This shows the two inequalities.

\vskip10pt

Next, to see \ $\nabla(A\mbox{-}{\rm Mod}, \  _B\mathcal I) \ne \Delta(A\mbox{-}{\rm Mod}, \  _B\mathcal P)^\perp$ and $\nabla(A\mbox{-}{\rm Mod}, \  _B\mathcal I)\ne \Delta (A\mbox{-}{\rm Mod}, \  B\mbox{-}{\rm Mod})^\perp$,
Let $_BJ$ be the injective envelope of $_BM$. By Lemma \ref{extadj2}(2),
\ ${\rm Z}_BJ=\binom{0}{J}_{0,0}$ is in $\Delta(A\mbox{-}{\rm Mod}, \  _B\mathcal P)^\perp$ and \ $\Delta(A\mbox{-}{\rm Mod}, \  B\mbox{-}{\rm Mod})^\perp$.
But $\widetilde{f}: 0\longrightarrow \Hom_B(M, J)$ is not epic, so $\binom{0}{J}_{0,0}\notin \nabla(A\mbox{-}{\rm Mod}, \  _B\mathcal I)$.

\vskip10pt

Finally,  to see the two inequalities involving \ $\nabla(_A\mathcal I, \ B\mbox{-}{\rm Mod})$.
Let $_AI$ be the injective envelope of $_AN$. By Lemma \ref{extadj2}(1),
\ ${\rm Z}_AI=\binom{I}{0}_{0,0}$ is in \ $\Delta(A\mbox{-}{\rm Mod}, \  _B\mathcal P)^\perp$ \ and \ $\Delta(A\mbox{-}{\rm Mod}, \  B\mbox{-}{\rm Mod})^\perp$.
But $\widetilde{g}: 0\longrightarrow \Hom_A(N, I)$ is not epic, so $\binom{I}{0}_{0,0}\notin \nabla(_A\mathcal I, \ B\mbox{-}{\rm Mod})$.
This shows \ $\Delta(A\mbox{-}{\rm Mod}, \  _B\mathcal P)^\perp\ne \nabla(_A\mathcal I, \ B\mbox{-}{\rm Mod})$
and \ $\nabla(_A\mathcal I, \ B\mbox{-}{\rm Mod})\ne {\rm Mon}(\Lambda)^\perp.$
\end{proof}

\subsection{``New" cotorsion pairs in Series II in Table 2}
In Series II of Table 2, taking off the projective cotorsion pair and
the injective one, the remaining six cotorsion pairs
are ``new".

\vskip5pt

\begin{prop}\label{newII} \ Let $\Lambda = \left(\begin{smallmatrix} A & N \\ M & B \end{smallmatrix}\right)$ be a Morita ring with $M\otimes_AN=0 = N\otimes_BM$.
Then all the six cotorsion pairs
\begin{align*} & ({\rm Mon}(\Lambda), \ {\rm Mon}(\Lambda)^\bot), \ \ \ \ \ \ \ \ \ \ \ \ \ \ \ \ \ \ \ \ \ \ \ \ \ \  (^\bot{\rm Epi}(\Lambda), \ {\rm Epi}(\Lambda))
\\ &(\Delta(_A\mathcal P, \ B\mbox{-}{\rm Mod}), \ \Delta(_A\mathcal P, \ B\mbox{-}{\rm Mod})^\bot),
\ \ \ \ \ \ (^\perp\nabla(A\mbox{-}{\rm Mod},  \ _B\mathcal I),  \ \nabla(A\mbox{-}{\rm Mod},  \ _B\mathcal I))
\\ &
(\Delta(A\mbox{-}{\rm Mod}, \ _B\mathcal P), \ \Delta(A\mbox{-}{\rm Mod}, \ _B\mathcal P)^\bot),
\ \ \ \ \ \  (^\bot\nabla(_A\mathcal I, \ B\mbox{-}{\rm Mod}),  \ \nabla(_A\mathcal I, \ B\mbox{-}{\rm Mod}))\end{align*} are ``new", in the sense of \ {\rm Definition \ref{new}}.
\end{prop}

\vskip5pt

To prove Proposition \ref{newII}, we first show

\vskip5pt

\begin{lem}\label{nongor3} \ Let $\Lambda = \left(\begin{smallmatrix} A & N \\ M & B \end{smallmatrix}\right)$ be a Morita ring which is an Artin algebra and a Gorenstein ring, with $M\otimes_AN=0 = N\otimes_BM$.
Then the cotorsion pairs
\ $({\rm Mon}(\Lambda), \ {\rm Mon}(\Lambda)^\bot)$ and $(^\bot{\rm Epi}(\Lambda), \ {\rm Epi}(\Lambda))$ are generally different from the Gorenstein-projective cotorsion pair
and the Gorenstein-injective one.
\end{lem}

\begin{proof} \ Take \ $\Lambda$ to be the Morita ring $\Lambda = \left(\begin{smallmatrix} A & N \\ N & A \end{smallmatrix}\right)$, constructed in Example \ref{ie}. Thus
 $A$ is the path algebra $k(1 \longrightarrow 2)$ with ${\rm char} \ k\ne 2$,  $N = Ae_2\otimes_ke_1A$,  and $N\otimes_AN = 0$. By Lemma \ref{agoralg},  $\Lambda$ is a Gorenstein algebra.

\vskip5pt

{\bf Claim 1.} \ $({\rm Mon}(\Lambda), \ {\rm Mon}(\Lambda)^\bot)$
is generally different from $({\rm GP}(\Lambda), \ _\Lambda\mathcal P^{<\infty})$.

\vskip5pt

In fact, $L = \left(\begin{smallmatrix}Ae_1\\ Ae_1\end{smallmatrix}\right)_{\sigma, \sigma}\in {\rm Mon}(\Lambda).$
By  {\bf Claim 2} in the proof of Lemma \ref{notgor1}, ${\rm proj.dim}_\Lambda L = 1.$ Thus $L$ is not Gorenstein-projective (otherwise $L$ is projective, which is absurd. Note that a Gorenstein-projective module of finite projective dimension is projective. See [EJ, 10.2.3]).
So $L\notin {\rm GP}(\Lambda)$. Thus ${\rm Mon}(\Lambda)\ne {\rm GP}(\Lambda)$, and hence
$$({\rm Mon}(\Lambda), \ {\rm Mon}(\Lambda)^\bot)
\ne ({\rm GP}(\Lambda), \ _\Lambda\mathcal P^{<\infty}).$$

\vskip5pt

{\bf Claim 2.} \ $({\rm Mon}(\Lambda), \ {\rm Mon}(\Lambda)^\bot)$
is generally different from  $(_\Lambda\mathcal P^{<\infty}, \ {\rm GI}(\Lambda))$.

\vskip5pt

In fact, the following $\Lambda$-projective resolution of $\left(\begin{smallmatrix}Ae_1\\ 0\end{smallmatrix}\right)_{0, 0}$
$$0\longrightarrow {\rm T}_BS_2 = \left(\begin{smallmatrix}0\\ S_2\end{smallmatrix}\right)_{0, 0}
\stackrel{\binom{0}{1}}\longrightarrow {\rm T}_A(Ae_1) = \left(\begin{smallmatrix} Ae_1\\ S_2\end{smallmatrix}\right)_{1, 0}
\stackrel{\left(\begin{smallmatrix} 1 \\ 0\end{smallmatrix}\right)}
\longrightarrow \left(\begin{smallmatrix}Ae_1\\ 0\end{smallmatrix}\right)_{0, 0}\longrightarrow 0$$
shows that $\left(\begin{smallmatrix}Ae_1\\ 0\end{smallmatrix}\right)_{0, 0}\in \  _\Lambda\mathcal P^{<\infty}$.
Since $N\otimes_AAe_1\cong S_2$, \ $\left(\begin{smallmatrix}Ae_1\\ 0\end{smallmatrix}\right)_{0, 0}\notin {\rm Mon}(\Lambda)$.
Thus \ ${\rm Mon}(\Lambda)\ne \  _\Lambda\mathcal P^{<\infty}$, and hence
$$({\rm Mon}(\Lambda), \ {\rm Mon}(\Lambda)^\bot)\ne (_\Lambda\mathcal P^{<\infty}, \ {\rm GI}(\Lambda)).$$

\vskip5pt

{\bf Claim 3.} \ $(^\bot{\rm Epi}(\Lambda), \ {\rm Epi}(\Lambda))$
is generally different from $({\rm GP}(\Lambda), \ _\Lambda\mathcal P^{<\infty})$.

\vskip5pt

In fact, by {\bf Claim 2}, $\left(\begin{smallmatrix}Ae_1\\ 0\end{smallmatrix}\right)_{0, 0}\in \  _\Lambda\mathcal P^{<\infty}$. Since $\Hom_A(N, Ae_1)\cong S_1 \ne 0$ (cf. Example\ref{ie}),
$\left(\begin{smallmatrix}Ae_1\\ 0\end{smallmatrix}\right)_{0, 0}\notin {\rm Epi}(\Lambda)$.
Thus
${\rm Epi}(\Lambda) \ne \  _\Lambda\mathcal P^{<\infty}$, and hence
$$(^\bot{\rm Epi}(\Lambda), \ {\rm Epi}(\Lambda))
\ne ({\rm GP}(\Lambda), \ _\Lambda\mathcal P^{<\infty}).$$

\vskip5pt

{\bf Claim 4.} \ $(^\bot{\rm Epi}(\Lambda), \ {\rm Epi}(\Lambda))$
is generally different from  $(_\Lambda\mathcal P^{<\infty}, \ {\rm GI}(\Lambda))$.

\vskip5pt

In fact, $L = \left(\begin{smallmatrix}Ae_1\\ Ae_1\end{smallmatrix}\right)_{\sigma, \sigma}\in {\rm Epic}(\Lambda)$ and $\Ext^1_\Lambda(L, L)\ne 0$ \ (cf. Example \ref{ie}).
So \ $L\notin \ ^\bot{\rm Epi}(\Lambda)$. However, $L\in \ _\Lambda\mathcal P^{<\infty}$.
Thus \ $^\bot{\rm Epi}(\Lambda)\ne \ _\Lambda\mathcal P^{<\infty}$, and hence
$$(^\bot{\rm Epi}(\Lambda), \ {\rm Epi}(\Lambda))
\ne (_\Lambda\mathcal P^{<\infty}, \ {\rm GI}(\Lambda)).$$

This completes the proof. \end{proof}

\vskip5pt

\begin{lem}\label{nongor4} \ Let $\Lambda = \left(\begin{smallmatrix} A & N \\ M & B \end{smallmatrix}\right)$ be a Morita ring which is an Artin algebra and a Gorenstein ring, with $M\otimes_AN=0 = N\otimes_BM$.
Then the cotorsion pairs
\begin{align*} & (\Delta(_A\mathcal P, \ B\mbox{-}{\rm Mod}), \ \Delta(_A\mathcal P, \ B\mbox{-}{\rm Mod})^\bot),
\ \ \ \ (^\perp\nabla(A\mbox{-}{\rm Mod},  \ _B\mathcal I),  \ \nabla(A\mbox{-}{\rm Mod},  \ _B\mathcal I))
\\ &
(\Delta(A\mbox{-}{\rm Mod}, \ _B\mathcal P), \ \Delta(A\mbox{-}{\rm Mod}, \ _B\mathcal P)^\bot),
\ \ \ \ (^\bot\nabla(_A\mathcal I, \ B\mbox{-}{\rm Mod}),  \ \nabla(_A\mathcal I, \ B\mbox{-}{\rm Mod}))\end{align*} are generally different from the Gorenstein-projective cotorsion pair
and the Gorenstein-injective one.
\end{lem}

\begin{proof} \ Choose rings $A$ and $B$, bimodules \ $_BM_A$ and $_AN_B$, such that

\vskip5pt

(i) \ \ \ $A$ and $B$ are quasi-Frobenius and not semisimple;

(ii) \ \ $_AN$ and $_BM$ are non-zero projective modules, and $M_A$ and $N_B$ are flat;

(iii) \ \ $M\otimes_A N = 0 = N\otimes_BM;$

(iv) \ \  $\Lambda$ is an Artin algebra.

\vskip5pt

By Remark \ref{examctp4}, such $\Lambda$'s always exist! By Theorem \ref{ctp4}, \ $\Lambda$ is a Gorenstein ring with ${\rm inj.dim} \Lambda \le 1$,  \ $({\rm GP}(\Lambda), \ \mathcal P^{\le 1}) = ({\rm Mon}(\Lambda), \ {\rm Mon}(\Lambda)^\bot)$,
\ and \ $(_\Lambda \mathcal P^{\le 1}, \ {\rm GI}(\Lambda)) = (^\bot{\rm Epi}(\Lambda), \ {\rm Epi}(\Lambda))$.

\vskip5pt

{\bf Claim 1.} \ $(\Delta(_A\mathcal P, \ B\mbox{-}{\rm Mod}), \ \Delta(_A\mathcal P, \ B\mbox{-}{\rm Mod})^\bot)$ and \ $(\Delta(A\mbox{-}{\rm Mod}, \ _B\mathcal P), \ \Delta(A\mbox{-}{\rm Mod}, \ _B\mathcal P)^\bot)$ are generally different from the Gorenstein-projective cotorsion pair.

\vskip5pt

In fact, since \ $A$ and $B$ are not semisimple, \ $A\mbox{-}{\rm Mod}\ne \ _A\mathcal P$ and
\ $B\mbox{-}{\rm Mod}\ne \ _B\mathcal P$. By Lemma \ref{different1},
\ $\Delta(_A\mathcal P, \ B\mbox{-}{\rm Mod})\ne \Delta(A\mbox{-}{\rm Mod}, \ B\mbox{-}{\rm Mod}) = {\rm Mon}(\Lambda)$, and
$\Delta(A\mbox{-}{\rm Mod}, \ _B\mathcal P) \ne {\rm Mon}(\Lambda)$.
Thus
$$(\Delta(_A\mathcal P, \ B\mbox{-}{\rm Mod}), \ \Delta(_A\mathcal P, \ B\mbox{-}{\rm Mod})^\bot)
\ne ({\rm Mon}(\Lambda), \ {\rm Mon}(\Lambda)^\bot) = ({\rm GP}(\Lambda), \ \mathcal P^{\le 1})$$
and
$$(\Delta(A\mbox{-}{\rm Mod}, \ _B\mathcal P), \ \Delta(A\mbox{-}{\rm Mod}, \ _B\mathcal P)^\bot)\ne ({\rm Mon}(\Lambda), \ {\rm Mon}(\Lambda)^\bot) = ({\rm GP}(\Lambda), \ \mathcal P^{\le 1}).$$

\vskip5pt

{\bf Claim 2.} \ $(^\perp\nabla(A\mbox{-}{\rm Mod},  \ _B\mathcal I),  \ \nabla(A\mbox{-}{\rm Mod},  \ _B\mathcal I))$
and \ $(^\bot\nabla(_A\mathcal I, \ B\mbox{-}{\rm Mod}),  \ \nabla(_A\mathcal I, \ B\mbox{-}{\rm Mod}))$ are generally different from the Gorenstein-projective cotorsion pair.

\vskip5pt

In fact, by Lemma \ref{extadj2}(3),
\ ${\rm Z}_AA=\binom{A}{0}_{0,0}\in \ ^\perp\nabla(A\mbox{-}{\rm Mod},  \ _B\mathcal I)\cap \ ^\bot\nabla(_A\mathcal I, \ B\mbox{-}{\rm Mod})$.
But $f: M\otimes_AA\longrightarrow 0$ is not monic, so $\binom{A}{0}_{0,0}\notin \Delta(A\mbox{-}{\rm Mod}, \ B\mbox{-}{\rm Mod}) = {\rm Mon}(\Lambda)$.
This shows \ $^\perp\nabla(A\mbox{-}{\rm Mod},  \ _B\mathcal I)\ne {\rm Mon}(\Lambda)$ and \ $^\bot\nabla(_A\mathcal I, \ B\mbox{-}{\rm Mod})\ne {\rm Mon}(\Lambda)$.
Thus
$$(^\perp\nabla(A\mbox{-}{\rm Mod},  \ _B\mathcal I),  \ \nabla(A\mbox{-}{\rm Mod},  \ _B\mathcal I))
\ne ({\rm Mon}(\Lambda), \ {\rm Mon}(\Lambda)^\bot) = ({\rm GP}(\Lambda), \ \mathcal P^{\le 1})$$
and
$$(^\bot\nabla(_A\mathcal I, \ B\mbox{-}{\rm Mod}),  \ \nabla(_A\mathcal I, \ B\mbox{-}{\rm Mod}))\ne ({\rm Mon}(\Lambda), \ {\rm Mon}(\Lambda)^\bot) = ({\rm GP}(\Lambda), \ \mathcal P^{\le 1}).$$

\vskip5pt

{\bf Claim 3.} \ $(\Delta(_A\mathcal P, \ B\mbox{-}{\rm Mod}), \ \Delta(_A\mathcal P, \ B\mbox{-}{\rm Mod})^\bot)$ and \ $(\Delta(A\mbox{-}{\rm Mod}, \ _B\mathcal P), \ \Delta(A\mbox{-}{\rm Mod}, \ _B\mathcal P)^\bot)$ are generally different from the Gorenstein-injective cotorsion pair.

\vskip5pt

In fact, let $I$ be the injective envelope of \ $_AN$. By Lemma \ref{extadj2}(1),
\ ${\rm Z}_AI=\binom{I}{0}_{0,0}\in \Delta(_A\mathcal P, \ B\mbox{-}{\rm Mod})^\perp\cap \Delta(A\mbox{-}{\rm Mod}, \ _B\mathcal P)^\perp$.
But $\widetilde{g}: 0\longrightarrow \Hom_A(N, I)$ is not epic, so $\binom{I}{0}_{0,0}\notin \nabla(A\mbox{-}{\rm Mod}, \ B\mbox{-}{\rm Mod}) = {\rm Epi}(\Lambda)$.
This shows $\Delta(_A\mathcal P, \ B\mbox{-}{\rm Mod})^\perp\ne {\rm Epi}(\Lambda)$ and \ $\Delta(_A\mathcal P, \ B\mbox{-}{\rm Mod})^\perp\ne {\rm Epi}(\Lambda)$.
Thus
$$(\Delta(_A\mathcal P, \ B\mbox{-}{\rm Mod}), \ \Delta(_A\mathcal P, \ B\mbox{-}{\rm Mod})^\bot)
\ne (^\bot{\rm Epi}(\Lambda), \ {\rm Epi}(\Lambda)) = (_\Lambda \mathcal P^{\le 1}, \ {\rm GI}(\Lambda))$$
and
$$(\Delta(_A\mathcal P, \ B\mbox{-}{\rm Mod}), \ \Delta(_A\mathcal P, \ B\mbox{-}{\rm Mod})^\bot)\ne (^\bot{\rm Epi}(\Lambda), \ {\rm Epi}(\Lambda)) = (_\Lambda \mathcal P^{\le 1}, \ {\rm GI}(\Lambda)).$$

\vskip5pt

{\bf Claim 4.} \ $(^\perp\nabla(A\mbox{-}{\rm Mod},  \ _B\mathcal I),  \ \nabla(A\mbox{-}{\rm Mod},  \ _B\mathcal I))$ and \ $(^\bot\nabla(_A\mathcal I, \ B\mbox{-}{\rm Mod}),  \ \nabla(_A\mathcal I, \ B\mbox{-}{\rm Mod}))$ are generally different from the Gorenstein-injective cotorsion pair.

\vskip5pt

In fact, since \ $A$ and \ $B$ are not semisimple, \ $_A\mathcal I \ne A\mbox{-}{\rm Mod}$ \ and
\ $_B\mathcal I\ne B\mbox{-}{\rm Mod}$. By Lemma \ref{different1},
\ $\nabla(_A\mathcal I, \ B\mbox{-}{\rm Mod})\ne \nabla(A\mbox{-}{\rm Mod}, \ B\mbox{-}{\rm Mod}) = {\rm Epi}(\Lambda)$
\ and \ $\nabla(A\mbox{-}{\rm Mod},  \ _B\mathcal I)\ne {\rm Epi}(\Lambda)$.
Thus
$$(^\bot\nabla(_A\mathcal I, \ B\mbox{-}{\rm Mod}),  \ \nabla(_A\mathcal I, \ B\mbox{-}{\rm Mod}))\ne (^\bot{\rm Epi}(\Lambda), \ {\rm Epi}(\Lambda)) = (_\Lambda \mathcal P^{\le 1}, \ {\rm GI}(\Lambda))$$
and $$(^\perp\nabla(A\mbox{-}{\rm Mod},  \ _B\mathcal I),  \ \nabla(A\mbox{-}{\rm Mod},  \ _B\mathcal I))
\ne (^\bot{\rm Epi}(\Lambda), \ {\rm Epi}(\Lambda)) = (_\Lambda \mathcal P^{\le 1}, \ {\rm GI}(\Lambda)).$$

This completes the proof. \end{proof}

\vskip5pt

\noindent {\bf Proof of Proposition \ref{newII}.} \ By Proposition \ref{different2}, these six cotorsion pairs
are generally different from the projective cotorsion pair and the injective one.
By Lemmas \ref{nongor3} and \ref{nongor4}, they are generally different from the Gorenstein-projective cotorsion pair
and the Gorenstein-injective one. It remains to show that they are generally different from the flat cotorsion pair.

\vskip5pt

In fact, choose rings $A$ and $B$ such that they admit non flat modules (such a ring of course exists! See the proof of Proposition \ref{newI}). Taking non flat modules \ $_AX$ and  $_BY$, by Lemma \ref{flat}, all the following  $\Lambda$-modules are not flat:
$$\left(\begin{smallmatrix} X\\ 0\end{smallmatrix}\right)_{0, 0}, \ \ \left(\begin{smallmatrix} 0\\ Y\end{smallmatrix}\right)_{0,0},
\ \ {\rm T}_A X = \left(\begin{smallmatrix}X\\ M\otimes_AX\end{smallmatrix}\right)_{1,0}, \ \ {\rm T}_BY = \left(\begin{smallmatrix} N\otimes_BY\\ Y\end{smallmatrix}\right)_{0,1}.
$$

\vskip5pt

\noindent However,

\vskip5pt

$\bullet$ \ For the cotorsion pair \ $({\rm Mon}(\Lambda), \ {\rm Mon}(\Lambda)^\bot)$ , one has \
${\rm T}_A X = \binom{X}{M\otimes_AX}_{1,0}\in {\rm Mon}(\Lambda)$.

\vskip5pt

$\bullet$ \ For the cotorsion pair \ $(\Delta(_A\mathcal P, \ B\mbox{-}{\rm Mod}), \ \Delta(_A\mathcal P, \ B\mbox{-}{\rm Mod})^\perp)$,
one has \ ${\rm T}_BY = \binom{N\otimes_BY}{Y}_{0,1}\in \Delta(_A\mathcal P, \ B\mbox{-}{\rm Mod})$.

\vskip5pt

$\bullet$ \ For the cotorsion pair \ $(^\perp\nabla(A\mbox{-}{\rm Mod},  \ _B\mathcal I),  \ \nabla(A\mbox{-}{\rm Mod},  \ _B\mathcal I))$, one has
\ $\binom{0}{Y}_{0,0}\in \ ^\perp\nabla(A\mbox{-}{\rm Mod}, \ _B\mathcal I)$, by Lemma \ref{extadj2}(4).

\vskip5pt

$\bullet$ \ For the cotorsion pair \ $(\Delta(A\mbox{-}{\rm Mod}, \ _B\mathcal P), \ \Delta(A\mbox{-}{\rm Mod}, \ _B\mathcal P)^\bot)$, one has \
${\rm T}_A X = \binom{X}{M\otimes_AX}_{1,0}\in \Delta(A\mbox{-}{\rm Mod}, \ _B\mathcal P)$.

\vskip5pt

$\bullet$ \ For the cotorsion pair \ $(^\bot\nabla(_A\mathcal I, \ B\mbox{-}{\rm Mod}),  \ \nabla(_A\mathcal I, \ B\mbox{-}{\rm Mod}))$, one has
\ $\binom{X}{0}_{0, 0}\in \ ^\bot\nabla(_A\mathcal I, \ B\mbox{-}{\rm Mod})$, by Lemma \ref{extadj2}(3).

\vskip5pt

In conclusion,  the five cotorsion pairs are different from the flat cotorsion pair.

\vskip5pt

Finally, to see $(^\bot{\rm Epi}(\Lambda), \ {\rm Epi}(\Lambda))$ is generally different from the flat cotorsion pair, choose a ring $A$ such that $A$ admits a flat (left) module which is not projective.

(For example, the ring $\Bbb Z$ of integers has a flat module $_{\Bbb Z}\Bbb Q$, but $_{\Bbb Z}\Bbb Q$ is not projective, or equivalently, $_{\Bbb Z}\Bbb Q$ is not free.)

Let
$\Lambda = \left(\begin{smallmatrix} A & 0 \\ 0 & A \end{smallmatrix}\right) = A\times A$. Then ${\rm Epi}(\Lambda) = \Lambda\mbox{-}{\rm Mod}$, and hence
$^\bot{\rm Epi}(\Lambda) = \ _\Lambda\mathcal P$. By the choice of $A$,  \ $^\bot{\rm Epi}(\Lambda) = \ _\Lambda\mathcal P$ is strictly contained in $_\Lambda{\rm F}$, the class of flat $\Lambda$-modules. It follows that
$(^\bot{\rm Epi}(\Lambda), \ {\rm Epi}(\Lambda))$ is generally different from the flat cotorsion pair. \hfill $\square$

\section{\bf Abelian model structures on Morita rings}

Based on results in the previous sections, we will see how abelian model structures on $A$-Mod and $B$-Mod induce abelian model structures on Morita rings; and
we will see that all these abelian model structures obtained on Morita rings are pairwise generally different, and they are generally different from the six well-known abelian model structures (cf. Proposition \ref{newmodel}).

\subsection{Cofibrantly generated Hovey triples in Morita rings}

Let $R$ be a ring. Recall that a Hovey triple $(\mathcal C, \mathcal F, \mathcal W)$  in  $R\mbox{-}{\rm Mod}$ is cofibrantly generated,
if both the cotorsion pairs $(\mathcal C\cap \mathcal W, \mathcal F)$ and $(\mathcal C, \mathcal F\cap\mathcal W)$ are cogenerated by sets.
If a model structure on $R\mbox{-}{\rm Mod}$ is clear in context, we write Quillen's homotopy category simply as ${\rm Ho}(R)$.

\vskip5pt

\begin{thm}\label{cofibrantlygenHtriple} \ Let \ $\Lambda=\left(\begin{smallmatrix} A & N \\
M & B\end{smallmatrix}\right)$ be a Morita ring with  $M\otimes_AN=0 = N\otimes_BM$,   \ \ $(\mathcal U', \ \mathcal X, \ \mathcal W_1)$ and \ $(\mathcal V', \ \mathcal Y, \ \mathcal W_2)$ cofibrantly generated
Hovey triples in $A\mbox{-}{\rm Mod}$ and  $B\mbox{-}{\rm Mod}$, respectively.

\vskip5pt

$(1)$ \ Suppose that  \ ${\rm Tor}^A_1(M, \ \mathcal U') = 0 = {\rm Tor}^B_1(N, \ \mathcal V')$,  \  $M\otimes_A\mathcal U' \subseteq \mathcal Y\cap \mathcal W_2$ and \ $N\otimes_B\mathcal V' \subseteq \mathcal X\cap \mathcal W_1.$ Then
$$({\rm T}_A(\mathcal U')\oplus {\rm T}_B(\mathcal V'), \ \left(\begin{smallmatrix} \mathcal X \\ \mathcal Y\end{smallmatrix}\right), \ \left(\begin{smallmatrix} \mathcal W_1 \\ \mathcal W_2\end{smallmatrix}\right))$$
is a cofibrantly generated Hovey triple in $\Lambda\mbox{-}{\rm Mod};$ and it is hereditary with \ ${\rm Ho}(\Lambda) \cong   {\rm Ho}(A)\oplus {\rm Ho}(B),$ provided that  \ $(\mathcal U',  \mathcal X,  \mathcal W_1)$ and \ $(\mathcal V',  \mathcal Y,  \mathcal W_2)$ are hereditary.

\vskip5pt

$(2)$ \ Suppose that   ${\rm Ext}_B^1(M,  \mathcal Y) = 0 = {\rm Ext}_A^1(N,  \mathcal X)$, $\Hom_B(M, \mathcal Y) \subseteq \mathcal U'\cap \mathcal W_1$ and  $\Hom_A(N, \mathcal X) \subseteq \mathcal V'\cap \mathcal W_2$. Then
$$(\left(\begin{smallmatrix} \mathcal U' \\ \mathcal V'\end{smallmatrix}\right),  \ {\rm H}_A(\mathcal X)\oplus {\rm H}_B(\mathcal Y), \ \left(\begin{smallmatrix} \mathcal W_1 \\ \mathcal W_2\end{smallmatrix}\right))$$
is a cofibrantly generated
Hovey triple$;$ and it is hereditary with  ${\rm Ho}(\Lambda) \cong   {\rm Ho}(A)\oplus {\rm Ho}(B)$, provided that $(\mathcal U', \mathcal X,  \mathcal W_1)$ and  $(\mathcal V',  \mathcal Y, \mathcal W_2)$ are hereditary.
\end{thm}

\noindent{\bf Proof.} \  Put \ $\mathcal U: = \mathcal U'\cap \mathcal W_1,  \ \ \mathcal X': = \mathcal X\cap \mathcal W_1, \ \ \mathcal V: = \mathcal V'\cap \mathcal W_2, \ \ \mathcal Y':= \mathcal Y\cap \mathcal W_2.$

\vskip5pt

Since \ $(\mathcal U', \ \mathcal X, \ \mathcal W_1)$ is a cofibrantly generated Hovey triple in $A$-{\rm Mod},
\ $(\mathcal U, \ \mathcal X)$ and $(\mathcal U', \ \mathcal X')$ are cotorsion pairs in $A$-{\rm Mod}, cogenerated by, say,
set  $S_1$ and set $S_1'$, respectively. \ Similarly,  \ $(\mathcal V, \ \mathcal Y)$ and $(\mathcal V', \ \mathcal Y')$ are cotorsion pairs in $B$-{\rm Mod}, cogenerated by, say,  set $S_2$ and set $S_2'$, respectively.

\vskip5pt

(1) \ Since \ ${\rm Tor}^A_1(M, \ \mathcal U) \subseteq {\rm Tor}^A_1(M, \ \mathcal U') = 0$
and \ ${\rm Tor}^B_1(N, \ \mathcal V) \subseteq {\rm Tor}^B_1(N, \ \mathcal V') = 0$,
it follows from Theorem \ref{ctp1}(1) that \ $(^\perp\binom{\mathcal X}{\mathcal Y}, \ \binom{\mathcal X}{\mathcal Y})$ is a cotorsion pair in $\Lambda$-Mod; and it is  cogenerated by set ${\rm T}_A(S_1)\oplus {\rm T}_B(S_2)$,  by Proposition \ref{generatingcomplete}(1).

\vskip5pt

Since \ $M\otimes_A\mathcal U \subseteq M\otimes_A\mathcal U' \subseteq \mathcal Y$ \  and \
$N\otimes_B\mathcal V \subseteq N\otimes_B\mathcal V' \subseteq \mathcal X,$  by Theorem \ref{identify1}(1), \ $^\perp\binom{\mathcal X}{\mathcal Y} = {\rm T}_A(\mathcal U)\oplus {\rm T}_B(\mathcal V)$.
Thus, \ $({\rm T}_A(\mathcal U)\oplus {\rm T}_B(\mathcal V), \ \left(\begin{smallmatrix} \mathcal X \\ \mathcal Y\end{smallmatrix}\right))$ is
a cotorsion pair, cogenerated by set ${\rm T}_A(S_1)\oplus {\rm T}_B(S_2)$.

\vskip5pt

Similarly, \ $({\rm T}_A(\mathcal U')\oplus {\rm T}_B(\mathcal V'), \ \left(\begin{smallmatrix} \mathcal X' \\ \mathcal Y'\end{smallmatrix}\right))$ is a cotorsion pair,
cogenerated by set ${\rm T}_A(S'_1)\oplus {\rm T}_B(S'_2)$.

\vskip5pt

Since \
$M\otimes_A\mathcal U' \subseteq \mathcal W_2$ \ and \ $N\otimes_B\mathcal V'\subseteq \mathcal W_1$, one has
$$({\rm T}_A(\mathcal U')\oplus {\rm T}_B(\mathcal V'))\cap \left(\begin{smallmatrix} \mathcal W_1 \\ \mathcal W_2\end{smallmatrix}\right)
= {\rm T}_A(\mathcal U'\cap\mathcal W_1)\oplus {\rm T}_B(\mathcal V'\cap \mathcal W_2) = {\rm T}_A(\mathcal U)\oplus {\rm T}_B(\mathcal V).$$
Also,  \ $\left(\begin{smallmatrix} \mathcal X \\ \mathcal Y\end{smallmatrix}\right)\cap \left(\begin{smallmatrix} \mathcal W_1 \\ \mathcal W_2\end{smallmatrix}\right)
= \left(\begin{smallmatrix} \mathcal X\cap \mathcal W_1 \\ \mathcal Y\cap \mathcal W_2\end{smallmatrix}\right) = \left(\begin{smallmatrix} \mathcal X' \\ \mathcal Y'\end{smallmatrix}\right).$
Since \ $\mathcal W_1$ and \ $\mathcal W_2$ are thick, \ $\left(\begin{smallmatrix} \mathcal W_1 \\ \mathcal W_2\end{smallmatrix}\right)$
is thick.
Thus
$$({\rm T}_A(\mathcal U')\oplus {\rm T}_B(\mathcal V'), \ \left(\begin{smallmatrix} \mathcal X\\ \mathcal Y\end{smallmatrix}\right), \ \left(\begin{smallmatrix} \mathcal W_1 \\ \mathcal W_2\end{smallmatrix}\right))$$
is a cofibrantly generated Hovey triple.

\vskip5pt

If \ $(\mathcal U',  \mathcal X,  \mathcal W_1)$ and \ $(\mathcal V',  \mathcal Y,  \mathcal W_2)$ are hereditary Hovey triples, then so is \ $({\rm T}_A(\mathcal U')\oplus {\rm T}_B(\mathcal V'), \ \left(\begin{smallmatrix} \mathcal X \\ \mathcal Y\end{smallmatrix}\right), \ \left(\begin{smallmatrix} \mathcal W_1 \\ \mathcal W_2\end{smallmatrix}\right))$. Since \ $M\otimes_A\mathcal U'\subseteq \mathcal Y\cap \mathcal W_2$ and \ $N\otimes_B\mathcal V'\subseteq \mathcal X\cap \mathcal W_1,$
by  Theorem \ref{Ho} one has
\begin{align*}{\rm Ho}(\Lambda)&
\cong (({\rm T}_A(\mathcal U')\oplus {\rm T}_B(\mathcal V'))\cap \left(\begin{smallmatrix} \mathcal X \\ \mathcal Y\end{smallmatrix}\right))/(({\rm T}_A(\mathcal U')\oplus {\rm T}_B(\mathcal V'))
\cap \left(\begin{smallmatrix} \mathcal X\cap \mathcal W_1 \\ \mathcal Y\cap\mathcal W_2\end{smallmatrix}\right))
\\ & \cong  ({\rm T}_A(\mathcal U'\cap \mathcal X)\oplus {\rm T}_B(\mathcal V'\cap\mathcal Y))/(({\rm T}_A(\mathcal U\cap\mathcal X)\oplus {\rm T}_B(\mathcal V\cap \mathcal Y))
\\ & \cong [{\rm T}_A(\mathcal U'\cap\mathcal X)/ {\rm T}_A(\mathcal U\cap \mathcal X)]\oplus [{\rm T}_B(\mathcal V'\cap\mathcal Y)/ {\rm T}_B(\mathcal V\cap \mathcal Y)]
\\ & \cong [(\mathcal U'\cap \mathcal X)/(\mathcal U\cap \mathcal X)]\oplus [(\mathcal V'\cap \mathcal Y)/(\mathcal V\cap \mathcal Y)]
\\& =   {\rm Ho}(A)\oplus {\rm Ho}(B).\end{align*}

\vskip5pt

$(2)$  Since \ ${\rm Ext}_A^1(M, \ \mathcal Y') \subseteq {\rm Ext}_A^1(M, \ \mathcal Y) = 0$
and \ ${\rm Ext}_B^1(N, \ \mathcal X') \subseteq {\rm Ext}_B^1(N, \ \mathcal X) = 0$, by Theorem \ref{ctp1}(2), \ $(\binom{\mathcal U'}{\mathcal V'}, \ \binom{\mathcal U'}{\mathcal V'}^\perp)$ is a cotorsion pair in $\Lambda$-Mod.

\vskip5pt

Since \ ${\rm Hom}_B(M,  \mathcal Y') \subseteq {\rm Hom}_B(M,  \mathcal Y) \subseteq \mathcal U'$ \  and \
${\rm Hom}_A(N,  \mathcal X') \subseteq {\rm Hom}_A(N,  \mathcal X) \subseteq \mathcal V',$  by Theorem \ref{identify1}(2) one has
$$(\left(\begin{smallmatrix} \mathcal U' \\ \mathcal V'\end{smallmatrix}\right), \ \left(\begin{smallmatrix} \mathcal U' \\ \mathcal V'\end{smallmatrix}\right)^\perp)= (^\perp\nabla(\mathcal X', \ \mathcal Y'), \ \nabla(\mathcal X', \ \mathcal Y'))$$
and  $\binom{\mathcal U'}{\mathcal V'}^\perp  = {\rm H}_A(\mathcal X')\oplus {\rm H}_B(\mathcal Y')$. By Proposition \ref{generatingcomplete}(2), $(^\perp\nabla(\mathcal X', \ \mathcal Y'), \ \nabla(\mathcal X', \ \mathcal Y'))$
is cogenerated by set ${\rm Z}_A(S'_1)\oplus {\rm Z}_B(S'_2)$.
Thus, \ $(\left(\begin{smallmatrix} \mathcal U' \\ \mathcal V'\end{smallmatrix}\right), \ \ {\rm H}_A(\mathcal X')\oplus {\rm H}_B(\mathcal Y'))$ is
a cotorsion pair, cogenerated by set ${\rm Z}_A(S'_1)\oplus {\rm Z}_B(S'_2)$.

\vskip5pt

Similarly, \ $(\left(\begin{smallmatrix} \mathcal U \\ \mathcal V\end{smallmatrix}\right), \ \ {\rm H}_A(\mathcal X)\oplus {\rm H}_B(\mathcal Y))$ is
a cotorsion pair, cogenerated by set ${\rm Z}_A(S_1)\oplus {\rm Z}_B(S_2)$.

\vskip5pt

Note that  \ $\left(\begin{smallmatrix} \mathcal U' \\ \mathcal V'\end{smallmatrix}\right)\cap \left(\begin{smallmatrix} \mathcal W_1 \\ \mathcal W_2\end{smallmatrix}\right)
= \left(\begin{smallmatrix} \mathcal U'\cap \mathcal W_1 \\ \mathcal V'\cap \mathcal W_2\end{smallmatrix}\right) = \left(\begin{smallmatrix} \mathcal U \\ \mathcal V\end{smallmatrix}\right).$
Since \
${\rm Hom}_A(N, \mathcal X) \subseteq \mathcal W_2$ \ and \ ${\rm Hom}_B(M, \mathcal Y) \subseteq \mathcal W_1$, one has
$$({\rm H}_A(\mathcal X)\oplus {\rm H}_B(\mathcal Y))\cap \left(\begin{smallmatrix} \mathcal W_1 \\ \mathcal W_2\end{smallmatrix}\right)
= {\rm H}_A(\mathcal X\cap\mathcal W_1)\oplus {\rm H}_B(\mathcal Y\cap \mathcal W_2) = {\rm H}_A(\mathcal X')\oplus {\rm H}_B(\mathcal Y').$$

Since \ $\mathcal W_1$ and \ $\mathcal W_2$ are thick, \ $\left(\begin{smallmatrix} \mathcal W_1 \\ \mathcal W_2\end{smallmatrix}\right)$
is thick.
Thus
$$(\left(\begin{smallmatrix} \mathcal U'\\ \mathcal V'\end{smallmatrix}\right),  \ {\rm H}_A(\mathcal X)\oplus {\rm H}_B(\mathcal Y),  \ \left(\begin{smallmatrix} \mathcal W_1 \\ \mathcal W_2\end{smallmatrix}\right))$$
is a cofibrantly generated Hovey triple.

\vskip5pt

If \ $(\mathcal U',  \mathcal X,  \mathcal W_1)$ and \ $(\mathcal V',  \mathcal Y,  \mathcal W_2)$ are hereditary, then  \ $(\left(\begin{smallmatrix} \mathcal U'\\ \mathcal V'\end{smallmatrix}\right),  \ {\rm H}_A(\mathcal X)\oplus {\rm H}_B(\mathcal Y),  \ \left(\begin{smallmatrix} \mathcal W_1 \\ \mathcal W_2\end{smallmatrix}\right))$
is  hereditary. Since \ ${\rm Hom}_A(N, \mathcal X) \subseteq \mathcal V'$ and  \
${\rm Hom}_B(M, \mathcal Y) \subseteq \mathcal U',$
by  Theorem \ref{Ho} one has
\begin{align*}\ \ \ \ \ \ \ \ \ \  \ \ \ \ \  \ \ \ \ \ \ {\rm Ho}(\Lambda)&
\cong (\left(\begin{smallmatrix} \mathcal U' \\ \mathcal V'\end{smallmatrix}\right)\cap ({\rm H}_A(\mathcal X)\oplus {\rm H}_B(\mathcal Y)))/(\left(\begin{smallmatrix} \mathcal U \\ \mathcal V\end{smallmatrix}\right)\cap ({\rm H}_A(\mathcal X)\oplus {\rm H}_B(\mathcal Y)))
\\ & \cong  ({\rm H}_A(\mathcal U'\cap \mathcal X)\oplus {\rm H}_B(\mathcal V'\cap\mathcal Y))/(({\rm H}_A(\mathcal U\cap\mathcal X)\oplus {\rm H}_B(\mathcal V\cap \mathcal Y))
\\ & \cong [{\rm H}_A(\mathcal U'\cap\mathcal X)/ {\rm H}_A(\mathcal U\cap \mathcal X)]\oplus [{\rm H}_B(\mathcal V'\cap\mathcal Y)/ {\rm H}_B(\mathcal V\cap \mathcal Y)]
\\ & \cong [(\mathcal U'\cap \mathcal X)/(\mathcal U\cap \mathcal X)]\oplus [(\mathcal V'\cap \mathcal Y)/(\mathcal V\cap \mathcal Y)]
\\& =   {\rm Ho}(A)\oplus {\rm Ho}(B). \ \ \ \ \ \ \ \ \ \  \ \ \ \ \ \  \ \ \ \ \  \ \ \ \ \ \  \ \ \ \ \  \ \ \ \ \ \  \ \  \ \ \ \ \  \ \ \ \ \ \ \ \ \  \ \ \ \ \ \  \ \ \ \ \  \ \ \ \ \ \ \square \end{align*}

From Theorem \ref{cofibrantlygenHtriple} and its proof, one easily sees the following.

\begin{cor}\label{cofibrantlygenGHtriple} \ Let $\Lambda=\left(\begin{smallmatrix} A & N \\
M & B\end{smallmatrix}\right)$ be a Morita ring with  $M\otimes_AN=0 = N\otimes_BM$.
Let   $(\mathcal U,  \mathcal X)$ and $(\mathcal U',  \mathcal X')$ be
compatible hereditary cotorsion pairs in $A$-{\rm Mod}, cogenerated by sets $S_1$ and $S_1'$, respectively, with Gillespie-Hovey triple  $(\mathcal U',  \mathcal X,  \mathcal W_1)$.
Let  $(\mathcal V,  \mathcal Y)$ and $(\mathcal V',  \mathcal Y')$ be
compatible hereditary cotorsion pairs in $B$-{\rm Mod}, cogenerated by sets $S_2$ and $S_2'$, respectively, with Gillespie-Hovey triple  $(\mathcal V',  \mathcal Y,  \mathcal W_2)$.

\vskip5pt

$(1)$ \ Assume that   ${\rm Tor}^A_1(M, \ \mathcal U') = 0 = {\rm Tor}^B_1(N, \ \mathcal V'),  \ \ M\otimes_A\mathcal U' \subseteq \mathcal Y', \
N\otimes_B\mathcal V' \subseteq \mathcal X'.$ Then
$$({\rm T}_A(\mathcal U)\oplus {\rm T}_B(\mathcal V), \ \left(\begin{smallmatrix} \mathcal X \\ \mathcal Y\end{smallmatrix}\right))
\ \ \ \mbox{and} \ \ \
({\rm T}_A(\mathcal U')\oplus {\rm T}_B(\mathcal V'), \ \left(\begin{smallmatrix} \mathcal X' \\ \mathcal Y'\end{smallmatrix}\right))$$
are compatible complete hereditary  cotorsion pairs in \ $\Lambda\mbox{-}{\rm Mod}$, with Gillespie-Hovey triple
$$({\rm T}_A(\mathcal U')\oplus {\rm T}_B(\mathcal V'), \ \left(\begin{smallmatrix} \mathcal X \\ \mathcal Y\end{smallmatrix}\right), \ \left(\begin{smallmatrix} \mathcal W_1 \\ \mathcal W_2\end{smallmatrix}\right))$$
and \ ${\rm Ho}(\Lambda) \cong
{\rm Ho}(A)\oplus {\rm Ho}(B).$

\vskip5pt

$(2)$ \ Assume that   ${\rm Ext}_B^1(M,  \mathcal Y) = 0 = {\rm Ext}_A^1(N,  \mathcal X)$, $\Hom_B(M, \mathcal Y) \subseteq \mathcal U$ and  $\Hom_A(N, \mathcal X) \subseteq \mathcal V$. Then
$$(\left(\begin{smallmatrix} \mathcal U \\ \mathcal V\end{smallmatrix}\right), \ {\rm H}_A(\mathcal X)\oplus {\rm H}_B(\mathcal Y))
\ \ \ \mbox{and} \ \ \
(\left(\begin{smallmatrix} \mathcal U' \\ \mathcal V'\end{smallmatrix}\right), \ {\rm H}_A(\mathcal X')\oplus {\rm H}_B(\mathcal Y'))$$
are compatible complete hereditary  cotorsion pairs in \ $\Lambda\mbox{-}{\rm Mod}$, with Gillespie-Hovey triple
$$(\left(\begin{smallmatrix} \mathcal U' \\ \mathcal V'\end{smallmatrix}\right), \ {\rm H}_A(\mathcal X)\oplus {\rm H}_B(\mathcal Y), \ \left(\begin{smallmatrix} \mathcal W_1 \\ \mathcal W_2\end{smallmatrix}\right))$$
and \ ${\rm Ho}(\Lambda) \cong {\rm Ho}(A)\oplus {\rm Ho}(B).$
\end{cor}

\subsection{Hovey triples in Morita rings} We stress that, all the results in the rest of this section are not consequences of Theorem \ref{cofibrantlygenHtriple}, or Corollary \ref{cofibrantlygenGHtriple},
since they need module-theoretical arguments on the completeness of cotorsion pairs in Morita rings, developed in Section 5. Thus, all these results are new even for $M = 0$ or $N = 0$.

\begin{thm}\label{Htriple1} \ Let $\Lambda=\left(\begin{smallmatrix} A & N \\
M & B\end{smallmatrix}\right)$ be a Morita ring with  $M\otimes_AN=0 = N\otimes_BM$.
Let \ $(\mathcal V', \ \mathcal Y, \ \mathcal W)$ be a Hovey triple in $B$\mbox{-}{\rm Mod}. Suppose that  \ $N_B$ is flat and  $_BM$ is projective.

\vskip5pt

$(1)$  \  If \ $M\otimes_A\mathcal P \subseteq \mathcal Y\cap \mathcal W$, then
$$({\rm T}_A(_A\mathcal P)\oplus {\rm T}_B(\mathcal V'), \ \left(\begin{smallmatrix} A\text{\rm\rm-Mod} \\ \mathcal Y\end{smallmatrix}\right), \ \left(\begin{smallmatrix} A\text{\rm\rm-Mod} \\ \mathcal W\end{smallmatrix}\right))$$
is a Hovey triple in $\Lambda$\mbox{-}{\rm Mod}$;$ and it is hereditary with \ ${\rm Ho}(\Lambda) \cong {\rm Ho}(B)$, provided that
 \ $(\mathcal V',  \mathcal Y,  \mathcal W)$ is hereditary.

\vskip5pt

$(2)$  \ If \ $\Hom_A(N, \ _A\mathcal I)\subseteq \mathcal V'\cap \mathcal W$, then
$$(\left(\begin{smallmatrix} A\mbox{-}{\rm Mod} \\ \mathcal V'\end{smallmatrix}\right), \ \ {\rm H}_A(_A\mathcal I)\oplus {\rm H}_B(\mathcal Y), \ \ \left(\begin{smallmatrix}A\text{\rm\rm-Mod}\\ \mathcal W\end{smallmatrix}\right))$$
is a Hovey triple in $\Lambda$\mbox{-}{\rm Mod}$;$ and it is hereditary with \ ${\rm Ho}(\Lambda) \cong {\rm Ho}(B)$, provided that
 \ $(\mathcal V',  \mathcal Y,  \mathcal W)$ is hereditary.
\end{thm}
\begin{proof} \  Put \ $\mathcal V: = \mathcal V'\cap \mathcal W, \ \ \mathcal Y':= \mathcal Y\cap \mathcal W.$
Since \ $(\mathcal V', \ \mathcal Y, \ \mathcal W)$ is a Hovey triple in $B$\mbox{-}{\rm Mod},
\ $(\mathcal V, \ \mathcal Y)$ and \ $(\mathcal V', \ \mathcal Y')$ are complete cotorsion pairs in $B$\mbox{-}{\rm Mod}.

\vskip5pt

$(1)$ \ Since \ $M\otimes_A \mathcal P \subseteq \mathcal Y,$
it follows from Theorem \ref{ctp2}(1) that
 \ $({\rm T}_A(_A\mathcal P)\oplus {\rm T}_B(\mathcal V), \ \left(\begin{smallmatrix}A\text{\rm\rm-Mod}\\ \mathcal Y\end{smallmatrix}\right))$ is a complete cotorsion pair in $\Lambda$-Mod.
Similarly, \ $({\rm T}_A(_A\mathcal P)\oplus {\rm T}_B(\mathcal V'), \ \left(\begin{smallmatrix}A\text{\rm\rm-Mod}\\ \mathcal Y'\end{smallmatrix}\right))$ is a complete cotorsion pair.

\vskip5pt

Since \ $M\otimes_A \mathcal P \subseteq \mathcal W$, it follows that
$$({\rm T}_A(_A\mathcal P)\oplus {\rm T}_B(\mathcal V'))\cap \left(\begin{smallmatrix} A\text{\rm-Mod} \\ \mathcal W\end{smallmatrix}\right) ={\rm T}_A(_A\mathcal P)\oplus {\rm T}_B(\mathcal V'\cap \mathcal W)={\rm T}_A(_A\mathcal P)\oplus {\rm T}_B(\mathcal V).$$
Clearly, $\left(\begin{smallmatrix} A\text{\rm-Mod} \\ \mathcal Y\end{smallmatrix}\right)\cap \left(\begin{smallmatrix} A\text{\rm-Mod} \\ \mathcal W\end{smallmatrix}\right) = \left(\begin{smallmatrix} A\text{\rm-Mod} \\ \mathcal Y'\end{smallmatrix}\right)$.
Since $\mathcal W$ is a thick class of $B$\mbox{-}{\rm Mod}, \ $\left(\begin{smallmatrix} A\text{\rm-Mod} \\ \mathcal W\end{smallmatrix}\right)$ is a thick class of $\Lambda$\mbox{-}{\rm Mod}.
By definition
\ $({\rm T}_A(_A\mathcal P)\oplus {\rm T}_B(\mathcal V'), \ \left(\begin{smallmatrix} A\text{\rm\rm-Mod} \\ \mathcal Y\end{smallmatrix}\right), \ \left(\begin{smallmatrix} A\text{\rm\rm-Mod} \\ \mathcal W\end{smallmatrix}\right))$
is a Hovey triple.

\vskip5pt

If \ $(\mathcal V', \ \mathcal Y, \ \mathcal W)$ is a hereditary Hovey triple,
then by Theorem \ref{ctp2}(1),   both \ $({\rm T}_A(_A\mathcal P)\oplus {\rm T}_B(\mathcal V), \ \left(\begin{smallmatrix}A\text{\rm\rm-Mod}\\ \mathcal Y\end{smallmatrix}\right))$ and \ $({\rm T}_A(_A\mathcal P)\oplus {\rm T}_B(\mathcal V'), \ \left(\begin{smallmatrix}A\text{\rm\rm-Mod}\\ \mathcal Y'\end{smallmatrix}\right))$ are hereditary cotorsion pairs, and hence  \ $({\rm T}_A(_A\mathcal P)\oplus {\rm T}_B(\mathcal V'), \ \left(\begin{smallmatrix} A\text{\rm\rm-Mod} \\ \mathcal Y\end{smallmatrix}\right), \ \left(\begin{smallmatrix} A\text{\rm\rm-Mod} \\ \mathcal W\end{smallmatrix}\right))$ is a hereditary Hovey triple. By Theorem \ref{Ho} one has
\begin{align*}{\rm Ho}(\Lambda)& \cong (({\rm T}_A(_A\mathcal P)\oplus {\rm T}_B(\mathcal V'))\cap \left(\begin{smallmatrix} A\text{\rm\rm-Mod} \\ \mathcal Y\end{smallmatrix}\right)) /(({\rm T}_A(_A\mathcal P)\oplus {\rm T}_B(\mathcal V'))\cap \left(\begin{smallmatrix} A\text{\rm\rm-Mod} \\ \mathcal Y\cap\mathcal W\end{smallmatrix}\right))
\\ & \cong  ({\rm T}_A(_A\mathcal P)\oplus {\rm T}_B(\mathcal V'\cap\mathcal Y))/(({\rm T}_A(_A\mathcal P)\oplus {\rm T}_B(\mathcal V'\cap \mathcal Y'))\\ & \cong {\rm T}_B(\mathcal V'\cap\mathcal Y)/ {\rm T}_B(\mathcal V'\cap \mathcal Y')
\\ & \cong (\mathcal V'\cap \mathcal Y)/(\mathcal V'\cap \mathcal Y')\cong {\rm Ho}(B).\end{align*}

\vskip5pt

$(2)$ \ The proof is similar as $(1)$. We include the main steps. Since \ $\Hom_A(N, \ _A\mathcal I) \subseteq \mathcal V$,
by Theorem \ref{ctp2}(2), $(\left(\begin{smallmatrix} A\mbox{-}{\rm Mod} \\ \mathcal V\end{smallmatrix}\right), \ {\rm H}_A(_A\mathcal I)\oplus {\rm H}_B(\mathcal Y))$ is a complete cotorsion pair.
Similarly, \ $(\left(\begin{smallmatrix} A\mbox{-}{\rm Mod} \\ \mathcal V'\end{smallmatrix}\right), \ {\rm H}_A(_A\mathcal I)\oplus {\rm H}_B(\mathcal Y'))$ is a complete cotorsion pair.

\vskip5pt

Clearly  \ $\left(\begin{smallmatrix} A\text{\rm\rm-Mod} \\ \mathcal V'\end{smallmatrix}\right)\cap \left(\begin{smallmatrix} A\text{\rm\rm-Mod} \\ \mathcal W\end{smallmatrix}\right) = \left(\begin{smallmatrix} A\text{\rm\rm-Mod} \\ \mathcal V\end{smallmatrix}\right).$
Since
$\Hom_A(N, \ _A\mathcal I) \subseteq \mathcal W$, it follows that
$$({\rm H}_A(_A\mathcal I)\oplus {\rm H}_B(\mathcal Y))\cap \left(\begin{smallmatrix} A\text{\rm\rm-Mod} \\ \mathcal W\end{smallmatrix}\right)
= {\rm H}_A(_A\mathcal I)\oplus {\rm H}_B(\mathcal Y\cap \mathcal W) = {\rm H}_A(_A\mathcal I)\oplus {\rm T}_B(\mathcal Y').$$
Also, $\left(\begin{smallmatrix} A\text{\rm-Mod} \\ \mathcal W\end{smallmatrix}\right)$ is a thick class of $\Lambda$\mbox{-}{\rm Mod}. By definition
$$(\left(\begin{smallmatrix} A\mbox{-}{\rm Mod} \\ \mathcal V'\end{smallmatrix}\right), \ \ {\rm H}_A(_A\mathcal I)\oplus {\rm H}_B(\mathcal Y), \ \ \left(\begin{smallmatrix}A\text{\rm\rm-Mod}\\ \mathcal W\end{smallmatrix}\right))$$
is a Hovey triple. Moreover, it is hereditary if \ $(\mathcal V', \ \mathcal Y, \ \mathcal W)$ is hereditary. In this case,
by Theorem \ref{Ho} one has
\begin{align*}{\rm Ho}(\Lambda\mbox{-}{\rm Mod})& \cong (\left(\begin{smallmatrix} A\text{\rm\rm-Mod} \\ \mathcal V'\end{smallmatrix}\right)\cap ({\rm H}_A(_A\mathcal I)\oplus {\rm H}_B(\mathcal Y))) /(\left(\begin{smallmatrix} A\text{\rm\rm-Mod} \\ \mathcal V'\cap \mathcal W\end{smallmatrix}\right)\cap ({\rm H}_A(_A\mathcal I)\oplus {\rm H}_B(\mathcal Y)))
\\ & \cong  ({\rm H}_A(_A\mathcal I)\oplus {\rm H}_B(\mathcal V'\cap\mathcal Y))/({\rm H}_A(_A\mathcal I)\oplus {\rm H}_B(\mathcal V'\cap\mathcal Y')) \\ & \cong {\rm H}_B(\mathcal V'\cap\mathcal Y)/{\rm H}_B(\mathcal V'\cap\mathcal Y')
\\ & \cong (\mathcal V'\cap \mathcal Y)/(\mathcal V'\cap\mathcal Y') \cong {\rm Ho}(B).\end{align*}
\end{proof}

From Theorem \ref{Htriple1} and its proof, one has

\begin{cor}\label{GHtriple1} \ Let $\Lambda=\left(\begin{smallmatrix} A & N \\
M & B\end{smallmatrix}\right)$ be a Morita ring with  $M\otimes_AN=0 = N\otimes_BM$, $(\mathcal V, \ \mathcal Y)$ and $(\mathcal V', \ \mathcal Y')$
compatible complete hereditary cotorsion pairs in $B$-{\rm Mod}, with Gillespie-Hovey triple
\ $(\mathcal V', \ \mathcal Y, \ \mathcal W)$. Suppose that  \ $N_B$ is flat and  $_BM$ is projective.

\vskip5pt

$(1)$  \  If \ $M\otimes_A\mathcal P \subseteq \mathcal Y'$, then
$$({\rm T}_A(_A\mathcal P)\oplus {\rm T}_B(\mathcal V), \ \left(\begin{smallmatrix} A\text{\rm\rm-Mod} \\ \mathcal Y\end{smallmatrix}\right))
\ \ \ \mbox{and} \ \ \
({\rm T}_A(_A\mathcal P)\oplus {\rm T}_B(\mathcal V'), \ \left(\begin{smallmatrix} A\text{\rm\rm-Mod} \\ \mathcal Y'\end{smallmatrix}\right))$$
are compatible complete hereditary  cotorsion pairs in \ $\Lambda\mbox{-}{\rm Mod}$, with Gillespie-Hovey triple
$$({\rm T}_A(_A\mathcal P)\oplus {\rm T}_B(\mathcal V'), \ \left(\begin{smallmatrix} A\text{\rm\rm-Mod} \\ \mathcal Y\end{smallmatrix}\right), \ \left(\begin{smallmatrix} A\text{\rm\rm-Mod} \\ \mathcal W\end{smallmatrix}\right))$$
and \ ${\rm Ho}(\Lambda)\cong {\rm Ho}(B)$.

\vskip5pt

$(2)$  If \ $\Hom_A(N, \ _A\mathcal I)\subseteq \mathcal V$, then
$$(\left(\begin{smallmatrix} A\mbox{-}{\rm Mod} \\ \mathcal V\end{smallmatrix}\right), \ {\rm H}_A(_A\mathcal I)\oplus {\rm H}_B(\mathcal Y)) \ \ \ \mbox{and} \ \ \ (\left(\begin{smallmatrix} A\mbox{-}{\rm Mod} \\ \mathcal V'\end{smallmatrix}\right), \ {\rm H}_A(_A\mathcal I)\oplus {\rm H}_B(\mathcal Y'))$$
are compatible complete hereditary cotorsion pairs in $\Lambda$-{\rm Mod}, with Gillespie-Hovey triple
$$(\left(\begin{smallmatrix} A\mbox{-}{\rm Mod} \\ \mathcal V'\end{smallmatrix}\right), \ \ {\rm H}_A(_A\mathcal I)\oplus {\rm H}_B(\mathcal Y), \ \ \left(\begin{smallmatrix}A\text{\rm\rm-Mod}\\ \mathcal W\end{smallmatrix}\right))$$
and \ ${\rm Ho}(\Lambda) \cong {\rm Ho}(B)$.
\end{cor}

Similar as Theorem \ref{Htriple1}, starting from a Hovey triple in $A$-{\rm Mod} and using Theorem \ref{ctp3},  we get

\begin{thm}\label{Htriple2} \ Let $\Lambda=\left(\begin{smallmatrix} A & N \\
M & B\end{smallmatrix}\right)$ be a Morita ring with  $M\otimes_AN=0 = N\otimes_BM$.
Let \ $(\mathcal U', \ \mathcal X, \ \mathcal W)$ be a Hovey triple in $A$\mbox{-}{\rm Mod}.
Suppose that \ $M_A$ is flat and \ $_AN$ is projective.

\vskip5pt

$(1)$ \ If \ $N\otimes_B\mathcal P \subseteq \mathcal X\cap \mathcal W$, then
$$({\rm T}_A(\mathcal U')\oplus {\rm T}_B(_B\mathcal P), \ \left(\begin{smallmatrix}\mathcal X \\ B\text{\rm\rm-Mod}\end{smallmatrix}\right), \ \left(\begin{smallmatrix}\mathcal W \\ B\text{\rm\rm-Mod}\end{smallmatrix}\right))$$
is a Hovey triple$;$ and it is hereditary with \ ${\rm Ho}(\Lambda) \cong {\rm Ho}(A)$, provided that
 \ $(\mathcal U',  \mathcal X,  \mathcal W)$ is hereditary.
\vskip5pt

$(2)$ \  If \ $\Hom_B(M, \ _B\mathcal I) \subseteq \mathcal U'\cap \mathcal W,$ then
$$(\left(\begin{smallmatrix}\mathcal U' \\ B\text{\rm\rm-Mod}\end{smallmatrix}\right), \ \ {\rm H}_A(\mathcal X)\oplus {\rm H}_B(_B\mathcal I), \ \ \left(\begin{smallmatrix}\mathcal W \\ B\text{\rm\rm-Mod}\end{smallmatrix}\right))$$
is a Hovey triple$;$ and it is hereditary with \ ${\rm Ho}(\Lambda) \cong {\rm Ho}(A)$, provided that
 \ $(\mathcal U',  \mathcal X,  \mathcal W)$ is hereditary.
\end{thm}

\begin{cor}\label{GHtriple2} \ Let $\Lambda=\left(\begin{smallmatrix} A & N \\
M & B\end{smallmatrix}\right)$ be a Morita ring with  $M\otimes_AN=0 = N\otimes_BM$, \ $(\mathcal U, \mathcal X)$ and \ $(\mathcal U', \mathcal X')$
compatible complete hereditary cotorsion pairs in $A$-{\rm Mod}, with Gillespie-Hovey triple
\ $(\mathcal U', \ \mathcal X, \ \mathcal W)$.  Suppose that \ $M_A$ is flat and \ $_AN$ is projective.

\vskip5pt

$(1)$ \ If \ $N\otimes_B\mathcal P \subseteq \mathcal X'$, then
$$({\rm T}_A(\mathcal U)\oplus {\rm T}_B(_B\mathcal P), \ \left(\begin{smallmatrix}\mathcal X \\ B\text{\rm\rm-Mod}\end{smallmatrix}\right))
\ \ \ \mbox{and} \ \ \
({\rm T}_A(\mathcal U')\oplus {\rm T}_B(_B\mathcal P), \ \left(\begin{smallmatrix}\mathcal X' \\ B\text{\rm\rm-Mod}\end{smallmatrix}\right))$$
are compatible complete hereditary cotorsion pairs in \ $\Lambda\mbox{-}{\rm Mod}$, with Gillespie-Hovey triple
$$({\rm T}_A(\mathcal U')\oplus {\rm T}_B(_B\mathcal P), \ \left(\begin{smallmatrix}\mathcal X \\ B\text{\rm\rm-Mod}\end{smallmatrix}\right), \ \left(\begin{smallmatrix}\mathcal W \\ B\text{\rm\rm-Mod}\end{smallmatrix}\right))$$
and \ ${\rm Ho}(\Lambda)\cong {\rm Ho}(A)$.

\vskip5pt

$(2)$ \  If \ $\Hom_B(M, \ _B\mathcal I) \subseteq \mathcal U,$ then
$$(\left(\begin{smallmatrix}\mathcal U \\ B\text{\rm\rm-Mod}\end{smallmatrix}\right), \ {\rm H}_A(\mathcal X)\oplus {\rm H}_B(_B\mathcal I)) \ \ \ \mbox{and} \ \ \
(\left(\begin{smallmatrix}\mathcal U' \\ B\text{\rm\rm-Mod}\end{smallmatrix}\right), \ {\rm H}_A(\mathcal X')\oplus {\rm H}_B(_B\mathcal I))$$
are compatible complete hereditary cotorsion pairs in $\Lambda$-{\rm Mod}, with Gillespie-Hovey triple
$$(\left(\begin{smallmatrix}\mathcal U' \\ B\text{\rm\rm-Mod}\end{smallmatrix}\right), \ \ {\rm H}_A(\mathcal X)\oplus {\rm H}_B(_B\mathcal I), \ \ \left(\begin{smallmatrix}\mathcal W \\ B\text{\rm\rm-Mod}\end{smallmatrix}\right))$$
and \ ${\rm Ho}(\Lambda) \cong {\rm Ho}(A)$.
\end{cor}

\subsection{Gillespie-Hovey triples in Morita rings, via generalized projective (injective) cotorsion pairs}

The notion of generalized projective (injective) cotorsion pairs is essentially due to H. Becker [Bec].

\begin{defn} \label{genprojctp}  $(1)$ \ \ A complete cotorsion pair \ $(\mathcal X, \ \mathcal Y)$  in an abelian category $\mathcal A$ with enough projective objects
is {\it a generalized projective cotorsion pair}, or in short, gpctp, provided that

${\rm (i)}$ \ \ $\mathcal X\cap \mathcal Y = \mathcal P$, where  $\mathcal P$ is the class of projective objects of $\mathcal A;$

${\rm (ii)}$  \ \ the class \ $\mathcal Y$ is thick.

\vskip5pt

${\rm (1')}$  \ \ A complete cotorsion pair \ $(\mathcal X, \ \mathcal Y)$  in an abelian category $\mathcal A$ with enough injective objects is {\it a generalized injective cotorsion pair}, or in short, gictp, provided that

${\rm (i')}$  \ \ $\mathcal X\cap \mathcal Y = \mathcal I$, where  $\mathcal I$ is the class of injective objects of $\mathcal A;$

${\rm (ii')}$ \ \ the class \ $\mathcal X$ is thick.
\end{defn}

\begin{exmrem} A {\rm gpctp} $($respectively, {\rm gictp}$)$  is not necessarily the projective $($respectively, injective$)$ cotorsion pair \ $(\mathcal P, \ \mathcal A)$ \
$($respectively, \ $(\mathcal A, \ \mathcal I))$.

\vskip5pt

$(1)$ \ {\rm ([H2])} \ For a Gorenstein ring $R$, the Gorenstein-projective cotorsion pair
\ $({\rm GP}(R), \ _R\mathcal P^{<\infty})$ is  a {\rm gpctp}. Dually,
\ $(_R\mathcal P^{<\infty}, \ {\rm GI}(R))$ is a {\rm gictp}.

\vskip5pt

$(2)$ \ Let ${\rm Ch}(R)$ be the complex category of modules over ring $R$, \ $\mathcal E$ the class of
acyclic complexes, and ${\rm dg}\mathcal P$ the class of dg projective complexes $Q$
$($see {\rm [Sp], [AF]}$)$, i.e., components of $Q$ are projective and \ $\Hom^\bullet (Q, \mathcal E)$ is acyclic.
By {\rm [EJX]}, \ $({\rm dg}\mathcal P, \ \mathcal E)$ is a cotorsion pair, and
${\rm dg}\mathcal P\cap \mathcal E$ is exactly the class of projective objects of  ${\rm Ch}(R)$. That is,
$${\rm dg}\mathcal P\cap \mathcal E = \{\bigoplus\limits_{i\in\Bbb Z} P^i(P) \ | \ P\in \ _R\mathcal P\}$$
where \ $P^i(P): \ \cdots \rightarrow 0 \rightarrow P \stackrel{{\rm Id}}\rightarrow P \rightarrow 0 \rightarrow \cdots $ is the complex with $i$-th and $(i+1)$-th component $P$. By {\rm [Sp]} $($also {\rm [BN]}$)$, for any complex $X$ there is
an epimorphism $Q\longrightarrow X$ which is a quasi-isomorphism. Thus, \ $({\rm dg}\mathcal P, \ \mathcal E)$ is complete, and hence generalized projective.
Dually, there is a {\rm gictp} \ $(\mathcal E, \ {\rm dg}\mathcal I)$. See {\rm [Gil1]} for an important development of this work.

\vskip5pt

$(3)$ \ Any {\rm gpctp}  $(\mathcal X,  \mathcal Y)$ is hereditary,   $\mathcal X$ is a Frobenius category $($with the canonical exact structure$)$, and  $\mathcal P$  is
the class of projective-injective objects.

\vskip5pt

$(3')$ \ Any {\rm gictp}   $(\mathcal X,  \mathcal Y)$ is hereditary,   $\mathcal Y$ is a Frobenius category, and  $\mathcal I$  is
the class of projective-injective objects.    \end{exmrem}

\vskip5pt

Taking gpctps or gictps in Corollary \ref{GHtriple1}, we get a stronger and an improved result without extra conditions (i.e., the conditions
``$M\otimes_A\mathcal P \subseteq \mathcal Y'$" and ``$\Hom_A(N, \ _A\mathcal I)\subseteq \mathcal V$" in Corollary \ref{GHtriple1} can be dropped).
This is the reason we list it as a theorem.

\vskip5pt

\begin{thm}\label{GHtriplegpgiB} \ Let $\Lambda=\left(\begin{smallmatrix} A & N \\
M & B\end{smallmatrix}\right)$ be a Morita ring with  $M\otimes_AN=0 = N\otimes_BM$.  Suppose that \ $N_B$ is flat and  $_BM$ is projective.

\vskip5pt

$(1)$ \ Let \  $(\mathcal V, \ \mathcal Y)$ and $(\mathcal V', \ \mathcal Y')$ be compatible gpctps in $B$-{\rm Mod}, with Gillespie-Hovey triple
\ $(\mathcal V', \ \mathcal Y, \ \mathcal W)$.
Then
$$({\rm T}_A(_A\mathcal P)\oplus {\rm T}_B(\mathcal V), \ \left(\begin{smallmatrix} A\text{\rm\rm-Mod} \\ \mathcal Y\end{smallmatrix}\right))
\ \ \ \mbox{and} \ \ \
({\rm T}_A(_A\mathcal P)\oplus {\rm T}_B(\mathcal V'), \ \left(\begin{smallmatrix} A\text{\rm\rm-Mod} \\ \mathcal Y'\end{smallmatrix}\right))$$
are compatible gpctps in \ $\Lambda\mbox{-}{\rm Mod}$, with Gillespie-Hovey triple
$$({\rm T}_A(_A\mathcal P)\oplus {\rm T}_B(\mathcal V'), \ \left(\begin{smallmatrix} A\text{\rm\rm-Mod} \\ \mathcal Y\end{smallmatrix}\right), \ \left(\begin{smallmatrix} A\text{\rm\rm-Mod} \\ \mathcal W\end{smallmatrix}\right))$$
and \ ${\rm Ho}(\Lambda)\cong (\mathcal V'\cap \mathcal Y)/_B\mathcal P\cong {\rm Ho}(B)$.

\vskip5pt

$(2)$   \ Let \ $(\mathcal V, \mathcal Y)$ and \ $(\mathcal V', \mathcal Y')$ be compatible gictps in $B$-{\rm Mod}, with Gillespie-Hovey triple
\ $(\mathcal V', \ \mathcal Y, \ \mathcal W)$.  Then
$$(\left(\begin{smallmatrix} A\mbox{-}{\rm Mod} \\ \mathcal V\end{smallmatrix}\right), \ {\rm H}_A(_A\mathcal I)\oplus {\rm H}_B(\mathcal Y)) \ \ \ \mbox{and} \ \ \ (\left(\begin{smallmatrix} A\mbox{-}{\rm Mod} \\ \mathcal V'\end{smallmatrix}\right), \ {\rm H}_A(_A\mathcal I)\oplus {\rm H}_B(\mathcal Y'))$$
are compatible gictps in $\Lambda$-{\rm Mod}, with Gillespie-Hovey triple
$$(\left(\begin{smallmatrix} A\mbox{-}{\rm Mod} \\ \mathcal V'\end{smallmatrix}\right), \ \ {\rm H}_A(_A\mathcal I)\oplus {\rm H}_B(\mathcal Y), \ \ \left(\begin{smallmatrix}A\text{\rm\rm-Mod}\\ \mathcal W\end{smallmatrix}\right))$$
and \ ${\rm Ho}(\Lambda) \cong (\mathcal V'\cap \mathcal Y)/ _B\mathcal I\cong {\rm Ho}(B)$.
\end{thm}
\begin{proof} $(1)$ \ Since \ $_BM$ is projective, $M\otimes_A \mathcal P \subseteq \ _B\mathcal P.$ \
Since cotorsion pair \ $(\mathcal V', \mathcal Y')$ \ is generalized projective,
\ $M\otimes_A \mathcal P \subseteq \ _B\mathcal P = \mathcal V'\cap \mathcal Y'\subseteq \mathcal Y'\subseteq \mathcal Y.$

\vskip5pt

Thus, by Corollary \ref{GHtriple1}(1), $$({\rm T}_A(_A\mathcal P)\oplus {\rm T}_B(\mathcal V), \ \left(\begin{smallmatrix} A\text{\rm\rm-Mod} \\ \mathcal Y\end{smallmatrix}\right))
\ \ \ \mbox{and} \ \ \
({\rm T}_A(_A\mathcal P)\oplus {\rm T}_B(\mathcal V'), \ \left(\begin{smallmatrix} A\text{\rm\rm-Mod} \\ \mathcal Y'\end{smallmatrix}\right))$$
are compatible complete hereditary cotorsion pairs in \ $\Lambda\mbox{-}{\rm Mod}$, with Gillespie-Hovey triple
$$({\rm T}_A(_A\mathcal P)\oplus {\rm T}_B(\mathcal V'), \ \left(\begin{smallmatrix} A\text{\rm\rm-Mod} \\ \mathcal Y\end{smallmatrix}\right), \ \left(\begin{smallmatrix} A\text{\rm\rm-Mod} \\ \mathcal W\end{smallmatrix}\right))$$
and \ ${\rm Ho}(\Lambda)\cong {\rm Ho}(B) \cong (\mathcal V'\cap \mathcal Y)/_B\mathcal P$. Since
$$_\Lambda\mathcal P = {\rm T}_A(_A\mathcal P)\oplus {\rm T}_B(_B\mathcal P)
= \{\left(\begin{smallmatrix}P \\ M\otimes_A P\end{smallmatrix}\right)\oplus \left(\begin{smallmatrix}N\otimes_B Q \\ Q\end{smallmatrix}\right) \ | \ P\in \ _A\mathcal P,
\ Q\in \ _B\mathcal P\}$$
and \ $M\otimes_A \mathcal P \subseteq \mathcal Y$, it follows that
$$({\rm T}_A(_A\mathcal P)\oplus {\rm T}_B(\mathcal V))\cap \left(\begin{smallmatrix} A\text{\rm-Mod} \\ \mathcal Y\end{smallmatrix}\right)
={\rm T}_A(_A\mathcal P)\oplus {\rm T}_B(\mathcal V\cap \mathcal Y)={\rm T}_A(_A\mathcal P)\oplus {\rm T}_B(_B\mathcal P)= \ _\Lambda\mathcal P.$$
Since $\mathcal Y$ is thick, $\left(\begin{smallmatrix}A\text{\rm\rm-Mod}\\ \mathcal Y\end{smallmatrix}\right)$ is thick.
Thus, cotorsion pair $({\rm T}_A(_A\mathcal P)\oplus {\rm T}_B(\mathcal V), \ \left(\begin{smallmatrix}A\text{\rm\rm-Mod}\\ \mathcal Y\end{smallmatrix}\right))$ is generalized projective.
Similarly, \ $({\rm T}_A(_A\mathcal P)\oplus {\rm T}_B(\mathcal V'), \ \left(\begin{smallmatrix}A\text{\rm\rm-Mod}\\ \mathcal Y'\end{smallmatrix}\right))$ is generalized projective.

\vskip5pt

$(2)$ \ Since \ $N_B$ is flat, $\Hom_A(N, \ _A\mathcal I) \subseteq  \ _B \mathcal I$.
Since \ $(\mathcal V, \mathcal Y)$ \ is generalized injective,
$\Hom_A(N, \ _A\mathcal I) \subseteq \ _B \mathcal I = \mathcal V\cap \mathcal Y \subseteq \mathcal V$.

\vskip5pt

Thus, by Corollary \ref{GHtriple1}(2), $$(\left(\begin{smallmatrix}\mathcal U \\ B\text{\rm\rm-Mod}\end{smallmatrix}\right), \ {\rm H}_A(\mathcal X)\oplus {\rm H}_B(_B\mathcal I)) \ \ \ \mbox{and} \ \ \
(\left(\begin{smallmatrix}\mathcal U' \\ B\text{\rm\rm-Mod}\end{smallmatrix}\right), \ {\rm H}_A(\mathcal X')\oplus {\rm H}_B(_B\mathcal I))$$
are compatible complete hereditary cotorsion pairs, with Gillespie-Hovey triple
$$(\left(\begin{smallmatrix}\mathcal U' \\ B\text{\rm\rm-Mod}\end{smallmatrix}\right), \ \ {\rm H}_A(\mathcal X)\oplus {\rm H}_B(_B\mathcal I), \ \ \left(\begin{smallmatrix}\mathcal W \\ B\text{\rm\rm-Mod}\end{smallmatrix}\right))$$
and \ ${\rm Ho}(\Lambda) \cong {\rm Ho}(B)\cong (\mathcal U'\cap \mathcal X)/ _A\mathcal I$.
Since
$$_\Lambda\mathcal I ={\rm H}_A(_A\mathcal I)\oplus {\rm H}_B(_B\mathcal I)= \{\left(\begin{smallmatrix} I \\ \Hom_A(N,I)\end{smallmatrix}\right)\oplus \left(\begin{smallmatrix} \Hom_B(M,J) \\ J \end{smallmatrix}\right) \ | \ I\in \ _A\mathcal I, \ J\in \ _B\mathcal I\}$$
and $\Hom_A(N, \ _A\mathcal I) \subseteq  \mathcal V$,
it follows that  $$\left(\begin{smallmatrix} A\text{\rm-Mod} \\ \mathcal V\end{smallmatrix}\right)\cap ({\rm H}_A(_A\mathcal I)\oplus {\rm H}_B(\mathcal Y))
={\rm H}_A(_A\mathcal I)\oplus {\rm H}_B(\mathcal V\cap \mathcal Y)={\rm H}_A(_A\mathcal I)\oplus {\rm H}_B(_B\mathcal I)= \ _\Lambda\mathcal I.$$
Since \ $\mathcal V$ is thick, \ $\binom{A\text{\rm\rm-Mod}}{\mathcal V}$ is thick.
Thus \ $(\left(\begin{smallmatrix} A\mbox{-}{\rm Mod} \\ \mathcal V\end{smallmatrix}\right), \ {\rm H}_A(_A\mathcal I)\oplus {\rm H}_B(\mathcal Y))$ is generalized injective.  Similarly, \ $(\left(\begin{smallmatrix} A\mbox{-}{\rm Mod} \\ \mathcal V'\end{smallmatrix}\right), \ {\rm H}_A(_A\mathcal I)\oplus {\rm H}_B(\mathcal Y'))$ is generalized injective.
\end{proof}

\vskip5pt

Similarly, taking gpctps or gictps in Corollary \ref{GHtriple2}, we get a stronger and an improved result.

\vskip5pt

\begin{thm}\label{GHtriplegpgiA} \ Let $\Lambda=\left(\begin{smallmatrix} A & N \\
M & B\end{smallmatrix}\right)$ be a Morita ring with  $M\otimes_AN=0 = N\otimes_BM$. Suppose that  \ $M_A$ is flat and \ $_AN$ is projective.

\vskip5pt

$(1)$ \ Let \  $(\mathcal U, \ \mathcal X)$ and $(\mathcal U', \ \mathcal X')$ be compatible gpctps in $A$-{\rm Mod}, with Gillespie-Hovey triple
\ $(\mathcal U', \ \mathcal X, \ \mathcal W)$.
Then
$$({\rm T}_A(\mathcal U)\oplus {\rm T}_B(_B\mathcal P), \ \left(\begin{smallmatrix}\mathcal X \\ B\text{\rm\rm-Mod}\end{smallmatrix}\right))
\ \ \ \mbox{and} \ \ \
({\rm T}_A(\mathcal U')\oplus {\rm T}_B(_B\mathcal P), \ \left(\begin{smallmatrix}\mathcal X' \\ B\text{\rm\rm-Mod}\end{smallmatrix}\right))$$
are compatible gpctps in \ $\Lambda\mbox{-}{\rm Mod}$, with Gillespie-Hovey triple
$$({\rm T}_A(\mathcal U')\oplus {\rm T}_B(_B\mathcal P), \ \left(\begin{smallmatrix}\mathcal X \\ B\text{\rm\rm-Mod}\end{smallmatrix}\right), \ \left(\begin{smallmatrix}\mathcal W \\ B\text{\rm\rm-Mod}\end{smallmatrix}\right))$$
and \ ${\rm Ho}(\Lambda)\cong (\mathcal U'\cap \mathcal X)/_A\mathcal P\cong {\rm Ho}(A)$.

\vskip5pt

$(2)$ \  Let \ $(\mathcal U, \mathcal X)$ and \ $(\mathcal U', \mathcal X')$ be compatible gictps in $A$-{\rm Mod}, with Gillespie-Hovey triple
\ $(\mathcal U', \ \mathcal X, \ \mathcal W)$.  Then
$$(\left(\begin{smallmatrix}\mathcal U \\ B\text{\rm\rm-Mod}\end{smallmatrix}\right), \ {\rm H}_A(\mathcal X)\oplus {\rm H}_B(_B\mathcal I)) \ \ \ \mbox{and} \ \ \
(\left(\begin{smallmatrix}\mathcal U' \\ B\text{\rm\rm-Mod}\end{smallmatrix}\right), \ {\rm H}_A(\mathcal X')\oplus {\rm H}_B(_B\mathcal I))$$
are compatible gictps in $\Lambda$-{\rm Mod}, with Gillespie-Hovey triple
$$(\left(\begin{smallmatrix}\mathcal U' \\ B\text{\rm\rm-Mod}\end{smallmatrix}\right), \ \ {\rm H}_A(\mathcal X)\oplus {\rm H}_B(_B\mathcal I), \ \ \left(\begin{smallmatrix}\mathcal W \\ B\text{\rm\rm-Mod}\end{smallmatrix}\right))$$
and \ ${\rm Ho}(\Lambda)\cong (\mathcal U'\cap \mathcal X)/ _A\mathcal I\cong {\rm Ho}(A)$.
\end{thm}

\subsection{Projective (Injective) models on Morita rings}
An abelian model structure on (abelian) category $\mathcal A$ is {\it projective} (respectively, {\it injective}) if
each object is fibrant (respectively, cofibrant), i.e.,
the  Hovey triple is of the form  $(\mathcal X,  \mathcal A,  \mathcal Y)$ \ (respectively,  $(\mathcal A,  \mathcal Y,  \mathcal X)$). See [H2], [Gil2].

\vskip5pt

The following observation clarifies the relation between projective (respectively, injective) models and gpctps (respectively, gictps).

\vskip5pt

\begin{lem}\label{Htriple} {\rm ([Bec, 1.1.9]; [Gil3, 1.1])}  \ Let  \ $(\mathcal X, \ \mathcal Y)$ be a complete cotorsion pair in abelian category $\mathcal A$ with enough projective objects and enough injective objects. Then

\vskip5pt

$(1)$ \ $(\mathcal X, \ \mathcal A, \ \mathcal Y)$ is a $($hereditary$)$ Hovey triple if and only if
\ $(\mathcal X, \ \mathcal Y)$ is a generalized projective cotorsion pair.

\vskip5pt

$(1')$ \ $(\mathcal A, \ \mathcal Y, \ \mathcal X)$  is a $($hereditary$)$ Hovey triple if and only if
\ $(\mathcal X, \ \mathcal Y)$ is a generalized injective cotorsion pair.
\end{lem}

\vskip5pt

Any gpctp $(\mathcal V, \ \mathcal Y)$ in $B$-{\rm Mod} gives compatible gpctps
$(_B\mathcal P, \ B\mbox{-}{\rm Mod})$ and $(\mathcal V, \ \mathcal Y)$.
Any gictp $(\mathcal V, \ \mathcal Y)$ in $B$-{\rm Mod} gives compatible gictps
$(\mathcal V, \ \mathcal Y)$ and $(B\mbox{-}{\rm Mod}, \ _B\mathcal I)$.
Thus, by Theorem \ref{GHtriplegpgiB} one gets:

\vskip5pt

\begin{cor}\label{projinjtripleB} \ Let $\Lambda=\left(\begin{smallmatrix} A & N \\
M & B\end{smallmatrix}\right)$ be a Morita ring with  $M\otimes_AN=0 = N\otimes_BM$. Suppose that \ $N_B$ is flat and \ $_BM$ is projective.

\vskip5pt

$(1)$ \  Let \ $(\mathcal V, \mathcal Y)$ be a gpctp in $B$-{\rm Mod}. Then
$$({\rm T}_A(_A\mathcal P)\oplus {\rm T}_B(\mathcal V), \ \ \Lambda\text{\rm\rm-Mod}, \ \ \left(\begin{smallmatrix} A\text{\rm\rm-Mod} \\ \mathcal Y\end{smallmatrix}\right))$$
is a hereditary Hovey triple, and \ ${\rm Ho}(\Lambda)\cong\mathcal V/_B\mathcal P$.

\vskip5pt

$(2)$   \ Let \ $(\mathcal V, \mathcal Y)$ be a gictp in $B$-{\rm Mod}. Then
$$(\Lambda\text{\rm\rm-Mod}, \ \ {\rm H}_A(_A\mathcal I)\oplus {\rm H}_B(\mathcal Y), \ \ \left(\begin{smallmatrix}A\text{\rm\rm-Mod}\\ \mathcal V\end{smallmatrix}\right))$$
is a hereditary Hovey triple, and \ ${\rm Ho}(\Lambda)\cong \mathcal Y/_B\mathcal I$.
\end{cor}

\vskip5pt

If $B$  is quasi-Frobenius, then  \ $(B\mbox{-}{\rm Mod}, \ _B\mathcal I)$
is a gpctp, and
\ $(_B\mathcal P, \ B\mbox{-}{\rm Mod})$ is a gictp. By Corollary \ref{projinjtripleB} one gets

\begin{cor}\label{frobB} \ Let $\Lambda=\left(\begin{smallmatrix} A & N \\
M & B\end{smallmatrix}\right)$ be a Morita ring with  $M\otimes_AN=0 = N\otimes_BM$. Suppose that $B$ is quasi-Frobenius,  $N_B$ is flat and \ $_BM$ is projective. Then

\vskip5pt

$(1)$ \  $({\rm T}_A(_A\mathcal P)\oplus {\rm T}_B(B\text{\rm\rm-Mod}), \ \ \Lambda\text{\rm\rm-Mod}, \ \ \left(\begin{smallmatrix} A\text{\rm\rm-Mod} \\ _B\mathcal I\end{smallmatrix}\right))$
is a hereditary Hovey triple$;$ and \ ${\rm Ho}(\Lambda)\cong B\mbox{-}\underline{{\rm Mod}}.$

\vskip5pt

$(2)$  \ $(\Lambda\text{\rm\rm-Mod}, \ \ {\rm H}_A(_A\mathcal I)\oplus {\rm H}_B(B\text{\rm\rm-Mod}), \ \ \left(\begin{smallmatrix}A\text{\rm\rm-Mod}\\ _B\mathcal P\end{smallmatrix}\right))$
is a hereditary Hovey triple$;$ and \ ${\rm Ho}(\Lambda)\cong B\mbox{-}\underline{{\rm Mod}}.$
\end{cor}

\vskip5pt

Similar as Corollary \ref{projinjtripleB}, by Theorem \ref{GHtriplegpgiA} one gets

\vskip5pt

\begin{cor}\label{projinjtripleA} \ Let $\Lambda=\left(\begin{smallmatrix} A & N \\
M & B\end{smallmatrix}\right)$ be a Morita ring with  $M\otimes_AN=0 = N\otimes_BM$. Suppose that \ $M_A$ is flat and \ $_AN$ is projective.

\vskip5pt

$(1)$ \  Let  \ $(\mathcal U, \mathcal X)$ be a gpctp in $A$-{\rm Mod}. Then
$$({\rm T}_A(\mathcal U)\oplus {\rm T}_B(_B\mathcal P), \ \ \Lambda\text{\rm\rm-Mod}, \ \ \left(\begin{smallmatrix}\mathcal X \\ B\text{\rm\rm-Mod}\end{smallmatrix}\right))$$
is a hereditary  Hovey triple, and \ ${\rm Ho}(\Lambda)\cong\mathcal U/_A\mathcal P$.

\vskip5pt

$(2)$ \  Let \ $(\mathcal U, \mathcal X)$ be a gictp in $A$-{\rm Mod}.  Then
$$(\Lambda\text{\rm\rm-Mod}, \ \ {\rm H}_A(\mathcal X)\oplus {\rm H}_B(_B\mathcal I), \ \ \left(\begin{smallmatrix}\mathcal U \\ B\text{\rm\rm-Mod}\end{smallmatrix}\right))$$
is a hereditary Hovey triple, and \ ${\rm Ho}(\Lambda)\cong \mathcal X/_A\mathcal I$.
\end{cor}

\vskip5pt

If $A$ is quasi-Frobenius, then  \ $(A\mbox{-}{\rm Mod}, \ _A\mathcal I)$ is a gpctp, and
\ $(_A\mathcal P, \ A\mbox{-}{\rm Mod})$
is a gictp. By Corollary \ref{projinjtripleA} one gets

\begin{cor}\label{frobA} \ Let $\Lambda=\left(\begin{smallmatrix} A & N \\
M & B\end{smallmatrix}\right)$ be a Morita ring with  $M\otimes_AN=0 = N\otimes_BM$. Suppose that $A$ is quasi-Frobenius, \ $M_A$ is flat and \ $_AN$ is projective. Then

\vskip5pt

$(1)$ \ $({\rm T}_A(A\text{\rm\rm-Mod})\oplus {\rm T}_B(_B\mathcal P), \ \ \Lambda\text{\rm\rm-Mod}, \ \ \left(\begin{smallmatrix}_A\mathcal I \\ B\text{\rm\rm-Mod}\end{smallmatrix}\right))$
is a hereditary Hovey triple$;$ and \ ${\rm Ho}(\Lambda)\cong A\mbox{-}\underline{{\rm Mod}}.$

\vskip5pt

$(2)$ \ $(\Lambda\text{\rm\rm-Mod}, \ \ {\rm H}_A(A\text{\rm\rm-Mod})\oplus {\rm H}_B(_B\mathcal I), \ \ \left(\begin{smallmatrix}_A\mathcal P \\ B\text{\rm\rm-Mod}\end{smallmatrix}\right))$
is a hereditary Hovey triple$;$ and \ ${\rm Ho}(\Lambda)\cong A\mbox{-}\underline{{\rm Mod}}.$
\end{cor}

\subsection{Generally different Hovey triples}

\begin{fact} \label{thesameHoveytriple} \ Let \ $(\mathcal C, \ \mathcal F, \ \mathcal W)$ and \ $(\mathcal C', \ \mathcal F', \ \mathcal W')$  be Hovey triples in  abelian category $\mathcal A$.
If $$(\mathcal C\cap \mathcal W, \ \mathcal F) = (\mathcal C'\cap \mathcal W', \ \mathcal F'), \ \  \ \ (\mathcal C, \ \mathcal F\cap \mathcal W) = (\mathcal C', \ \mathcal F'\cap \mathcal W')$$
then \ $(\mathcal C, \ \mathcal F, \ \mathcal W) = (\mathcal C', \ \mathcal F', \ \mathcal W')$.
\end{fact}

In fact, by Theorem \ref{hoveycorrespondence}, the corresponding two abelian model structures  are the same. Thus $\mathcal W = \mathcal W'$.

\begin{defn} \label{Hdifference} \ Let \ $\Omega$ be a class of Morita rings,  \ $(\mathcal C, \ \mathcal F, \ \mathcal W)$ and \ $(\mathcal C', \ \mathcal F', \ \mathcal W')$  Hovey triples defined in $\Lambda\mbox{-}{\rm Mod}$, for arbitrary Morita rings $\Lambda\in \Omega$.
They are said to be generally different Hovey triples, provided that there is $\Lambda\in \Omega$, such that \ $(\mathcal C, \ \mathcal F, \ \mathcal W)\ne (\mathcal C', \ \mathcal F', \ \mathcal W')$ in $\Lambda\mbox{-}{\rm Mod}$.
\end{defn}

\begin{lem} \label{Hdifference1} \ Hovey triples \ $(\mathcal C, \ \mathcal F, \ \mathcal W)$ and \ $(\mathcal C', \ \mathcal F', \ \mathcal W')$ in $\Lambda\mbox{-}{\rm Mod}$ are generally different if and only if
 \ $(\mathcal C\cap \mathcal W, \ \mathcal F)$ \  and $(\mathcal C'\cap \mathcal W', \ \mathcal F')$ are generally different, or,
\ $(\mathcal C, \ \mathcal F\cap \mathcal W)$ and \ $(\mathcal C', \ \mathcal F'\cap \mathcal W')$ are generally different, as cotorsion pairs.
\end{lem}
\begin{proof} The ``only if" part follows from Fact \ref{thesameHoveytriple}. Conversely, without loss of generality, we may assume that
there are $A$, $B$, \ $_BM_A$ and $_AN_B$, such that \ $\mathcal F = \mathcal F'$ and  $\mathcal C\cap \mathcal W\ne \mathcal C'\cap \mathcal W'$.
Then, either $\mathcal C \ne \mathcal C'$, or $\mathcal W\ne \mathcal W'$. Hence $(\mathcal C, \ \mathcal F, \ \mathcal W) \ne (\mathcal C', \ \mathcal F', \ \mathcal W')$ for the corresponding $\Lambda$. \end{proof}

\begin{exm} \label{Hgdsame} \ Generally different Hovey triples could be the same in special cases.

\vskip5pt

For example,
$(_\Lambda\mathcal P, \ \Lambda\mbox{\rm-Mod}, \ \Lambda\mbox{\rm-Mod})$ and $(\binom{_A\mathcal P}{_B\mathcal P}, \ \binom{_A\mathcal P}{_B\mathcal P}^\perp, \ \Lambda\mbox{\rm-Mod})$ are Hovey triples.
Since \ $(_\Lambda\mathcal P, \ \Lambda\mbox{\rm-Mod})$ and $(\binom{_A\mathcal P}{_B\mathcal P}, \ \binom{_A\mathcal P}{_B\mathcal P}^\perp)$ are generally different (cf. Example \ref{gdsame}), by
Lemma \ref{Hdifference1},  the two Hovey triples are
generally different. But, if $M = 0 = N$, then they are the same.
\end{exm}

\begin{prop} \label{newmodel} $(1)$ \ The two Hovey triples in {\rm Theorem \ref{cofibrantlygenHtriple}} are generally different.

\vskip5pt

$(2)$ \ The four Hovey triples  in {\rm Theorems \ref{Htriple1} and \ref{Htriple2}} are pairwise generally different.

\vskip5pt

$(3)$ \ The four Hovey triples in {\rm Theorems \ref{GHtriplegpgiB} and \ref{GHtriplegpgiA}} are pairwise generally different.

\vskip5pt

$(4)$ \ The four Hovey triples in  {\rm Corollaries \ref{projinjtripleB} and \ref{projinjtripleA}} are pairwise generally different.

\vskip5pt

$(5)$ \ The four Hovey triples in  {\rm Corollaries \ref{frobB} and \ref{frobA}} are pairwise generally different.

\vskip5pt

$(6)$ \ All the Hovey triples in $(1)$- $(5)$ are generally different from the following Hovey triples$:$

\vskip5pt

\hskip20pt  $\bullet$  \ $(_\Lambda\mathcal P, \ \Lambda\mbox{-}{\rm Mod},  \ \Lambda\mbox{-}{\rm Mod});$
\vskip5pt
\hskip20pt $\bullet$  \ $(\Lambda\mbox{-}{\rm Mod},  \ _\Lambda\mathcal I, \ \Lambda\mbox{-}{\rm Mod});$
\vskip5pt
\hskip20pt $\bullet$  \ the Frobenius model ${\rm ([Gil2])}:$  \ $(\Lambda\mbox{-}{\rm Mod}, \  \Lambda\mbox{-}{\rm Mod}, \ _\Lambda\mathcal P)$ $($if  $\Lambda$ is quasi-Frobenius$);$
\vskip5pt
\hskip20pt $\bullet$  \  $({\rm GP}(\Lambda), \  \Lambda\mbox{-}{\rm Mod}, \ _\Lambda\mathcal P^{<\infty})$ $($if  $\Lambda$ is Gorenstein$);$
\vskip5pt
\hskip20pt $\bullet$   \ $(\Lambda\mbox{-}{\rm Mod},  \ {\rm GI}(\Lambda), \ _\Lambda\mathcal P^{<\infty})$ $($if $\Lambda$ is Gorenstein$);$
\vskip5pt
\hskip20pt $\bullet$ \ the flat-cotorsion Hovey triple  \ $({\rm F}(\Lambda),  \ {\rm C}(\Lambda), \  \Lambda\mbox{-}{\rm Mod})$ $($see {\rm [BBE], [EJ, 7.4.3]}$)$.
\end{prop}
\begin{proof} \ (1) \ Let $k$ be a field. In Theorem \ref{cofibrantlygenHtriple},
taking \ $\Lambda = \left(\begin{smallmatrix} k & k \\
	0 & k\end{smallmatrix}\right)$ and
 \ $\mathcal U' = k\mbox{-}{\rm Mod} = \mathcal X = \mathcal W_1 = \mathcal V' = \mathcal Y = \mathcal W_2$, then all the conditions are satisfied. To see that
$({\rm T}_A(\mathcal U')\oplus {\rm T}_B(\mathcal V'), \ \left(\begin{smallmatrix} \mathcal X \\ \mathcal Y\end{smallmatrix}\right), \ \left(\begin{smallmatrix} \mathcal W_1 \\ \mathcal W_2\end{smallmatrix}\right))$
\ and
$(\left(\begin{smallmatrix} \mathcal U' \\ \mathcal V'\end{smallmatrix}\right),  \ {\rm H}_A(\mathcal X)\oplus {\rm H}_B(\mathcal Y), \ \left(\begin{smallmatrix} \mathcal W_1 \\ \mathcal W_2\end{smallmatrix}\right))$
are different Hovey triples, it suffices to see that cotorsion pairs
$(_\Lambda\mathcal P, \ \Lambda\mbox{-}{\rm Mod})$ and $(\Lambda\mbox{-}{\rm Mod}, \ _\Lambda\mathcal I)$ are different. This is clear since $_\Lambda\mathcal P \subsetneqq \Lambda\mbox{-}{\rm Mod}$.

\vskip5pt

To show (2), (3), (4), (5), by the definition of generally different Hovey triples, it suffices to prove (5), since the Hovey triples in
Corollaries \ref{frobB} and \ref{frobA} are respectively the special cases of the Hovey triples in {\rm Theorems \ref{Htriple1} and \ref{Htriple2}}
(or, in {\rm Theorems \ref{GHtriplegpgiB} and \ref{GHtriplegpgiA}}; or, in {\rm Corollaries \ref{projinjtripleB} and \ref{projinjtripleA}}).
While for the four kinds of Hovey triples in Corollaries \ref{frobB} and \ref{frobA}, one can easily see that they are pairwise generally different.

\vskip5pt

(6) \ It suffices to show that the four Hovey triples in  Corollaries \ref{frobB} and \ref{frobA} are generally different from the six Hovey triples listed above.
Then, all together there are 24 cases, and all these 24 cases are easy, except the following cases.

\vskip5pt

To see Hovey triple $({\rm T}_A(_A\mathcal P)\oplus {\rm T}_B(B\text{\rm\rm-Mod}), \ \ \Lambda\text{\rm\rm-Mod}, \ \ \left(\begin{smallmatrix} A\text{\rm\rm-Mod} \\ _B\mathcal I\end{smallmatrix}\right))$
in Corollary \ref{frobB}(1) is generally different from $({\rm GP}(\Lambda), \  \Lambda\mbox{-}{\rm Mod}, \ _\Lambda\mathcal P^{<\infty})$ (if $\Lambda$ is Gorenstein), we take $\Lambda$ to be the Morita rings as in Theorem \ref{ctp4}.
Then $_\Lambda \mathcal P^{<\infty} = \binom{_A\mathcal I}{_B\mathcal I}\ne \left(\begin{smallmatrix} A\text{\rm\rm-Mod} \\ _B\mathcal I\end{smallmatrix}\right)$ if $A$ is not semisimple.

\vskip5pt

To see the Hovey triple $(\Lambda\text{\rm\rm-Mod}, \ \ {\rm H}_A(_A\mathcal I)\oplus {\rm H}_B(B\text{\rm\rm-Mod}), \ \ \left(\begin{smallmatrix}A\text{\rm\rm-Mod}\\ _B\mathcal P\end{smallmatrix}\right))$
in Corollary \ref{frobB}(2) is generally different from  $(\Lambda\mbox{-}{\rm Mod},  \ {\rm GI}(\Lambda), \ _\Lambda\mathcal P^{<\infty})$ \ (if $\Lambda$ is Gorenstein),
we take $\Lambda$ to be the Morita rings as in Theorem \ref{ctp4}.
Then $_\Lambda \mathcal P^{<\infty} = \binom{_A\mathcal P}{_B\mathcal P}\ne \left(\begin{smallmatrix} A\text{\rm\rm-Mod} \\ _B\mathcal P\end{smallmatrix}\right)$ if $A$ is not semisimple.

\vskip5pt

To see the Hovey triple $({\rm T}_A(A\text{\rm\rm-Mod})\oplus {\rm T}_B(_B\mathcal P), \ \ \Lambda\text{\rm\rm-Mod}, \ \ \left(\begin{smallmatrix}_A\mathcal I \\ B\text{\rm\rm-Mod}\end{smallmatrix}\right))$
in Corollary \ref{frobA}(1) is generally different from $({\rm GP}(\Lambda), \  \Lambda\mbox{-}{\rm Mod}, \ _\Lambda\mathcal P^{<\infty})$ (if $\Lambda$ is Gorenstein), we take $\Lambda$ to be the Morita rings as in Theorem \ref{ctp4}.
Then $_\Lambda \mathcal P^{<\infty} = \binom{_A\mathcal I}{_B\mathcal I}\ne \left(\begin{smallmatrix} _A\mathcal I \\ B\text{\rm\rm-Mod}\end{smallmatrix}\right)$ if $B$ is not semisimple.

\vskip5pt

To see the Hovey triple $(\Lambda\text{\rm\rm-Mod}, \ \ {\rm H}_A(A\text{\rm\rm-Mod})\oplus {\rm H}_B(_B\mathcal I), \ \ \left(\begin{smallmatrix}_A\mathcal P \\ B\text{\rm\rm-Mod}\end{smallmatrix}\right))$
in Corollary \ref{frobA}(2) is generally different from $(\Lambda\mbox{-}{\rm Mod},  \ {\rm GI}(\Lambda), \ _\Lambda\mathcal P^{<\infty})$ \ (if $\Lambda$ is Gorenstein), we take $\Lambda$ to be the Morita rings as in Theorem \ref{ctp4}.
Then $_\Lambda \mathcal P^{<\infty} = \binom{_A\mathcal P}{_B\mathcal P}\ne \left(\begin{smallmatrix} _A\mathcal P \\ B\text{\rm\rm-Mod}\end{smallmatrix}\right)$ if $B$ is not semisimple.
\end{proof}

\vskip20pt

{\bf Acknowledgement}: We thank the anonymous referee for helpful comments and suggestions.

\end{document}